\documentclass{amsart}

\usepackage{graphicx}
\usepackage{amssymb, comment, color}

\title[Systems of transversal sections near critical energy levels]{Systems of transversal sections near critical energy levels of Hamiltonian systems in $\R^4$}
\author{Naiara V. de Paulo}

\address[Naiara V. de Paulo]{Universidade de S\~ao Paulo,  Instituto de Matem\'atica e Estat\'istica -- Departamento de Matem\'atica, Rua do Mat\~ao, 1010 - Cidade Universit\'aria - S\~ao Paulo SP, Brazil 05508-090.}
\email{naiaravp@ime.usp.br}

\author{Pedro A. S. Salom\~ao}
\address[Pedro A. S. Salom\~ao]{Universidade de S\~ao Paulo,  Instituto de Matem\'atica e Estat\'istica -- Departamento de Matem\'atica, Rua do Mat\~ao, 1010 - Cidade Universit\'aria - S\~ao Paulo SP, Brazil 05508-090.}
\email{psalomao@ime.usp.br}

\newcommand{\tr}{\text{tr}}
\newcommand{\C}{\mathbb{C}}
\newcommand{\D}{\mathbb{D}}
\newcommand{\R}{\mathbb{R}}
\newcommand{\Z}{\mathbb{Z}}

\newcommand{\N}{\mathbb{N}}
\newcommand{\M}{\mathcal{M}}
\newcommand{\LL}{\mathcal{L}}

\newcommand{\J}{\mathcal{J}}

\newcommand{\T}{\mathcal{T}}
\renewcommand{\P}{\mathcal{P}}
\newcommand{\F}{\mathcal{F}}
\newcommand{\V}{\mathcal{V}}
\newcommand{\W}{\mathcal{W}}

\renewcommand{\sl}{\text{sl}}

\newcommand{\wind}{\text{wind}}

\newcommand{\U}{\mathcal U}

\newcounter{newcounter}[section]

\numberwithin{equation}{section}
\numberwithin{newcounter}{section}
\numberwithin{figure}{section}
\numberwithin{footnote}{section}

\newtheorem{corollary}[newcounter]{Corollary}
\newtheorem{definition}[newcounter]{Definition}

\newtheorem{lem}[newcounter]{Lemma}
\newtheorem{prop}[newcounter]{Proposition}
\newtheorem{remark}[newcounter]{Remark}
\newtheorem{theo}[newcounter]{Theorem}

\begin{document}

\begin{abstract}
In this article we study Hamiltonian flows associated to smooth functions $H:\R^4 \to \R$ restricted to energy levels close to critical levels. We assume the existence of a saddle-center equilibrium point $p_c$ in the zero energy level $H^{-1}(0)$. The Hamiltonian function near $p_c$ is assumed to satisfy Moser's normal form and $p_c$ is assumed to lie in a strictly convex singular subset $S_0$ of $H^{-1}(0)$. Then for all $E>0$ small, the energy level $H^{-1}(E)$ contains a subset $S_E$ near $S_0$, diffeomorphic to the closed $3$-ball, which admits a system of transversal sections $\F_E$, called a $2-3$ foliation. $\F_E$ is a singular foliation of $S_E$ and contains two periodic orbits $P_{2,E}\subset \partial S_E$ and $P_{3,E}\subset  S_E\setminus \partial S_E$ as binding orbits. $P_{2,E}$ is the Lyapunoff  orbit lying in the center manifold of $p_c$, has Conley-Zehnder index $2$ and spans two rigid planes in $\partial S_E$. $P_{3,E}$  has Conley-Zehnder index $3$ and spans a one parameter family of planes in $S_E \setminus \partial S_E$. A rigid cylinder connecting $P_{3,E}$ to $P_{2,E}$ completes $\F_E$. All regular leaves are transverse to the Hamiltonian vector field. The existence of a homoclinic orbit to $P_{2,E}$ in $S_E\setminus \partial S_E$ follows from this foliation.
\end{abstract}

\date{Received 22 Nov 2013}
\subjclass[2010]{(Primary) 53D35 (Secondary) 37J55 37J45}
\keywords{Hamiltonian dynamics, systems of transversal sections, pseudo-holomorphic curves, periodic orbits, homoclinic orbits}

\maketitle

\tableofcontents

\section{Introduction}\label{sec_introduction}

Pseudo-holomorphic curves have been used as a powerful tool to study global properties of Hamiltonian dynamics. In his pioneering work \cite{93}, H. Hofer showed that one can study Hamiltonian flows restricted to a contact type energy level by considering pseudo-holomorphic curves in its symplectization. For that, he introduced the notion of finite energy pseudo-holomorphic curve  and proved many cases of Weinstein conjecture in dimension 3. Later, in  a series of papers \cite{char1,char2,props1,props2,props3}, Hofer, Wysocki and Zehnder developed the foundations of the theory of pseudo-holomorphic curves in symplectizations. Deep dynamical results followed from this theory such as the construction of disk-like global surfaces of section for strictly convex energy levels in $\R^4$, see \cite{convex}, and the  existence of global systems of transversal sections for generic star-shaped energy levels in $\R^4$, see \cite{fols}. 

In this paper we use the theory of pseudo-holomorphic curves in symplectizations in order to investigate the existence of systems of transversal sections near critical energy levels of Hamiltonian systems in $\R^4$.

Remounting the works of Poincar\'e and Birkhoff in the restricted three body problem,  the use of global surfaces of section revealed to be an important tool to study dynamics of two degree of freedom Hamiltonian systems restricted to an energy level. Starting with the classical Poincar\'e-Birkhoff fixed point theorem, this type of two-dimensional reduction motivated a vast development of conservative surface dynamics and, in particular, led to various existence results of periodic orbits. The method of global surfaces of sections, however, has its own limitations since either they may not exist or their existence may be hard to  be proven. One recent way to overcome this difficulty is to look after the so called global systems of transversal sections introduced in \cite{fols}. These are singular foliations of the energy level where the singular set is formed by finitely many periodic orbits, called bindings, and the regular leaves are punctured surfaces transverse to the Hamiltonian vector field and asymptotic to the bindings at the punctures. It naturally generalizes the notion of global surfaces of section where, in this particular case, the regular leaves lie in an $S^1$-family of punctured surfaces and a first return map to every regular leaf is available.  In some situations, a global system of transversal sections  gives valuable information about dynamics since, as in the case of a global surface of section, it may force the existence of periodic orbits, homoclinics to a hyperbolic orbit etc. As mentioned above, results in \cite{fols} imply that the Hamiltonian flow on a generic star-shaped energy level in $\R^4$ admits a global system of transversal sections where the regular leaves are punctured spheres.

Here we assume that a smooth autonomous Hamiltonian system in $\R^4$  admits a saddle-center equilibrium point at the zero energy level. The linearized Hamiltonian vector field at such an equilibrium has a pair of real eigenvalues and a pair of purely imaginary eigenvalues. The Hamiltonian function near the saddle-center is assumed to have a nice normal form which will be referred  as Moser's normal form. For energies slightly below zero, the energy level near the saddle-center contains two connected components and, as the energy goes above zero, a $1$-handle neck parametrized by $[-1,1] \times S^2$ is attached to the energy level as the connected sum of the two previous components. The special separating $2$-sphere $\{0\} \times S^2$ contains the Lyapunoff periodic orbit as its equator, one for each positive value of energy, forming the center manifold of the saddle-center associated to the purely imaginary eigenvalues. The separating $2$-sphere has  two open hemispheres separated by the Lyapunoff orbit which correspond to transit trajectories that cross the neck from one side of the energy level to the other. Local trajectories which do not cross the separating $2$-sphere and escape from the neck for positive and negative times are called non-transit. The Lyapunoff orbit is hyperbolic within its energy level and its stable and unstable manifolds are invariant cylinders with local branches in both sides of the separating $2$-sphere. The intersection of these cylinders with the $2$-spheres $\{s\} \times S^2, s\neq 0$, are two simple closed curves separating the transit from the non-transit trajectories.

It is a classical problem to study the existence of Hamiltonian solutions which are doubly asymptotic to the Lyapunoff orbits, the so called homoclinics.  The existence of homoclinics depends on global assumptions of the Hamiltonian function which may force intersections of their stable and unstable manifolds. We refer to Conley's papers \cite{con1,con2,con3} (see also McGehee's thesis \cite{McG}) as an attempt to study homoclinics in the circular planar restricted three body problem for energies above the first Lagrange value. Using Moser's normal form, Conley considers relative surfaces of section and their transition maps representing transit, non-transit and global trajectories. The notion of global systems of transversal sections was not yet known and thus not considered.

As mentioned above, the existence of homoclinics to the Lyapunoff orbits depends on global assumptions. In this paper, we assume that the saddle-center  lies in  a strictly convex subset of the zero energy level so that this subset is homeomorphic to the $3$-sphere and  has the saddle-center as its unique singularity. It follows that for energies slightly above zero, the energy level contains a subset diffeomorphic to the $3$-ball and close to the singular subset, so that its boundary coincides with the separating $2$-sphere. Our main result states that for all small positive values of energy, this $3$-ball admits a system of transversal sections, called  $2-3$ foliation, and the Lyapunoff orbit is one of the two existing binding orbits. Both hemispheres of the separating $2$-sphere are regular leaves of this foliation. The other binding orbit  lies in the interior of the $3$-ball and spans a one parameter family of disk-like regular leaves. Such family of disks `breaks' into a cylinder connecting both bindings and into the hemispheres of the separating $2$-sphere. The organization of the flow induced by the $2-3$ foliation allows us to conclude that the Lyapunoff orbit admits at least one homoclinic in the interior of the $3$-ball.

The proof of the main result relies on the existence of finite energy foliations in the symplectization of the $3$-sphere equipped with a nondegenerate tight contact form. For small positive energies, we take a copy of the $3$-ball close to the strictly convex subset and patch them together along their boundaries to obtain a smooth $3$-sphere. Using Moser's coordinates near the saddle-center and the global convexity assumption, we show that the Hamiltonian flow on this $3$-sphere is the Reeb flow of a tight contact form.  Moreover, except for the Lyapunoff orbit, all other periodic orbits have Conley-Zehnder index at least $3$ and those periodic orbits which cross the separating $2$-sphere must have Conley-Zehnder index $>3$. Under a slight perturbation of the contact form, we may assume that the Reeb flow is nondegenerate. Hence, with the aid of suitable almost complex structures, we apply results of \cite{fols} in order to obtain a finite energy foliation in the symplectization of the $3$-sphere. This foliation projects down to a global system of transversal sections for the perturbed flow. Using the estimates of the Conley-Zehnder indices mentioned above, we conclude that it must be a $3-2-3$ foliation, i.e., in both sides of the separating $2$-sphere it restricts to a $2-3$ foliation. The last and hardest step in the proof consists of taking the limit of such finite energy foliations to the unperturbed Hamiltonian dynamics in order to obtain the desired singular foliation.

There are many classical Hamiltonian systems in $\R^4$ which fit our hypotheses. If the Hamiltonian function is of the form kinetic plus potential energy, then a saddle critical point of the potential function corresponds to a saddle-center equilibrium point. The global convexity assumption can be directly checked on the corresponding Hill's region. In fact, it depends on the positivity of an expression involving the derivatives of the potential function up to second order \cite{Sa1}. In some cases this computation is relatively simple. Here we show two examples in this class of Hamiltonians  satisfying our hypotheses, one of them being a perturbation of the well known H\'enon-Heiles Hamiltonian.

In the following subsections \S \ref{sec_saddle_center} and \S\ref{sec_strictly_convex} we establish the local and global hypotheses of our Hamiltonian system, respectively. After defining $2-3$ foliations in  \S\ref{sec_syst_transversal}, we  state our main theorem in \S\ref{sec_main_theorem}. Applications of the main theorem to Hamiltonian systems of the form kinetic plus potential energy are given in \S\ref{sec_applications}. In Section \ref{sec_proof_main}, we split the proof of the main theorem in many steps which are essentially stated in Propositions \ref{prop_step1}, \ref{prop_step2}, \ref{prop_step3}, \ref{prop_step4} and \ref{prop_step5}-i)-ii)-iii). These propositions are proved  in Sections \ref{sec_prop_step1} to \ref{sec_5-iii)}, respectively. There are three appendices, one containing the basics of the theory of pseudo-holomorphic curves in symplectizations, another containing a linking property of a homoclinic  to the saddle-center and, finally, the last appendix containing uniqueness and intersection properties of pseudo-holomorphic curves.

\subsection{Saddle-center equilibrium points}\label{sec_saddle_center}

Let $H:\R^4 \to \R$ be a smooth function. Consider the coordinates $(x_1,x_2,y_1,y_2)\in \R^4$ and let $\omega_0$ be the standard symplectic form $$\omega_0=\sum_{i=1}^2 dy_i \wedge dx_i.$$ The Hamiltonian vector field $X_H$ is determined by $i_{X_H} \omega_0 = -dH$ and is given by $$X_H = \sum_{i=1}^2 \partial_{y_i}H \partial_{x_i} - \partial_{x_i}H \partial_{y_i}.$$
We denote by $\{\psi_t, t\in \R\}$ the flow of $X_H$, assumed to be complete. Clearly, $\psi_t$ preserves the energy levels of $H$ since $dH \cdot X_H = -\omega_0(X_H,X_H)=0$.

\begin{definition}An equilibrium point of
the Hamiltonian flow is a point $p_c\in \R^4$ so that $$ dH(p_c)=0  \Leftrightarrow X_H(p_c)=0\Leftrightarrow \psi_t(p_c)=p_c,\forall t.$$ We say that the equilibrium point $p_c$ is a saddle-center if the matrix representing $DX_H(p_c)$  has a pair of real eigenvalues $\alpha$ and  $-\alpha$, and a pair of purely imaginary eigenvalues $\omega i$ and
$-\omega i$, where $\alpha, \omega
>0$.\end{definition}

Saddle-center equilibrium points appear quite often in literature. As an example, let $$H(x,y)=\frac{\left<y,y \right>}{2} + U_0(x),$$ where $\left< \cdot, \cdot\right>$ is a positive definite inner product on $\R^2$, $y=(y_1,y_2)\in \R^2,$ and $U_0$ is a smooth potential function in $x=(x_1,x_2)\in \R^2$. Assume $U_0$  admits a saddle-type critical point at $x_c=(x_{1c},x_{2c})\in \R^2$. Then $p_c=(x_{1c},x_{2c},0,0)\in \R^4$ is a saddle-center equilibrium point for $H$.

Throughout this article we assume that $H$ admits a saddle-center equilibrium point $p_c\in \R^4$ so that near $p_c$ we have the following  canonical coordinates which put the Hamiltonian $H$ in a very special normal form, possibly after adding a constant to $H$ and changing its sign.

\vspace{0.3cm}
\noindent
{\bf Hypothesis 1 (local coordinates).} {\it There exist neighborhoods $V,U\subset \R^4$ of $0$, $p_c,$ respectively, and a smooth symplectic diffeomorphism
$\varphi:(V,0) \to (U,p_c)$  so that, denoting
$(x_1,x_2,y_1,y_2)=\varphi(q_1,q_2,p_1,p_2)$, the symplectic form on $V$ is given by $\sum_{i=1,2} dp_i \wedge dq_i=\varphi^*\left(\sum_{i=1,2} dy_i \wedge dx_i \right)$ and the Hamiltonian $K := \varphi^*
H$   is given by
\begin{equation}\label{Kloc}K(q_1,q_2,p_1,p_2) = \bar K(I_1,I_2)= -\alpha I_1 + \omega I_2 +
R(I_1,I_2),
\end{equation}
where \begin{equation}\label{I1I2} I_1 = q_1p_1, I_2 = \frac{q_2^2+ p_2^2}{2} \mbox{ and } R(I_1,I_2)= O(I_1^2 + I_2^2).\end{equation} Here $\bar K$ is a smooth function on $I_1$ and $I_2$.}

\begin{remark}Throughout the paper we use the convention $$g(x)=O(f(x))\Leftrightarrow \exists k>0 \mbox{ {\rm s.t. }} |g(x)|\leq k|f(x)|, \forall |x| \mbox{ {\rm small}}.$$ \end{remark}

\begin{remark}\label{remreal} Moser's theorem \cite{moser} and a refinement of Russmann \cite{Russ} (see also Delatte's paper \cite{del}) show that if the Hamiltonian function is real-analytic near the saddle-center then one can find local real-analytic symplectic coordinates so that the Hamiltonian function takes the form \eqref{Kloc}. Hence Hypothesis 1 is automatically satisfied if the Hamiltonian function is real-analytic near the saddle-center. This normal form of the Hamiltonian function will be called Moser's normal form. \end{remark}

The flow of $\dot z = X_K \circ z$, $z=(q_1,q_2,p_1,p_2)\in V,$ is integrable with first integrals $I_1$ and $I_2$. Its solutions satisfy
\begin{equation}\label{SolHam}
\left\{\begin{aligned} & q_1(t)  =q_1(0)e^{ -\bar \alpha t},\\ & p_1(t) =p_1(0)e^{ \bar \alpha t},\\ & q_2(t) + ip_2(t) = (q_2(0) + ip_2(0))e^{-i \bar \omega t}. \end{aligned} \right.
\end{equation}
where $\bar \alpha = -\partial_{I_1}\bar K = \alpha -
\partial_{I_1} R$ and $\bar \omega= \partial_{I_2}\bar K = \omega +
\partial_{I_2} R$ are constant along the trajectories. We then observe a saddle and a center behavior of the flow in the planes $(q_1,p_1)$ and $(q_2,p_2)$, respectively. See Figure \ref{fig_selaecentro}.

\begin{figure}[ht!!]
  \centering
  \includegraphics[width=0.7\textwidth]{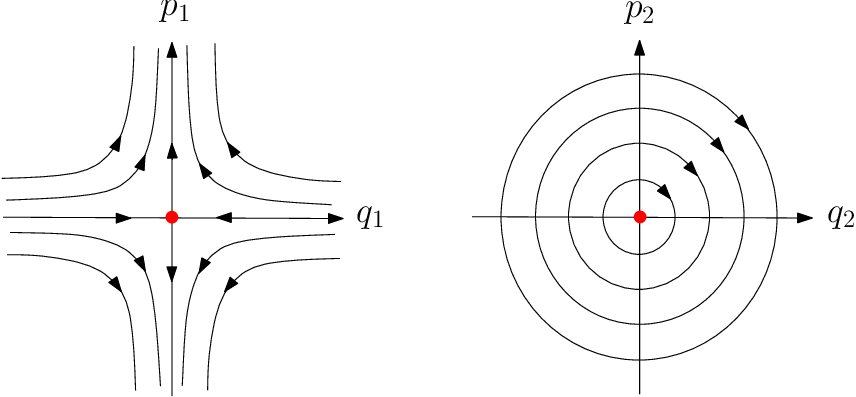}
  \caption{Local behavior of the flow near a saddle-center equilibrium point projected into the planes $(q_1,p_1)$ and $(q_2,p_2)$.}
  \label{fig_selaecentro}
\hfill
\end{figure}

In Figure \ref{fig_projecao}, we have a local description of the energy levels $H^{-1}(E)$, $|E|$ small, projected into the plane $(q_1,p_1)$, by means of the local Hamiltonian function $K$. We see that for $E<0$, $K^{-1}(E)$ contains  two components projected into the first and third quadrants. For $E=0$, these components touch each other at the origin and for $E>0$, $K^{-1}(E)$ contains only one component corresponding topologically to a connected sum of the two components in $E<0$.

\begin{figure}[ht!!]
  \centering
  \includegraphics[width=0.85\textwidth]{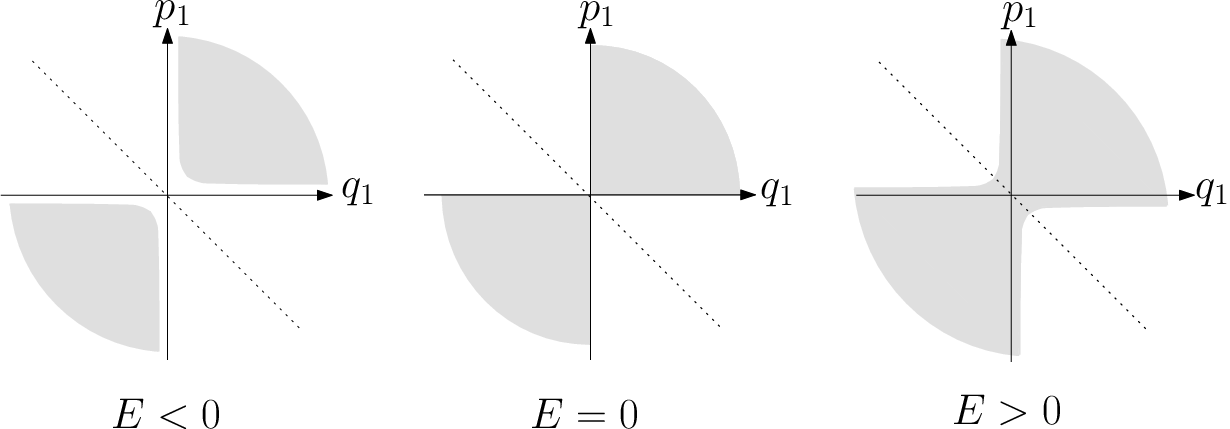}
  \caption{The projections of $K^{-1}(E)$ into the plane $(q_1,p_1)$ for $E<0$, $E=0$ and $E>0$.  }
  \label{fig_projecao}
\hfill
\end{figure}

\subsection{Strictly convex subsets of the critical energy level}\label{sec_strictly_convex} Let us start with the following lemma which will be useful in our construction.

\begin{lem}\label{2-sphere}For all $(E,\delta)\neq (0,0),$ with $E,\delta \geq 0$  sufficiently small, the set $$N^\delta_E:= \{q_1+p_1=\delta\}\cap K^{-1}(E) $$  is an embedded $2$-sphere in $K^{-1}(E).$
\end{lem}
\begin{proof}Using $q_1+p_1=\delta$ and $K=E$ in \eqref{Kloc}, we obtain $$\alpha\left(q_1 - \frac{\delta}{2}\right)^2 + \omega \frac{q_2^2+p_2^2}{2} + O\left(q_1^2(q_1-\delta)^2+ \left(q_2^2+p_2^2\right)^2\right) = E+ \frac{\alpha \delta^2}{4}.$$ Let $k=\sqrt{E+\frac{\alpha \delta^2}{4}}>0$. Defining the rescaled variables $q_1 = \frac{k}{\sqrt{\alpha}}\bar q_1+\frac{\delta}{2},$ $q_2 = \frac{k\sqrt{2}}{\sqrt{\omega}} \bar q_2,p_2 = \frac{k\sqrt{2}}{\sqrt{\omega}}  \bar p_2$, with $|(\bar q_1,\bar q_2,\bar p_2)|\leq 2$, we get $$\bar q_1^2 + \bar q_2^2 + \bar p_2^2+O\left( k^4 \right) = 1.$$ This is an embedded $2$-sphere in coordinates $(\bar q_1,\bar q_2,\bar p_2)$ for all $E,\delta \geq 0$   sufficiently small, with $(E,\delta)\neq (0,0)$. 
\end{proof}

A more detailed description of $N^\delta_E$ is the following. Since $\partial_{I_2}\bar K(0,0)=\omega \neq 0,$ one can write using the implicit function theorem \begin{equation}\label{eqI2}I_2 = I_2(I_1,E) = \frac{\alpha}{\omega} I_1 + \frac{E}{\omega} + O\left(I_1^2+E^2\right),\end{equation} for $|I_1|,|E|$ small. For a fixed $E\geq 0$ small, there exists $I_1^-(E)\leq 0$ so that $I_2 \geq 0 \Rightarrow I_1 \geq I_1^-(E)$. Fixing also $\delta \geq 0$ small, with $(E,\delta) \neq (0,0)$, we see that $\{q_1 + p_1 =\delta\}\cap K^{-1}(E)$ projects onto a bounded segment $I_{\delta,E}$ in the plane $(q_1,p_1)$. Its mid-point corresponds in $K^{-1}(E)$ to the circle $I_1 = I_1^+(\delta):=\frac{\delta^2}{4}$ and $I_2^+(\delta,E):=I_2(I_1^+(\delta),E)$. Each internal point $(q_1,p_1)$ of $I_{\delta,E}$ corresponds in $K^{-1}(E)$ to a circle $I_1 = q_1p_1$ and $I_2 = I_2(I_1,E)\leq I_2^+(\delta,E)$. These circles collapse at the extreme points of $I_{\delta,E}$, which correspond in $K^{-1}(E)$ to points $I_1 = I_1^-(E)$ and $I_2=0$, forming the embedded $2$-sphere $N_E^\delta$. See Figure \ref{fig_esferinha} below.

The embedded $2$-sphere $N^\delta_0\subset K^{-1}(0), \delta>0$ small, as in Lemma \ref{2-sphere}, is the boundary of a topological closed $3$-ball inside $K^{-1}(0)$ given by $$B_0^\delta:=\{0\leq q_1+p_1 \leq \delta\} \cap K^{-1}(0).$$ Note that the $3$-ball $B_0^\delta$ contains the saddle-center equilibrium point in its center.

Fixing $\delta>0$ small, we  assume that $\varphi(N^\delta_0)$ is also the boundary of an embedded closed $3$-ball $B_\delta \subset H^{-1}(0)$ containing only regular points of $H$ and so that in local coordinates  it projects in $\{q_1+p_1 \geq \delta\}$. In this case $\varphi(B^\delta_0)\cap B_\delta = \varphi(N^\delta_0)$. Let $$S_0:=\varphi(B_0^\delta) \cup B_\delta \subset H^{-1}(0).$$  It follows that $S_0$ is homeomorphic to the $3$-sphere and has $p_c$ as its unique singularity. See Figure \ref{fig_nreg}.

\begin{definition} The set $S_0$ constructed above is called a sphere-like singular subset of $H^{-1}(0)$ containing $p_c$ as its unique singularity. \end{definition}

\begin{figure}[ht!!]
  \centering
  \includegraphics[width=0.4\textwidth]{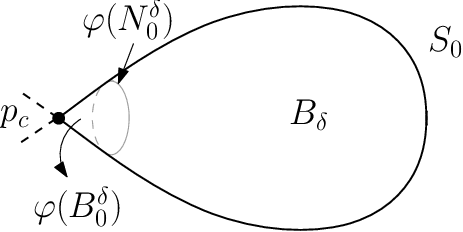}
  \caption{The sphere-like singular subset $S_0=\varphi(B_0^\delta) \cup B_\delta \subset H^{-1}(0).$ }
  \label{fig_nreg}
\hfill
\end{figure}

Let $\U \subset \R^4$ be a small neighborhood of $S_0$. It follows  that for all $E>0$ small we find an embedded closed $3$-ball  $S_E\subset H^{-1}(E)\cap \U$  so that
in  local coordinates  $S_E$ projects in $\{(q_1,p_1)\in \R^2: q_1 + p_1 \geq 0\}$, see Figures \ref{fig_se} and \ref{fig_esferinha}.
\begin{figure}[ht!!]
  \centering
  \includegraphics[width=0.42\textwidth]{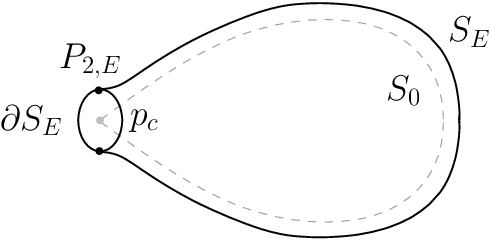}
  \caption{The embedded closed $3$-ball $S_E \subset H^{-1}(E),$ $E>0$ small.}
  \label{fig_se}
\hfill
\end{figure}
By Lemma \ref{2-sphere}, its boundary $\partial S_E\subset H^{-1}(E)$ is an embedded $2$-sphere  and contains the periodic orbit $P_{2,E} \subset H^{-1}(E)$ given in local coordinates by
$$ P_{2,E}:=\{q_1=p_1=0, I_2=I_2(0,E)\} \subset K^{-1}(E),$$ see \eqref{eqI2}. Note that $I_2(0,E) \to 0$ as $E \to 0^+$.

The one parameter family of periodic orbits $P_{2,E}, E>0$ small,  lies in the center manifold of $p_c$. We shall see later that $P_{2,E}$
is hyperbolic inside $H^{-1}(E)$ and its Conley-Zehnder index equals
$2$. $P_{2,E}$ is called the equator of $\partial S_E$. The stable and unstable manifolds of $P_{2,E}$ are locally given by $$\begin{aligned} W^s_{\rm E, loc} & := \{ p_1=0,I_2=I_2(0,E)\} \\ W^u_{\rm E, loc} & := \{ q_1=0,I_2=I_2(0,E)\},\end{aligned}$$ respectively. Both $W^s_{\rm E,loc}$ and $W^u_{\rm E,loc}$ are transverse to $\partial S_E$ inside $H^{-1}(E)$ and both have branches inside $\dot S_E:=S_E \setminus \partial S_E$.

The hemispheres of $\partial S_E$, defined in local coordinates by \begin{equation}\label{hemis} \begin{aligned} U_{1,E}:=\{q_1+p_1=0,q_1<0\}\cap K^{-1}(E),\\ U_{2,E}:=\{q_1+p_1=0,q_1>0\}\cap K^{-1}(E),\end{aligned}\end{equation}  are transverse to the Hamiltonian vector field, which points inside $S_E$ at $U_{1,E}$ and outside $S_E$ at $U_{2,E}$. See Figure \ref{fig_esferinha}.

\begin{figure}[ht!!]
  \centering
  \includegraphics[width=0.8\textwidth]{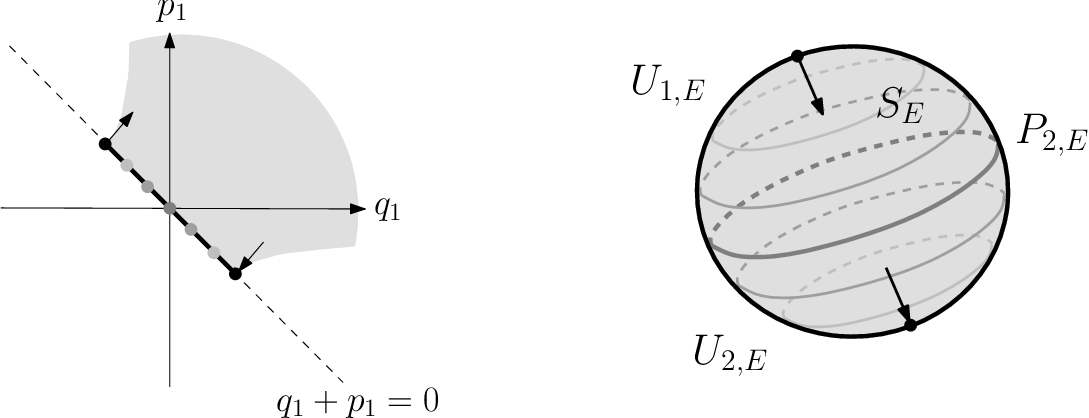}
  \caption{The embedded $2$-sphere $\partial S_E =\varphi(N_E^0)\subset H^{-1}(E),$  $E>0$, and its hemispheres $U_{1,E}$ and $U_{2,E}$. The arrows point in the direction of the Hamiltonian vector field. }
  \label{fig_esferinha}
\hfill
\end{figure}

The local stable and unstable manifolds of $p_c$ are given  by $$\begin{aligned} W^s_{\rm 0,loc}&:=\{p_1=q_2=p_2=0\} \\ W^u_{\rm 0,loc}&:=\{q_1=q_2=p_2=0\}.\end{aligned}$$ Both are one-dimensional and contain a branch inside $\dot S_0:=S_0 \setminus \{p_c\}$.

\begin{definition} We say that the sphere-like singular subset $S_0\subset H^{-1}(0)$ containing the saddle-center equilibrium point $p_c$ as its unique singularity is a strictly convex singular subset of $H^{-1}(0)$ if $S_0$ bounds a convex subset of $\R^4$ and the following local convexity assumption is satisfied
\begin{equation} \label{hessiano} H_{ww}(w)|_{T_w\dot S_0} \mbox{ is positive definite for all regular points } w\in
\dot S_0,\end{equation} where $H_{ww}$ denotes the Hessian of $H$. \end{definition}

\noindent
{\bf Hypothesis 2 (global convexity).} {\it The saddle-center equilibrium point $p_c$ lies in a strictly convex singular subset of the critical energy level $H^{-1}(0)$.}

\begin{remark}\label{remconvexity} In \cite{Sa1},  Hamiltonian functions satisfying Hypotheses 1 and 2 are presented. It is proved  that the local convexity assumption \eqref{hessiano}  implies that $S_0$ bounds a convex domain in $\R^4$, thus Hypothesis 2 could be relaxed by only requiring the local condition \eqref{hessiano}. In this case, all tangent hyperplanes $H_w:=w+T_w \dot S_0,w \in \dot S_0,$ are non-singular support hyperplanes of $\dot S_0$. This means that $H_w \cap S_0 = \{w\}$ and given any smooth curve $\gamma:(-\epsilon,\epsilon) \to \dot S_0,\epsilon>0,$ with $\gamma(0)=w$ and $0\neq \gamma'(0)\in T_w \dot S_0$, its second derivative $\gamma''(0)$ is transverse to $T_w\dot S_0$. One of the consequences of this geometric property is that given any point $w_0$ belonging to the bounded component of $\R^4 \setminus S_0$, each ray through $w_0$ intersects $S_0$ precisely at a single point and, at a regular point of $\dot S_0$, this intersection is transverse. It follows that $\dot S_0$ has contact type and the Hamiltonian flow restricted to $\dot S_0$ is equivalent to a Reeb flow. This will be discussed below. \end{remark}

\subsection{$2-3$ foliations}\label{sec_syst_transversal} We are interested in finding a particular system of transversal
sections for the Hamiltonian flow restricted to the closed $3$-ball $S_E,$ with $E>0$ small. We shall call
this system a $2-3$ foliation. The numbers correspond to the
Conley-Zehnder indices of the binding orbits.

\begin{definition}\label{feef2}Let $p_c$ be a saddle-center equilibrium point for the Hamiltonian function $H$ for which Hypothesis 1 is satisfied. Let $S_0\subset H^{-1}(0)$ be a sphere-like singular subset of $H^{-1}(0)$ containing $p_c$ as its unique singularity. Given $E>0$ small, let $S_E\subset H^{-1}(E)$ be the embedded closed $3$-ball near $S_0$ as explained above, so that in local coordinates it projects in $\{q_1 + p_1 \geq 0\}$. A $2-3$ foliation adapted to $S_E$ is a
singular foliation $\F_E$ of $S_E$ with the following properties:
\begin{itemize}
\item The singular set of $\F_E$ is formed by the hyperbolic periodic orbit $P_{2,E}\subset \partial S_E$ and by an unknotted periodic orbit $P_{3,E}\subset \dot S_E:= S_E \setminus \partial S_E$. These periodic orbits are called the binding orbits of $\F_E$ and their Conley-Zehnder indices are $CZ(P_{2,E})=2$ and $CZ(P_{3,E})=3$, see Appendix \ref{ap_basics} for definitions.
\item $\F_E$ contains the hemispheres $U_{1,E}$ and $U_{2,E}$ of $\partial S_E$ defined in local coordinates by \eqref{hemis}. Both hemispheres are called rigid planes. \item $\F_E$ contains an
embedded open cylinder $V_{E} \subset \dot S_E\setminus P_{3,E}$ and its closure has boundary $\partial V_{E} = P_{2,E} \cup P_{3,E}$. This cylinder is also called rigid. \item $\F_E$ contains a smooth one parameter family of embedded open disks $D_{\tau,E} \subset \dot S_E\setminus (V_{E} \cup P_{3,E}),\tau \in (0,1)$,  so that the closure of $D_{\tau,E}$ has boundary $\partial D_{\tau,E} = P_{3,E}, \forall \tau.$ This family foliates $\dot S_E \setminus (V_{E} \cup P_{3,E})$ and as $\tau \to 0^+$, $D_{\tau,E} \to V_{E}\cup P_{2,E} \cup U_{1,E}$ and as $\tau\to 1^-$, $D_{\tau,E} \to  V_{E}\cup P_{2,E} \cup U_{2,E}$.
\item The Hamiltonian vector field $X_H$ is transverse to all regular leaves $U_{1,E},$ $U_{2,E},$ $V_{E}$ and  $D_{\tau,E},\tau\in (0,1)$.
\end{itemize}
\end{definition}

Later on we shall give a better description of the asymptotic behavior of the leaves $D_{\tau,E} \in \F_E,\tau \in (0,1),$ and $V_{E}$ near the boundary $P_{3,E}$. We shall also describe in a more precise way the convergence of the one parameter family of disks $D_{\tau,E},\tau \in (0,1),$ to $V_{E}\cup P_{2,E} \cup U_{1,E}$ as $\tau \to 0^+$ and to $V_{E}\cup P_{2,E} \cup U_{2,E}$ as $\tau \to 1^-$.  See Figure \ref{fig2-3}.

\begin{figure}[ht!!]
  \centering
  \includegraphics[width=0.5\textwidth]{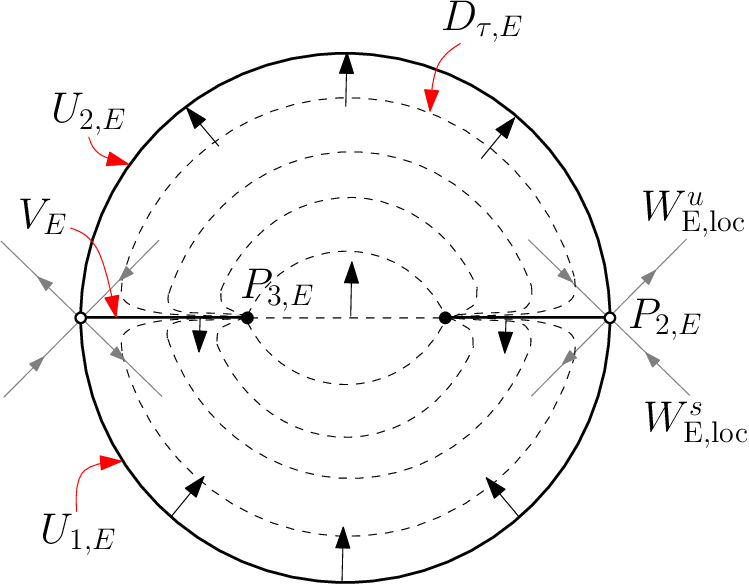}
  \caption{This figure represents a slice of a $2-3$ foliation adapted to $S_E$. The black dots represent the binding orbit $P_{3,E}$ with Conley-Zehnder index $3$ and the white dots represent the hyperbolic orbit $P_{2,E}$ with Conley-Zehnder index $2$. The disks $D_{\tau,E}, \tau \in (0,1),$ are represented by dashed curves. The rigid planes $U_{1,E}$ and $U_{2,E}$ and the rigid cylinder $V_{E}$ are represented by bold curves. The black arrows point in the direction of the Hamiltonian vector field. The local stable and unstable manifolds of $P_{2,E}$ are also represented in the figure. }
  \label{fig2-3}
\hfill
\end{figure}

\subsection{Main statement}\label{sec_main_theorem}
The main result of this article is the following theorem.

\begin{theo}\label{main1}Let $H:\R^4 \to \R$ be a smooth Hamiltonian function admitting a saddle-center equilibrium point $p_c \in H^{-1}(0)$. Assume the existence of suitable smooth canonical coordinates $(q_1,q_2,p_1,p_2)$ near $p_c$ so that $H$ has the special normal form \eqref{Kloc} as in Hypothesis 1. Assume that Hypothesis 2 is also satisfied, i.e., $p_c$ lies in a strictly convex singular subset $S_0$ of the critical level $H^{-1}(0)$ projecting in local coordinates in the first quadrant $\{q_1,p_1 \geq 0\}$. Let $S_E\subset H^{-1}(E),E>0$ small,
be the embedded closed $3$-ball near $S_0$, defined as above, and which in local coordinates projects in $\{q_1+p_1 \geq 0\}$. Then for all $E>0$ sufficiently small, there exists a $2-3$ foliation $\F_E$ adapted to $S_E$. In particular, there exists at least one homoclinic orbit to $P_{2,E}\subset \partial S_E$ contained in  $\dot S_E=S_E \setminus \partial S_E$.
\end{theo}


\begin{corollary}Let $H:\R^4 \to \R$ be a smooth Hamiltonian function admitting a saddle-center equilibrium point $p_c \in H^{-1}(0)$ so that $H$ is real-analytic near $p_c$, see Remark \ref{remreal}. Assume that $p_c$ lies in a strictly convex singular subset $S_0 \subset H^{-1}(0)$ as in Theorem \ref{main1}.  Then for all $E>0$ sufficiently small, the energy level $H^{-1}(E)$ contains a $3$-ball $S_E$ near $S_0$ which admits an adapted $2-3$ foliation. The hyperbolic binding orbit $P_{2,E}\subset \partial S_E$ has a homoclinic orbit in  $\dot S_E=S_E \setminus \partial S_E$. \end{corollary}

\begin{remark}\label{remdouble} It might happen that $H^{-1}(0)$ contains a pair of strictly convex singular subsets $S_0, S'_0$ intersecting precisely at the saddle-center $p_c$. In this case, $H^{-1}(E), E>0$ small, also contains a subset $S'_E$ near $S'_0$, diffeomorphic to the closed $3$-ball,  so that $\partial S'_E = \partial S_E$. In local coordinates $(q_1,q_2,p_1,p_2)$, $S_E$ projects in $\{q_1+p_1\geq 0\}$ and  $S'_E$ projects in $\{q_1+p_1 \leq 0\}$. Using the symmetry of the Hamiltonian function in local coordinates given by $K(-q_1,q_2,-p_1,p_2)=K(q_1,q_2,p_1,p_2)$, one can also obtain a $2-3$ foliation $\F'_E$ adapted to $S'_E$. In this way, $S'_E \cup S_E$ is diffeomorphic to the $3$-sphere and it admits the singular foliation $\F'_E \cup \F_E$,  called a $3-2-3$ foliation adapted to $S'_E \cup S_E$, which has three periodic orbits as binding orbits: $P_{3,E}\subset \dot S_E, P_{3,E}'\subset \dot S'_E:= S'_E \setminus \partial  S'_E,$ both having Conley-Zehnder index $3$, and the hyperbolic orbit $P_{2,E}\subset \partial S'_E = \partial S_E$ with Conley-Zehnder index $2$. At least two homoclinic orbits to $P_{2,E}$ are obtained, one in each subset $\dot S_E$ and $\dot  S'_E$.\end{remark}

\begin{remark}We expect that the geometric convexity assumptions of $\dot S_0$ in Theorem \ref{main1}, given by Hypothesis 2, may be replaced by the weaker assumption that $\dot S_0$ is dynamically convex, i.e., all of its periodic orbits  have Conley-Zehnder index at least $3$.  \end{remark}

\begin{remark}Recent results on the circular planar restricted three body problem \cite{AFKP} show that for energies slightly above the first Lagrange value, the energy level has contact type. This allows the use of pseudo-holomorphic curves to study Hamiltonian dynamics restricted to such energy levels. Considering the Levi-Civita regularization near one of the primaries and working on the double covering space, one has a symmetric subset of the critical energy level in $\R^4$ which is homeomorphic to the $3$-sphere and has precisely two symmetric saddle-center equilibriums as singular points. Here the symmetry is given by the antipodal map and these two equilibriums correspond to the first Lagrange point. Since we consider the case where the singular subset has only one singularity, our results do not apply to that situation. Even if one considers the dynamics on the manifold quotiented by the antipodal symmetry, where the critical subset has only one singularity, one still has a different topological setup which is not covered by results in this paper.

 \end{remark}

\subsection{Applications}\label{sec_applications} It is not hard to find Hamiltonian functions satisfying the hypotheses of Theorem \ref{main1}. For instance, the real-analytic Hamiltonian  \begin{equation}\label{eq_Ham1} H(x_1,x_2,y_1,y_2)=\frac{y_1^2+y_2^2}{2} + \frac{x_1^2+kx_2^2}{2} + \frac{1}{2}\left(x_1^2+x_2^2\right)^2,  \end{equation} admits a saddle-center equilibrium point at $p_c = 0\in \R^4$ for all $k<0$. Moreover, it is proved in \cite{Sa1} that $p_c$ lies in a pair of strictly convex singular subsets $S_0, S_0'\subset H^{-1}(0)$. The projections of $S_0$ and $S_0'$ into the plane $(x_1,x_2)$ are topological disks $D_0\subset \{x_2 \geq 0\}$ and $D_0'\subset \{x_2 \leq 0\}$, respectively, whose boundaries are strictly convex and contain a singularity at $0\in \R^2$, see Figure \ref{fig_estconv}. Thus hypotheses of Theorem \ref{main1} are satisfied and we conclude that for all $E>0$ sufficiently small, there exists a pair of subsets $S_E,S_E'\subset H^{-1}(E)$, with $\partial S_E = \partial S_E'$, so that $S_E,S_E'$ admit $2-3$ foliations $\F_E,\F_E'$ and $W_E=S_E \cup S_E'$ admits a $3-2-3$ foliation. In this case $\partial S_E$ contains the periodic orbit $P_{2,E}=\{y_1^2+x_1^2+x_1^4=2E,x_2=y_2=0\}$ and, by the symmetry of $H$, its hemispheres  can be chosen to be $U_{1,E}=\{y_1^2+y_2^2+x_1^2+x_1^4=2E,x_2=0,y_2> 0\}$ and $U_{2,E}=\{y_1^2+y_2^2+x_1^2+x_1^4=2E,x_2=0,y_2< 0\}$. The binding orbits of $\F_E \cup \F_E'$ are $P_{2,E}$, $P_{3,E}$ and $P_{3,E}'$, with $P_{3,E}$ projecting inside $\{x_2 >0\}$ and $P_{3,E}'$ projecting inside $\{x_2 <0\}$. In this simple example the Liouville vector field $Y(w)=\frac{w}{2}$ is transverse to $H^{-1}(E),E>0,$ and possibly one might do many computations explicitly. There exists at least one homoclinic to $P_{2,E}$ in $\dot S_E$ and in $\dot S_E'$.

\begin{figure}[h!!]
  \centering
  \includegraphics[width=0.75\textwidth]{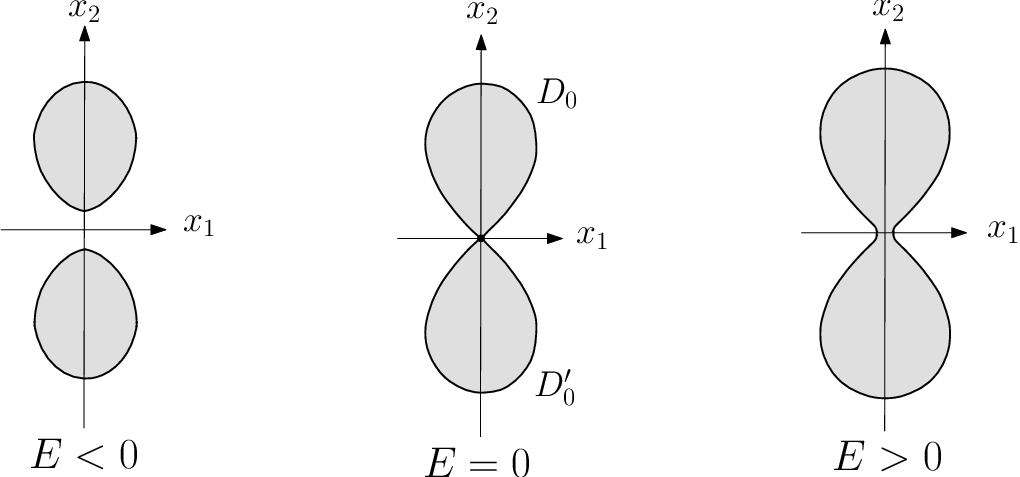}
  \caption{Hill's region of the Hamiltonian \eqref{eq_Ham1} for $E<0$, $E=0$ and $E>0$. Observe the symmetry under the involutions $(x_1,x_2)\mapsto (\pm x_1, \pm x_2)$. }
  \label{fig_estconv}
\hfill
\end{figure}

Another Hamiltonian function satisfying hypotheses of Theorem \ref{main1} is given by \begin{equation} \label{eq_Ham2}H(x_1,x_2,y_1,y_2)=\frac{y_1^2+y_2^2}{2}+ \frac{x_1^2+x_2^2}{2}+bx_1^2x_2-\frac{x_2^3}{3}, \end{equation} where $0<b<1$. The case $b=1$ corresponds to the well known H\'enon-Heiles Hamiltonian. It is proved in \cite{Sa1} that if $0<b<1$ then the saddle-center equilibrium point $p_c=(0,1,0,0)\in H^{-1}\left(\frac{1}{6}\right)$ lies in a strictly convex singular subset $S_0\subset H^{-1}\left(\frac{1}{6}\right)$. The projection of $S_0$ into the plane $(x_1,x_2)$ is a topological disk with strictly convex boundary containing a singularity at $(0,1)\in \R^2$, see Figure \ref{fig_ham2}. Thus for each $E=\frac{1}{6}+\varepsilon$, with $\varepsilon>0$ sufficiently small, $H^{-1}(E)$ contains a subset $S_E$, diffeomorphic to the closed $3$-ball and close to $S_0$, admitting a $2-3$ foliation $\F_E$. The boundary $\partial S_E$ is close to $p_c$ and contains a periodic orbit $P_{2,E}$. The existence of a homoclinic to $P_{2,E}$ in $\dot S_E$ follows from this foliation.

\begin{figure}[h!!]
  \centering
  \includegraphics[width=0.90\textwidth]{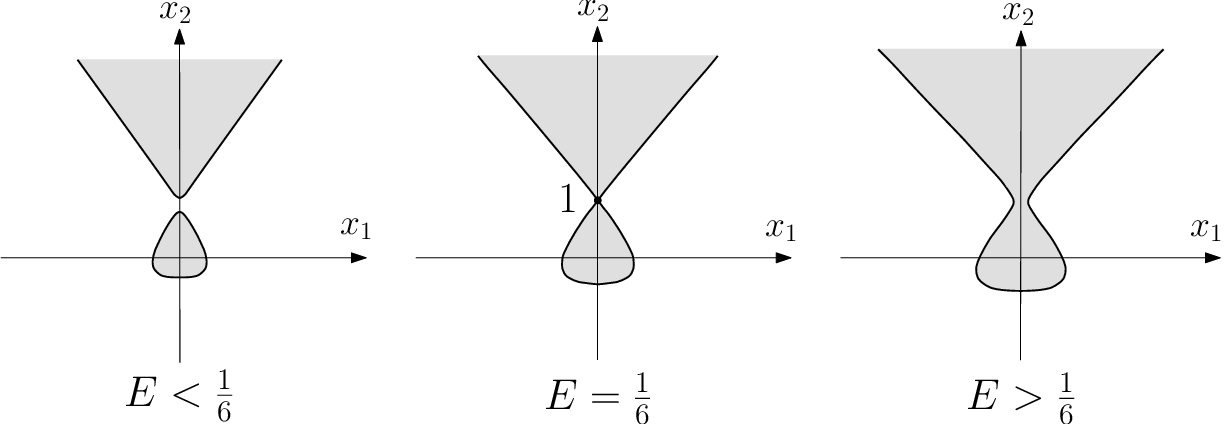}
  \caption{Hill's region of the Hamiltonian \eqref{eq_Ham2} for $E<\frac{1}{6}$, $E=\frac{1}{6}$ and $E>\frac{1}{6}$. Observe the symmetry under the involution $(x_1,x_2)\mapsto (- x_1,  x_2)$.}
  \label{fig_ham2}
\hfill
\end{figure}

\subsection{An open question}  One could try to obtain a system of transversal sections  adapted to the strictly convex  singular subset $S_0\subset H^{-1}(0)$. The expected picture is the following: there exists an unknotted periodic orbit $P_{3,0}\subset \dot S_0$, which has Conley-Zehnder index $3$ and is the binding of a singular foliation $\F_0$. $\F_0$ contains an open cylinder $V_{0}\subset \dot S_0$ so that its closure $D_\infty:= {\rm closure} (V_{0}\cup \{p_c\})\subset S_0$ is a topological embedded closed disk containing $p_c$ in its interior and whose boundary coincides with $P_{3,0}$. In the complement $S_0 \setminus  D_\infty $, $\F_0$ contains a smooth one parameter family of embedded open disks $D_{\tau,0}\subset \dot S_0,\tau \in (0,1)$, so that the closure of each $D_{\tau,0}$ has boundary $\partial D_{\tau,0} = P_{3,0}, \forall \tau,$ and converges to  $D_\infty$  as $\tau \to 0^+$ and as   $\tau \to 1^-$, through opposite sides of $D_\infty$. All regular leaves in $\dot S_0$ are transverse to the Hamiltonian vector field. In fact, this is a bifurcation picture for  foliations adapted to $S_E$, as one goes from the open book decomposition with disk-like pages for a component near $S_0$, also denoted $S_E$ for all $E<0$ small,  obtained in \cite{convex},  to the $2-3$ foliation adapted to $S_E$ for $E>0$, obtained in Theorem \ref{main1}. If such a singular foliation adapted to $S_0$ exists, then the disks $D_{\tau,0}, \tau \in (0,1),$ and the disk $V_{0} \cup \{p_c\}$ are disk-like global surfaces of section for the Hamiltonian flow restricted to $S_0$ in the following sense: every orbit in $\dot S_0 \setminus P_{3,E}$ hits one such disk infinitely many times forward in time, except the branch $\gamma^s$ of the stable manifold of $p_c$ in $\dot S_0$, which is one-dimensional. In local coordinates $(q_1,q_2,p_1,p_2)$, the local branch of $\gamma^s$ is given by $\{p_1=q_2=p_2=0,q_1 \geq 0\}$. Considering the continuation of the local branch of $\gamma^s$ to be the local branch of the unstable manifold of $p_c$ in $S_0$, given  by $\{q_1=q_2=p_2=0, p_1 \geq 0\}$, one can define a continuous area-preserving return map for all of these disks. It turns out that the return map and its positive iterates are smooth except perhaps at those points intersecting $\gamma^s$, where an infinite twist near each of such points appears after finitely many iterations. It follows that $\dot S_0$ contains either $2$ or infinitely many periodic orbits, where a homoclinic orbit to $p_c$, if it exists, is counted as a periodic orbit. In \cite{GS1}, the binding orbit $P_{3,0}$ is proved to exist and in \cite{GS2} such a foliation is proved to exist under very restrictive assumptions such as integrability or quasi-integrability of the flow. The general case, possibly assuming only dynamical convexity for $\dot S_0$, is still open.

\begin{figure}[ht!!]
  \centering
  \includegraphics[width=1\textwidth]{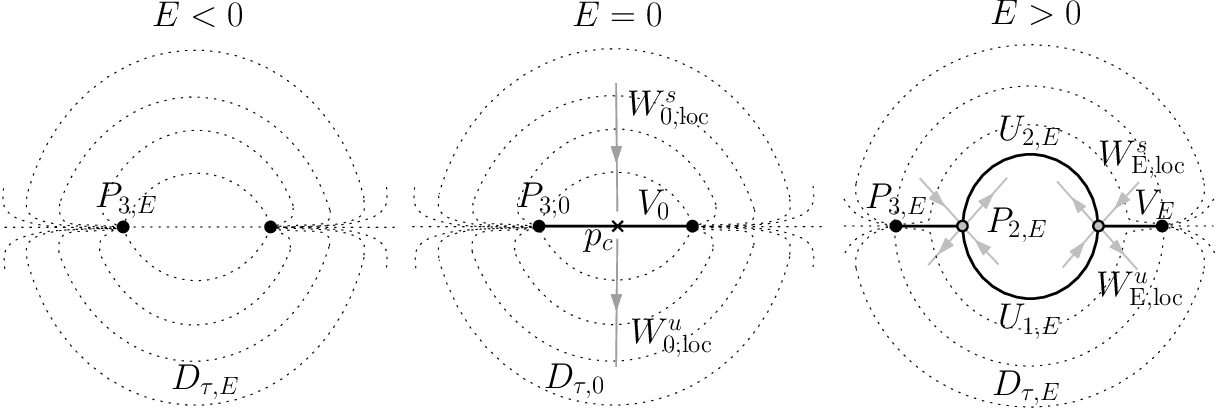}
  \caption{A bifurcation diagram of the foliation from $E<0$ to $E>0$. The expected foliation for $E=0$ contains a leaf passing through the singularity $p_c$ which is represented as a bold line. If $E<0$, this singularity is removed and if $E>0$ the singularity gives place to an embedded $2$-sphere $\partial  S_E$, formed by a pair of rigid planes and the hyperbolic orbit $P_{2,E}$ in the center manifold of $p_c$. Notice that for $E>0$ the foliation is the same as in Figure \ref{fig2-3}, but  represented outside $\partial S_E$.}
  \label{figopenquestion}
\end{figure}

\section{Proof of the main statement}\label{sec_proof_main}
The main steps of the proof of Theorem \ref{main1} are stated in Propositions \ref{prop_step1}, \ref{prop_step2}, \ref{prop_step3}, \ref{prop_step4} and \ref{prop_step5} below. In the following we sketch these steps: first we take a copy $S_E'$ of the $3$-ball $S_E \subset H^{-1}(E)$ and glue them together along their boundaries to form a smooth $3$-sphere $W_E=S_E \sqcup S_E'$ so that $\partial S_E\equiv \partial S_E'$ is the separating $2$-sphere containing the hyperbolic orbit $P_{2,E}$. For all $E>0$ small, we show that the Hamiltonian flow on $W_E$  coincides, up to a reparametrization, with the Reeb flow of a tight contact form $\lambda_E$ on $W_E$. Then we estimate the Conley-Zehnder indices of periodic orbits and prove that the closer a periodic orbit in $W_E \setminus P_{2,E}$ is to $P_{2,E}$, the higher its Conley-Zehnder index is. In particular, we show that periodic orbits crossing the separating $2$-sphere must have arbitrarily high Conley-Zehnder index, for small values of $E>0$. Moreover, we prove that all periodic orbits, except for $P_{2,E}$, have Conley-Zehnder index $\geq 3$. These properties also hold for the Reeb flow of any $\lambda_n$, $n$ large, for any sequence $\lambda_n\to \lambda_E$ in $C^\infty$ as $n\to \infty$. We may assume that the Reeb flow of $\lambda_n$ is nondegenerate and then we can apply results from \cite{fols} to obtain a global system of transversal sections for $\lambda_n$. For topological reasons, this must be a $3-2-3$ foliation. Such a foliation is the projection of a finite energy foliation $\tilde \F_n$ in the symplectization $\R \times W_E$ for suitable almost complex structures $\tilde J_n=(\lambda_n,J_n)\to \tilde J_E=(\lambda_E,J_E)$ in $C^\infty$ as $n\to \infty$, where $J_n$ and $J_E$ are compatible complex structures on the contact structures $\ker \lambda_n$ and $\ker \lambda_E$, respectively. Finally, we show that $\tilde \F_n$ converges to a finite energy foliation $\tilde \F_E$ for $\tilde J_E$ which projects down to a $3-2-3$ foliation on $W_E$. Its restriction to $S_E$ is a $2-3$ foliation. The compatible complex structure $J_E$ on $\ker \lambda_E$ is chosen so that the hemispheres $U_{1,E}$ and $U_{2,E}$ of $\partial S_E$ are regular leaves of the foliation.

We start with the construction of the smooth $3$-sphere $W_E$. Let $\varphi: (V \subset \R^4,0) \to (U \subset
\R^4,p_c)$ be the symplectomorphism given in Hypothesis 1. In local coordinates $(q_1,q_2,p_1,p_2) \in V$, $\partial S_E$ is the boundary of a topological embedded $3$-ball $B_{\leq E}:=\{q_1+p_1 =0,0\leq K \leq E\}=\bigcup_{0 \leq K\leq E} N^0_K$ containing the saddle-center at the origin in its interior, where $N^0_K$ is defined as in Lemma \ref{2-sphere}. Then $S_E \cup \varphi(B_{\leq E})$ is a topological embedded $3$-sphere in $\R^4$ which bounds  a closed topological $4$-ball $U_E$. Let $U_E'$ be a copy of $U_E$. The boundary of $U_E'$ is a topological $3$-sphere $\partial U_E' = S_E' \cup \varphi(B_{\leq E}')$, where $S_E'$ and $B_{\leq E}'$ are copies of $S_E$ and $B_{\leq E}$, respectively. In local coordinates, a point of $U_E$ is projected inside $\{q_1 + p_1 \geq 0\}$. Using the map $T:V \to V$, defined by \begin{equation}\label{eq_T}  T(q_1,q_2,p_1,p_2):=(-q_1,q_2,-p_1,p_2),\end{equation} we identify the points of $U_E'$, in local coordinates,  with points inside $\{q_1+p_1\leq 0\}$. Hence in these local coordinates, $B_{\leq E}'=B_{\leq E}$ and, therefore, $\partial  S_E'= \partial S_E$. Gluing $U_E'$ with $U_E$ by the identification $B_{\leq E}'\sim B_{\leq E}$ via the identity map, we obtain a smooth $4$-manifold $$\U_E:=U_E' \bigsqcup_{B_{\leq E}' \sim B_{\leq E}} U_E$$ with boundary $$W_E:=S_E' \bigsqcup_{\partial S_E \sim \partial S_E'} S_E$$  diffeomorphic to $S^3$. See Figure \ref{fig_construcaoWE}. There exists a natural extension of $T$ to $W_E$.

Since $T^* \omega_0=\omega_0$ and $K\circ T = K$, the symplectic form $\omega_0$ and the Hamiltonian function $H$ naturally descend to $\U_E$ both denoted by the same letters. In this way,  $p_c\in \U_E, \forall E>0$ small, and the Hamiltonian function $K$ defined in  \eqref{Kloc} is a local model for $H$ near $p_c\in \U_E$. Note that if $0<E'<E$ then we naturally have the inclusion $\U_{E'} \subset \U_E$. We denote by $S_0'\subset \U_E$ the corresponding copy of $S_0$ in $U_E'$. Note that $S_0'$ is a strictly convex singular subset of $H^{-1}(0)$ and intersects $S_0$ precisely at $p_c$. In local coordinates near $p_c$, $S_0'$ projects inside $\{q_1,p_1 \leq 0\}$. Let $$W_0:=S_0' \cup S_0.$$ $S_0$ and $S_0'$ are the boundary of closed subsets $U_0\subset U_E$ and $U_0'\subset U_E'$, respectively, both homeomorphic to the closed $4$-ball.

\begin{figure}[ht!!]
  \centering
  \includegraphics[width=.75\textwidth]{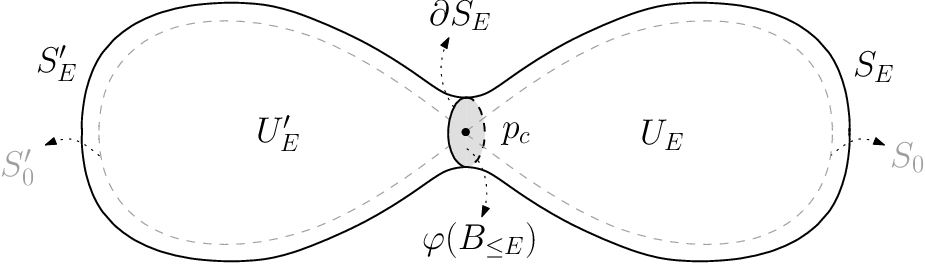}
  \caption{This figure represents the construction of the $4$-ball $\U_E=U_E' \sqcup U_E$, with boundary $W_E=S_E' \sqcup S_E$ diffeomorphic to the $3$-sphere, and $W_0=S_0' \cup S_0$.}
  \label{fig_construcaoWE}
\end{figure}

In Section \ref{sec_prop_step1} we prove the following proposition.

\begin{prop}\label{prop_step1}There exists $E^*>0$ such that if $0<E<E^*$, then there exists a Liouville vector field $Y_E$ defined in a neighborhood of $W_E$ in $\U_{E^*}$ which is transverse  to $W_E$. Moreover, $Y_E$ is invariant under the symmetry $T$ defined in local coordinates by \eqref{eq_T}, i.e., $T_* Y_E = Y_E$.  The $1$-form $i_{Y_E} \omega_0$ restricts to a tight contact form $\lambda_E$ on $W_E$ and its Reeb flow is a reparametrization of the Hamiltonian flow of $H$ restricted to $W_E$.  In local coordinates $(q_1,q_2,p_1,p_2)$, $\lambda_E$ coincides with the Liouville form $\lambda_0=\frac{1}{2}\sum_{i=1,2} p_i dq_i -q_i dp_i$ near $\partial S_E= \{q_1+p_1=0\} \cap K^{-1}(E)$.  There exists a Liouville vector field $\bar X_0$ defined in a neighborhood of $\dot W_0=\dot S_0 \cup \dot S_0':=W_0 \setminus \{p_c\}$ in $\U_{E^*}$, transverse to $\dot W_0$, so that given any small neighborhood $U_0 \subset \U_{E^*}$ of $p_c$, we have that $Y_E=\bar X_0$ outside $U_0$, for all $E>0$ sufficiently small. We denote by  $\bar \lambda_0=i_{\bar X_0} \omega_0|_{\dot W_0}$ the induced contact form on $\dot W_0$. \end{prop}

It follows from Proposition \ref{prop_step1} that the Hamiltonian flow of $H$ restricted to $W_E$ can be studied using all the tools and results coming from the theory of pseudo-holomorphic curves in symplectizations developed for Reeb flows on the tight $3$-sphere. For definitions, see Appendix \ref{ap_basics}.

Let $P\subset W_E$ be an unknotted periodic orbit of $\lambda_E$. We say that a periodic orbit $P'\subset W_E$, $P'$ geometrically distinct from $P$, is linked to $P$ if $$0\neq [P'] \in H_1(W_E \setminus P,\Z)\simeq \Z.$$ Otherwise, we say that $P'$ is not linked to $P$.

In Section \ref{sec_step2} we prove the following proposition.

\begin{prop}\label{prop_step2}If $E>0$ is sufficiently small then the following holds: let $\lambda_n$ be a sequence of contact forms on $W_E$ so that $\lambda_n \to \lambda_E$ in $C^\infty$ as $n \to \infty$, where $\lambda_E$ is given by Proposition \ref{prop_step1}. If $n$ is sufficiently large then the Reeb flow of $\lambda_n$ on $W_E$ admits only one periodic orbit $P_{2,n}$ with Conley-Zehnder index $2$, which is unknotted, hyperbolic and converges to $P_{2,E}$ as $n \to \infty$. Moreover, all other periodic orbits of $\lambda_n$ have Conley-Zehnder index $\geq 3$ and the periodic orbits with Conley-Zehnder index equal to $3$ are not linked to $P_{2,n}$.
\end{prop}

\begin{remark}Taking the constant sequence $\lambda_n = \lambda_E, \forall n,$ we see that if $E>0$ is sufficiently small then all conclusions of Proposition \ref{prop_step2} also hold for $\lambda_E$. \end{remark}

From now on we fix $E>0$ sufficiently small so that the conclusions of Propositions \ref{prop_step1} and \ref{prop_step2} hold.

Now we are in position to apply results from \cite{fols}. Motivated by Proposition \ref{prop_step2} we focus our attention on weakly convex contact forms as in the following definition.

\begin{definition}[Hofer-Wysocki-Zehnder]A tight contact form $\lambda$ on a $3$-manifold $M$ diffeomorphic to the $3$-sphere is called weakly convex if all of its periodic orbits have Conley-Zehnder index $\geq 2$.
\end{definition}

\begin{definition}Let $\lambda$ be a weakly convex tight contact form on a $3$-manifold $M$ diffeomorphic to the $3$-sphere. A global system of transversal
sections adapted to $\lambda$ is a singular foliation $\F$ of
$M$ with the following properties:
\begin{itemize}
\item The singular set of $\F$ is formed by a set $\P$ of finitely many unknotted  periodic orbits of the Reeb flow of $\lambda$, called the binding orbits of $\F$, so that all periodic orbits in $\P$ have Conley-Zehnder index either $2$ or $3$.
\item $\F \setminus \P$ consists of embedded punctured spheres foliating $M \setminus \bigcup_{P\in \P} P$, which are either planes or cylinders. The closure ${\rm cl}(F)$ of each leaf $F\in \F \setminus \P$ is an embedded sphere with either one or two open disks removed. Its boundary components are distinct elements of $\P$ and are called the asymptotic limits of $F$. One may have the following possibilities:
\begin{itemize}
\item[(i)] $F$ has only one asymptotic limit with Conley-Zehnder index $3$. These leaves appear in a one parameter family of leaves with the same asymptotic limit. This family is called a one parameter family of planes.
\item[(ii)] $F$ has one asymptotic limit with Conley-Zehnder index $3$ and one asymptotic limit with Conley-Zehnder index $2$. $F$ is called a rigid cylinder.
\item[(iii)] $F$ has only one asymptotic limit with Conley-Zehnder index $2$. $F$ is called a rigid plane.
\end{itemize}
\item The Reeb vector field $X_{\lambda}$ is transverse to each leaf in $\F \setminus \P$.
\end{itemize}
\end{definition}

As we shall see below in Theorem \ref{sfef}, the existence of a global system of transversal sections $\F$ adapted to a weakly convex tight contact form $\lambda$ on $M$, diffeomorphic to the $3$-sphere, depends on the choice of an extra structure. In fact, the contact structure $\xi=\ker \lambda$ is a trivial symplectic bundle over $M$ and the space of $d\lambda$-compatible complex structures $J:\xi \to \xi$, denoted by $\J(\lambda)$, is non-empty and contractible in the $C^\infty$-topology. For each $J\in \J(\lambda)$, the pair $(\lambda,J)$ determines a natural almost complex structure $\tilde J:TW \to TW$ in the symplectization $W:=\R \times M$ of $M$ and one can talk about pseudo-holomorphic punctured spheres $\tilde u=(a,u): (S^2 \setminus \Gamma,j) \to (W,\tilde J)$ satisfying $\tilde J(\tilde u) \circ  d\tilde u= d\tilde u \circ j$ and having finite energy $0<E(\tilde u)<\infty$. Here $(S^2,j)$ is the Riemann sphere and $\Gamma\subset S^2$ is a finite set of punctures. We may use the notation $\tilde J=(\lambda,J)$. See Appendix \ref{ap_basics} for details. A non-removable puncture $z\in \Gamma$ is either positive or negative, meaning that either $a \to +\infty$ or $a \to -\infty$ at $z$, respectively. If $\lambda$ is nondegenerate then  there exists a periodic orbit $P_z$ of the Reeb flow of $\lambda$ so that $\tilde u$ is asymptotic to $P_z$ at $z$. $P_z$ is called the asymptotic limit of $\tilde u$ at $z\in \Gamma$. Its Conley-Zehnder index is denoted by $CZ(P_z)$. We define the Conley-Zehnder index of $\tilde u$ by $$CZ(\tilde u)=\sum_{z\in \Gamma_+}CZ(P_z) - \sum_{z\in \Gamma_-} CZ(P_z),$$ where $\Gamma = \Gamma_+ \cup \Gamma_-$ is the decomposition of $\Gamma$ according to the sign of the punctures.

The global system of transversal sections $\F$ adapted to $\lambda$ is thus obtained as the projection onto $M$ of a stable finite energy  foliation $\tilde \F$ of the symplectization $W=\R \times M$ as in  Definition \ref{deffef} below. Each leaf of $\tilde \F$ is the image of an embedded finite energy pseudo-holomorphic punctured sphere associated to the pair $\tilde J=(\lambda,J)$.

\begin{definition}[Hofer-Wysocki-Zehnder]\label{deffef}Let $\lambda$ be a weakly convex tight contact form on a manifold $M$ diffeomorphic to the $3$-sphere, and $J \in \J(\lambda).$ A stable finite energy foliation for $\tilde J=(\lambda, J)$ is a smooth foliation $\tilde \F$ of $\R \times M$ with the following properties:

\begin{itemize}
\item Each leaf $\tilde F \in \tilde \F$ is the image of an embedded finite energy pseudo-holomorphic curve $\tilde u: S^2 \setminus \Gamma \to \R \times M$, where $\R \times M$ is equipped with the almost complex structure $\tilde J=(J,\lambda)$. The asymptotic limits of $\tilde F$ are defined to be the asymptotic limits of $\tilde u$ at its punctures.  We define ${\rm Fred}(\tilde F) := {\rm Fred}(\tilde u)$, where $${\rm Fred}(\tilde u):=CZ(\tilde u)-2 + \# \Gamma$$ is the Fredholm index of $\tilde u$. ${\rm Fred}(\tilde u)$ represents the local dimension of the space of unparametrized pseudo-holomorphic curves near $\tilde u$ with the same asymptotic limits. ${\rm Fred}(\tilde F)$ does not depend on the choice of $\tilde u$.
\item There exists a uniform constant $C>0$ such that $0<E(\tilde u) < C$ for all embedded finite energy curves $\tilde u:S^2 \setminus \Gamma \to \R \times M$ satisfying $\tilde u(S^2 \setminus \Gamma) = \tilde F \in \tilde \F$, where $E(\tilde u)$ denotes the energy of $\tilde u$.
\item For all $\tilde F_1,\tilde F_2\in \tilde \F$ and $c\in \R$, either $T_c(\tilde F_1) \cap \tilde F_2 = \emptyset$ or $T_c(\tilde F_1) = \tilde F_2$, where $T_c(a,z)=(a+c,z),\forall (a,z)\in \R\times M$.
\item  All the asymptotic limits of $\tilde F$ are  unknotted periodic orbits and have Conley-Zehnder index either $2$ or $3$.
\item If $\tilde u=(a,u): S^2 \setminus \Gamma \to \R \times M$ is an embedded finite energy curve such that $\tilde u(S^2 \setminus \Gamma) = \tilde F \in \tilde \F$, then $\tilde u$ has precisely one positive puncture. If $T_c(\tilde F) \cap \tilde F = \emptyset,\forall c \neq 0$, then ${\rm Fred}(\tilde u) \in \{1,2\}$ and $u$ is an embedding, transverse to the Reeb vector field. If $T_c(\tilde F) = \tilde F, \forall c,$ then ${\rm Fred}(\tilde u) = 0$ and $\tilde u$ is a cylinder over a periodic orbit.
\end{itemize}
See Appendix \ref{ap_basics} for missing definitions.
\end{definition}

It follows  that each leaf $\tilde F$ of a finite energy foliation must have one of the following types:
\begin{itemize}
\item $\tilde F$ has one positive puncture with Conley-Zehnder index $3$ and no negative punctures. In this case ${\rm Fred}(\tilde F)=2$.
\item $\tilde F$ has one positive puncture with Conley-Zehnder index $2$ and no negative punctures.  In this case ${\rm Fred}(\tilde F)=1$.
\item $\tilde F$ has one positive puncture with Conley-Zehnder index $3$ and one negative puncture with Conley-Zehnder index $2$. In this case ${\rm Fred}(\tilde F)=1$.
\item $\tilde F$ has a positive puncture and a negative puncture with the same asymptotic limit. It satisfies $T_c(\tilde F)=\tilde F,\forall c\in \R$, i.e., $\tilde F=\R \times P$, where $P\subset M$ is an unknotted periodic orbit. In this case ${\rm Fred}(\tilde F)=0$.
\end{itemize}

The existence of a stable finite energy foliation for generic almost complex structures $\tilde J=(\lambda,J)$ is given in Theorem \ref{sfef} below. Denote by $p:\R \times M \to M$ the projection onto the second factor.

\begin{theo}\cite[Theorem 1.6, Proposition 1.7 and Corollary 1.8]{fols}\label{sfef} Let $\lambda$ be a nondegenerate weakly convex tight contact form on a $3$-manifold $M$ diffeomorphic to the $3$-sphere. There exists a dense subset $\J_{\rm reg}(\lambda) \subset \J(\lambda),$ in the $C^\infty$-topology, such that if $J \in \J_{\rm reg}(\lambda)$ then the pair $\tilde J=(\lambda,J)$ admits a stable finite energy foliation $\tilde \F$. The projection $\F:= p(\tilde \F)$ is a global system of transversal sections adapted to $\lambda$. Denote by $\P$  the set of binding orbits of $\F$. We have
\begin{itemize}
\item[i)] The binding orbits of $\F$ with Conley-Zehnder index $2$ are hyperbolic and those having Conley-Zehnder index $3$ are either hyperbolic or elliptic.

\item[ii)] There exists at least one binding orbit with Conley-Zehnder index $3$ and if $P_3$ is one such binding orbit then $\F$ contains at least a one parameter family of planes asymptotic to $P_3$.

\item[iii)] If $\P =\{P\}$ contains only one periodic orbit $P$ then $P$ has Conley-Zehnder index $3$ and $\F \setminus \P$ consists of an $S^1$-family of planes asymptotic to $P$ which foliates $M \setminus P$. In this case we say that $P$ is the binding orbit of an open book decomposition with disk-like pages  adapted to $\lambda$. Moreover, each page is a global surface of section for the Reeb flow of $\lambda$.

\item[iv)] If $\P$ contains an orbit $P_2$ with Conley-Zehnder index $2$ then $\F$  contains a pair of rigid planes $U^1,U^2$ both asymptotic to $P_2$ so that $U:=U^1 \cup P_2 \cup U^2$ is a topological embedded $2$-sphere which separates $M$ in two disjoint components $M_1 \dot \cup M_2 = M \setminus U$. In each $M_i, i=1,2,$ there exists a binding orbit $P_{3,i}$ with Conley-Zehnder index $3$ and a rigid cylinder $V^i$ asymptotic to $P_{3,i}$ and $P_2$. If $M_i$ contains no binding orbit with Conley-Zehnder index $2$, then $P_{3,i}$ is the only binding orbit in $M_i$ and $\F$ contains a one parameter family of planes $D_\tau^i,\tau\in (0,1),$ all of them asymptotic to $P_{3,i}$, which foliates $M_i \setminus (P_{3,i} \cup V^i)$. $D_\tau^i$ converges to $U^1 \cup P_2 \cup V^i$ as $\tau \to 0^+$ and to $U^2 \cup P_2 \cup V^i$ as $\tau \to 1^-$. Moreover, $\lambda$ admits a homoclinic orbit to $P_2$ in $M_i$. In particular, if $\F$ contains precisely one binding orbit with Conley-Zehnder index $2$ then $\F$ contains precisely three binding orbits $P_3,P_2,P_3'$ with Conley-Zehnder indices $3$, $2$, $3$, respectively, and $\F$ is called a $3-2-3$ foliation adapted to $\lambda$.
\end{itemize}
\end{theo}

In \cite{fols}, a more general theorem is proved under much weaker assumptions. In fact, it only requires $\lambda$ to be nondegenerate and tight. In this case the global system of transversal sections may also have periodic orbits with Conley-Zehnder index $1$ as binding orbits. As mentioned before, we are only interested in the weakly convex case.

We start the construction of a global system of transversal sections adapted to $\lambda_E$ on $W_E$ by choosing a special complex structure $J_E\in \J(\lambda_E)$ as in the following proposition proved in Section \ref{sec_step3}.

\begin{prop}\label{prop_step3}If $E>0$ is sufficiently small, then there exists $J_E \in \J(\lambda_E)$ so that the almost complex structure $\tilde J_E=(\lambda_E,J_E)$ on the symplectization $\R \times W_E$ admits a pair of finite energy  planes $\tilde u_{1,E}=(a_{1,E},u_{1,E}),\tilde u_{2,E}=(a_{2,E},u_{2,E}):\C \to \R\times W_E$ and their projections $u_{1,E}(\C),u_{2,E}(\C)$ into $W_E$ are, respectively, the hemispheres $U_{1,E},U_{2,E} \subset \partial S_E$ defined in local coordinates by \eqref{hemis}, both asymptotic to the hyperbolic periodic orbit $P_{2,E}$.  \end{prop}

In order to apply Theorem \ref{sfef}, we may need to perturb  the contact form $\lambda_E$ and the complex structure $J_E$ to achieve the required generic hypotheses.  So we take a sequence of nondegenerate contact forms $\lambda_n=f_n \lambda_E$, $f_n:W_E \to (0,\infty)$ smooth $\forall n$,   and a sequence of complex structures $J_n\in \J_{\rm reg}(\lambda_n)\subset \J(\lambda_n)=\J(\lambda_E)$ satisfying $f_n \to 1$ and $J_n \to J_E$ in $C^\infty$ as $n \to \infty$. The existence of $f_n$ follows from Proposition 6.1 in \cite{convex}. If $E>0$ is sufficiently small and $n$ is sufficiently large then, by Proposition \ref{prop_step2},  $\lambda_n$ is weakly convex and contains precisely one periodic orbit with Conley-Zehnder index 2 which is not linked to any periodic orbit with Conley-Zehnder index 3. From Theorem \ref{sfef}, the pair $\tilde J_n=(\lambda_n,J_n)$ admits a stable finite energy foliation $\tilde \F_n$, which projects onto a global system of transversal sections $\F_n$ adapted to $\lambda_n$. From the description of $\F_n$ given in Theorem \ref{sfef}, we conclude that $\F_n$ is a $3-2-3$ foliation of $W_E$ adapted to $\lambda_n$. In Section \ref{sec_step4} we prove the following proposition.

\begin{prop}\label{prop_step4} If $E>0$ is sufficiently small, then the following holds:
\begin{itemize}
\item[i)] There exist a sequence of nondegenerate weakly convex contact forms $\lambda_n$ on $W_E$ satisfying $\ker \lambda_n = \ker \lambda_E, \forall n,$ $\lambda_n \to \lambda_E$ in $C^\infty$ as $n \to \infty$ and a sequence of $d\lambda_E$-compatible complex structures $J_n \in \J_{\rm reg}(\lambda_n) \subset \J(\lambda_E)$ satisfying $J_n \to J_E$ in $C^\infty$ as $n \to \infty$ so that for all $n$ sufficiently large, $\tilde J_n=(\lambda_n,J_n)$ admits a stable finite energy foliation $\tilde \F_n$ of \,$\R \times W_E$ which projects in $W_E$  onto a $3-2-3$ foliation $\F_n$ adapted to $\lambda_n$. Let $P_{3,n},P_{2,n},P_{3,n}'$ be the nondegenerate binding orbits of $\F_n$ with Conley-Zehnder indices $3,2,3,$ respectively. Then $P_{2,n} \to P_{2,E}$ as $n \to \infty$ and there exist unknotted periodic orbits $P_{3,E}\subset \dot S_E$ and $P_{3,E}'\subset \dot S_E'$ of $\lambda_E$, both with Conley-Zehnder index $3$, such that  $P_{3,n} \to P_{3,E}$ and $P_{3,n}'\to P_{3,E}'$ in $C^\infty$ as $n \to \infty$.
\item[ii)] After changing coordinates with suitable contactomorphisms $C^\infty$-close to the identity map, we can assume that $P_{3,n}=P_{3,E},P_{2,n}=P_{2,E},P_{3,n}'=P_{3,E}'$ as point sets in $W_E$ for all large $n$, and that there are sequences of constants $c_{3,n},c_{3,n}',c_{2,n} \to 1$ as $n \to \infty$ such that $$ \begin{aligned} \lambda_n|_{P_{3,E}} &= c_{3,n}\lambda_E|_{P_{3,E}}, \\ \lambda_n|_{P_{3,E}'}&=  c_{3,n}'\lambda_E|_{P_{3,E}'}, \\ \lambda_n|_{P_{2,E}}&=  c_{2,n}\lambda_E|_{P_{2,E}}.\end{aligned}$$
\end{itemize}
\end{prop}
Proposition \ref{prop_step4} provides candidates $P_{3,E},P_{2,E},P_{3,E}'$ for the  binding orbits of a $3-2-3$ foliation adapted to $\lambda_E$. For each $n$, the stable finite energy foliation $\tilde \F_n$ contains the following finite energy pseudo-holomorphic curves associated to $\tilde J_n=(\lambda_n,J_n)$:

\begin{itemize}
\item A pair of rigid planes $\tilde u_{1,n}=(a_{1,n},u_{1,n}), \tilde u_{2,n}=(a_{2,n},u_{2,n}):\C \to \R \times W_E$, both asymptotic to $P_{2,n}$. The topological embedded $2$-sphere  $u_{1,n}(\C) \cup P_{2,n} \cup u_{2,n}(\C)$ separates $W_E$ in two components denoted by $\dot S_n$ and $\dot S_n'$, which contain $P_{3,n}$ and $P_{3,n}'$, respectively.
\item A pair of rigid cylinders $\tilde v_n=(b_n,v_n), \tilde v_n'=(b_n',v_n'):\R \times S^1 \to \R \times W_E$ with $v_n(\R \times S^1)\subset \dot S_n \setminus P_{3,n}$ and $v_n'(\R \times S^1)\subset \dot S_n' \setminus P_{3,n}'$, $\tilde v_n$ asymptotic to $P_{3,n}$ and $P_{2,n}$, and $\tilde v_n'$ asymptotic to $P_{3,n}'$ and $P_{2,n}$.
\item Smooth one parameter families $\tilde w_{\tau,n}=(d_{\tau,n},w_{\tau,n}),\tilde w_{\tau,n}'=(d_{\tau,n}',w_{\tau,n}'):\C \to \R \times W_E$, $\tau \in (0,1)$, with  $\tilde w_{\tau,n}$ asymptotic to $P_{3,n}$, $D_{\tau,n}=w_{\tau,n}(\C),\tau \in (0,1),$ foliating $\dot S_n \setminus (v_n(\R\times S^1) \cup P_{3,n})$,  $\tilde w_{\tau,n}'$ asymptotic to $P_{3,n}'$  and $D_{\tau,n}'=w_{\tau,n}'(\C),\tau \in (0,1),$ foliating $\dot S_n' \setminus (v_n'(\R\times S^1) \cup P_{3,n}')$.
\end{itemize}

Our last and hardest step is to show that $\tilde \F_n$ converges as $n \to \infty$ to a stable finite energy foliation $\tilde \F_E$ associated to $\tilde J_E=(\lambda_E,J_E)$ so that the projection $\F_E = p(\tilde \F_E)$ is  a $3-2-3$ foliation of $W_E$ adapted to $\lambda_E$ containing $P_{3,E},P_{2,E}$ and $P_{3,E}'$ as binding orbits and the hemispheres $U_{1,E}$ and $U_{2,E}$ as regular leaves. We prove

\begin{prop}\label{prop_step5} If $E>0$ is sufficiently small, then the following holds. Let $\tilde \F_n$ be the stable finite energy foliation on $\R \times W_E$ associated to $\tilde J_n=(\lambda_n,J_n)$ as in Proposition \ref{prop_step4}, which contains the curves $\tilde u_{1,n},\tilde u_{2,n},\tilde v_n,\tilde v_n',\tilde w_{\tau,n},\tilde w_{\tau,n}',\tau\in (0,1),$ as above, and  projects onto a $3-2-3$ foliation $\F_n$ adapted to $\lambda_n$ on $W_E$ with binding orbits $P_{3,n},P_{2,n},P_{3,n}'$. Then,  up to suitable reparametrizations and $\R$-translations of such $\tilde J_n$-holomorphic curves, we have:
\begin{itemize}
\item[i)] $\tilde u_{i,n} \to \tilde u_{i,E},i=1,2,$ in $C^\infty_{\rm loc}$ as $n \to \infty$, where $\tilde u_{i,E}$ are the rigid planes obtained in Proposition \ref{prop_step3}. Given a small neighborhood $\U \subset W_E$ of $\partial S_E=u_{1,E}(\C)\cup P_{2,E} \cup u_{2,E}(\C)$, there exists $n_0\in \N$ so that if $n \geq n_0$, then $u_{1,n}(\C)\cup  P_{2,n} \cup u_{2,n}(\C) \subset \U.$

\item[ii)] There exist finite energy $\tilde J_E$-holomorphic embedded cylinders $\tilde v_E=(b_E,v_E),$ $\tilde v_E'=(b_E',v_E'):\R \times S^1 \to \R \times W_E$, with $v_E$ and $v_E'$ embeddings,  $\tilde v_E$ asymptotic to $P_{3,E}$ and $P_{2,E}$, $\tilde v_E'$ asymptotic to $P_{3,E}'$ and $P_{2,E}$ at their punctures, respectively,  so that $\tilde v_n \to \tilde v_E$ and $\tilde v_n' \to \tilde v_E'$ in $C^\infty_{\rm loc}$ as $n \to \infty$. We denote $V_{E}=v_{E}(\R \times S^1)$ and $V_{E}'=v_{E}'(\R \times S^1)$. Then $V_{E}\subset \dot S_E\setminus P_{3,E}$ and $V_{E}'\subset \dot S_E'\setminus P_{3,E}'$. Given small neighborhoods $\U \subset W_E$ of $V_{E}\cup P_{2,E} \cup P_{3,E}$ and $\U' \subset W_E$ of $V_{E}'\cup P_{2,E} \cup P_{3,E}'$, there exists $n_0\in \N$ so that if $n \geq n_0$, then $v_{n}(\R \times S^1) \subset \U$ and $v_{n}'(\R \times S^1) \subset \U'.$

\item[iii)] There exist smooth one parameter families of $\tilde J_E$-holomorphic embedded planes $\tilde w_{\tau,E}=(d_{\tau,E},w_{\tau,E}),$ $\tilde w_{\tau,E}'=(d_{\tau,E}',w_{\tau,E}'):\C \to \R \times W_E$, $\tau \in (0,1)$, with $w_{\tau,E}$ and  $w_{\tau,E}'$ embeddings,   $\tilde w_{\tau,E}$ asymptotic to $P_{3,E}$, $D_{\tau,E}=w_{\tau,E}(\C),\tau \in (0,1),$ foliating $\dot S_E \setminus (V_E\cup P_{3,E})$,  $\tilde w_{\tau,E}'$ asymptotic to $P_{3,E}'$  and $D_{\tau,E}'=w_{\tau,E}'(\C),\tau \in (0,1),$ foliating $\dot S_E' \setminus (V_E' \cup P_{3,E}')$.  Given small neighborhoods $\U_1,\U_2\subset W_E$ of $P_{3,E} \cup V_{E} \cup P_{2,E} \cup U_{1,E}$ and  $P_{3,E} \cup V_{E} \cup P_{2,E} \cup U_{2,E}$, respectively, we find $0<\tau_1<\tau_2<1$  so that if $\tau\in (0,\tau_1)$ then $D_{\tau,E} \subset \U_1$ and if $\tau\in (\tau_2,1)$ then $D_{\tau,E} \subset \U_2$. An analogous statement holds for the family $D_{\tau,E}',\tau \in (0,1)$. Moreover, given $p_0 \in \dot S_E \setminus (V_E \cup P_{3,E})$, we find a sequence $\tau_n \in (0,1)$ and $\bar \tau \in(0,1)$ so that $p_0 \in w_{\tau_n,n}(\C)$ for each large $n$, $p_0 \in w_{\bar \tau,E}(\C)$ and $\tilde w_{\tau_n,n} \to \tilde w_{\bar \tau,E}$ in $C^\infty_{\rm loc}$ as $n \to \infty$. An analogous statement holds if $p_0 \in \dot S_E' \setminus (V_E' \cup P_{3,E}')$.
\end{itemize}
The curves $\tilde u_{1,E},\tilde u_{2,E},\tilde v_E, \tilde v_E', \tilde w_{\tau,E}, \tilde w_{\tau,E}',\tau \in(0,1),$ and the cylinders over $P_{3,E},$ $P_{2,E}$ and $P_{3,E}'$ determine a stable finite energy foliation $\tilde \F_E$ adapted to $\tilde J_E=(\lambda_E,J_E)$, which projects to  a $3-2-3$ foliation $\F_E$ adapted to $\lambda_E$ on $W_E$, with binding orbits $P_{3,E},P_{2,E}$ and $P_{3,E}'$. In particular, the $3-2-3$ foliation $\F_E$ adapted to $\lambda_E$ on $W_E$ restricts to a $2-3$ foliation adapted to $S_E$.
\end{prop}
The proofs of Proposition \ref{prop_step5}-i), ii) and iii) are left to Sections \ref{sec_5-i)}, \ref{sec_5-ii)} and  \ref{sec_5-iii)}, respectively. Now the proof of Theorem \ref{main1} is complete, remaining only to prove Propositions \ref{prop_step1}, \ref{prop_step2}, \ref{prop_step3}, \ref{prop_step4} and \ref{prop_step5} above. This is accomplished in the next sections. The existence of homoclinics to $P_{2,E}$ is discussed below.

The $2-3$ foliation adapted to $S_E$ obtained in Proposition \ref{prop_step5} and an argument found in \cite[Proposition 7.5]{fols} imply the existence of at least one homoclinic orbit to  $P_{2,E}$ in $\dot S_E$ (see also \cite{BGS} for a similar argument). We include it here for completeness.

Consider the one parameter family of disks $D_{\tau,E}$ for $\tau \in (0,1)$, which foliates $S_E \setminus (\partial S_E \cup V_E\cup P_{3,E})$. The local branch of $W^u_{E,{\rm loc}}$ intersects $D_{\tau,E}$ transversally for all $\tau >0$ small in an embedded circle $C^u_{\tau,0}$ which encloses an embedded closed disk $B^u_{\tau,0} \subset D_{\tau,E}$ with $d\lambda_E$-area $T_{2,E}>0$. In the same way,  the local branch of $W^s_{E,{\rm loc}}$ intersects $D_{\tau',E}$ transversally for all $\tau'<1$ sufficiently close to $1$ in an embedded circle $C^s_{\tau'}$ which encloses an embedded closed disk $B^s_{\tau'} \subset D_{\tau',E}$ with the same area $T_{2,E}$.

The existence of the $2-3$ foliation adapted to $S_E$ implies that the forward flow sends $B^u_{\tau,0}$, $\tau$ close to $0$,  into a disk $B_1$ inside $D_{\tau',E}$ for any $\tau'$ close to $1$. If $B_1$ intersects $B^s_{\tau'}$ then, since both disks have the same area, their boundaries also intersect and a homoclinic to $P_{2,E}$ must exist. Otherwise, $B_1$ is contained in $D_{\tau',E} \setminus B^s_{\tau'}$. The forward flow sends $B_1$ into a disk $B^u_{\tau,1} \subset D_{\tau,E}$ which, by uniqueness of solutions, must be disjoint from $B^u_{\tau,0}$. Its area is preserved by the flow and hence equals $T_{2,E}$. Arguing  in the same way, the forward flow sends $B^u_{\tau,1}$ into a disk $B_2\subset D_{\tau',E}$ which is disjoint from $B_1$ and has area $T_{2,E}$. If $B_2$ intersects $B^s_{\tau'}$ then, as before, a homoclinic to $P_{2,E}$ must exist. Otherwise, $B_2$ is contained in $D_{\tau',E} \setminus (B^s_{\tau'} \cup B_1)$ and the forward flow sends $B_2$ into a disk $B^u_{\tau,2}\subset D_{\tau,E}$ which is disjoint from $B^u_{\tau,0}\cup B^u_{\tau,1}$ and has area $T_{2,E}$. Analogously we construct the disks $B_n\subset D_{\tau',E}$ and $B^u_{\tau, n}\subset D_{\tau,E}$ using the forward flow and the $2-3$ foliation. This procedure must terminate at some point since all disks $B^u_{\tau, i}$ are disjoint and have the same positive area $T_{2,E}$ while the area of $D_{\tau,E}$ is finite and equal to $T_{3,E}$. This forces the existence of a positive integer $n< \frac{T_{3,E}}{T_{2,E}}$ so that $B_n$ intersects $B^s_{\tau'}$ and hence a homoclinic to $P_{2,E}$ must exist.

Some works on   homoclinics to saddle-centers are found in \cite{AR,BMO,ber1,ber2,BGS,BLJ,CM,CM2,rag,LK,MHO}  and references therein.

\begin{remark}In \cite{FS}, J. Fish and R. Siefring construct stable finite energy foliations for connected sums of contact $3$-manifolds with  prescribed stable finite energy foliations. The local model for this surgery looks like a saddle-center model and the new foliation admits an additional hyperbolic orbit $P$ with Conley-Zehnder index $2$ as a binding orbit, two new rigid planes asymptotic to $P$, and two leaves with additional negative punctures asymptotic to $P$. Some generic conditions are required.  \end{remark}

\section{Proof of Proposition \ref{prop_step1}}\label{sec_prop_step1}
In this section we prove Proposition \ref{prop_step1}. We restate it for clearness.

\begin{prop}\label{prop_step1b}There exists $E^*>0$ such that if $0<E<E^*$, then there exists a Liouville vector field $Y_E$ defined on a neighborhood of $W_E$ in $\U_{E^*}$ which is transverse  to $W_E$. Moreover, $Y_E$ is invariant under the symmetry $T$ defined in local coordinates by \eqref{eq_T}, i.e., $T_* Y_E = Y_E$.  In particular,  the $1$-form $i_{Y_E} \omega_0$ restricts to a tight contact form $\lambda_E$ on $W_E$ and its Reeb flow is a reparametrization of the Hamiltonian flow of $H$ restricted to $W_E$.  In local coordinates $(q_1,q_2,p_1,p_2)$, $\lambda_E$ coincides with the Liouville form $\lambda_0=\frac{1}{2}\sum_{i=1,2} p_i dq_i -q_i dp_i$ near $\partial S_E= \{q_1+p_1=0\} \cap K^{-1}(E)$.  There exists a Liouville vector field $\bar X_0$ defined in a neighborhood of $\dot W_0=\dot S_0 \cup \dot S_0':=W_0 \setminus \{p_c\}$ in $\U_{E^*}$, transverse to $\dot W_0$, so that given any small neighborhood $U_0 \subset \U_{E^*}$ of $p_c$, we have that $Y_E=\bar X_0$ outside $U_0$, for all $E>0$ sufficiently small. We denote by  $\bar \lambda_0=i_{\bar X_0} \omega_0|_{\dot W_0}$ the induced contact form on $\dot W_0$. \end{prop}

We keep using the notation established in Section \ref{sec_proof_main}. We start with a few definitions.

\begin{definition} A Liouville vector field defined on an open subset $U\subset W$ of a symplectic manifold $(W,\omega)$ is a
vector field $Y$ on $U$ satisfying $\LL_Y \omega = \omega$,
where $\LL$ is the Lie derivative. Given a regular hypersurface $S \subset W$, we say that $S$ is of contact type if
there exists a Liouville vector field $Y$ defined on an open
neighborhood of $S$ in $W$ such that $Y$ is everywhere transverse to
$S$. \end{definition}

Let us fix the dimension of the symplectic manifold $(W,\omega)$  to be equal to $4$. It is immediate to verify from the definition above that the $1$-form  $\lambda:=i_Y \omega|_S$ is a contact form on $S$, i.e., $\lambda \wedge d\lambda \neq 0$.  The Reeb
vector field associated to  a contact form $\lambda$ is the vector field
$X_\lambda$ on $S$ uniquely determined by $$i_{X_\lambda} d\lambda =0
\mbox{ and } i_{X_\lambda} \lambda =1.$$

Recall that $\R^4$ is equipped with the standard symplectic structure $$\omega_0 = \sum_{i=1}^2 dy_i \wedge dx_i.$$ Then the radial vector field $Y(z):=\frac{1}{2}z$ is a Liouville vector field on $\R^4$ and the Liouville form $$\lambda_0:=i_Y \omega_0=\frac{1}{2}\sum_{i=1}^2 y_i dx_i - x_i dy_i$$ is a primitive of $\omega_0$, i.e., $d\lambda_0 = \omega_0$. Thus any hypersurface $S \subset \R^4$, star-shaped with respect to the origin $0\in \R^4$, is of contact type and $\lambda_0|_S$ is a contact form on $S$. Its Reeb vector field coincides up to a non-vanishing factor with the Hamiltonian vector field on $S$ associated to any Hamiltonian function admitting $S$ as a regular energy level.

The contact structure induced by a contact form $\lambda$ on a $3$-manifold $M$ is the non-integrable hyperplane distribution $$\xi:=\ker \lambda \subset TM.$$ The contact structure $\xi$ is called overtwisted if there exists an embedded disk $D \hookrightarrow M$ such that $T_z \partial D \subset \xi_z$ and $T_z D \neq
\xi_z$ for all $z\in \partial D$. Otherwise it is called tight.

Now we start the proof of Proposition \ref{prop_step1b}. In order to prove that  $W_E=\partial \U_E$ is of contact type if
$E>0$ is sufficiently small, we construct a
Liouville vector field $Y_E$ near  $W_E$ and
prove that $Y_E$ is transverse to $W_E$.

Let $U_0\subset \U_E$ be the closed set, homeomorphic to the closed $4$-ball, so that $\partial U_0 = S_0$. For simplicity, we may view $U_0$ as a subset of $\R^4$. Let
\begin{equation}\label{x0}
X_0(w)  = \frac{w-w_0}{2}
\end{equation}
be the Liouville vector field with respect to $\omega_0$, where $w_0\in U_0 \setminus S_0$.  The point $w_0$ will be determined below. Since $S_0$ is a strictly convex singular subset of $H^{-1}(0)$, $X_0$ is transverse to $\dot S_0$, see Remark \ref{remconvexity}. This implies that $X_0$ is transverse to $S_E$ outside a fixed neighborhood of $p_c$ for all $E>0$ sufficiently small.

Now let us denote by $\varphi: (V \subset \R^4,0) \to (U \subset
\R^4,p_c)$ the symplectomorphism as in Hypothesis 1. In local coordinates
$z=(q_1,q_2,p_1,p_2)\in V$, the Hamiltonian function takes the
form
$$K(z)=\bar K(I_1,I_2)= - \alpha I_1 + \omega I_2 + R(I_1,I_2),$$ where $I_1 =
q_1p_1$, $I_2 = \frac{q_2^2 + p_2^2}{2}$ and $R(I_1,I_2) = O(I_1^2
+ I_2^2)$.

Let $C(I_1,I_2,E) := \bar K(I_1,I_2)-E$. Notice that $C(0,0,0)=0$. Since
$\partial_{I_2} C(0,0,0) = \omega \neq 0$, we find  a function $I_2=I_2(I_1,E)$ such that
$C(I_1,I_2(I_1,E),E)=0,\forall (I_1,E)$ in a neighborhood of $0\in \R^2$. A direct
computation shows that \begin{equation} \label{I2} I_2(I_1,E) =
\frac{\alpha}{\omega} I_1 + \frac{1}{\omega}E + O(I_1^2 +
E^2).\end{equation}

The lemma below is crucial for proving Proposition \ref{prop_step1b}.

\begin{lem}\label{lem_prod}The following inequalities holds.
\begin{itemize}\item[(i)]If $|I_1|$ and $E\geq 0$ are small enough, then \begin{equation}\label{prod} I_1\partial_{I_1} \bar K + I_2 \partial_{I_2}\bar K \geq \frac{E}{2} + O(I_1^2).\end{equation} \item[(ii)]If $I_1\leq 0$, $z \in B_{\delta_0}(0)\cap K^{-1}(E),E>0$, with $\delta_0>0$ small enough, then \begin{equation}\label{prod2}I_1 \partial_{I_1} \bar K + I_2 \partial_{I_2}\bar K>0.\end{equation} \item[(iii)]If $z \in B_{\delta_0}(0)\cap K^{-1}(E),E<0$, with $\delta_0>0$ small enough, then \begin{equation}\label{prod3} I_1 \partial_{I_1}\bar K + I_2 \partial_{I_2}\bar K \geq  O(I_1).\end{equation} \end{itemize} \end{lem}

\begin{proof} If $|I_1|$ and $E\geq 0$ are small enough, then we obtain
$$\begin{aligned} I_1
\partial_{I_1} \bar K + I_2 \partial_{I_2}\bar K & =  I_1(-\alpha + \partial_{I_1} R) + I_2(\omega + \partial_{I_2} R) \\ & = E-R + I_1\partial_{I_1} R + I_2 \partial_{I_2} R \\ & =  E + O(I_1^2 + I_2^2) \\ & =  E + O(I_1^2+E^2)\\ & \geq \frac{E}{2} + O(I_1^2),\end{aligned}
$$
proving (i).

First observe that $|z|$ small implies $|I_1|$ and $|I_2|$ small. If $I_1\leq 0$, $z \in B_{\delta_0}(0)\cap K^{-1}(E),E>0$, with $\delta_0>0$ small enough, then $$ \begin{aligned} I_1 \partial_{I_1} \bar K + I_2 \partial_{I_2}\bar K = &  -\alpha I_1 + \omega I_2  + I_1 \partial_{I_1} R + I_2 \partial_{I_2} R \\   = &  -\alpha I_1 + \omega I_2 +O(I_1^2 + I_2^2) \\ > & 0.\end{aligned} $$ Notice that if $I_1=0$ and $E>0$ then $I_2>0$. This proves (ii).

To prove (iii), observe that
$$\begin{aligned} I_1 \partial_{I_1}\bar K + I_2 \partial_{I_2}\bar K   = &  -\alpha I_1 + \omega I_2 +O(I_1^2 + I_2^2) \\ > & \frac{\omega}{2}I_2 + O(I_1)\\ \geq & O(I_1),\end{aligned}$$ if $z \in B_{\delta_0}(0)\cap K^{-1}(E),E<0$, with $\delta_0>0$ small enough. See that if $E<0$ and $I_2=0$, then $I_1>0$.\end{proof}

We can also write \begin{equation} \label{I1}I_1= I_1(I_2,E)=
\frac{\omega}{\alpha}I_2 -\frac{1}{\alpha}E + O(I_2^2+E^2),
\end{equation}for sufficiently small neighborhood of $(I_2,E)=0\in
\R^2$. In this case we have
\begin{equation}\begin{aligned} (q_1+p_1)^2=& q_1^2+p_1^2 + 2\left(\frac{\omega}{\alpha}I_2 -\frac{1}{\alpha}E
\right) + O(I_2^2+E^2) \\
\geq & \min \left\{\frac{\omega}{\alpha},1\right\}|z|^2-2\frac{E}{\alpha} +
O(|z|^4+E^2) \\ \geq &
\min \left\{\frac{\omega}{\alpha},1\right\}\frac{|z|^2}{2}-4\frac{E}{\alpha}
\\ \geq & \min\left\{\frac{\omega}{\alpha},1\right\}\frac{|z|^2}{4},
\end{aligned} \end{equation} if $|z|$ is sufficiently small and
$0<E<\hat c|z|^2$, where $\hat c=\frac{\alpha}{16}\min\left\{\frac{\omega}{\alpha},1\right\}>0$ is a fixed small constant. This implies that
if $\delta_0>0$ is fixed sufficiently small  then
\begin{equation}\label{zc2} \delta_0 \geq |z|\geq \frac{\delta_0}{2} \Rightarrow q_1+p_1 \geq
\frac{\delta_0}{c_2},
\end{equation} for all $z\in B_{\delta_0}(0)\cap \{q_1+p_1 \geq 0\}\cap K^{-1}(E)$ and $E>0$ sufficiently small,
where
\begin{equation} \label{c2}
c_2=\frac{4}{\sqrt{\min\{\frac{\omega}{\alpha},1\}}}\geq 4.\end{equation}

We have proved the following lemma.

\begin{lem}For all fixed $\delta_0>0$ sufficiently small, implication \eqref{zc2} holds for $z\in B_{\delta_0}(0)\cap \{q_1+p_1\geq 0\}\cap K^{-1}(E),$ where $0<E<\hat c |z|^2$ and $\hat c=\frac{\alpha}{16}\min\left\{\frac{\omega}{\alpha},1\right\}>0$. \end{lem}

Let us assume that $w_0 \in U$ and denote $z_0 =
\varphi^{-1}(w_0)\in V$. Let $\widetilde Y_0 = \varphi^{-1}_* X_0$
be the Liouville vector field on $V$, where $X_0$ is given as in \eqref{x0}. Although $\widetilde Y_0$ is
not necessarily radial with respect to $z_0$, it  satisfies
$$\widetilde Y_0(z) = \frac{z-z_0}{2} + R_{z_0}(z),$$ where
\begin{equation}\label{Rz} |R_{z_0}(z)| \leq c |z-z_0|^2,\end{equation}  for a constant $c>0$  which
does not depend on $z_0$, if both $z$ and $z_0$ lie on
$B_{\delta_0}(0) = \{|z|\leq \delta_0\} \subset V$, for
$\delta_0>0$ small.

Let $$Y_0(z)= \frac{z-z_0}{2}$$ be the radial Liouville vector field
on $B_{\delta_0}(0)$. Notice that
\begin{equation}\label{fechada} \begin{aligned} d(i_{R_{z_0}}\omega_0) = &
d(i_{\widetilde Y_0} \omega_0 - i_{Y_0}\omega_0)\\ = &
\mathcal{L}_{\widetilde Y_0} \omega_0 - \mathcal{L}_{Y_0} \omega_0
\\ = & \omega_0 - \omega_0\\ = & 0
\end{aligned} \end{equation}
Since $B_{\delta_0}(0)$ is simply connected, equation
\eqref{fechada} implies that
\begin{equation}\label{defiG} i_{R_{z_0}}\omega_0=-dG_{z_0}\end{equation}  for a smooth function
$G_{z_0}:B_{\delta_0}(0) \to \R$. Observe that $R_{z_0}=\widetilde
Y_0 - Y_0$ is the Hamiltonian vector field associated to
$G_{z_0}$. From \eqref{defiG} we see that
\begin{equation}\label{partialG} \frac{\partial G_{z_0}}{\partial q_i} = -n_i \mbox{ and }
\frac{\partial G_{z_0}}{\partial p_i} = m_i,\end{equation}
$i=1,2,$ where $R_{z_0}(z) = \sum_{i=1,2} m_i(z)
\partial_{q_i} + n_i(z) \partial_{p_i}$. The functions
$m_i,n_i,i=1,2,$ depend on $z_0$ and satisfy
\begin{equation}\label{ineqmi} |m_i(z)|,|n_i(z)|\leq c
|z-z_0|^2,\end{equation} for all $z,z_0\in B_{\delta_0}(0)$, where
$c>0$ is a constant which does not depend on $z_0$.

From \eqref{partialG} and \eqref{ineqmi}, we have $$|\nabla
G_{z_0}(z)| \leq 4c |z-z_0|^2,$$ and therefore, assuming that
$G_{z_0}(z_0)=0$, we get \begin{equation}\label{Gz}
\begin{aligned} |G_{z_0}(z)| = &
\left|\int_0^1 \nabla G_{z_0}(z_0 + s(z-z_0)) \cdot (z-z_0) ds \right| \\
\leq & |z-z_0| \int_0^1 4c s^2 |z-z_0|^2 ds  \\ = & \frac{4c}{3}
|z-z_0|^3,
\end{aligned}\end{equation} for all $z,z_0\in B_{\delta_0}(0)$, with $\delta_0>0$ sufficiently small.

We choose a cut-off function $f: [0,\infty) \to [0,1]$
 satisfying $f(t) = 1$ if
$t\in \left[0,\frac{\delta_1}{3}\right]$, $f(t)=0$ if
$t\geq \frac{2\delta_1}{3}$, and $-\frac{6}{\delta_1}< f'(t)<0$ if
$t\in \left(\frac{\delta_1}{3},\frac{2\delta_1}{3}\right)$, where $0<\delta_1:=
\frac{\delta_0}{c_2}$, $c_2>0$ defined in \eqref{c2}. This implies that
\begin{equation} \label{ineqf} 0\leq 1-f(t) \leq \frac{6}{\delta_1} t,\forall t\geq 0, \end{equation}
and
\begin{equation} \label{ineqfl} |f'(t)|\leq \frac{18}{\delta_1^2} t ,\forall t\geq 0. \end{equation}

Our first step is to interpolate the Liouville vector fields $Y_0$
and $\widetilde Y_0$ so that we get a new Liouville vector field
$Y$ defined on $B_{\delta_0}(0)\cap\{q_1+p_1\geq 0\}$,
$\delta_0>0$ small, which coincides with $Y_0$ on
$B_{\frac{\delta_1}{6}}(0)\cap \{q_1+p_1\geq 0\}$ and with $\widetilde
Y_0$ on $\left(B_{\delta_0}(0) \setminus B_{\frac{\delta_0}{2}}(0)\right)\cap
\{q_1+p_1\geq 0\}$. Then we show that $Y$ is transverse to
$ \{q_1+p_1 \geq 0\} \cap B_{\delta_0}(0)\cap K^{-1}(E)$, if $E>0$ is
sufficiently small. The choice of $\frac{\delta_1}{6}$ above is due to the fact that
$|z|<\frac{\delta_1}{6}$ implies $q_1 + p_1 < \frac{\delta_1}{3}$. In the same
way, we can use \eqref{zc2} to see that if $\delta_0>0$ is
chosen small enough, then $\delta_0 \geq |z|\geq \frac{\delta_0}{2}$
implies $q_1+p_1 \geq \delta_1$ for all $E>0$ sufficiently small.

In order to do that, define
\begin{equation}\label{defi_Y} Y := \widetilde Y_0 - Y_{fG_{z_0}},\end{equation} where $Y_{fG_{z_0}}$ is the
Hamiltonian vector field on $\{q_1 + p_1 \geq 0\}\cap
B_{\delta_0}(0)$ associated to the function $K_{fG_{z_0}}(z) :=
f(q_1+p_1)G_{z_0}(z)$, $z=(q_1,q_2,p_1,p_2),$ and defined by
$i_{Y_{fG_{z_0}}}\omega_0 = -dK_{fG_{z_0}}$.

Since $\mathcal{L}_{Y_{fG_{z_0}}} \omega_0 =0$ and $\widetilde
Y_0$ is Liouville with respect to $\omega_0$, we get $\mathcal{L}_Y
\omega_0 = \omega_0,$ and therefore
$Y$ is Liouville with respect to $\omega_0$ as well.

It remains to show that $Y$ is transverse to $S_E$ in local
coordinates, if $E>0$ is sufficiently small. Let $$z_0 =
\frac{1}{\sqrt{2}}(\delta_0,0,\delta_0,0).$$

Notice that $Y_{fG_{z_0}} = G_{z_0}Y_f + K_fR_{z_0}$, with $Y_f$
defined by $i_{Y_f} \omega_0 = -dK_f$, where $K_f(z)=f(q_1+p_1)$.
We compute
\begin{equation} \label{dKY} \begin{aligned} dK\cdot Y  = & dK \cdot
\widetilde Y_0 - K_f dK \cdot R_{z_0} - G_{z_0}dK \cdot Y_f  \\
= & (1-K_f)dK\cdot R_{z_0} + dK \cdot Y_0 - G_{z_0}dK \cdot Y_f.
\end{aligned}
\end{equation}

We begin by estimating the first term of the last line of equation
\eqref{dKY}.

\begin{lem}\label{termo1}There exists a constant $A_1>0$ such that for all $\delta_0>0$ sufficiently small and $z \in \{q_1 + p_1 \geq 0\} \cap
B_{\delta_0}(0)$ we have
\begin{equation} |(1-K_f)dK\cdot R_{z_0}(z)|\leq A_1 \delta_0^2(q_1+p_1).
\end{equation}
\end{lem}

\begin{proof}
First observe that $|dK(z)| \leq c_0|z|$ for some
constant $c_0>0$ and $z\in B_{\delta_0}(0)$, for $\delta_0>0$ small. Now we use that
$\delta_1=\frac{\delta_0}{c_2}$, $|z_0|=\delta_0>0$ and $|z-z_0|^2\leq |z|^2
+ 2|z|\delta_0+\delta_0^2$. From \eqref{Rz}
  and \eqref{ineqf} we have
\begin{equation} \begin{aligned} |(1-K_f)dK(z)\cdot R_{z_0}(z)|
\leq & \frac{6c_0c}{\delta_1}|z||z-z_0|^2(q_1+p_1)\\
\leq & \frac{6c_0cc_2}{\delta_0}|z|(|z|^2+2|z|\delta_0+\delta_0^2)(q_1+p_1)\\ \leq &
24c_0cc_2\delta_0^2(q_1+p_1),
\end{aligned}
\end{equation} for all $z\in \{q_1 + p_1 \geq 0\} \cap B_{\delta_0}(0)$. Let $A_1=24c_0cc_2$.
\end{proof}

Regarding the second term of \eqref{dKY}, we observe that
\begin{equation}\label{expr2} dK \cdot Y_0 (z)=dK \cdot \frac{1}{2}(z-z_0)= I_1 \partial_{I_1}\bar K + I_2 \partial_{I_2}\bar K +(q_1+p_1)\frac{\delta_0 \bar
\alpha }{2\sqrt{2}},\end{equation} where $\bar \alpha =
-\partial_{I_1}\bar K = \alpha -
\partial_{I_1} R$. Then we obtain the estimate below.

\begin{lem}\label{termo2}We have \begin{itemize}
\item[(i)] For all $\delta_0>0$ sufficiently small, the following holds: if $0 \neq z\in B_{\delta_0}(0)\cap \{q_1 +p_1\geq 0\}\cap \{E \geq 0\}$ then \begin{equation} dK \cdot Y_0 (z)>(q_1+p_1) \frac{\delta_0 \alpha}{8 \sqrt{2}}.\end{equation}

\item [(ii)] For any fixed $\delta_0>0$ small, there exists $\delta_0 \gg \bar \delta_0 >0$ depending on $\delta_0$ such that if $z\in B_{\bar \delta_0}(0)\cap \{q_1 + p_1\geq 0\}\cap \{E <0\}$ then \begin{equation}  dK \cdot Y_0 (z)>(q_1+p_1)\frac{\delta_0 \alpha}{8 \sqrt{2}}.\end{equation}
\end{itemize}
\end{lem}

\begin{proof}Taking $|z|$ sufficiently small we can assume that $\frac{\alpha}{2} < \bar \alpha< 2\alpha$. To prove $(i)$ we assume first that $E>0$. If $I_1> 0$ then $q_1,p_1>0$ since $q_1+p_1\geq 0$ and, therefore, from Lemma \ref{lem_prod}-(i) and \eqref{expr2} we get $$\begin{aligned} dK \cdot Y_0 (z)>& (q_1+p_1) \frac{\delta_0 \alpha}{4 \sqrt{2}}+O(I_1^2) \\ > & (q_1+p_1) \frac{\delta_0 \alpha}{8 \sqrt{2}} ,\end{aligned}$$ if $\delta_0>0$ is small enough and $z\in B_{\delta_0}(0)\cap \{q_1+p_1 \geq 0\} \cap \{E>0\} \cap \{I_1>0\}$.

If $E>0$ and $I_1 \leq 0$, then Lemma \ref{lem_prod}-(ii) and \eqref{expr2} imply that $$dK \cdot Y_0 (z)> (q_1+p_1) \frac{\delta_0 \alpha}{4 \sqrt{2}},$$ if $\delta_0>0$ is small enough and $z\in B_{\delta_0}(0)\cap \{q_1+p_1 \geq 0\} \cap \{E>0\} \cap \{I_1 \leq 0\}$.

If $E=0$, then $q_1,p_1 \geq 0$ and since $z \neq 0$ by assumption, $q_1>0$ or $p_1>0$. From Lemma \ref{lem_prod}-(i) and \eqref{expr2} we get $$\begin{aligned} dK \cdot Y_0 (z)>& (q_1+p_1) \frac{\delta_0 \alpha}{4 \sqrt{2}}+O(I_1^2) \\ > & (q_1+p_1) \frac{\delta_0 \alpha}{8 \sqrt{2}} ,\end{aligned}$$ if $\delta_0>0$ is small enough and $z\in B_{\delta_0}(0)\cap \{q_1+p_1 \geq 0\}\cap \{E= 0\}$. This proves (i).

Now assume $E<0$. To prove (ii), we observe first that $q_1,p_1>0$. From Lemma \ref{lem_prod}-(iii) and \eqref{expr2} we have $$\begin{aligned} dK \cdot Y_0 (z)>& (q_1+p_1) \frac{\delta_0 \alpha}{4 \sqrt{2}}+O(I_1) \\ > & (q_1+p_1) \frac{\delta_0 \alpha}{8 \sqrt{2}} ,\end{aligned}$$ if $0<\bar \delta_0 \ll \delta_0$ are chosen small enough and $z\in B_{\bar \delta_0}(0)\cap \{q_1+p_1 \geq 0\}\cap \{E<0\}$. \end{proof}

Now we estimate the third term of \eqref{dKY}.

\begin{lem}\label{termo3} There exists a constant $A_3>0$ such that for all $\delta_0>0$ sufficiently small and $z \in \{q_1 + p_1 \geq 0\} \cap
B_{\delta_0}(0)$ we have \begin{equation}|G_{z_0}dK \cdot Y_f(z)|\leq A_3\delta_0^2(q_1+p_1). \end{equation}
\end{lem}

\begin{proof}
Observe that
$Y_f(z) = f'(q_1+p_1)(\partial_{q_1} -
\partial_{p_1})$, hence \begin{equation}\label{dKYf} dK \cdot Y_f (z) = \bar \alpha
(q_1-p_1)f'(q_1+p_1).\end{equation}
Now using that
$\delta_1=\frac{\delta_0}{c_2}$, $|z_0|=\delta_0>0$, $|q_1-p_1| \leq 2|z|$, $|z-z_0|^3\leq |z|^3
+ 3|z|^2\delta_0+ 3|z|\delta_0^2+\delta_0^3$, and the equations \eqref{Gz},
\eqref{ineqfl} and \eqref{dKYf}, we get
\begin{equation}\begin{aligned} |G_{z_0}dK \cdot Y_f(z)| \leq
& \frac{24c}{\delta_1^2}|q_1-p_1|\bar \alpha|z-z_0|^3(q_1+p_1) \\
\leq & \frac{96c}{\delta_1^2}\alpha |z||z-z_0|^3(q_1+p_1) \\
\leq & \frac{96cc_2^2}{\delta_0^2}\alpha
|z|(|z|^3+3|z|^2\delta_0+3|z|\delta_0^2+\delta_0^3)(q_1+p_1)  \\ \leq
& 768cc_2^2\alpha\delta_0^2(q_1+p_1),
\end{aligned}
\end{equation} for all $z \in \{q_1 + p_1 \geq 0\} \cap
B_{\delta_0}(0)$, with $\delta_0>0$ sufficiently small. Let $A_3=768cc_2^2\alpha$.
\end{proof}

Using Lemmas \ref{termo1}, \ref{termo2} and \ref{termo3}, we prove:

\begin{prop}\label{propYtransversal}For all $\delta_0>0$ sufficiently small, the following properties hold:
\begin{itemize}
\item[i)] If $0 \neq z\in B_{\delta_0}(0)\cap \{q_1 +p_1\geq 0\}\cap \{E \geq 0\}$, then the Liouville vector field $Y(z)$ constructed above and defined on $B_{\delta_0}(0)\cap \{q_1 + p_1 \geq 0\}$  is positively transverse to its energy level $K^{-1}(K(z)),$ i.e.,  $dK \cdot Y (z)>0$. Moreover, for all $z \in B_{\delta_2}(0)$, $Y(z)=\frac{z-z_0}{2}$, where $z_0 = \frac{1}{\sqrt{2}}(\delta_0,0,\delta_0,0)$, $\delta_2:=\frac{\delta_1}{6}=\frac{\delta_0}{6c_2}$ and $c_2>0$ is given by \eqref{c2}.
\item[ii)] There exists $0<\bar \delta_0  \ll \delta_0,$ depending on $\delta_0$, such that if $z\in B_{\bar \delta_0}(0)\cap \{q_1 + p_1\geq 0\}\cap \{E <0\}$, then the Liouville vector $Y(z)$ is positively transverse to its energy level $K^{-1}(K(z))$.
\end{itemize}
\end{prop}

\begin{proof}We can choose $\delta_0>0$ sufficiently small satisfying Lemmas \ref{termo1}, \ref{termo2}-(i) and \ref{termo3}  so that in addition it satisfies $\frac{\delta_0 \alpha }{8\sqrt{2}}-
(A_1+A_3)\delta_0^2>0$. Using equation \eqref{dKY} and these lemmas, we see that for all  $0 \neq z\in B_{\delta_0}(0)\cap \{q_1 +p_1\geq 0\}\cap \{E \geq 0\}$ and $\delta_0>0$ sufficiently small, we have
$$dK \cdot Y(z) >  \left(\frac{\delta_0 \alpha }{8\sqrt{2}}-
(A_1+A_3)\delta_0^2\right) (q_1+p_1) \geq  0.$$ This proves i).

Now for any $\delta_0>0$ small as above, pick $\bar \delta_0$ as in Lemma \ref{termo2}-(ii). We also get $$dK \cdot Y(z) >  \left(\frac{\delta_0 \alpha }{8\sqrt{2}}-
(A_1+A_3)\delta_0^2\right) (q_1+p_1) \geq 0,$$ for all $z\in B_{\bar \delta_0}(0)\cap \{q_1 +p_1\geq 0\}\cap \{E <0 \}$. This proves ii). \end{proof}

For a fixed $\delta_0>0$ sufficiently small, we have constructed a local Liouville vector field $Y$ which interpolates $\widetilde Y_0$ and $Y_0$,  coincides with $\widetilde Y_0$ in $\left(B_{\delta_0}(0)\setminus B_{\frac{\delta_0}{2}}(0)\right)\cap \{q_1+p_1 \geq 0\}$ and with $Y_0(z)=\frac{z-z_0}{2}$ in $B_{\delta_2}(0) \cap \{q_1+p_1\geq 0\}$, for all $E>0$ sufficiently small. Moreover, $Y$ is transverse to $B_{\delta_0}(0) \cap \{q_1+p_1\geq 0\} \cap K^{-1}(E)$. Here $\delta_2 = \frac{\delta_0}{6c_2}
$, where $c_2$ is given by \eqref{c2}.

From Proposition \ref{propYtransversal}-i), we also see that $Y$ is transverse to $B_{\delta_0}(0) \cap \{q_1+p_1 > 0\} \cap K^{-1}(0)$. Patching $\varphi_*Y$ together with $X_0(w)=\frac{w-w_0}{2}$ outside $\varphi(B_{\delta_0}(0))$, with $w_0 = \varphi(z_0)$, we obtain a Liouville vector field $\bar X_0$ in a small neighborhood of $\dot S_0 = S_0 \setminus \{p_c\}$ which is transverse to $\dot S_0$, see Remark \ref{remconvexity}.

Now we fix $0<\delta_2=\frac{\delta_0}{6c_2}<\delta_0$  small enough as in Proposition \ref{propYtransversal}. We also fix $\bar E>0$ sufficiently small as above so that $\partial S_{\bar E} = \{q_1+p_1=0\} \cap K^{-1}(\bar E)\subset B_{\delta_2}(0)$. Our next step is to perform an interpolation between $Y_0$ and
\begin{equation}\label{y0barra}
\overline Y_0(z) = \frac{z}{2}
\end{equation}
in $B_{\delta_2}(0)\cap\{q_1+p_1\geq 0\}\cap K^{-1}(E),$ for all $E>0$ close to $\bar E$, so that we find a Liouville vector field $Y_{\bar E}$ defined on a neighborhood $\W_{\bar E}\subset B_{\delta_2}(0)\cap \{q_1+p_1 \geq0\}$ of $\partial S_{\bar E}$, which coincides with $\overline Y_0$ near $\partial S_{\bar E}$ and with $Y_0$ outside a small neighborhood of $\partial S_{\bar E}$. Moreover, $Y_{\bar E}$ is transverse to $S_{E}$ at $\W_{\bar E}$ for all $E>0$ close to $\bar E$.

Recall that we can write $I_2 =I_2(I_1,E)$ as in \eqref{I2}.
From Lemma \ref{lem_prod}-(ii), we can assume that  \begin{equation}\label{prod6} I_1 \partial_{I_1}\bar K +I_2 \partial_{I_2}\bar K=(-\alpha + \partial_{I_1}R)I_1 + (\omega + \partial_{I_2}R)I_2=-\bar \alpha I_1 +\bar \omega I_2>0\end{equation} if $I_1\leq I_{1,\bar E}^*$, $z \in B_{\delta_2}(0)\cap \{q_1+p_1 \geq 0\} \cap K^{-1}(E)$ and $E>0$ is sufficiently close to $\bar E$, where $I_{1,\bar E}^*>0$ is fixed sufficiently small. This set contains an open neighborhood $\W_{\bar E} \subset B_{\delta_2}(0)\cap \{q_1+p_1 \geq 0\}$ of $\partial S_{\bar E}$. So we can find $\epsilon_{\bar E}>0$ so that $\{0\leq q_1 + p_1 \leq \epsilon_{\bar E}\} \cap K^{-1}(E) \subset \W_{\bar E}$, for all $E>0$ sufficiently close to $\bar E$. We choose $\epsilon_{\bar E}>0$ so that it satisfies \begin{equation}\label{epsbarra}\epsilon_{\bar E} \to 0 \mbox{ as } \bar E \to 0^+.\end{equation}

Consider a smooth cut-off function $f_{\bar E}:[0,\infty) \to [0,1]$ so that $f_{\bar E}(t) =1$ if $t\in \left[0,\frac{\epsilon_{\bar E}}{4}\right]$, $f_{\bar E}(t) =0$ if $t\geq\frac{\epsilon_{\bar E}}{2}$ and $f_{\bar E}'(t) \leq 0$ for all $t\geq 0$, and define $$Y_{\bar E}:=Y_0-Y_{f_{\bar E} G_{z_0}},$$ where $Y_{f_{\bar E} G_{z_0}}$ is the Hamiltonian vector field associated to the function $K_{f_{\bar E} G_{z_0}}(z):=f_{\bar E}(q_1+p_1)G_{z_0}(z)$, $z=(q_1,q_2,p_1,p_2)$, and $G_{z_0}$ is now defined by $i_{-\frac{z_0}{2}}\omega_0 = -dG_{z_0}$ and $G_{z_0}(0)=0$, i.e.,
$$G_{z_0}=(q_1-p_1)\frac{\delta_0}{2\sqrt{2}}.$$ Observe that $Y_{\bar E}$ is Liouville with respect to $\omega_0$ since $\LL_{Y_{\bar E}} \omega_0 = \omega_0$.

Denoting by $Y_{f_{\bar E}}$ the Hamiltonian vector field associated to $K_{f_{\bar E}}(z):=f_{\bar E}(q_1+p_1)$, we have $Y_{f_{\bar E}}=f'_{\bar E}(q_1+p_1) (\partial_{q_1}-\partial_{p_1})$, which implies that $dK \cdot Y_{f_{\bar E}} (z) = \bar \alpha
(q_1-p_1)f'_{\bar E}(q_1+p_1).$ Moreover, notice that $Y_{f_{\bar E}G_{z_0}} = G_{z_0}Y_{f_{\bar E}} - K_{f_{\bar E}}\frac{z_0}{2}$. Therefore
\begin{equation}\label{fim}\begin{aligned}  dK \cdot Y_{\bar E}  =   &(1-K_{f_{\bar E}})dK\cdot
\left( \frac{-z_0}{2} \right) + dK \cdot \frac{z}{2} - G_{z_0}dK
\cdot Y_{f_{\bar E}}\\ = & (1-K_{f_{\bar E}})\frac{\delta_0}{2\sqrt{2}} \bar \alpha
(q_1+p_1) + I_1 \partial_{I_1}\bar K + I_2 \partial_{I_2}\bar K
\\ & -(q_1-p_1)^2\frac{\delta_0}{2\sqrt{2}} \bar \alpha f_{\bar E}'(q_1+p_1)  >
0, \end{aligned} \end{equation} if $z\in \W_{\bar E}$. The last inequality follows from
\eqref{prod6} and from the fact that the two other terms in \eqref{fim} are
non-negative due to the properties of $f_{\bar E}$.

At this point we have found a Liouville vector field, still denoted by $Y_{\bar E}$ and defined
near $S_{\bar E}\cap B_{\delta_0}(0)$, which is transverse to $S_E$ for all $E>0$ sufficiently close to $\bar E$,  coincides with
$\overline Y_0$ near $\partial S_E$ and with
$\widetilde Y_0=\varphi_*^{-1}X_0$ on $B_{\delta_0}(0) \setminus
B_{\frac{\delta_0}{2}}(0)$, where $\overline Y_0$ and $X_0$ are defined in \eqref{y0barra} and \eqref{x0} respectively.

Patching $\varphi_* Y_{\bar E}$  together with the Liouville vector field $X_0$ outside $\varphi(B_{\delta_0}(0))$, we finally obtain a Liouville vector field defined near $S_{\bar E}\subset U_{\bar E}$, still denoted by $Y_{\bar E}$, which is transverse to $S_{\bar E}$.

Now we can extend $Y_{\bar E}$ to $\{q_1+p_1 \leq 0\}$ by defining
$$Y_{\bar E}(z) = DT^{-1}(T(z))\cdot (Y_{\bar E}\circ T(z)),$$ for all $z\in
 \{q_1+p_1 \leq 0\}\cap B_{\delta_0}(0)\cap K^{-1}(E)$, for all $E>0$ sufficiently close to $\bar E$,
where $T$ is the involution map $(q_1,q_2,p_1,p_2) \mapsto
(-q_1,q_2,-p_1,p_2)$. Since $T$ is symplectic, this extension is also Liouville. As above, this naturally induces an extension of $Y_{\bar E}$ to a neighborhood of $W_{\bar E}\subset \U_{\bar E}$, again denoted by $Y_{\bar E}$, which is transverse to $W_{\bar E}$. Thus $\lambda_{\bar E}:= i_{Y_{\bar E}} \omega_0|_{W_{\bar E}}$ is a contact form on $W_{\bar E}$ and since $W_{\bar E}=\partial \U_{\bar E}$, $W_{\bar E}$ admits a symplectic filling and its contact structure $\xi:=\ker \lambda_{\bar E}$ is diffeomorphic to the standard tight contact structure on the $3$-sphere, see \cite{eli}.

Notice that $i_{\bar Y_0} \omega_0=\frac{1}{2}\sum_{i=1,2} p_i dq_i - q_idp_i$ is the Liouville form $\lambda_0$ and, therefore, $\lambda_{\bar E}$ coincides with $\lambda_0|_{K^{-1}(\bar E)}$ near $\partial S_{\bar E}$. The Reeb vector field of $\lambda_{\bar E}$ is parallel to the Hamiltonian vector field on $S_{\bar E}$ since both lie in the line bundle $\ker d\lambda_{\bar E} = \ker \omega_0|_{S_{\bar E}}$.

Thus we have found $E^*>0$ so that for all $0<E<E^*$, there exists a Liouville vector field $Y_E$ defined near $W_E$, which is transverse to $W_E$ and has the desired properties near $\partial S_E$.

Finally observe that the Liouville vector field $\bar X_0$, constructed above and defined near $\dot S_0$, can be reflected via the involution map $T$ to a Liouville vector field, also denoted by $\bar X_0$, now defined near $W_0 \setminus \{p_c\}$ and transverse to it. This induces a contact form $\bar \lambda_0$ on $W_0 \setminus\{p_c\}.$  It follows from the construction of $Y_E, E>0,$ and $\bar X_0$, see \eqref{epsbarra} that, given any small neighborhood $U_0\subset \U_{E^*}$ of $p_c$, we have $\bar X_0=Y_E$ outside $U_0$ for all $E>0$ sufficiently small. This concludes the proof of Proposition \ref{prop_step1b}.

\begin{remark}
In \cite{AFKP}, Albers, Frauenfelder, van Koert and Paternain construct a Liouville vector field for the circular planar restricted three-body problem, which is transverse to energy levels for energies slightly above the first Lagrangian critical value. For that, they consider interpolations of Liouville vector fields similar to the ones explained above.
\end{remark}

\section{Proof of Proposition \ref{prop_step2}}\label{sec_step2}

Let $\lambda_E$ be the contact form on $W_E$ obtained in Proposition \ref{prop_step1} for $E>0$ sufficiently small. In this section we prove the following proposition which implies Proposition \ref{prop_step2}.
\begin{prop}\label{prop_EM} The following assertions hold:
\begin{itemize}
\item[i)] There exists $E^*>0$ sufficiently small such that if $0<E<E^*$ then  $P_{2,E}\subset W_E$ is the unique periodic orbit of $\lambda_E$ with Conley-Zehnder index $2$. All other periodic orbits have Conley-Zehnder index $\geq 3$.

\item[ii)] Given an integer $M\geq 3$, there exists $E_M>0$, such that if $0<E<E_M$ and $P \subset W_E\setminus P_{2,E}$ is a periodic orbit linked to $P_{2,E}$, then $CZ(P)> M$. In particular, if $0<E<E_3$ and $P\subset W_E \setminus P_{2,E}$ is a periodic orbit with $CZ(P)=3$, then $P$ is not linked to $P_{2,E}$.
\item[iii)] There exists a small neighborhood $U_{4 \pi}\subset \U_{E^*}$ of $p_c$ such that for all $E>0$ sufficiently small the following holds: let $\lambda_n$ be a sequence of contact forms on $W_E$ so that $\lambda_n \to \lambda_E$ in $C^\infty$ as $n \to \infty$. If  $n$ is sufficiently large then $\lambda_n$ is weakly convex and the Reeb flow of $\lambda_n$ admits only one periodic orbit $P_{2,n}$ with Conley-Zehnder index $2$. Moreover, $P_{2,n}$ is unknotted, hyperbolic, converges to $P_{2,E}$ as $n \to \infty$ and all periodic orbits of $\lambda_n$ with Conley-Zehnder index equal to $3$ do not intersect $U_{4 \pi}$ and are not linked to $P_{2,n}$.
\end{itemize}
\end{prop}

We denote by $j_0,j_1,j_2,j_3:\R^4 \to \R^4$ the following linear
maps: \begin{equation}\label{eqframe} \begin{aligned} j_0(a_1 \partial_{q_1}+a_2
\partial_{q_2}+b_1 \partial_{p_1} + b_2\partial_{p_2}) &=  a_1 \partial_{q_1}+a_2
\partial_{q_2}+b_1
\partial_{p_1} + b_2\partial_{p_2}, \\
j_1(a_1 \partial_{q_1}+a_2 \partial_{q_2}+b_1 \partial_{p_1} +
b_2\partial_{p_2}) &=  b_2 \partial_{q_1} -b_1
\partial_{q_2} + a_2 \partial_{p_1} - a_1 \partial_{p_2},
\\
j_2(a_1 \partial_{q_1}+a_2 \partial_{q_2}+b_1 \partial_{p_1} +
b_2\partial_{p_2}) &=  a_2 \partial_{q_1} -a_1
\partial_{q_2} - b_2 \partial_{p_1} + b_1 \partial_{p_2},
\\
j_3(a_1 \partial_{q_1}+a_2 \partial_{q_2}+b_1 \partial_{p_1} +
b_2\partial_{p_2}) &=  b_1 \partial_{q_1} + b_2
\partial_{q_2} - a_1 \partial_{p_1} -a_2 \partial_{p_2}.
\end{aligned} \end{equation} One can check that $\{j_0=\text{Id},j_1,j_2,j_3,-j_0=-\text{Id},-j_1,-j_2,-j_3\}$ is isomorphic
to the quaternion group, i.e., $j_i^2 = -\text{Id},i=1,2,3$,
$j_1j_2= j_3$ etc.

For each regular point $w\in \R^4$ of $H$, there exists an
orthonormal frame $${\rm span}\{X_0,X_1,X_2,X_3\}_w=T_w\R^4\simeq\R^4$$ determined by
\begin{equation}\label{mfX} X_i(w) = j_i \left( \frac{\nabla H(w)}{|\nabla
H(w)|}\right), \, i=0,1,2,3.\end{equation} One easily checks that
$T_w(H^{-1}(H(w)))= \text{span}\{X_1(w),X_2(w),X_3(w)\}$. Let us
denote by $S_w = H^{-1}(H(w))$ the energy level containing $w$.

We denote by $u(t)=\displaystyle \sum_{i=0}^{3}
\alpha_i(t)X_i(w(t))\in T_{w(t)}\R^4 \simeq \R^4$ a linearized solution of
 \begin{equation}\label{eqlin}
\dot u(t) = j_3 H_{ww}(w(t)) u(t)
\end{equation}
along a non-constant trajectory $w(t)$ of $X_H$, i.e., $$u(t) = D\psi_t(w(0))u(0),u(0)\in
T_{w(0)}\R^4,$$ where $\psi_t$ is the flow of $H$. Since $H$ is preserved by the flow we can assume
that $\alpha_0(t)=0,\forall t,$ which holds if $\alpha_0(0)=0$.
Moreover, $D\psi_t$ preserves the Hamiltonian vector direction $\R
X_3$. We are interested in understanding the linearized flow
projected into $\text{span}\{X_1,X_2\}$, i.e., we consider the
projection
$$\pi_{12}(u(t)) := \alpha_1(t)X_1(w(t))+\alpha_2(t)X_2(w(t)).$$
Using the equation \eqref{eqlin} one can check that $(\alpha_1(t),\alpha_2(t))\in \R^2$ satisfies
\begin{equation}\label{linflow} \left(\begin{array}{c}\dot \alpha_1(t) \\ \dot
\alpha_2(t)
\end{array} \right) = -J M(w(t)) \left(\begin{array}{c}\alpha_1(t) \\ \alpha_2(t) \end{array}
\right),
\end{equation}
where $$J=\left(\begin{array}{cc} 0 & 1 \\ -1 & 0 \end{array}
\right)$$ and $M$ is the symmetric matrix given by
\begin{equation}\label{defiM} M(w) =
\left(\begin{array}{cc}\kappa_{11}(w) & \kappa_{12}(w) \\
\kappa_{21}(w) & \kappa_{22}(w) \end{array}
 \right) + \kappa_{33}(w)I,\end{equation} with \begin{equation}\label{defikappa}\kappa_{ij}(w) =
 \left<H_{ww}(w)X_i(w),X_j(w)\right>,i,j\in\{1,2,3\}.\end{equation} Notice that if $H_{ww}(w(t))|_{T_{w(t)}S_{w(t)}}$
 is positive definite then $M(w(t))$ is positive definite as well. In fact, in this case we have
 $$\begin{aligned} \det M & =  \kappa_{11}\kappa_{22}-\kappa_{12}^2 + \kappa_{33}(\kappa_{11}+\kappa_{22}+
 \kappa_{33})>0,\\
 \tr M & =  \kappa_{11} + \kappa_{22} +2\kappa_{33}>0.\end{aligned}$$ Thus, from
 \eqref{linflow},  $$\alpha_1(t)\dot \alpha_2(t) - \alpha_2(t)\dot \alpha_1(t) =
 (\alpha_1(t),\alpha_2(t))^t M(w(t)) (\alpha_1(t),\alpha_2(t))>0,$$ i.e.,
 $0\neq (\alpha_1(t),\alpha_2(t))\in \R^2$ rotates
 counter-clockwise. Let $\alpha_1(t) + i \alpha_2(t) \in \R^+ e^{i
 \eta(t)}$ for a continuous argument $\eta(t)$. Then $$\dot \eta(t) = \frac{\alpha_1(t)\dot \alpha_2(t) - \alpha_2(t)\dot \alpha_1(t)}
 {|(\alpha_1(t),\alpha_2(t))|^2}.$$ We immediately obtain the following proposition which will be
useful
 later on.

\begin{prop}\label{etabarra} Let $W\subset \R^4$ be a compact set so that $H$ is regular at all $w\in W$.  Assume that $H_{ww}(w)$ is positive definite when restricted to
$T_{w}(H^{-1}(H(w))),\forall w\in W.$
Then there exists $\bar \eta >0$  such that if $w(t)\in W$ is a
Hamiltonian trajectory then
$$ \dot \eta(t)>\bar \eta, $$ for any non-vanishing transverse linearized solution
$\alpha_1(t)+i\alpha_2(t)\in \R^+ e^{\eta(t) i}$ along $w(t)$.
\end{prop}

Now we consider the local coordinates $z=(q_1,q_2,p_1,p_2)$ as in Hypothesis 1 where the Hamiltonian function $H$ is locally given by the expression \eqref{Kloc}. For any regular point $z$ of
$K$, we have $T_z\R^4 = \text{span} \{Y_0,Y_1,Y_2,Y_3\}_z$, where
\begin{equation}\label{Yi}
Y_i(z) = j_i\left(\frac{\nabla K(z)}{|\nabla K(z)|}\right), i=0,1,2,3,
\end{equation}
and $j_i,i=0,1,2,3,$ are defined as in \eqref{eqframe}.

We denote by $\phi_t$ the Hamiltonian flow of $K$ and by $v(t)\in
T_{z(t)}\R^4$ a  solution of the linearized flow along a non-constant trajectory $z(t) =
\phi_t(z(0))$. Projecting $v(t)$ into
$\text{span}\{Y_1,Y_2\}_{z(t)}$ we have
$\pi_{12}(v(t)):=\beta_1(t)Y_1(z(t)) + \beta_2(t)Y_2(z(t))$ where
$(\beta_1(t),\beta_2(t))\in \R^2$ satisfies
\begin{equation}\label{linflow2} \left(\begin{array}{c}\dot
\beta_1(t) \\ \dot \beta_2(t)
\end{array} \right) = -J M(z(t)) \left(\begin{array}{c}\beta_1(t) \\ \beta_2(t) \end{array}
\right).
\end{equation} The matrix $M$ is  defined as in \eqref{defiM}
and \eqref{defikappa}, replacing $H_{ww}$ with $K_{zz}$ and $X_i$
with $Y_i,i=1,2,3$.

For a non-vanishing solution $(\beta_1(t),\beta_2(t))\in \R^2
\simeq \C$ of \eqref{linflow2}, we write $ \beta_1(t) + i
\beta_2(t) \in \R^+e^{\theta(t)i} $, where $\theta(t)$ is a
continuous argument.
Given $z_0 \in V \setminus \{0\}$ we define as before $S_{z_0} =
K^{-1}(K(z_0))$. For each $v_0 \in T_{z_0} S_{z_0}$ satisfying
$\pi_{12}(v_0) \neq 0$ and $a\leq b\in \R$, we denote by
\begin{equation}\label{deltatheta} \Delta \theta (z_0,v_0,[a,b]) =
\theta(b) - \theta(a)\end{equation} the variation of the argument
of $(\beta_1(t),\beta_2(t))$, which corresponds to the linearized
solution $v(t)=D\phi_t(z_0) v_0$ along $z(t)=\phi_t(z_0)$
projected into $\text{span}\{Y_1,Y_2\}_{z(t)}$, in the time
interval $[a,b]$.


We are ready to state our first estimate.

\begin{lem}\label{lem1} There exists $\delta>0$ sufficiently small such that the following holds: let $z_0=(q_{10},q_{20},p_{10},p_{20}) \in B_\delta(0), q_{10}p_{10} \neq 0,$ and $[t^-,t^+],t^-<0<t^+,$ be the maximal interval containing $t=0$ such that  $z(t) \in \overline{B_\delta(0)}, \forall t\in [t^-,t^+]$, where $z(t)=\phi_t(z_0)$. Let $v_0\in T_{z_0}S_{z_0}$ be such that $\pi_{12}(v_0)\neq 0$. Then $$\Delta \theta(z_0,v_0,[t^-,t^+])>\frac{\omega}{2}(t^+-t^-) - \pi.$$ Moreover, given $N\geq 0$, there exists $0<\delta_N< \delta$ such that if $z_0\in B_{\delta_N}(0)\subset B_\delta(0)$ and $v_0\in T_{z_0}S_{z_0}$ are as above, then $$\Delta \theta(z_0,v_0,[t^-,t^+])>N.$$
\end{lem}

\begin{proof} By the maximality of $[t^-,t^+]$ and since $q_{10}p_{10} \neq 0$, there exists $t^*\in
(t^-,t^+)$ such that $$z(t^*)=(b,q_{20}(t^*),\pm b,p_{20}(t^*)),b \neq 0.$$
We may assume $t^*=0$ since $\Delta \theta$ is independent of
reparametrizations in time of the linearized solutions, i.e., $$\Delta
\theta(z_0,v_0,[t^-,t^+])= \Delta \theta(z_1,v_1,[t_1^-,t_1^+]),$$
if $z_1 = \phi_{t^*}(z_0)$, $v_1 = D\phi_{t^*}(z_0) \cdot v_0$,
$t_1^- = t^- - t^*$ and $t_1^+ = t^+ - t^*$. Notice that $t^*=0$ implies $t^-=-t^+$ by symmetry.

Let us assume that $z_0 = (b,q_{20},b,p_{20})\in B_\delta(0)$ with
$b>0$, the other cases are completely analogous. Then
\begin{equation}\label{z(t)}\begin{aligned} z(t) =& (q_1(t),q_2(t),p_1(t),p_2(t)) \\=& (be^{-\bar \alpha t }, r \sin (\bar \omega
(t + \zeta_0)), be^{\bar \alpha t }, r \cos (\bar \omega (t +
\zeta_0)))\end{aligned} \end{equation} where
\begin{equation}\label{alfaomega} \begin{aligned} \bar \alpha = &
-\partial_{I_1}\bar K\left(b^2, \frac{r^2}{2}\right )  =
 \alpha +O(b^2 + r^2),
\\\bar \omega =&
\hspace{0.5cm} \partial_{I_2}\bar K\left(b^2, \frac{r^2}{2}\right)  =
\omega + O(b^2+r^2), \end{aligned} \end{equation} are constants
along the solutions, $(q_{20},p_{20})= (r\sin (\bar \omega
\zeta_0),r \cos (\bar \omega  \zeta_0))$ and $r = \sqrt{q_{20}^2 +
p_{20}^2},$ see \eqref{SolHam}. For simplicity in the notation we may assume that
$\zeta_0=0$.

Using \eqref{Yi} we compute \begin{equation}\label{nablaK}
\begin{aligned} Y_0(z(t))  = & g(t) (-\bar \alpha be^{\bar \alpha
t}, \bar \omega r \sin \bar \omega t, -\bar \alpha
be^{-\bar \alpha t}, \bar \omega r \cos \bar \omega t), \\
Y_1(z(t)) = & g(t)(\bar \omega r \cos \bar \omega t,\bar \alpha
be^{-\bar \alpha t},\bar \omega r \sin \bar \omega t,\bar \alpha
be^{\bar \alpha t}),\\
Y_2(z(t))  = & g(t) (\bar \omega r \sin \bar \omega t,\bar \alpha
be^{\bar \alpha t},-\bar \omega r \cos \bar \omega t ,- \bar
\alpha be^{-\bar \alpha t}), \\
Y_3(z(t))  = & g(t) (-\bar \alpha be^{-\bar \alpha t}, \bar \omega
r \cos \bar \omega t,\bar \alpha be^{\bar \alpha t}, -\bar \omega
r \sin \bar \omega t),
\end{aligned}
\end{equation} where $$g(t) = \frac{1}{|\nabla K(z(t))|}=\frac{1}{\sqrt{2\bar \alpha^2 b^2\cosh (2\bar \alpha
t) + \bar \omega^2 r^2}}.$$ The Hessian matrix $K_{zz}(z(t))$ of $K$
along $z(t)$ is given by
$$
\left(\begin{array}{llll} r_{11}b^2e^{2 \bar \alpha t} & r_{12}
bre^{\bar \alpha t} \sin \bar \omega t & -\bar \alpha + r_{11} b^2
& r_{12} br e^{\bar \alpha t} \cos \bar \omega t \\
r_{12}bre^{\bar \alpha t} \sin \bar \omega t & \bar \omega +
r_{22}r^2\sin^2 \bar \omega t & r_{12}br e^{-\bar \alpha t} \sin
\bar \omega t & r_{22} r^2 \sin \bar \omega t \cos \bar \omega t
\\ -\bar \alpha + r_{11}b^2 & r_{12} br e^{-\bar \alpha t} \sin
\bar \omega t & r_{11} b^2 e^{-2 \bar \alpha t} & r_{12} br
e^{-\bar \alpha t} \cos \bar \omega t \\ r_{12} br e^{\bar \alpha
t} \cos \bar \omega t & r_{22} r^2 \sin \bar \omega t \cos \bar
\omega t & r_{12} br e^{-\bar \alpha t} \cos \bar \omega t & \bar
\omega + r_{22} r^2 \cos^2 \bar \omega t
\end{array} \right).
$$  The functions $r_{ij}=\partial_{I_iI_j}
R(b^2,\frac{r^2}{2}),i,j=1,2,$ are zero order terms, constant
along the solutions, and we may assume that $|r_{ij}|\leq c_0$ on
$V$ for a uniform constant $c_0>0$. We compute

$$\begin{aligned}\left<K_{zz}Y_1,Y_1\right>_{z(t)}= & g(t)^2( -\bar
\alpha \bar \omega^2 r^2 \sin 2\bar \omega t + 2\bar\omega \bar
\alpha^2
b^2\cosh 2\bar \alpha t+ \\
& b^2r^2\bar r(\cosh 2\bar \alpha t + \sinh 2\bar a t \cos
2\bar
\omega t + \sin 2\bar \omega t)),\\
\left<K_{zz}Y_1,Y_2\right>_{z(t)}= & g(t)^2 (\bar \alpha \bar
\omega^2  r^2 \cos 2\bar \omega t +
 b^2r^2\bar r(-\cos 2\bar \omega t + \sin 2\bar \omega t \sinh 2\bar \alpha t)),\\
\left<K_{zz}Y_2,Y_2\right>_{z(t)}= & g(t)^2( \bar \alpha \bar
\omega^2 r^2 \sin 2\bar \omega t + 2\bar\omega \bar \alpha^2 b^2\cosh
2\bar \alpha t+ \\
& b^2r^2\bar r(\cosh 2\bar \alpha t - \sinh 2\bar a t \cos
2\bar
\omega t - \sin 2\bar \omega t)),\\
\left<K_{zz}Y_3,Y_3\right>_{z(t)}= & g(t)^2( 2 \bar \alpha^3b^2 +
\bar \omega^3 r^2 ),
\end{aligned}$$ where $\bar r(b,r)=r_{11}\bar \omega ^2 +2r_{12}\bar
\alpha \bar \omega +r_{22}\bar \alpha^2$ is constant along the trajectories and satisfies
\begin{equation}\label{rb} |\bar r(b,r)|\leq c_1\end{equation} for some uniform constant $c_1>0$.

In this way, the linearized solution $v(t)$ along $z(t)$ projected
into $\text{span}\{Y_1,Y_2\}_{z(t)}$ is represented by
$(\beta_1(t),\beta_2(t))\neq 0$ satisfying
\begin{equation}\label{betaponto}\left(\begin{array}{c} \dot \beta_1(t) \\ \dot
\beta_2(t) \end{array} \right) = \left(\begin{array}{cc}-m(t) &
p(t)-n(t) \\ p(t)+ n(t) & m(t)
\end{array} \right)\left(\begin{array}{c} \beta_1(t) \\ \beta_2(t)  \end{array}
\right)\end{equation} where
\begin{equation}\label{pmn} \begin{aligned}
m(t)= & g(t)^2[\bar\alpha \bar \omega^2 r^2\cos 2 \bar \omega t+b^2r^2\bar r(\sin 2\bar \omega t\sinh 2\bar \alpha t-\cos 2\bar \omega t)], \\
p(t)= & g(t)^2[-\bar \alpha \bar \omega^2 r^2 \sin 2\bar \omega t+b^2r^2\bar r(\cos 2\bar \omega t \sinh 2\bar \alpha t +\sin 2\bar\omega t)],\\
n(t)= & \bar \omega+ g(t)^2 (2\bar\alpha^3b^2+ b^2r^2\bar r \cosh 2\bar \alpha t).
\end{aligned}\end{equation}
Denoting
$\beta_1(t)+i\beta_2(t)=\rho(t)e^{\theta(t)i}$ for continuous
functions $\rho(t)>0$ and $\theta(t)$, we find from
\eqref{betaponto} and \eqref{pmn} that $$\begin{aligned} \dot \theta(t)
= & \frac{\beta_1(t)\dot \beta_2(t) -\beta_2(t)\dot
\beta_1(t)}{\beta_1(t)^2+\beta_2(t)^2} \\
= & n(t)+p(t)\cos 2\theta(t) + m(t)\sin 2\theta(t)\\
= &  \bar \omega+g(t)^2\{2\bar \alpha^3 b^2+(b^2r^2 \bar r  -\bar \alpha \bar
\omega^2r^2)\sin (2\bar \omega t-2\theta(t)) +\\
& b^2r^2\bar r[\cosh 2\bar \alpha t + \sinh 2 \bar \alpha t \cos (2\bar \omega t-2\theta(t))]\}\\
 =: & G(t,\theta(t)),\end{aligned}
$$ with $G(t,\theta)$ being $\pi$-periodic in $\theta$.
Let $\tilde \theta(t) = \bar \omega t+\frac{\pi}{4}$.  We have
$$G(t, \tilde \theta(t))  =  \bar \omega  + g(t)^2\{2 \bar \alpha^3
b^2 + r^2[\bar \alpha \bar \omega^2 +b^2\bar r (\cosh 2\bar \alpha t - 1)] \}$$
Since $0<q_1(t)=be^{-\bar \alpha t} \leq  \delta$ and $0<p_1(t)=be^{\bar \alpha t}
\leq \delta$ for all $t\in [t^-,t^+],$ we have
\begin{equation}\label{csh}
b^2\cosh 2\bar \alpha t\leq  2\delta^2
\end{equation} for all $t\in[t^-,t^+]$. Using \eqref{rb} and
\eqref{csh} we obtain \begin{equation} G(t, \tilde \theta(t))  =
\bar \omega  + g(t)^2\{2 \bar \alpha^3 b^2 + r^2[\bar \alpha \bar
\omega^2+ O(\delta^2)]\}.\end{equation} Thus choosing $\delta>0$
sufficiently small we finally get
\begin{equation}\label{eq_Gtetatil}
G(t, \tilde \theta(t)) > \bar\omega,
\end{equation}
for all $t\in[t^-,t^+]$.

Let $\theta(0)=k_0 \pi + \frac{\pi}{4} + \mu_0,$ for some
$k_0\in \Z$ and $\mu_0\in [0, \pi )$.  Letting $\Delta(t) =
\theta(t) - \tilde \theta(t)-k_0 \pi, \forall t,$ we have $\dot
\Delta(t) = G(t, \Delta(t)+ \tilde \theta(t)) - \bar \omega$ and
hence, from \eqref{eq_Gtetatil}, $\dot \Delta(t)>0$ if $\Delta(t)=k \pi, \forall k\in \Z$.
Since $\Delta(0)=\mu_0 \geq 0$, this implies that $\Delta(t)>0$ for all $t>0$, i.e., $\theta(t)>\tilde\theta(t) +k_0 \pi,
\forall t>0$. Therefore
\begin{equation}\label{mintheta}
 \theta(t)-\theta(0)=  \Delta(t)+\bar \omega t -\mu_0
 >  \bar \omega t-\pi, \forall t>0.\end{equation}  Since
$$\Delta \theta(z_0,v_0,[t^-,t^+]) = \Delta
\theta(\phi_{t^-}(z_0), D\phi_{t^-}(z_0) \cdot v_0,[0,t^+-t^-]),$$
we use estimate \eqref{mintheta} considering $\phi_{t^-}(z_0)$ and
$D\phi_{t^-}(z_0) \cdot v_0$ as initial conditions in order to
find \begin{equation}\label{res1} \Delta
\theta(z_0,v_0,[t^-,t^+])> \bar \omega (t^+-t^-) - \pi
> \frac{\omega}{2}(t^+-t^-) -\pi,\end{equation} if $\delta>0$ is fixed small
enough.

The last statement of this lemma clearly follows from \eqref{res1}. In fact, due to the local behavior of the flow near the saddle-center equilibrium, given $N\geq 0$ we can choose $\delta_N>0$ small enough such that if
$z_0 \in B_{\delta_N}(0)$ then $t^+- t^- >
\frac{2}{\omega}(N+\pi)$. This implies
$$\Delta \theta(z_0,v_0,[t^-,t^+])>N.$$
\end{proof}

Next we prove a similar estimate for the Hamiltonian flow of $H$
in the original coordinates $w=(x_1,x_2,y_1,y_2)$. As before, we
consider the moving frame $\{X_0,X_1,X_2,X_3\}_w$ defined in
\eqref{mfX} for regular points $w$ of $H$. Let $$u(t)=\displaystyle
\sum_{i=1}^3 \alpha_i(t)X_i(w(t))=D\psi_t(w_0) \cdot u_0$$ be a
solution of the linearized flow along a non-constant trajectory $w(t)=\psi_t(w_0)$ of $X_H$, with $u_0=u(0)$. As we
have seen, the projection $\pi_{12}(u(t))$ into
$\text{span}\{X_1,X_2\}_{w(t)}$ is represented by
$(\alpha_1(t),\alpha_2(t))\in \R^2$ which satisfies
\eqref{linflow}. We assume $\pi_{12}(u_0) \neq 0$. Let $\eta(t)$ be
a continuous argument defined by $\alpha_1(t)+ i \alpha_2(t) \in
\R^+ e^{i\eta(t)}$. In the same way, for $a \leq b$ we define
$$\Delta \eta(w_0,u_0,[a,b])=\eta(b)-\eta(a).$$


\begin{lem}\label{lem2} Given $\W_0\subset U$ neighborhood of $p_c$,  there exists a neighborhood $U_*\subset \W_0$ of $p_c$ such that the following holds: let
 $w_0=(x_{10},x_{20},y_{10},y_{20}) \in U_*$ so that $\varphi^{-1}(w_0)=(q_{10},q_{20},p_{10},p_{20})$
 satisfies $q_{10}p_{10} \neq 0$. Let $[t^-,t^+],t^-<0<t^+,$ be
 the maximal interval containing $t=0$ such that  $w(t) \in {\rm closure}(U_*), \forall t\in [t^-,t^+]$, where $w(t)=\psi_t(w_0)$. Let
$u_0\in T_{w_0}S_{w_0}$ be such that $\pi_{12}(u_0)\neq 0$, where $S_{w_0} = H^{-1}(H(w_0))$. Then
$$\Delta \eta (w_0,u_0,[t^-,t^+]) > \frac{\omega}{2}(t^+-t^-) -
C,$$ where $C>0$ is a uniform constant depending only on $U_*$. In particular, given $N\geq 0$,
there exists a neighborhood $U_N\subset U_*$ of $p_c$ such that if $w_0\in U_N$ and $u_0\in T_{w_0}S_{w_0}$ are as above, then
$$\Delta \eta(w_0,u_0,[t^-,t^+])>N.$$

\end{lem}
\begin{proof}

Let $\delta>0$ be such that $\overline {B_\delta(0)}\subset \varphi^{-1}(\W_0)$.
Consider the moving frame on $\overline{B_\delta(0)} \setminus
\{0\}$ given by
$$\{\widetilde X_i(z)= D\varphi^{-1}(\varphi(z))\cdot X_i(\varphi(z)),i=1,2,3\}.$$
We denote by $\widetilde \pi_{12}$ the projection into
$\text{span}\{\widetilde X_1, \widetilde X_2\}$ along $\widetilde X_3 \parallel Y_3
$.
Let $\widetilde Y_i=\widetilde \pi_{12}(Y_i),i=1,2$. Notice that
$\{\widetilde Y_1,\widetilde Y_2\}$ is a symplectic basis of
$\text{span}\{\widetilde X_1,\widetilde X_2\}$. Writing $\widetilde
Y_1(z) = a_{1}(z) \widetilde X_1(z) + a_{2}(z) \widetilde X_2(z)$
 we find a continuous function $\zeta_1: \overline{B_\delta(0)}
\setminus \{0\} \to S^1$ given by $z\mapsto
\frac{(a_{1}(z),a_{2}(z))}{|(a_{1}(z),a_{2}(z))|}$. Let $\tilde
\zeta_1: \overline{B_\delta(0)} \setminus \{0\} \to \R$ be a lift
of $\zeta_1$, which exists since $\overline{B_\delta(0)} \setminus
\{0\}$ is simply connected. In this way we have $a_{1}(z) + i a_{2}(z)
\in \R^+ e^{i\tilde \zeta_1(z)},\forall z\in \overline{B_\delta(0)}\setminus\{0\}.$
Let \begin{equation}\label{mc1} C_1 = \sup_{z_1,z_2\in \partial
B_\delta(0)} |\tilde \zeta_1(z_1)-\tilde
\zeta_1(z_2)|<\infty.\end{equation}

The flow $\psi_t$ of $H$ in $U$ is conjugated to the flow $\phi_t$ of $K$ in $V$
by the symplectomorphism $\varphi$, i.e., $\varphi \circ \phi_t =
\psi_t \circ \varphi$. Thus the linearized flows satisfy
$D\varphi \cdot D\phi_t = D\psi_t \cdot D\varphi$. Let $z(t) =
\varphi^{-1}(w(t))$ and $v(t)=D\phi_t(z(0))\cdot
D\varphi^{-1}(w_0) \cdot u_0$, where $w(t)$ and $u_0$ are as in
the statement. Thus $$\begin{aligned} \widetilde \pi_{12}(v(t)) = &
\alpha_1(t) \widetilde X_1(z(t))+ \alpha_2(t) \widetilde
X_2(z(t))\\ = & \beta_1(t) \widetilde Y_1(z(t))+ \beta_2(t)
\widetilde Y_2(z(t)),\end{aligned}$$ where $(\alpha_1(t),\alpha_2(t))$ satisfies \eqref{linflow}$-$\eqref{defikappa} and
$(\beta_1(t),\beta_2(t))$ satisfies \eqref{linflow2} with $H$ replaced with $K$. Let
$\theta(t)$ be a continuous argument of $\beta_1(t)+ i
\beta_2(t)$. Now let \begin{equation}\label{tetat} \tilde
\theta(t) = \eta(t) - \tilde \zeta_1(z(t)).\end{equation} Observe
that $\tilde \theta(t)$ corresponds to the angle between $\tilde
\pi_{12}(v(t))$ and $\widetilde Y_1(z(t))$ in the trivialization
induced by $\{\widetilde X_1,\widetilde X_2\}$ while $\theta(t)$
is the angle between $\pi_{12}(v(t))$ and $Y_1(z(t))$ in the
trivialization induced by $\{Y_1,Y_2\}$. We can assume that $$-\pi
<\tilde \theta(0) - \theta(0) < \pi.$$ Since $Y_1$ and $Y_2$
correspond to $\widetilde Y_1$ and $\widetilde Y_2$ in
$\text{span} \{\widetilde X_1,\widetilde X_2\}$, respectively, we
must have
\begin{equation}\label{dtheta} -\pi <\tilde \theta(t) - \theta(t)
< \pi,\forall t.\end{equation}

We can take $\delta>0$ even smaller so that the conclusions of Lemma \ref{lem1}  hold in $B_\delta(0)$. Let $U_*:=\varphi(B_\delta(0)) \subset \W_0$. From
\eqref{mc1}, \eqref{tetat} and \eqref{dtheta}, if $w_0 \in U_*$, we find
$$\begin{aligned} \Delta \eta (w_0,u_0,[t^-,t^+]) =  & \eta(t^+)-\eta(t^-) \\ =
& (\tilde \theta(t^+) - \tilde \theta(t^-)) +(\tilde
\zeta_1(z(t^+)) - \tilde \zeta_1(z(t^-))) \\ = & (\tilde
\theta(t^+) - \theta(t^+)) + (\theta(t^+) - \theta(t^-)) \\ &  +
(\theta(t^-) - \tilde \theta(t^-)) +(\tilde \zeta_1(z(t^+)) -
\tilde \zeta_1(z(t^-)))
\\ > & \Delta \theta(z_0,v_0,[t^-,t^+]) -2\pi - C_1 \\ > &
\frac{\omega}{2}(t^+-t^-) - C, \end{aligned}$$ where $C = 3\pi +
C_1$.

Given $N\geq 0$, choose $\delta_N>0$ sufficiently small such that
if $z_0 \in B_{\delta_N}(0)$ then $t^+- t^- >
\frac{2}{\omega}(N+C)$. Let $U_N := \varphi(B_{\delta_N}(0)) \subset U_*$. If
$w_0 \in U_N$ then $$\Delta \eta(w_0,u_0,[t^-,t^+])>N.$$ \end{proof}

Given a contractible periodic orbit $P=(w,T)\subset H^{-1}(E)$ with period $T>0$, we
have the following characterization of its Conley-Zehnder index $CZ(P)$: let $u(t)$ be a solution of the linearized flow along $w(t)$
with initial condition $ u(0) \in T_{w(0)}S_{w(0)}$ satisfying
$\pi_{12} (u(0)) \neq 0$. We write $\pi_{12}(u(t)) = \alpha_1(t)
X_1(w(t))+\alpha_2(t) X_2(w(t))$, where the frame $\{X_1,X_2\}$ is defined as in \eqref{mfX}. Let $\eta(t)$ be a continuous
argument of $\alpha_1(t) + i \alpha_2(t)$. Let
$$I = \left\{\frac{\Delta\eta(w(0),u(0),[0,T])}{2\pi} : \pi_{12} (u(0)) \neq 0\right\}.$$
Then $I$ is a compact interval with
length $< \frac{1}{2}$. Let
\begin{equation}\label{Iepsilon} I_\epsilon = I + \epsilon,\end{equation} with $\epsilon \in \R$. We find
an integer $k$ such that for all $\epsilon <0$ small, either
$I_\epsilon \subset (k,k+1)$ or $k\in
\text{interior}(I_\epsilon)$. In the first case,  $CZ(P) =
2k+1$ and in the second case  $CZ(P) = 2k$, see \cite{hryn2}.  Notice that
we are using the symplectic trivialization of ${\rm span}\{X_1,X_2\}$ induced by
$\{X_1,X_2\}$ to study the transverse linearized flow. In our case we always assume that this
trivialization is defined over an embedded disk $D \subset H^{-1}(E)$ such that $\partial D=P$.

\begin{prop}\label{prop_p2e} For each $E>0$ small, let $P_{2,E} = (z_E, T^H_{2,E})\subset K^{-1}(E)$ be the periodic orbit
 in the center manifold of the saddle-center, given in local
coordinates $(q_1,q_2,p_1,p_2)$ by $\{q_1=p_1=0, I_2=I_2(0,E)\}$, see \eqref{eqI2}. Then $P_{2,E}$ is unknotted,
hyperbolic and $CZ(P_{2,E})=2$.
\end{prop}
\begin{proof} We consider the symplectic frame $\{Y_1,Y_2\}$ defined in
\eqref{Yi}. Notice that $P_{2,E}$ is the boundary of the embedded disk
$ {\rm closure}(U_{1,E})=\{q_1+p_1=0, q_1\leq 0\}\cap K^{-1}(E)\subset V,$ for $E>0$
small. Hence $P_{2,E}$ is unknotted.

Assume that $z_E(0)= (0,0,0,r_E)$, where $r_E =\sqrt{2I_2(0,E)}=
\sqrt{2\frac{E}{\omega} + O(E^2)}.$ The period of $P_{2,E}$ with respect to the Hamiltonian flow is given by $T^H_{2,E} = \frac{2 \pi}
{\bar \omega}$, where $\bar \omega = \omega + O(r_E^2)$ is constant in $t$.  Thus $z_E(t)=(0,r_E \sin \bar \omega t,0, r_E \cos
\bar \omega t)$. Using \eqref{linflow2}, we see that the projection of a linearized solution $v(t)$ along $z_E(t)$, with $\pi_{12} (v(0))\neq 0,$ is represented by
$\beta(t)=\beta_1(t)+i \beta_2(t) \in \R^+ e^{i \theta(t)}$,
 satisfying
\begin{equation} \left(\begin{array}{c}\dot
\beta_1(t) \\ \dot \beta_2(t)
\end{array} \right) = \left(\begin{array}{cc}-\bar \alpha \cos 2\bar \omega t &
 -\bar \alpha \sin 2\bar \omega t - \bar \omega \\ -\bar \alpha \sin 2 \bar \omega t + \bar \omega &
 \bar \alpha \cos 2\bar \omega t \end{array} \right) \left(\begin{array}{c}\beta_1(t) \\ \beta_2(t) \end{array}
\right).
\end{equation}
Defining new coordinates $\tilde \beta(t) =\tilde \beta_1(t) + i\tilde \beta_2(t):= (\beta_1(t)+i\beta_2(t)) e ^{-i\bar
\omega t} \in \R^+ e^{i \tilde \theta(t)}$,  we get
$$\left(\begin{array}{c}
\dot {\tilde \beta}_1(t) \\
\dot {\tilde \beta}_2(t)
\end{array}\right)=
\left(\begin{array}{cc}
-\bar \alpha & 0\\
0 & \bar\alpha
\end{array}\right)
\left(\begin{array}{c}
\tilde \beta_1(t) \\
\tilde\beta_2(t)
\end{array}\right).$$
Since $\tilde \beta(T^H_{2,E}) = \beta(T^H_{2,E})$, $P_{2,E}$ is hyperbolic. Now since $\Delta \theta(z_E(0),v(0),[0,T^H_{2,E}])=\Delta \tilde\theta(z_E(0),v(0),[0,T^H_{2,E}])+ 2\pi$ and
$$0 \in {\rm interior}(\{\Delta \tilde\theta(z_E(0),v(0),[0,T^H_{2,E}]):\pi_{12}(v(0)) \neq 0\}),$$ we
have that $1 \in \text{interior}(I_\epsilon)$ for all $\epsilon<0$
small, where $I_\epsilon$ is defined in \eqref{Iepsilon}. This
implies that $CZ(P_{2,E})=2$. \end{proof}

\begin{prop}\label{prop_mu} The following assertions hold:
\begin{itemize}
\item[i)] There exists $E_0>0$ sufficiently small such
that  given an integer $M>0$, there exists a
neighborhood $\widetilde U_M$ of $p_c$ such that if $0<E<E_0$ and
$P\subset W_E\setminus P_{2,E}$ is a periodic orbit intersecting
$\widetilde U_M$, then $CZ(P)> M$.
\item[ii)] Given an integer $M>0$, there exists $0<E_M<E_0$, $E_0>0$ as in i), so that if $0<E<E_M$  and $P\subset W_E\setminus P_{2,E}$ is linked to $P_{2,E}$ then $CZ(P) > M$. In particular, if $0<E<E_3$ and $P \subset W_E\setminus P_{2,E}$ is a periodic orbit with $CZ(P)=3$, then $P$ is not linked to $P_{2,E}$.
\end{itemize}
\end{prop}

\begin{proof}Let $U_0\subset U_* \subset U$ be the neighborhoods of $p_c$ given in Lemma
\ref{lem2}, where $U_0$ is obtained choosing $N=0$. We can find
$E_0>0$ and a constant $k_0>0$ such that for all $0\leq E<E_0$ we
have $\left< H_{ww}(w) \cdot u, u\right>
> k_0$ for all $w \in W_E \setminus U_0$ and $u\in T_w S_{w}$
with $|u|=1$. This follows from the convexity assumption on the
critical level given by Hypothesis 2. Hence, from Proposition \ref{etabarra}
  we find $\bar \eta>0$ such that if $\alpha_1(t)+ i \alpha_2(t) \in \R^+ e^{i \eta(t)}$ represents a non-vanishing linearized solution $u(t)$ along a trajectory $w(t)$ projected into $\text{span} \{X_1,X_2\}_{w(t)}$, and so that $w(t)\in W_E \setminus U_0$ for $t \in [a,b]$, then \begin{equation} \label{eqq1} \dot \eta(t)
> \bar \eta >0, \forall t\in[a,b].\end{equation} Given $\widetilde M>0$,  let $U_{\widetilde M}\subset U_0$ be
as in Lemma \ref{lem2} so that if $w(0)=\varphi(q_{10},q_{20},p_{10},p_{20})$\\ $\in U_{\widetilde M}$ with $q_{10}p_{10}\neq 0$
 then \begin{equation}\label{eqq2} \Delta \eta(w(0),u(0),[t^-,t^+]) > \widetilde M,\end{equation}
for any $u(0)\in T_{w(0)}S_{w(0)}$ satisfying $\pi_{12}(u(0)) \neq 0$, where $t^-<0<t^+$
and $[t^-,t^+]$ corresponds to the maximal time interval containing
$t=0$ such that $w([t^-,t^+])\subset {\rm closure} (U_*)$. Any other maximal segment of orbit
intersecting $U_0$ and contained in $U_*$ contributes positively
to $\Delta \eta$, by Lemma \ref{lem2}.  Now if $w(t)$ corresponds
to a periodic orbit $P=(w,T)\subset W_E$ intersecting
$U_{\widetilde M}$, with $ P\neq P_{2,E},$ then by \eqref{eqq1} and \eqref{eqq2} we obtain
$$CZ(P) > 2 \left(\left \lfloor \frac{\Delta \eta(w(0),u(0),[t^-,t^+])}{2\pi}
\right \rfloor -1\right )
 \geq 2 \left(\left \lfloor \frac{\widetilde M}{2\pi} \right \rfloor -1\right ).$$ Therefore, given an integer $M>0$ we can take $\widetilde U_M = U_{\widetilde
M}$ for $\widetilde M= 2\pi (M+1)$. This proves i).

Finally there exists $0<E_M<E_0$ so that $\partial S_E \subset \widetilde U_M$ for all $0<E<E_M$. Hence if $P \subset W_E\setminus P_{2,E},0<E<E_M,$ is linked to $P_{2,E}$ then $P$ must intersect $\partial S_E \setminus P_{2,E}$ and therefore must also intersect $\widetilde U_M$. This implies $CZ(P)>M$ and proves ii). \end{proof}

The next proposition shows that $W_0 \setminus \{p_c\}$ is dynamically convex, i.e., all of its periodic orbits have Conley-Zehnder index $\geq 3$.
\begin{prop}\label{prop_ghomi}Let $P \subset W_0 \setminus \{p_c\}$ be a periodic orbit. Then $CZ(P) \geq 3$.  \end{prop}

\begin{proof} Assume $P \subset \dot S_0$, the case $\dot S_0'$ is analogous. Let $\delta>0$ be small enough such that $\overline {B_\delta(p_c)} \cap
P = \emptyset$.  The set $\dot S_0 \setminus
B_\delta(p_c)$ is strictly convex in the sense that any hyperplane tangent to it is a non-singular support hyperplane. By a theorem of M. Ghomi \cite[Theorem 1.2.5]{ghomi}, $\dot S_0 \setminus
B_\delta(p_c)$ can be extended to a strictly convex smooth
hypersurface $\widetilde S_0$ which is diffeomorphic to $S^3$.
Since $P\subset S_0\cap\widetilde S_0$ and $CZ(P)$ does not depend on the
fact that $P$ lies either in $S_0$ or in $\widetilde S_0$, we must
have $CZ(P) \geq 3$ by a theorem of Hofer, Wysocki and Zehnder,
see \cite[Theorem 3.4]{convex}. \end{proof}

Using Propositions \ref{prop_mu} and \ref{prop_ghomi}, we prove the following result.

\begin{prop}\label{prop_fracconvexo} There exists $E^*>0$ such that for each $0<E<E^*$, we have $CZ(P)\geq 3$ for all periodic
orbits $P \subset W_E\setminus P_{2,E}$. \end{prop}

\begin{proof}Let $E_0>0$
and $\widetilde U_3 \subset U_0 \subset U_*$ be as in Lemma
\ref{lem2} and Proposition \ref{prop_mu}. Let $0<E_1<E_0$ be such that
$\left< H_{ww}(w) \cdot u, u\right>
> k_0>0$ for all $w \in W_E \setminus \widetilde U_3,0<E<E_1,$ and $u\in T_w S_{w}$
with $|u|=1$. As in the proof of Proposition \ref{prop_mu}, we find a
constant $\bar \eta>0$ such that \begin{equation} \label{eqq3}
\dot \eta
> \bar \eta,\end{equation} for any non-vanishing transverse linearized solution $\alpha_1(t) + i
\alpha_2(t) \in \R^+ e^{i \eta(t)}$ along a segment of trajectory
in $W_E \setminus \widetilde U_3$, where $0<E<E_1$.

Arguing indirectly, assume that there
exists a sequence $E_n \to 0^+, 0<E_n< E_1$ such that each $W_{E_n}$ admits
a periodic orbit $P_n=(w_n,T_n)\neq P_{2,E},$ with $CZ(P_n) <3$. By definition
of $\widetilde U_3$, we have $P_n \subset W_{E_n} \setminus
\widetilde U_3$ for all $n$. From estimate \eqref{eqq3}, we get
that $T_n$ is uniformly bounded and therefore by Arzel\`a-Ascoli theorem
we find a subsequence of $E_n$, also denoted by $E_n$, such that
$T_n \to T
>0$ and $w_n \to w$, where $P=(w,T)\subset W_0\setminus \{p_c\}$ is a periodic orbit. Here the periods of the periodic orbits correspond to the Hamiltonian function $H$.
By the definition of the Conley-Zehnder index we see that $CZ(P)
\leq \liminf_{n \to \infty} CZ(P_n)<3$. This leads to a
contradiction with Proposition \ref{prop_ghomi} and shows that there exists $0<E^* < E_1$ so that all periodic orbits $P \subset W_E \setminus P_{2,E}$ have $CZ(P)\geq 3$ for all $0<E<E^*$.
\end{proof}

Propositions \ref{prop_mu} and \ref{prop_fracconvexo} imply i) and ii) of Proposition \ref{prop_EM}, remaining only to prove iii).

For $E>0$ sufficiently small, consider the contact form $\lambda_E$ on $W_E$ given in Proposition \ref{prop_step1} so that its Reeb flow is equivalent to the Hamiltonian flow of $H$ restricted to $W_E$ under a reparametrization of the trajectories. This reparametrization also reparametrizes  the transverse linearized flow.

From Propositions \ref{prop_p2e}, \ref{prop_mu}-ii) and \ref{prop_fracconvexo}, we know that for all $E>0$ sufficiently small, $\lambda_E$ is weakly convex (i.e., all of its periodic orbits have Conley-Zehnder index $\geq 2$), $P_{2,E}$ is the only periodic orbit of $\lambda_E$ with Conley-Zehnder index $2$, and if $P\subset W_E \setminus P_{2,E}$ is a periodic orbit of $\lambda_E$ with Conley-Zehnder index $3$ then $P$ is not linked to $P_{2,E}$.

Let $\lambda_n \to \lambda_E$ in $C^\infty$ as $n \to \infty$ with $E>0$ small enough as above. Since $P_{2,E}$ is unknotted and hyperbolic, it follows that for all large $n$, $\lambda_n$ admits a unique unknotted and hyperbolic periodic orbit $P_{2,n}$ satisfying $CZ(P_{2,n})=2$ and, moreover, $P_{2,n} \to P_{2,E}$ in $C^\infty$ as $n \to \infty$.

We fix the symplectic frame $\{X_1,X_2\}\subset TW_E$, defined in \eqref{mfX}, which is transverse to the Reeb vector field $X_{\lambda_E}$, where $E>0$ is sufficiently small. This frame was used to project the linearized flow of $\lambda_E$ and in the same way we use it to project the linearized flow of $\lambda_n$ for all large $n$ in order to estimate the Conley-Zehnder index of periodic orbits of $\lambda_n$.

Since the flow of $\lambda_n$ converges to the flow of $\lambda_E$ in $C^\infty_{\rm loc}$ as $n \to \infty$ and the estimates in Lemma \ref{lem2}, Propositions \ref{prop_mu} and \ref{prop_fracconvexo} are open, we obtain similar estimates for the Conley-Zehnder index of periodic orbits of $\lambda_n$ for all large $n$. To be more precise, as in Lemma \ref{lem2} we can find small neighborhoods $U_{4\pi} \subset U_*\subset \U_{E^*}$ of $p_c$ so that the argument $\eta(t)$ of any non-vanishing transverse linearized solution along a maximal segment of trajectory of $\lambda_n$ contained in $U_*$ and intersecting $U_{4 \pi}$ satisfies $\Delta \eta >4\pi$ for all large $n$. Now proceeding as in Proposition \ref{prop_mu}, we find $\bar \eta>0$ so that \begin{equation}\label{eq_doteta} \dot \eta(t)> \bar \eta\end{equation} for any segment of trajectory outside $U_{4\pi}$, for all $E>0$ fixed sufficiently small and $n$  sufficiently large. The fact that $P_{2,n} \to P_{2,E}$ as $n \to \infty$ implies that, for all $E>0$ fixed sufficiently small, every periodic orbit $P$ of $\lambda_n$ in $W_E \setminus P_{2,n}$ which is linked to $P_{2,n}$ intersects $U_{4\pi}$ for all large $n$ and therefore satisfies $CZ(P)> 3$ .

Now arguing as in Proposition \ref{prop_fracconvexo}, we can show that for all $E>0$  fixed even smaller, $\lambda_n$ is weakly convex and $P_{2,n}$ is the only periodic orbit with Conley-Zehnder index $2$ for all large $n$. In fact, arguing indirectly,  assume that for all large $n$, $\lambda_n$ admits a periodic orbit $P_n\neq P_{2,n}$ with $CZ(P_n)\leq 2$. Then $P_n$ does not intersect $U_{4\pi}$ for all large $n$ and the period $T_n$ of $P_n$ is uniformly bounded in $n$. This last statement follows from \eqref{eq_doteta}. From Arzel\`a-Ascoli theorem, we can extract a subsequence so that $P_n \to \bar P$ in $C^\infty$ as $n \to \infty$, where $\bar P$ is a periodic orbit of $\lambda_E$, geometrically distinct to $P_{2,E}$ and  $CZ(\bar P)\leq \liminf_{n \to \infty} CZ(P_n)\leq 2$. This contradicts Proposition \ref{prop_fracconvexo} and concludes the proof of Proposition \ref{prop_EM}-iii).

Thus the proof of Proposition \ref{prop_EM} is complete.

\section{Proof of Proposition \ref{prop_step3}}\label{sec_step3}

In this section we prove Proposition \ref{prop_step3} which is restated below.

\begin{prop}\label{prop_step3b}If $E>0$ is sufficiently small, then there exists $J_E \in \J(\lambda_E)$ so that the almost complex structure $\tilde J_E=(\lambda_E,J_E)$ on the symplectization $\R \times W_E$ admits a pair of finite energy  planes $\tilde u_{1,E}=(a_{1,E},u_{1,E}),\tilde u_{2,E}=(a_{2,E},u_{2,E}):\C \to \R\times W_E$ and their projections $u_{1,E}(\C),u_{2,E}(\C)$ into $W_E$ are, respectively, the hemispheres $U_{1,E},U_{2,E} \subset \partial S_E$ defined in local coordinates by \eqref{hemis}, both asymptotic to the hyperbolic periodic orbit $P_{2,E}$.  \end{prop}

To prove Proposition \ref{prop_step3b}, we first define a suitable complex structure $J_E \in \J(\lambda_E)$ and then explicitly construct a pair of finite energy $\tilde J_E=(\lambda_E,J_E)$-holomorphic planes $\tilde u_{1,E}$ and $\tilde u_{2,E}$ satisfying the desired properties for $E>0$  sufficiently small. By explicit construction, we mean that $\tilde u_{i,E},i=1,2,$ are given in terms of solutions of certain one-dimensional O.D.E.'s.

Consider the local coordinates $(q_1,q_2,p_1,p_2)\in V$ as in Hypothesis 1. Recall from Proposition \ref{prop_step1} that the contact form $\lambda_E$ on $W_E$ coincides near $\partial S_E$ with $$\lambda_E =\frac{1}{2} \sum_{i=1}^2 p_i dq_i - q_i dp_i|_{K^{-1}(E)}$$
and therefore the Reeb vector field $X_{\lambda_E}$ near $\partial S_E$ is given by
\begin{equation}\label{eq_XlambdaE}
X_{\lambda_E}=\frac{1}{-\bar \alpha I_1 + \bar \omega I_2}(-\bar \alpha q_1, \bar \omega p_2, \bar \alpha p_1, -\bar \omega q_2).
\end{equation}
The Hamiltonian function in these coordinates is $$K=-\alpha I_1 + \omega I_2 + O(I_1^2 + I_2^2),$$ where $I_1 = q_1p_1$ and $I_2 = \frac{q_2^2 + p_2^2}{2}$. We know that if  $|I_1|,E>0$ are small enough, we can write \begin{equation}\label{i2pseudo} I_2 = I_2(I_1,E) = \frac{E}{\omega} +\frac{\alpha}{\omega}I_1 + O(I_1^2 + E^2).\end{equation} Thus for each $E>0$ small we find $I_1^-(E)=-\frac{E}{\alpha} + O(E^2)<0$ so that $I_2(I_1^-(E),E)=0$ and $I_2(I_1,E)>0$ if $I_1>I_1^-(E)$.

Given  $z=(q_1,q_2,p_1,p_2)\in V$, let $S_z = K^{-1}(K(z))$. We assume $E=K(z)>0$. Then $T_z S_z = {\rm span} \{e_1,e_2,e_3 \}_z$, where $e_i=j_i e_0$, $e_0 = \nabla K\neq 0$, and $j_i,i=1,2,3,$ are  defined in \eqref{eqframe}. The vector $e_3$ is parallel to $X_{\lambda_E}$ and the contact structure $\xi = \ker \lambda_E$ is isomorphic to $\Pi_{12}:={\rm span}\{e_1,e_2\}$, as a tangent hyperplane distribution, via the projection $\pi_{12}:TW_E \to \Pi_{12}$ along $e_3$. Let $\bar e_i \in \xi$ be the vectors determined by $\pi_{12}(\bar e_i) = e_i, i=1,2$. We define locally the $d\lambda_E$-compatible complex structure $J_E:\xi \to \xi$ in these coordinates  by \begin{equation}\label{Jxi}J_E \cdot \bar e_1 = \bar e_2.\end{equation}
Using a suitable auxiliary Riemannian metric on $W_E$, we can smoothly extend $J_E$ to all $W_E$. We have $$\begin{aligned}e_1 & = (\bar \omega p_2, \bar \alpha q_1,\bar \omega q_2, \bar \alpha p_1),\\ e_2 & = (\bar \omega q_2, \bar \alpha p_1,-\bar \omega p_2, -\bar \alpha q_1), \end{aligned}$$ where both $\bar \alpha = \alpha - \partial_{I_1} R$ and $\bar \omega = \omega + \partial_{I_2} R$ depend on $I_1,I_2$.

Let $\tilde J_E=(\lambda_E,J_E)$ be the natural almost complex structure on the symplectization $\R \times W_E$. We look for a pair of $\tilde J_E$-holomorphic rigid planes $\tilde u_{1,E},\tilde u_{2,E}:\C \to \R \times W_E$, both asymptotic to $P_{2,E}=(x_{2,E}, T_{2,E})$ at $+\infty$ and  approaching $P_{2,E}$ through opposite directions. Assuming that $x_{2,E}(0)=(0,r_E,0,0)$, we write in local coordinates $$x_{2,E}(t) = \left(0,r_E \cos \frac{2 \pi t}{T_{2,E}},0, -r_E \sin \frac{2 \pi t}{T_{2,E}}\right),t\in \R,$$ where $\frac{r_E^2}{2}=I_2(0,E)\sim \frac{E}{\omega}>0,$ if $E>0$ is small, see \eqref{i2pseudo}. Since $X_{\lambda_E}(0,q_2,0,p_2)=\frac{1}{I_2}(0,p_2,0,-q_2)$ we see that the period of $P_{2,E}$ with respect to the Reeb flow of $\lambda_E$ is given by $$T_{2,E}=\pi r_E^2= \frac{2 \pi E}{\omega}+O(E^2).$$ Note that $T_{2,E}$  differs from the period $T_{2,E}^H= \frac{2\pi}{\omega}+O(E)$ of the Hamiltonian flow of $H$.

Our candidates for the planes $\tilde u_{1,E}$ and $\tilde u_{2,E}$ are the ones which project into the embedded $2$-sphere $$\partial S_E=\{q_1+p_1=0\} \cap K^{-1}(E),$$ with $E>0$ small. For simplicity  these planes will be defined in the cylinder $\R \times S^1$ with the help of the change of coordinates $\R \times S^1 \ni (s,t)\mapsto e^{2\pi (s+it)}\in \C$.

Let us construct $\tilde u_{1,E}=(a_{1,E},u_{1,E})$ which will be denoted simply by $\tilde u=(a,u)$. The case $\tilde u_{2,E}$ is completely analogous due to the symmetry of our coordinates. Trying an ansatz, we assume that $\tilde u(s,t)=(a(s,t),u(s,t)):\R \times S^1 \to \R \times W_E$ has the form \begin{equation}\label{eqpseudou}\begin{aligned}u(s,t) & =(q_1(s,t),q_2(s,t),p_1(s,t),p_2(s,t))\\ & =(-h(s),f(s)\cos 2 \pi t,h(s),-f(s) \sin 2 \pi t),\end{aligned}\end{equation} where $f:\R \to (0,\infty)$ and $h:\R \to \R \setminus\{0\}$ are suitable smooth functions to be determined below. The conditions for $\tilde u$ to be $\tilde J_E$-holomorphic are expressed by the equations
\begin{equation}\label{eqpseudosist} \left\{ \begin{array}{c}d(u^* \lambda_E \circ i)=0,\\ \pi u_s + J_E(u) \pi u_t=0,  \end{array} \right. \end{equation} where $\pi:TW_E \to \xi$ is the projection along the Reeb vector field $X_{\lambda_E}$ given by \eqref{eq_XlambdaE}. Using \eqref{eqpseudou}, one can easily check that $$d(u^* \lambda_E \circ i) =\frac{1}{2}(q_1 \Delta p_1 - p_1 \Delta q_1 +q_2 \Delta p_2 - p_2 \Delta q_2)ds \wedge dt =0,$$ and, therefore, the first equation in \eqref{eqpseudosist} is automatically satisfied.

Now we find conditions on $f(s)$ and $h(s)$ so that the second equation of \eqref{eqpseudosist} is satisfied. Restricting the frame $\{\bar e_1,\bar e_2\} \subset \xi$ to $\partial S_E,$ where $p_1=-q_1$, we obtain
\begin{equation}\label{eibarra} \begin{aligned} \bar e_1 & =e_1 - \lambda_E(e_1) X_{\lambda_E} \\ & = (\bar \omega p_2,\bar \alpha q_1,\bar \omega q_2, -\bar \alpha q_1) -\frac{1}{2}\frac{q_1(q_2+p_2)(\bar \omega - \bar \alpha)}{\bar \alpha I_1 - \bar \omega I_2}(-\bar \alpha q_1,\bar \omega p_2,-\bar \alpha q_1,-\bar \omega q_2), \\ \bar e_2 & =e_2 - \lambda_E(e_2) X_{\lambda_E} \\ & = (\bar \omega q_2,-\bar \alpha q_1,-\bar \omega p_2, -\bar \alpha q_1) -\frac{1}{2}\frac{q_1(q_2-p_2)(\bar \omega - \bar \alpha)}{\bar \alpha I_1 - \bar \omega I_2}(-\bar \alpha q_1,\bar \omega p_2,-\bar \alpha q_1,-\bar \omega q_2).   \end{aligned}\end{equation}

A direct computation shows that \begin{equation}\label{piust}\begin{aligned}\pi u_s  &=  u_s - \lambda_E(u_s) X_{\lambda_E} = (-h'(s),f'(s) \cos 2 \pi t, h'(s), -f'(s) \sin 2 \pi t),\\ \pi u_t  &=  u_t - \lambda_E(u_t) X_{\lambda_E}  = (0,-2\pi f(s) \sin 2 \pi t, 0,-2 \pi f(s) \cos 2 \pi t)\\ &  - \frac{\pi f(s)^2}{\bar \alpha h(s)^2 + \frac{\bar \omega f(s)^2}{2}}(\bar \alpha h(s),-\bar \omega f(s) \sin 2 \pi t, \bar \alpha h(s),-\bar \omega f(s) \cos 2 \pi t)  \end{aligned}\end{equation}

Using \eqref{eibarra} e \eqref{piust} we find

\begin{equation}\begin{aligned} \bar e_1= & \frac{\bar \omega}{2h'(s)} \, (q_2-p_2) \, \pi u_s - \left( \frac{\bar \omega - N(\bar \omega - \bar \alpha)\bar \alpha h(s)^2}{2N \bar \alpha \pi f(s)^2 h(s)}\right)(q_2+p_2) \, \pi u_t,\\  \bar e_2= & -\frac{\bar \omega}{2h'(s)} \, (q_2+p_2) \, \pi u_s - \left( \frac{\bar \omega - N(\bar \omega - \bar \alpha)\bar \alpha h(s)^2}{2N \bar \alpha\pi f(s)^2 h(s)} \right) (q_2-p_2) \, \pi u_t,\end{aligned} \end{equation} where $N = \left(\bar \alpha h(s)^2 + \frac{\bar \omega}{2} f(s)^2\right)^{-1}$. Now using \eqref{Jxi} and the second equation of \eqref{eqpseudosist}, we end up finding the following differential equation \begin{equation}\label{eqdifpseudo}h'(s) =-\frac{2 \pi \bar \alpha \bar \omega h(s) f(s)^2}{\bar \omega^2 f(s)^2 + 2 \bar \alpha ^2 h(s)^2}. \end{equation} Recall that $u(s,t)\in K^{-1}(E),\forall (s,t)$. Using \eqref{i2pseudo}, we have \begin{equation}\label{eqfdes} f(s)^2 = \frac{2}{\omega}E - \frac{2\alpha}{\omega} h(s)^2 +  O(h(s)^4 + E^2)\end{equation}for all $s\in \R$. Hence we may view \eqref{eqdifpseudo} as a differential equation of the type $$h'(s) = G(h(s)),$$ where $G=G(h)$ is a smooth function defined in the interval $[-h_E^*,h_E^*]$, where $$h_E^*=\sqrt{-I_1^-(E)}>0$$ is the positive small value of $h$ so that $f$ vanishes in \eqref{eqfdes}, with fixed $E>0$ small enough, see definition of $I_1^-(E)$ above. Thus $G$  vanishes at $h=0$ and at $h=\pm h_E^*$, $G$ is positive in $(-h_E^*,0)$ and negative in $(0,h_E^*)$. Any solution $h=h(s)$ of \eqref{eqdifpseudo} with $h(0)\in (0,h_E^*)$ is strictly decreasing and satisfies \begin{equation} \label{limiteh} \lim_{s \to -\infty} h(s) = h_E^* \mbox{ and } \lim_{s \to +\infty} h(s) =0. \end{equation} It follows from \eqref{eqfdes} that \begin{equation}\label{limitef} \lim_{s \to -\infty} f(s) =0 \mbox{ and } \lim_{s \to +\infty} f(s) =  r_E>0.\end{equation}  Equations \eqref{limiteh} and \eqref{limitef} imply that the loops $S^1 \ni t\mapsto u(s,t)$ converge to $S^1 \ni t\mapsto x_{2,E}(T_{2,E}t)$ in $C^\infty$ as $s \to +\infty$. Moreover, $u(s,t) \to  (-h_E^*,0,h_E^*,0)$ uniformly in $t$ as $s \to -\infty$.

The function $a(s,t)$ can be  defined by
\begin{equation}\label{a}
a(s,t) = \pi \int_0^s f(\tau )^2 d\tau,
\end{equation}
depends only on $s\in \R$ and
\begin{equation}\label{eq_asat}
\begin{aligned}
0 & <a_s(s,t)=\lambda_E(u_t(s,t))=\pi f(s)^2 \to \pi r_E^2=T_{2,E} \mbox{ as } s \to +\infty\\
0 & =a_t(s,t)=\lambda_E(u_s(s,t)).
\end{aligned}
\end{equation}

Let $\tilde u=(a,u):\R\times S^1 \to \R\times W_E$ with $a(s,t)$ given by \eqref{a} and $u(s,t)$ defined as in \eqref{eqpseudou}, where $f(s)$ is determined in terms of $h(s)$ by \eqref{eqfdes} and $h(s)$ is a solution of \eqref{eqdifpseudo} satisfying $h(0)\in (0,h_E^*)$. By construction, $\tilde u$ is a $\tilde J_E$-holomorphic curve since satisfies \eqref{eqpseudosist}. Moreover, \eqref{eq_asat} implies that $\tilde u$ has finite energy $0<E(\tilde u) = T_{2,E}<\infty$. In fact, by Stoke's theorem one has $$\begin{aligned} \int_{\R \times S^1} \tilde u^* d \lambda_\varphi = & \lim_{R \to +\infty} \left\{-\varphi(a(-R)) \pi f(-R)^2 + \varphi(a(R))\pi f(R)^2\right\}\\ = & \varphi(+\infty)\pi r_E^2 \\ = & \varphi(+\infty)T_{2,E},\end{aligned}$$ where $\lambda_\varphi(a,u) =\varphi(a) \lambda_E(u)$,  $\varphi:\R \to [0,1]$ is smooth and  $\varphi'\geq0$. Thus $E(\tilde u)= T_{2,E}$.

The mass $\lim_{s \to -\infty}\int_{\{s\} \times S^1} u^* \lambda_E$ of $\tilde u$ at $s=-\infty$ is equal to 0 and this implies that $\tilde u$ has a removable singularity at $s=-\infty$, see \cite{props2}. Removing it, we obtain the desired $\tilde J_E$-holomorphic plane $\tilde u_{1,E}=(a_{1,E},u_{1,E}):\C \to \R \times W_E$ asymptotic to $P_{2,E}$ at its positive puncture.  One can check that $u_{1,E}$ is an embedding and by construction we have $u_{1,E}(\C)= U_{1,E}$, where $U_{1,E}=\{q_1+p_1 = 0,q_1<0\} \cap K^{-1}(E)$ is a hemisphere of the embedded $2$-sphere $\partial S_E$, which has $P_{2,E}$ as the equator.

Now starting with a solution $h=h(s)$ of \eqref{eqdifpseudo} with $h(0)\in (-h_E^*,0)$, we can proceed in the same way in order to find a finite energy $\tilde J_E$-holomorphic plane $\tilde u_{2,E}=(a_{2,E},u_{2,E}):\C \to \R \times W_E$, also asymptotic to $P_{2,E}$ at its positive puncture, but now converging to $P_{2,E}$ through the opposite direction, i.e., $u_{2,E}(\C)=U_{2,E}$, where $U_{2,E}=\{q_1+p_1 = 0,q_1>0\} \cap K^{-1}(E)$ is the other hemisphere of $\partial S_E$.  This completes the proof of Proposition \ref{prop_step3b}.

\begin{figure}[ht!!]
  \centering
  \includegraphics[width=0.25\textwidth]{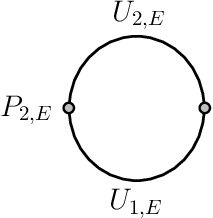}
  \caption{So far we have obtained the binding orbit $P_{2,E}$ and the pair of rigid planes $U_{1,E}=u_{1,E}(\C)$ and $U_{2,E}=u_{2,E}(\C).$}
  \label{fig_planos}
\hfill
\end{figure}

\section{Proof of Proposition \ref{prop_step4}}\label{sec_step4}

In this section we prove Proposition \ref{prop_step4} which is restated below.

\begin{prop}\label{prop_step4b} If $E>0$ is sufficiently small, then the following holds:
\begin{itemize}
\item[i)] There exist a sequence of nondegenerate weakly convex contact forms $\lambda_n$ on $W_E$ satisfying $\ker \lambda_n = \ker \lambda_E, \forall n,$ $\lambda_n \to \lambda_E$ in $C^\infty$ as $n \to \infty$ and a sequence of $d\lambda_E-$compatible complex structures $J_n \in \J_{\rm reg}(\lambda_n) \subset \J(\lambda_E)$ satisfying $J_n \to J_E$ in $C^\infty$ as $n \to \infty$ so that for all $n$ sufficiently large, $\tilde J_n=(\lambda_n,J_n)$ admits a stable finite energy foliation $\tilde \F_n$ of $\R \times W_E$ which projects in $W_E$  onto a $3-2-3$ foliation $\F_n$ adapted to $\lambda_n$. Let $P_{3,n},P_{2,n},P_{3,n}'$ be the nondegenerate binding orbits of $\F_n$ with Conley-Zehnder indices $3,2,3,$ respectively. Then $P_{2,n} \to P_{2,E}$ as $n \to \infty$ and there exist unknotted periodic orbits $P_{3,E}\subset \dot S_E$ and $P_{3,E}'\subset \dot S_E'$ of $\lambda_E$, both with Conley-Zehnder index $3$, such that  $P_{3,n} \to P_{3,E}$ and $P_{3,n}'\to P_{3,E}'$ in $C^\infty$ as $n \to \infty$.
\item[ii)] After changing coordinates with suitable contactomorphisms $C^\infty$-close to the identity map, we can assume that $P_{3,n}=P_{3,E},P_{2,n}=P_{2,E},P_{3,n}'=P_{3,E}'$ as point sets in $W_E$ for all large $n$, and that there are sequences of constants $c_{3,n},c_{3,n}',c_{2,n} \to 1$ as $n \to \infty$ such that $$ \begin{aligned} \lambda_n|_{P_{3,E}} & =  c_{3,n}\lambda_E|_{P_{3,E}}, \\ \lambda_n|_{P_{3,E}'}&=  c_{3,n}'\lambda_E|_{P_{3,E}'}, \\ \lambda_n|_{P_{2,E}}&= c_{2,n}\lambda_E|_{P_{2,E}}.\end{aligned}$$
\end{itemize}
\end{prop}

In order to prove Proposition \ref{prop_step4b}, fix $E>0$ sufficiently small and let $\lambda_n=f_n \lambda_E, f_n\in C^\infty(W_E,(0,+\infty)),$ be a sequence of nondegenerate contact forms on $W_E$ satisfying $\lambda_n \to \lambda_E$ in $C^\infty$ as $n \to \infty$. The existence of such a sequence is proved in \cite[Proposition 6.1]{convex}.

For each large $n$, we choose $J_n \in \J_{\rm reg} (\lambda_n)\subset \J(\lambda_E)$ so that the pair $\tilde J_n=(\lambda_n,J_n)$ admits a stable finite energy foliation $\tilde \F_n$ as in Theorem \ref{sfef}. Since $\J_{\rm reg} (\lambda_n)$ is dense in the set $\J(\lambda_E)$ of $d\lambda_E$-compatible complex structures on $\xi$, with the $C^\infty$-topology, we can assume that $J_n \to J_E$ in $C^\infty$ as $n \to \infty$, where $J_E\in \J(\lambda_E)$ is given by Proposition \ref{prop_step3}.

The projection $\F_n=p(\tilde \F_n)$ onto $W_E$ is a global system of transversal sections adapted to $\lambda_n$, where $p:\R \times W_E \to W_E$ denotes the projection onto the second factor. We know from Proposition \ref{prop_step2} that, for all large $n$, $\lambda_n$ is weakly convex and admits only one periodic orbit $P_{2,n}$ with Conley-Zehnder index $2$ and, moreover, any periodic orbit of $\lambda_n$ with Conley-Zehnder index $3$ is not linked to $P_{2,n}$. Hence $\F_n$ must be a $3-2-3$ foliation adapted to $\lambda_n$ as described in Theorem \ref{sfef}-iv). Denote by $P_{3,n}=(x_{3,n}, T_{3,n})$ and $P_{3,n}'=(x_{3,n}',T_{3,n}')$ the binding orbits of $\F_n$ having Conley-Zehnder index $3$.

From Proposition \ref{prop_EM} we find $E^*>0$ small and a neighborhood $U_{4 \pi}\subset \U_{E^*}$ of $p_c$ so that for all fixed $0<E<E^*$ sufficiently small and large $n$, $P_{3,n}$ and $P_{3,n}'$ do not intersect $U_{4\pi}$, each one lying on a different component of $W_E \setminus U_{4 \pi}$, in particular, on a different component of  $W_E \setminus \partial S_E$, say $\dot S_E$ and $\dot S_E'$, respectively, without loss of generality. Moreover, from the proof of Proposition \ref{prop_EM}-iii) we see that there exists $\bar \eta>0$ so that the argument $\eta(t)$ of any non-vanishing transverse linearized solution along $x_{3,n}$ and $x_{3,n}'$ satisfies $\dot \eta > \bar \eta$, which implies that $T_{3,n}$ and $T_{3,n}'$ are uniformly bounded in $n$. This follows from the fact that $CZ(P_{3,n})=CZ(P_{3,n}')=3, \forall n$. From Arzel\`a-Ascoli theorem we find periodic orbits $P_{3,E}$ and $P_{3,E}'$ of $\lambda_E$, geometrically distinct to $P_{2,E}$, each one lying on a different component of $W_E \setminus \partial S_E$, so that $P_{3,n} \to P_{3,E}$ and $P_{3,n}' \to P_{3,E}'$ in $C^\infty$ as $n \to \infty$, up to extraction of a subsequence. Since $CZ(P_{3,n}) = CZ(P_{3,n}')=3$, since $\lambda_E$ is weakly convex and since $P_{2,E}$ is the only periodic orbit with Conley-Zehnder index $2$, we get that $CZ(P_{3,E}) = CZ(P_{3,E}') = 3$. Moreover, $P_{3,E}$ and $P_{3,E}'$ are simply covered since otherwise their indices would be $\geq 5$, see \cite{convex}. This also implies that $P_{3,E}$ and $P_{3,E}'$ are unknotted. This proves Proposition \ref{prop_step4b}-i).

The proof of Proposition \ref{prop_step4b}-ii) makes use of Moser's trick but a preparation  is necessary. Let $\U\subset W_E$ be a small tubular neighborhood of $P$, where $P$ is one of the orbits $P_{2,E}$, $P_{3,E}$ or $P_{3,E}'$ and denote by $P_n$ the corresponding sequence of periodic orbits of $\lambda_n$ converging to $P$ in $C^\infty$ as $n \to \infty$.  Since $P_n \to
P$ in $C^\infty$ as $n \to \infty$, we find a sequence of diffeomorphisms $\varphi_n:
W_E \to W_E$ supported in $\U$ so that $\varphi_n \to
\text{Id}|_{W_E}$ in $C^\infty$ as $n \to \infty$ and $\varphi_n(P) = P_n$. See \cite[Lemma 4.2]{HS3} for a proof. Letting
$\mu_n = \varphi_n^* \lambda_n$, we have that $\mu_n \to \lambda_E$ in $C^\infty$  as $n \to \infty$ and that $P$ is the geometric image of a periodic orbit of $\mu_n$.

Consider Martinet's coordinates $(\vartheta,x,y)\in S^1 \times
\R^2$ in $\U$ around $P \equiv S^1 \times  \{0\}$, see Appendix \ref{ap_basics}. In these
coordinates we have $\lambda_E = g (\vartheta,x,y) \cdot
(d\vartheta + x dy),$ where $g$ is a smooth and positive function. We may work on the universal cover $\R \times \R^2$ of $S^1 \times \R^2$ and in this case we consider $\vartheta \in \R$. We modify
$\mu_n$, using these coordinates, in order that \begin{equation}\label{partialx} \ker \mu_n|_P =
\ker \lambda_n|_P=\ker \lambda_E|_P=\text{span}\{\partial_x,\partial_y\}.\end{equation}

To achieve \eqref{partialx} let $a_{0,n}(\vartheta),b_{0,n}(\vartheta),\vartheta\in \R,$ be $1$-periodic functions such that $$\ker \mu_n|_{(\vartheta,0,0)} = \text{span}\{a_{0,n}(\vartheta)\partial_\vartheta + \partial_x,
b_{0,n}(\vartheta)\partial_\vartheta + \partial_y\}, \forall \vartheta\in \R.$$ Since $\mu_n \to \lambda_E$, we have $a_{0,n},b_{0,n} \to 0$ in $C^\infty$ uniformly in $\vartheta$ as $n \to \infty$. Fix a cut-off function $f:\U \to [0,1]$ which does not depend on $\vartheta$, is equal to $1$ near $P$ and $0$ outside a small tubular neighborhood of $P$. Let $\psi_{0,n}:\U \to \U$ be the diffeomorphism defined by $$\psi_{0,n}(\vartheta,x,y) = (\vartheta + f(\vartheta,x,y)(a_{0,n}(\vartheta)x +b_{0,n}(\vartheta)y), x,y).$$ Let $\bar \mu_n:= \psi_{0,n}^*\mu_n$. One easily checks that $d \psi_{0,n}(\vartheta,0,0) \cdot \partial_x=a_{0,n}(\vartheta)\partial_\vartheta + \partial_x$ and $d \psi_{0,n}(\vartheta,0,0) \cdot \partial_y=b_{0,n}(\vartheta)\partial_\vartheta + \partial_y$, hence $\bar \mu_n$ satisfies
\begin{equation}\label{eqcontato}
\ker \bar \mu_n|_P = \text{span}\{\partial_x,\partial_y\}=\ker\lambda_E|_P.
\end{equation}  Observe that $\psi_{0,n}$ fixes $P \equiv S^1 \times  \{0\}$ and that $\bar \mu_n \to \lambda_E$ in $C^\infty$ as $n \to \infty$ since $\psi_{0,n} \to \text{Id}|_\U$ in $C^\infty$ as $n\to \infty$. We replace $\mu_n$ with $\bar \mu_n$ keeping $\mu_n$ in the notation for simplicity.

Another modification on $\mu_n$ is necessary. From \eqref{eqcontato}, we can find $h_n:S^1 \to \R$ such that
\begin{equation}\label{eq_hn}
\mu_n|_{(\vartheta,0,0)}=h_n(\vartheta)\lambda_E|_{(\vartheta,0,0)}.
\end{equation}
Define $T_n(\vartheta)$ by
\begin{equation}\label{eq_Tn}
\int_0^{T_n(\vartheta)} h_n(\tau) d \tau = c_n\vartheta,
\end{equation}
where $c_n:= \int_0^1 h_n(\tau) d \tau.$ Note that $T_n(0)=0$ and $T_n(1)=1$.

Let $\psi_{1,n}:\U \to \U$ be the diffeomorphism defined by $$\psi_{1,n}(\vartheta,x,y) = \left( \vartheta +f(\vartheta,x,y)\left(T_n(\vartheta) - \vartheta \right),x,y\right),$$ where  $f$ is the cut-off function constructed above. Note that $h_n \to 1$ in $C^\infty$ uniformly in $\vartheta$ as $n \to \infty$ since $\mu_n \to \lambda_E$ in $C^\infty$ as $n \to \infty$. Let $\bar \mu_n = \psi_{1,n}^* \mu_n$. From \eqref{eq_hn}, \eqref{eq_Tn} and the fact that $\lambda_E|_{(\vartheta,0,0)} \cdot \partial_\vartheta=T$, $\forall\vartheta \in S^1$, where $T$ is the period of $P$ with respect to $\lambda_E$, we obtain that $ \bar \mu_n|_{(\vartheta,0,0)} \cdot \partial_\vartheta = c_n \lambda_E|_{(\vartheta,0,0)}\cdot \partial_\vartheta$. Thus equation \eqref{eqcontato} is improved by \begin{equation}\label{eqmultiplo} \bar \mu_n|_{(\vartheta,0,0)} = c_n \lambda_E|_{(\vartheta,0,0)}, \forall \vartheta \in S^1.\end{equation} Again we replace $\mu_n$ with $\bar \mu_n$ keeping $\mu_n$ in the notation. Since $h_n \to 1$, we have $c_n \to 1$ and, therefore, $T_n \to {\rm Id}$ and $\psi_{1,n} \to  \text{Id}|_{\U}$ in $C^\infty$ uniformly in $\vartheta$ as $n \to \infty$. This implies that $\mu_n \to \lambda_E$ in $C^\infty$ as $n \to \infty$. Note further that $\psi_{1,n}$ also fixes $P \equiv S^1 \times  \{0\}$.

Since the diffeomorphisms above are supported near $P$, we can change the initial sequence $\lambda_n$ to a new sequence of contact forms $\mu_n$ which satisfy \begin{equation}\label{eqmultiplo2}\mu_n|_P = c_n \lambda_E|_P \end{equation} simultaneously for each one of the periodic orbits $P_{3,E},P_{3,E}'$ and $P_{2,E}$ for constants $c_{3,n},c_{3,n}',c_{2,n} \to 1$ as $n \to \infty$. Observe that by construction $P_{3,E},P_{3,E}'$ and $P_{2,E}$ are geometric images of periodic orbits of $\mu_n$.

Finally, we use Moser's trick in order to modify $\mu_n$ so that $\ker \mu_n=\ker\lambda_E$ for all $n$.
Consider for $t\in [0,1]$ the family of $1$-forms $\mu_t = t
\mu_n + (1-t) \lambda_E$, which are contact forms for all large $n$. We fix $n$ large and omit the dependence of
the functions on $n$ to simplify the notation. We construct an
isotopy $\rho_t:W_E \to W_E, t\in [0,1],$  so that
\begin{equation}\label{eq_gt}
\rho_t^* \mu_t = g_t \lambda_E,
\end{equation}
where $g_t,t\in[0,1],$ is a
smooth family of positive functions, satisfying $g_0\equiv 1$.
This isotopy is generated by a time-dependent vector field $X_t$
on $W_E$, so that $X_t|_P \equiv 0$ for
$t\in[0,1]$. In fact in order to have these conditions, we choose
$X_t \subset \ker \mu_t$ satisfying
\begin{equation}\label{eqXt} i_{X_t} d\mu_t|_{\ker \mu_t} = (\lambda_E
- \mu_n)|_{\ker \mu_t}.\end{equation} Since $d\mu_t|_{\ker \mu_t}$
is nondegenerate, $X_t$ is uniquely defined for all $t\in[0,1]$.
From \eqref{eqXt} and since $X_t \subset \ker \mu_t$, we have
$$\begin{aligned} \frac{d}{dt} \rho_t^* \mu_t = &
\rho_t^*\left(\LL_{X_t}\mu_t + \frac{d}{dt}\mu_t \right)
\\ = & \rho_t^* \left(i_{X_t} d \mu_t + d i_{X_t} \mu_t + \mu_n -
\lambda_E \right) \\ = & \rho_t^*(a_t \mu_t) \\ = & (a_t \circ
\rho_t) \rho_t^* \mu_t,\end{aligned}$$
which implies that
$\rho_t^*\mu_t=e^{\int_0^t a_s \circ \rho_s ds} \lambda_E.$
Thus defining $g_t: W_E \to (0,\infty)$ by $$g_t = e^{\int_0^t a_s \circ \rho_s ds}$$ we obtain \eqref{eq_gt}.

For each large $n$, we define $\psi_n: W_E \to W_E$ by
$\psi_n:= \rho_1$, $\bar f_n: W_E \to (0,\infty)$ by $\bar f_n:=g_1$ and $\bar \lambda_n:= \psi_n^*
\mu_n$. Since $\mu_n \to \lambda_E$ in $C^\infty$ as $n \to \infty$, it follows from \eqref{eqXt} that $X_t,t\in [0,1],$ converges to the
null vector field in $C^\infty$ as $n \to \infty$. This implies
that $\psi_n \to \text{Id}_{W_E}$ and $\bar \lambda_n \to \lambda_E$ in $C^\infty$ as $n \to \infty$. Moreover, from \eqref{eq_gt} we see that $\bar\lambda_n=\bar f_n \lambda_E$ and therefore $\ker\bar\lambda_n=\ker\lambda_E$. Since
$\ker \mu_t|_P = \ker \mu_n|_P = \ker \lambda_E|_P = \text{span}
\{\partial_x, \partial_y\}$, we have from \eqref{eqXt} that
$X_t|_P \equiv 0,\forall t\in[0,1]$, hence $P$ is the geometric image of a periodic orbit of $\bar
\lambda_n$ and we still have $\bar \lambda_n|_P = c_n \lambda_E|_P$, with $c_n \to 1$ as $n\to \infty$.

This procedure is done simultaneously for the periodic orbits $P=P_{3,E},$ $P=P_{3,E}'$ and $P=P_{2,E}$, and we obtain a sequence of nondegenerate weakly convex contact forms $\bar \lambda_n$ on
$W_E$ satisfying the requirements
of the statement for the constants $c_{3,n},c_{3,n}'$ and $c_{2,n}$ constructed above. Notice that $c_{3,n}T_{3,E},c_{3,n}'T_{3,E}'$ and $c_{2,n}T_{2,E}$ are the periods of $P_{3,n},P_{3,n}'$ and $P_{2,n}$, with respect to the Reeb flow of $\bar \lambda_n,$ respectively.

Observe that the map $T_n:= \varphi_n\circ\psi_{0,n}\circ\psi_{1,n}\circ\psi_n:W_E \to W_E$ supported near $P_{2,E}\cup P_{3,E}\cup P_{3,E}'$ induces the contact form $\bar \lambda_n:= T_n^* \lambda_n$, with $\ker \bar \lambda_n = \ker \lambda_n = \ker \lambda_E$,  and also induces the complex structure $\bar J_n = T_n^* J_n\in \J_{\rm reg}(\bar \lambda_n)$. Let $\tilde {\bar J}_n$ be the almost complex structure induced by the pair $(\bar \lambda_n,\bar J_n)$. The map $\bar T_n: \R \times W_E \to \R \times W_E$ given by $(a,u)\mapsto (a,T_n(u))$ satisfies $\tilde {\bar J}_n= \bar T_n ^* \tilde J_n$ and therefore $\tilde {\bar J}_n$ admits a stable finite energy foliation $\bar {\tilde \F}_n:=\bar T_n^{-1}(\tilde \F_n)$ with the same properties of $\tilde \F_n$. We keep using the notation $\lambda_n,J_n,\tilde J_n,\tilde \F_n$ instead of $\bar \lambda_n,\bar J_n,\tilde {\bar J}_n, \bar {\tilde \F}_n$. Proposition \ref{prop_step4b}-ii) is proved.

\begin{figure}[ht!!]
  \centering
  \includegraphics[width=0.35\textwidth]{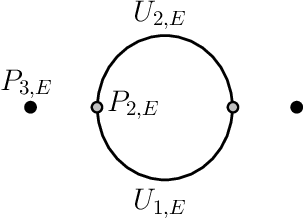}
  \caption{So far we have obtained the binding orbits $P_{2,E}$, $P_{3,E}$ and the pair of rigid planes $U_{1,E}=u_{1,E}(\C)$ and $U_{2,E}=u_{2,E}(\C).$}
  \label{fig_planosbindings}
\hfill
\end{figure}

\section{Proof of Proposition \ref{prop_step5}-${\rm i})$ }\label{sec_5-i)}

At this point, for all $E>0$ sufficiently small, we have sequences  $$\begin{aligned} C^\infty(W_E,(0,\infty)) \ni f_n & \to 1,\\ \J(\lambda_E)\supset \J_{\rm reg}(\lambda_n) \ni J_n & \to J_E,\end{aligned}$$ as $n \to \infty$, both convergences in  $C^\infty$. Moreover,  the following holds:

\begin{itemize}
\item $\lambda_E$ admits unknotted periodic orbits $$P_{3,E}=(x_{3,E},T_{3,E}), P_{3,E}'=(x_{3,E}',T_{3,E}') \mbox{ and } P_{2,E}=(x_{2,E},T_{2,E})$$ with $CZ(P_{3,E})=CZ(P_{3,E}')=3$ and $CZ(P_{2,E})= 2$. $P_{3,E}\subset \dot S_E$, $P_{3,E}'\subset \dot S_E'$ and $P_{2,E} \subset \partial S_E=\partial S_E'$.  $\tilde J_E=(\lambda_E,J_E)$ admits a pair of rigid planes $\tilde u_{1,E}=(a_{1,E},u_{1,E}),\tilde u_{2,E}=(a_{2,E},u_{2,E}):\C \to \R\times W_E$, both asymptotic to $P_{2,E}$, so that $u_{1,E}(\C)=U_{1,E}$ and $u_{2,E}(\C) = U_{2,E}$ are the hemispheres of $\partial S_E$.

\item For all $n$, $\lambda_n=f_n \lambda_E$ admits unknotted periodic orbits $$P_{3,n}=(x_{3,n},T_{3,n}), P_{3,n}'=(x_{3,n}',T_{3,n}') \mbox{ and } P_{2,n}=(x_{2,n},T_{2,n})$$ with $CZ(P_{3,n})=CZ(P_{3,n}')=3 \mbox{ and }CZ(P_{2,n})= 2.$ Moreover, $P_{3,n}=P_{3,E}$, $P_{3,n}'= P_{3,E}'$ and $P_{2,n}=P_{2,E}$ as point sets in $W_E$ and for all $n$ 
$$\begin{aligned} 
f_n|_{P_{3,E}}=c_{3,n}, f_n|_{P_{3,E}'}=c_{3,n}', f_n|_{P_{2,E}}=c_{2,n}, \\ df_n|_{P_{3,E}}=df_n|_{P_{3,E}'}=df_n|_{P_{2,E}}=0, 
\end{aligned}$$ 
where the constants $c_{3,n},c_{3,n}',c_{2,n} \to 1$ as $n \to \infty$. We also have $$\begin{aligned} T_{3,n} &= c_{3,n}T_{3,E} \to T_{3,E}, \\ T_{3,n}' &= c_{3,n}'T_{3,E}' \to T_{3,E}'\\ T_{2,n} &= c_{2,n}T_{2,E} \to T_{2,E},\end{aligned}$$  as  $n\to \infty.$
\item The pair $\tilde J_n=(\lambda_n,J_n)$ admits a stable finite energy foliation $\tilde \F_n$, which projects onto a $3-2-3$ foliation $\F_n$ adapted to $\lambda_n$ having $P_{3,n}, P_{3,n}'$ and $P_{2,n}$ as binding orbits.
\item $\tilde \F_n$ contains a pair of rigid planes $\tilde u_{1,n}=(a_{1,n},u_{1,n}),\tilde u_{2,n}=(a_{2,n},u_{2,n}):\C \to \R \times W_E$, both asymptotic to $P_{2,n}$. $u_{1,n}(\C) \cup P_{2,n} \cup u_{2,n}(\C)$ is a topological embedded $2$-sphere in $W_E$. In fact this sphere is at least $C^1$ and separates $W_E$ in two components whose closures, denoted by $S_n$ and $S_n'$, are both topological closed $3$-balls and satisfy $\partial S_n = \partial S_n'= u_{1,n}(\C) \cup P_{2,n} \cup u_{2,n}(\C)$. Moreover,  $P_{3,n} \subset \dot S_n:= S_n \setminus \partial S_n$ and $P_{3,n}' \subset \dot S_n':= S_n' \setminus \partial S_n'$ for all $n$.
\item The Reeb vector field $X_{\lambda_n}$ is transverse to $u_{1,n}$ and $u_{2,n}$, points inside $S_n$ at $u_{1,n}(\C)$ and points inside $S_n'$ at $u_{2,n}(\C)$.
\end{itemize}

In this section we prove Proposition \ref{prop_step5}-i) restated below.

\begin{prop}\label{prop_step5ib} If $E>0$ is sufficiently small then,  up to suitable reparametrizations and $\R$-translations,
$\tilde u_{i,n} \to \tilde u_{i,E},i=1,2,$ in $C^\infty_{\rm loc}$ as $n \to \infty$, where $\tilde u_{i,E}$ are the rigid planes obtained in Proposition \ref{prop_step3}. Moreover, given a small neighborhood $\U \subset W_E$ of $\partial S_E=u_{1,E}(\C)\cup P_{2,E} \cup u_{2,E}(\C)$, there exists $n_0\in \N$ so that if $n \geq n_0$, then $u_{1,n}(\C)\cup  P_{2,n} \cup u_{2,n}(\C) \subset \U.$
\end{prop}

We will need  the following lemma.

\begin{lem}\label{lemamin} $P_{2,E}$  is the periodic orbit of $\lambda_E$ with minimal action and, except by its covers, all other periodic orbits have action bounded below by a constant $T_{\rm min}>0$ which does not depend on $E>0$ small.\end{lem}
\begin{proof} According to Proposition \ref{prop_step1}, the contact form $\lambda_E$ near $P_{2,E}$ is given in local coordinates by $\lambda_E = \frac{1}{2}\sum_{i=1,2} p_i dq_i - q_idp_i|_{K^{-1}(E)}$. In such coordinates we have $P_{2,E} = \left\{ I_2 = \frac{E}{\omega} + O(E^2), q_1=p_1=0\right\}$, see \eqref{I2}. Then the action of $P_{2,E}$ satisfies $$\int_{P_{2,E}} \lambda_E = T_{2,E}  = \frac{2\pi E}{\omega} + O(E^2)\to 0 \mbox{ as } E \to 0^+.$$

Arguing indirectly, let us assume that there exists a sequence $E_k \to 0^+$ and periodic orbits $Q_k=(y_k,T_k)\in\P(\lambda_{E_k})$, which are not covers of $P_{2,E_k}$, with \begin{equation}\label{eqaction0}0<\int_{Q_k} \lambda_{E_k} \to 0 \mbox{ as } k \to \infty.\end{equation}

Let $\delta_0>0$ be any fixed small constant. From Proposition \ref{prop_step1}, we have $W_E \to W_0$ and $\lambda_E \to \bar \lambda_0$ in $C^\infty$ as $E \to 0^+$ outside $B_{\delta_0}(0)\subset \U_{E^*},$ where $\bar \lambda_0$ is a contact form on $\dot S_0 \cup \dot S_0'$ and $E^*>0$ is small.

If $Q_k$ is contained outside $B_{\delta_0}(0)$ for all $k$, then by  Arzel\`a-Ascoli theorem we find a constant Reeb orbit of $\bar \lambda_0$ on $\dot S_0\cup \dot S_0'=W_0 \setminus \{p_c\}$. This is a contradiction and we conclude that $Q_k$ must intersect $B_{\delta_0}(0)$. Now since $Q_k$ is not a cover of $P_{2,E_k}$, $Q_k$ must contain a branch $q_k$ inside $B_{2 \delta_0}(0)\setminus B_{\delta_0}(0)$, from $\partial B_{\delta_0}(0)$ to $\partial B_{2\delta_0}(0)$, for all $k$. This follows from the local behavior of the flow near the saddle-center equilibrium point. Therefore we find a uniform constant $\epsilon>0$ so that $$\int_{Q_k} \lambda_{E_k}>\int_{q_k} \lambda_{E_k} > \epsilon, \forall k \mbox{ large.} $$ This contradicts \eqref{eqaction0} and proves the existence of $T_{\rm min}>0$ as in the statement.
\end{proof}

Our aim is to show that, after reparametrization and $\R$-translation, $\tilde u_{i,n} \to \tilde u_{i,E}$ in $C^{\infty}_{\rm loc}$ as $n \to \infty$, $i=1,2$. We denote $\tilde u_{i,n}=(a_{i,n},u_{i,n})$ simply by $\tilde u_n=(a_n,u_n)$.

Now we have \begin{equation}\label{eqElimit} E(\tilde u_n) = \int_{\C} u_n^* d\lambda_n = T_{2,n} \to T_{2,E}=\frac{2\pi E}{\omega} + O(E^2) \mbox{ as } n \to \infty. \end{equation}

After a reparametrization $z \mapsto \alpha z,\alpha>0,$ and a translation in the $\R$-direction of $\tilde u_n$, we can assume that for all $n$ we have \begin{equation}\label{eqa1}  a_n(0)=  0=\min_{z\in \C} a_n(z) \mbox{ and } \int_{\D}u_n^*d\lambda_n =  T_{2,n} - \gamma_0,  \end{equation} where $0<\gamma_0\ll T_{2,E}$ is fixed and $\D = \{z\in \C: |z| \leq 1\}$ is the closed unit disk.

For each $n$, we have the Riemannian metric on $\R \times W_E$ induced by $\tilde J_n=(\lambda_n,J_n)$, given by $$\left< u, v \right>_n:= da(u)da(v)+ \lambda_n(u) \lambda_n(v)+ d\lambda_n (u,J_n \cdot v),$$ where $\lambda_n$ is seen as a $1$-form on $\R \times W_E$ via the projection onto the second factor $p:\R \times W_E \to W_E$. Its norm is denoted by $|u|_n = \sqrt{\left<u,u\right>_n}.$ Let $$\left|\nabla \tilde u_n(z)\right|_n=\max\left \{\left|\frac{\partial \tilde u_n}{\partial s}(z)\right|_n, \left|\frac{\partial \tilde u_n}{\partial t}(z)\right|_n\right\}$$

\begin{prop}\label{propbounds}If $E>0$ is sufficiently small, then $|\nabla \tilde u_n(z)|_n$ is uniformly bounded in $z\in \C$ and in $n\in \N$. Moreover,  there exists a finite energy $\tilde J_E$-holomorphic plane $\tilde u:\C \to \R \times W_E$, asymptotic to $P_{2,E},$ so that, up to extraction of a subsequence,  $\tilde u_n \to \tilde u$ in $C^\infty_{\rm loc}$ as $n \to \infty$. \end{prop}

To prove this proposition we make use of the Lemma \ref{lemamin}, the normalization \eqref{eqa1} and the topological lemma below.

\begin{lem}[{Ekeland-Hofer's Lemma }\cite{93}]\label{lemtop} Let $(X,d)$ be a complete metric space and $f:X \to [0,\infty)$ be a continuous function. Given $z\in X$ and $\epsilon>0$, there exists $0<\epsilon'\leq \epsilon$ and $z'\in X$ such that $\epsilon'f(z') \geq \epsilon f(z), d(z,z') \leq 2 \epsilon'$ and $f(y) \leq 2f(z')$ if $d(y,z')\leq \epsilon'$.
\end{lem}

\begin{proof}[Proof of Proposition \ref{propbounds}]Arguing indirectly we find a subsequence of $\tilde u_n$, still denoted by $\tilde u_n$, and a sequence $z_n \in \C$ such that $$|\nabla \tilde u_n(z_n)|_n \to \infty.$$ Take any sequence of positive numbers $\epsilon_n \to 0^+$ satisfying \begin{equation} \epsilon_n |\nabla \tilde u_n(z_n)|_n \to \infty.\end{equation} Using Lemma \ref{lemtop} for $X=(\C,|\cdot|_\C)$ and $f = |\nabla \tilde u_n(\cdot) |_n$, we assume furthermore, perhaps after slightly changing  $z_n,\epsilon_n$, that \begin{equation}\label{eqalimit}|\nabla \tilde u_n(z)|_n \leq 2 |\nabla \tilde u_n(z_n)|_n, \mbox { for all } |z-z_n| \leq \epsilon_n.\end{equation} Defining the new sequence of $\tilde J_n$-holomorphic maps $\tilde w_n: \C \to \R \times W_E$  by $$\tilde w_n(z)=(b_n(z),w_n(z)) = \left(a_n\left(z_n+\frac{z}{R_n}\right)-a_n(z_n),u_n\left(z_n+ \frac{z}{R_n}\right)\right),$$ with $R_n = |\nabla \tilde u_n(z_n)|_n$, we see from \eqref{eqalimit} that $$\begin{aligned} \tilde w_n(0) & \in \{0\} \times W_E, \\ |\nabla \tilde w_n(z)|_n & \leq 2, \forall |z|\leq \epsilon_n R_n.\end{aligned}$$ From an elliptic bootstrapping argument, we get $C^\infty_{\text{loc}}$-bounds for the sequence $\tilde w_n$ and, since $\epsilon_nR_n \to \infty$, after taking a subsequence we have $\tilde w_n \to \tilde w$ in $C^\infty_{\text{loc}}$ as $n\to \infty$, where $\tilde w=(b,w):\C \to \R \times W_E$ is a $\tilde J_E$-holomorphic plane satisfying $b(0)=0$ and $|\nabla \tilde w(0)|_E=1$, where $|\cdot |_E$ denotes the norm induced by $\tilde J_E=(\lambda_E,J_E)$ on $\R\times W_E$. Therefore $\tilde w$ is non-constant and from \eqref{eqElimit}, we have \begin{equation}\label{desigdl} E(\tilde w)\leq T_{2,E} \Rightarrow \int_{\C} w^*d\lambda_E \leq T_{2,E}.\end{equation} Since $P_{2,E}$ is hyperbolic, we get from Lemma \ref{lemamin} and Theorem \ref{asymptotic} that $\tilde w$ is asymptotic to $P_{2,E}$ at $\infty$. In particular, \begin{equation} \label{eqTE} E(\tilde w) = \int_\C w^* d\lambda_E = T_{2,E}.\end{equation}  We must have $\displaystyle \lim_{n \to \infty} |z_n| = 1$, otherwise for any $\epsilon>0$ we take a subsequence with either $|z_n|\geq 1+\epsilon$ or $|z_n|\leq 1-\epsilon,\forall n$, and then we obtain from \eqref{eqa1} that $$\begin{aligned} \int_{B_R(0)}w^*d\lambda_E = &  \lim_{n \to \infty} \int_{B_R(0)}w_n^*d\lambda_n  \\ \leq & \limsup_{n \to \infty} \int_{B_{\epsilon_nR_n}(0)}w_n^* d\lambda_n \\ = & \limsup_{n \to \infty} \int_{B_{\epsilon_n}(z_n)} u^*_n d\lambda_n  \\ \leq & \limsup_{n \to \infty} \max \{\gamma_0,T_{2,n} - \gamma_0\} =   \max \{\gamma_0,T_{2,E} - \gamma_0\}=T_{2,E} - \gamma_0\end{aligned} $$  for any $R>0$. This implies $\int_\C w^* d\lambda_E \leq T_{2,E} - \gamma_0<T_{2,E}$,  contradicting \eqref{eqTE}. Thus after taking a subsequence we assume $z_n \to z_* \in \partial \D$. The point $z_*$ is called a bubbling-off point for $\tilde u_n$. Since $z_*$ takes away $T_{2,E}>0$ from the $d\lambda_E$-energy, there cannot be any other bubbling-off point for the sequence $\tilde u_n$ other than $z_*$. Hence from \eqref{eqa1} and elliptic regularity we have $C^\infty_{\text{loc}}$-bounds for $\tilde u_n$ in $\C \setminus \{ z_*\}$. Taking a subsequence we have $\tilde u_n \to \tilde v$ for a $\tilde J_E$-holomorphic cylinder $\tilde v=(d,v):\C \setminus \{z_*\} \to \R \times W_E$ which is non-constant since from \eqref{eqa1} we have \begin{equation}\label{eqS2}\int_{\{|z|=2\}}v^* \lambda_E = \lim_{n \to \infty} \int_{\{|z|=2\}}u^*_n \lambda_n \geq T_{2,E} - \gamma_0 >0.\end{equation} Moreover, from Fatou's lemma, we get \begin{equation}\label{desigE3} E(\tilde v) \leq T_{2,E},\end{equation} and from \eqref{eqTE} we have \begin{equation}\label{eqdl0}\int_{\C \setminus \{z_*\}} v^* d\lambda_E = 0.\end{equation} From \eqref{eqS2} and \eqref{eqdl0} we see that $z_*$ is non-removable and from the characterization of finite energy cylinders with vanishing $d\lambda_E$-energy, see \cite[Theorem 6.11]{props2}, we find a periodic orbit $P\subset W_E$ of $\lambda_E$ so that $\tilde v$ maps $\C \setminus \{z_*\}$ onto $\R \times P$. In particular, from \eqref{desigE3} and Lemma \ref{lemamin} we must have $P=P_{2,E}$. However, from \eqref{eqa1}, the $\R$-component of $\tilde v$ is bounded from below by $0$. This contradiction proves the first assertion of this proposition.
From \eqref{eqa1}  and usual elliptic regularity, we obtain $C^\infty_{\text{loc}}$-bounds for our sequence $\tilde u_n$, thus we find a $\tilde J_E$-holomorphic plane $\tilde u=(a,u):\C \to \R \times W_E$ so that, up to extraction of a subsequence, $\tilde u_n \to \tilde u$ in $C^\infty_{\text{loc}}$ as $n\to \infty$. From \eqref{eqElimit} and \eqref{eqa1} we have  $\begin{aligned}  0<E(\tilde u) \leq T_{2,E} \end{aligned}$. From Lemma \ref{lemamin} and Theorem \ref{asymptotic}, $\tilde u$ is asymptotic to $P_{2,E}$ at $\infty$ and hence $E(\tilde u)=T_{2,E}$.
\end{proof}

\begin{lem}\label{lemvizi10} Let $\tilde u=(a,u):\C \to \R \times W_E$ be as in Proposition \ref{propbounds} so that $\tilde u_n \to \tilde u$ in $C^\infty_{\rm loc}$ as $n \to \infty$. Given an $S^1$-invariant open neighborhood $\W$ of the loop $S^1 \ni t \mapsto x_{2,E}(tT_{2,E})$ in $C^\infty(S^1,W_E)$,  there exists $R_0>0$ such that for all $R\geq R_0$ and all large $n$, the loop $t \mapsto u_n(Re^{2 \pi  i t})$ is contained in $\W$. In particular, given any small neighborhood $\U$ of $u(\C)\cup P_{2,E}$ in $W_E$, there exists $n_0 \in\N$ so that if $n \geq n_0$ then $u_n(\C) \subset \U$. \end{lem}

\begin{proof}
 First we claim that for every $\epsilon>0$ and every sequence $R_n \to +\infty$, there exists $n_0 \in \N$ such that for all $n \geq n_0$, we have $$\int_{\C\setminus B_{R_n}(0)} u_n^* d\lambda_n \leq \epsilon.$$ Arguing indirectly we may assume there exist $\epsilon >0$ and a sequence $R_n \to +\infty$ satisfying $$\int_{\C \setminus B_{R_n}(0)} u_n^* d\lambda_n > \epsilon,\forall n.$$ The $d\lambda_n$-energy of $\tilde u_n$ is equal to $T_{2,n}$. We conclude that for all $S_0>0$ and large $n$ we have $$\int_{B_{S_0}(0)} u_n^* d\lambda_n \leq T_{2,n}  - \epsilon.$$ Since $\tilde u_n \to \tilde u$ in $C^\infty_{\rm loc}$, we have  $$\int_{B_{S_0}(0)} u^* d\lambda_E \leq T_{2,E}  - \epsilon,$$ which implies that $\int_{\C} u^*d\lambda_E \leq T_{2,E}  - \epsilon,$ since $S_0$ is arbitrary. This is in contradiction with $E(\tilde u)=T_{2,E}$ and hence proves the claim.

In order to prove the lemma we again argue indirectly and assume that the loops $t\mapsto u_{n_j}(R_{n_j}e^{2 \pi i t})$ are not contained in $\W$ for  subsequences $\tilde u_{n_j}=(a_{n_j},u_{n_j})$ and $R_{n_j}$, also denoted by $\tilde u_n$ and $R_n$, respectively. Let $\tilde v_n=(b_n,v_n):\R \times S^1 \to \R \times W_E$ be defined by $$\tilde v_n(s,t) = (b_n(s,t),v_n(s,t))=\left(a_n\left(R_ne^{2 \pi  (s + it)}\right) - a_n(R_n),u_n\left(R_ne^{2 \pi  (s + it)}\right)\right).$$ Observe that for all $n$, $b_n(0,0)=0$ and $t\mapsto v_n(0,t)$ is not contained in $\W$.

From the claim above we see that the sequence  $\tilde v_n$ has $C^\infty_{\rm loc}$-bounds since any bounded sequence $z_n \in \R \times S^1$ with $|\nabla \tilde v_n(z_n)|_n \to \infty$ would take away at least $T_{2,E}$ from the $d\lambda_E$-energy, which is impossible since $z_n$ corresponds to points $w_n\in \C$ for $\tilde u_n$ with $\lim_{n \to \infty} |w_n|=+\infty$. So we find a $\tilde J_E$-holomorphic cylinder $\tilde v=(a,v):\R \times S^1 \to \R \times W_E$ such that, up to extraction of a subsequence, $\tilde v_n \to \tilde v$ in $C^\infty_{\rm loc}$ as $n\to \infty$. Hence for all fixed $S_0>0$ and all $\epsilon>0$ $$\begin{aligned} \int_{[-S_0,S_0] \times S^1} v^*d\lambda_E \leq & \limsup_n \int_{[-S_0,S_0] \times S^1} v_n^*d\lambda_n \\ \leq & \limsup_n \int_{B_{R_ne^{2\pi S_0}}(0) \setminus B_{R_ne^{-2\pi S_0}}(0)} u_n^* d\lambda_n \\ \leq &  \epsilon,\end{aligned}$$ for all large $n$. This implies that $\int_{\R \times S^1} v^* d\lambda_E = 0.$  Moreover, from  \eqref{eqa1}   $$T_{2,E} \geq \int_{\{0\} \times S^1} v^* \lambda_E = \lim_{n \to \infty} \int_{\{0\} \times S^1} v_n^* \lambda_n = \lim_{n \to \infty} \int_{\partial B_{R_n}(0)} u_n^* \lambda_n \geq T_{2,E} - \gamma_0>0.$$ It follows that $\tilde v$ is non-constant and, from Lemma \ref{lemamin}, $\tilde v$ is a cylinder over $P_{2,E}$. This contradicts the fact that the loop $t \mapsto v(0,t)$ is not contained in $\W$. \end{proof}

Now we prove that  $\tilde u_{1,n} \to \tilde u_{1,E}$ and $\tilde u_{2,n} \to \tilde u_{2,E}$ in $C^\infty_{\rm loc}$ as $n \to \infty$, up to reparametrization and $\R$-translation. From Proposition \ref{propbounds},  $\tilde u_{1,n}$ satisfying \eqref{eqa1} converges, up to extraction of a subsequence, to a finite energy $\tilde J_E$-holomorphic plane $\tilde u:\C \to \R \times W_E$ as $n \to \infty$, which is asymptotic to $P_{2,E}$. From  the uniqueness of such $\tilde J_E$-holomorphic planes, see Proposition \ref{propunique}, we find $j_1\in \{1,2\}$ so that $\tilde u=\tilde u_{j_1,E}$ up to reparametrization and $\R$-translation. In the same way, we find $j_2 \in \{1,2\}$ so that $\tilde u_{2,n} \to \tilde u_{j_2,E}$ in $C^\infty_{\rm loc}$ as $n \to \infty$, up to a subsequence, possibly after reparametrization and $\R$-translation.

Let $\U \subset W_E$ be a small neighborhood of $\partial S_E = u_{1,E}(\C) \cup P_{2,E}\cup u_{2,E}(\C)$ so that $\U$ does not intersect $P_{3,E} \cup P_{3,E}'$.
Denote by $\bar u_{1,n},\bar u_{2,n}$ the continuous extensions of $u_{1,n},u_{2,n}$ over the circle $S^1$ at $\infty$ for all  $n$, so that $\bar u_{1,n}(S^1) = \bar u_{2,n}(S^1) = P_{2,E}$. From Lemma \ref{lemvizi10}, $\bar u_{1,n}(\C),\bar u_{2,n}(\C) \subset \U$ for all large $n$. Any homotopy from $\bar u_{1,n}$ to $\bar u_{2,n}$ inside $W_E$, which keeps the boundary $P_{2,E}$ fixed, cannot be supported inside $\U$ and must intersect $P_{3,E}\cup P_{3,E}'$, for all large $n$.

Let $\U_1,\U_2\subset \U$ be small tubular neighborhoods of $u_{1,E}(\C)\cup P_{2,E},u_{2,E}(\C)\cup P_{2,E},$ respectively.  Arguing indirectly assume $j_1=j_2$. By Lemma \ref{lemvizi10},  $\bar u_{1,n}(\C),\bar u_{2,n}(\C) \subset \U_{j_1}$ for large $n$. This implies that for large $n$, $\bar u_{1,n}$ is homotopic to $\bar u_{2,n}$, keeping the boundary $P_{2,E}$ fixed, through a homotopy which does not intersect $P_{3,E} \cup P_{3,E}'$, a contradiction. Thus $j_1\neq j_2$. Now since $X_{\lambda_n}|_{u_{1,n}(\C)}$ points inside the component $\dot S_n$ of $W_E \setminus (u_{1,n}(\C)\cup P_{2,E} \cup u_{2,n}(\C) )$ containing $P_{3,E}$, $X_{\lambda_n}|_{u_{2,n}(\C)}$ points outside $\dot S_n$, since $X_{\lambda_E}|_{u_{1,E}(\C)}$ points inside the component $\dot S_E$ of $W_E \setminus (u_{1,E}(\C) \cup P_{2,E}\cup u_{2,E}(\C) )$ containing $P_{3,E}$, $X_{\lambda_E}|_{u_{2,E}(\C)}$ points outside $\dot S_E$, since $\tilde u_{1,n} \to \tilde u_{j_1,E}$ and $\tilde u_{2,n} \to \tilde u_{j_2,E}$ in $C^\infty_{\rm loc}$ as $n\to \infty$, and since $X_{\lambda_n} \to X_{\lambda_E}$ in $C^\infty$ as $n \to \infty$, we conclude that $j_1=1$ and $j_2=2$. This finishes the proof of Proposition  \ref{prop_step5ib}.

\section{Proof of Proposition \ref{prop_step5}-{\rm ii)}}\label{sec_5-ii)}

As before we have a sequence of almost complex structures $\tilde J_n=(\lambda_n, J_n) \to \tilde J_E=(\lambda_E, J_E)$ in $C^\infty$ as $n\to \infty$. For all $n$, $\tilde J_n$ admits a stable finite energy foliation $\tilde \F_n$, which projects onto a $3-2-3$ foliation $\F_n$ of $W_E$.  $\tilde \F_n$ contains a pair of rigid planes $\tilde u_{1,n}=(a_{1,n},u_{1,n}),\tilde u_{2,n}=(a_{2,n},u_{2,n}):\C \to \R \times W_E$ so that $u_{1,n}(\C) \cup P_{2,n} \cup u_{2,n}(\C)\subset W_E$ is a topological embedded $2$-sphere which separates $W_E$ in two components $\dot S_n$ and $\dot S_n '$, both diffeomorphic to the open $3$-ball. $\dot S_n$ and $\dot S_n'$ contain the binding orbits $P_{3,n}$ and $P_{3,n}',$ respectively.
Moreover, $\tilde \F_n$ contains a pair of $\tilde J_n$-rigid cylinders $\tilde v_{n}=(b_n,v_n),\tilde v_n'=(b_n',v_n'):\R \times S^1 \to \R \times W_E$ so that $v_n(\R \times S^1) \subset \dot S_n \setminus P_{3,n}$ and $v_n'(\R \times S^1) \subset \dot S_n'\setminus P_{3,n}'$. $\tilde v_n$ is asymptotic to $P_{3,n}$ and $P_{2,n}$ at $s=+\infty$ and $s=-\infty$, respectively. In the same way, $\tilde v_n'$ is asymptotic to $P_{3,n}'$ and $P_{2,n}$ at $s=+\infty$ and $s=-\infty$, respectively.

In this section we prove Proposition \ref{prop_step5}-ii) restated below.

\begin{prop}\label{prop_step5iib} If $E>0$ is sufficiently small then,   up to suitable reparametrizations and $\R$-translations, the following holds: there exist finite energy $\tilde J_E$-holomorphic embedded cylinders $\tilde v_E=(b_E,v_E),$ $\tilde v_E'=(b_E',v_E'):\R \times S^1 \to \R \times W_E$, with $v_E$ and $v_E'$ embeddings,  $\tilde v_E$ asymptotic to $P_{3,E}$ and $P_{2,E}$, $\tilde v_E'$ asymptotic to $P_{3,E}'$ and $P_{2,E}$ at $s=+\infty$ and $s=-\infty$, respectively,  so that $\tilde v_n \to \tilde v_E$ and $\tilde v_n' \to \tilde v_E'$ in $C^\infty_{\rm loc}$ as $n \to \infty$. We denote $V_{E}=v_{E}(\R \times S^1)$ and $V_{E}'=v_{E}'(\R \times S^1)$. Then $V_{E}\subset \dot S_E\setminus P_{3,E}$ and $V_{E}'\subset \dot S_E'\setminus P_{3,E}'$. Furthermore, given small neighborhoods $\U \subset W_E$ of $V_{E}\cup P_{2,E} \cup P_{3,E}$ and $\U' \subset W_E$ of $V_{E}'\cup P_{2,E} \cup P_{3,E}'$, there exists $n_0\in \N$ so that if $n \geq n_0$, then $v_{n}(\R \times S^1) \subset \U$ and $v_{n}'(\R \times S^1) \subset \U'.$
\end{prop}

In order to prove Proposition \ref{prop_step5iib},  we only deal with the compactness properties of $\tilde v_n$, since the case $\tilde v_n'$ is completely analogous.

We can choose embedded $2$-spheres in $W_E$ which separate $P_{3,n}$ and $P_{2,n}$ for all $n$. In local coordinates, for instance, these embedded $2$-spheres can be defined by $$N_E^\delta:=\{q_1 + p_1 = \delta\}\cap K^{-1}(E),$$ where $\delta>0$ is small, see Lemma \ref{2-sphere}.

Now we reparametrize the sequence $\tilde v_n(s,t)$ and slide it in the $\R$-direction in the following way: let $(s_n,t_n) \in \R \times S^1$ be such that $v_n(s_n,t_n) \in N_E^\delta$ and $v_n(s,\cdot) \cap N_E^\delta = \emptyset,\forall s>s_n$. The existence of $(s_n,t_n)$ follows from the compactness of $N_E^\delta$ and from the asymptotic properties of $\tilde v_n$. We define the sequence $\tilde w_n = (d_n, w_n):\R \times S^1 \to \R \times W_E$ of $\tilde J_n$-holomorphic cylinders  by $$\tilde w_n(s,t)=(b_n(s+s_n,t+t_n)-b_n(1+s_n,t_n), v_n(s+s_n,t+t_n)),\forall (s,t)\in \R \times S^1,$$ which we again denote by $\tilde v_n=(b_n,v_n)$. Thus we can assume that \begin{equation}\label{separa} \begin{aligned} v_n(s,t) & \in  U^\delta_E, \forall s>0, t\in S^1, \\ \tilde v_n(0,0) & \in \R \times N_E^\delta \mbox{ and } \tilde v_n(1,0)\in \{0\} \times U_E^\delta,\end{aligned} \end{equation} where $U_E^\delta$ is the component of $W_E \setminus N_E^\delta$ containing $P_{3,E}$.

\begin{prop}\label{propcylinder}For $E>0$ sufficiently small, $|\nabla \tilde v_n(z)|_n$ is uniformly bounded in $z\in \R \times S^1$ and in $n \in \N$. Moreover, there exists a $\tilde J_E$-holomorphic cylinder $\tilde v_E=(b_E,v_E):\R \times S^1 \to \R \times W_E$ such that  $\tilde v_n \to \tilde v_E$ in $C^\infty_{{\rm loc}}$ as $n \to \infty$, up to a shift in the $S^1$-direction in the domain for each $\tilde v_n$. The cylinder $\tilde v_E$ is asymptotic to $P_{3,E}$ at its positive puncture $+\infty$ and asymptotic to $P_{2,E}$ at its negative puncture $-\infty$.  The image $v_E(\R \times S^1)$ is contained in $\dot S_E\setminus P_{3,E}$.\end{prop}

\begin{proof}Arguing indirectly we assume that up to extraction of a subsequence we find a sequence $z_n=(s_n,t_n) \in \R \times S^1$ so that $|\nabla \tilde v_n(z_n)|_n \to \infty$ as $n \to \infty$.

We first claim that $\limsup_{n \to \infty} s_n \leq 0.$ Otherwise, after extracting a subsequence, we find $\varepsilon>0$ such that $s_n>\varepsilon, \forall n.$ Now we proceed as in Proposition \ref{propbounds}. Let $R_n=|\nabla \tilde v_n(z_n)|_n \to \infty$ and choose a sequence $\epsilon_n \to 0^+$ satisfying $\epsilon_n R_n \to +\infty$. Using Lemma \ref{lemtop}, we can slightly change $z_n$ and $\epsilon_n$ such that in addition the following holds $$|\nabla \tilde v_n(z)|_n \leq 2 |\nabla \tilde v_n(z_n)|_n, \forall |z-z_n|\leq \epsilon_n.$$ Now let $\tilde w_n=(d_n,w_n):B_{\epsilon_n R_n}(0) \to \R \times W_E$ be the $\tilde J_n$-holomorphic map defined by
\begin{equation}\label{reescala}
\tilde w_n(z)=(d_n(z),w_n(z))=\left(b_n\left(z_n + \frac{z}{R_n}  \right)-b_n(z_n),v_n \left(z_n+\frac{z}{R_n} \right)  \right).
\end{equation}
Thus $\tilde w_n(0) \in \{0\} \times W_E$ and $|\nabla \tilde w_n(z)|_n\leq 2,\forall |z| \leq \epsilon_n R_n$. Since $\epsilon_n R_n \to +\infty$, we have $\tilde w_n \to \tilde w=(d,w):\C \to \R \times W_E$ in $C^{\infty}_{\rm loc}$ up to extraction of a subsequence. By construction $\tilde w$ is non-constant and has bounded energy from Fatou's Lemma. Moreover, for any fixed $R>0$  and large $n$, we have $$ \begin{aligned} \int_{B_R(0)} w^* d\lambda_E \leq & \limsup_n \int_{B_{\epsilon_n R_n}(0)} w_n^* d\lambda_n \\ = & \limsup_n \int_{B_{\epsilon_n}(z_n)}v_n^*d\lambda_n \\ \leq & \limsup_n \int_{\R \times S^1} v_n^*d\lambda_n \\ = & \limsup_n (T_{3,n}-T_{2,n}) \\ = & T_{3,E} - T_{2,E},\end{aligned}$$ since $\epsilon_n \to 0$. This implies that \begin{equation}\label{eqbub1}\int_\C w^* d\lambda_E \leq T_{3,E} - T_{2,E}.\end{equation} Since $\tilde w$ has finite energy and is non-constant, it follows from Theorem \ref{Ho93} that for any sequence $r_n \to +\infty$, we can extract a subsequence, also denoted $r_n$, and find a periodic orbit $Q=(x,T)$ so that $w\left(e^{2\pi (r_n +it)}\right)\to x(Tt)$ as $n \to \infty$. From \eqref{eqbub1}, we have $T = \int_\C w^* d\lambda_E <T_{3,E}$. It follows that $Q$ is geometrically distinct from $P_{3,E}$. Moreover, by the construction of the sequence $w_n$, $Q$ lies on $U_E^\delta$, the component of $W_E \setminus  N_E^\delta$ which contains $P_{3,E}$. Thus $Q$ is also geometrically distinct from $P_{2,E}$. Now observe that since $v_n$ is an embedding which does not intersect $P_{3,E}$ for each $n$, it follows that the image of any contractible loop in $\R \times S^1$ under $v_n$ is not linked to $P_{3,E}$. This implies that the image of any loop under $w$ is not linked to $P_{3,E}$. In particular,  $Q$ is not linked to $P_{3,E}$ as well. However, this contradicts Theorem \ref{teolink} in Appendix \ref{ap_link} if $E>0$ is taken sufficiently small. We conclude that $\limsup_{n\to\infty} s_n \leq 0$.

In the following we exclude the possibility of bubbling-off points in $(-\infty,0] \times S^1$. Let $\Gamma \subset \R \times S^1$ be the set of points $z=(s,t)\in \R \times S^1$ such that for a subsequence $\tilde v_{n_j}$ of $\tilde v_n$, there exists a sequence $z_{j} \to z$ satisfying $|\nabla \tilde v_{n_j}(z_j)|_{n_j} \to \infty$ as $j \to \infty$. We claim that $\Gamma = \emptyset$. Arguing indirectly, assume that $\Gamma \neq \emptyset$ and  choose $z^*_0\in \Gamma$. The subsequence $\tilde v_{n_j}$ of $\tilde v_n$ satisfying $|\nabla \tilde v_{n_j}(z_j)|_{n_j} \to \infty$ and $z_j \to z^*_0$ as $j \to \infty$ is again denoted by $\tilde v_n$. Let $\Gamma_1$ be the set of points $z=(s,t)\in (\R \times S^1)\setminus \{z^*_0\}$ such that for a subsequence $\tilde v_{n_j}$ of $\tilde v_n$, there exists a sequence $z_{j} \to z$ satisfying $|\nabla \tilde v_{n_j}(z_j)|_{n_j} \to \infty$ as $j \to \infty$. From \eqref{separa}, if $\Gamma_1=\emptyset$ then, up to extraction of a subsequence, $\tilde v_n \to \tilde v:(\R \times S^1) \setminus \{z^*_0\} \to \R \times W_E$ in $C^\infty_{\rm loc}$  as $n \to \infty$. If $\Gamma_1 \neq \emptyset$, then we proceed in the same way: choose $z^*_1 \in \Gamma_1$ and denote again by $\tilde v_n$  the subsequence  $\tilde v_{n_j}$ satisfying $|\nabla \tilde v_{n_j}(z_j)|_{n_j} \to \infty$ where $z_j \to z^*_1$ as $j \to \infty$. Inductively, we define $\Gamma_{i+1}$  to be the set of points $z=(s,t)\in (\R \times S^1)\setminus \{z_0^*,\ldots,z_i^*\}$ such that for a subsequence $\tilde v_{n_j}$ of $\tilde v_n$, there exists a sequence $z_{j} \to z$ satisfying $|\nabla \tilde v_{n_j}(z_j)|_{n_j} \to \infty$ as $j \to \infty$. If $\Gamma_{i+1}=\emptyset$ then, up to extraction of a subsequence, $\tilde v_n \to \tilde v:(\R \times S^1) \setminus \{z^*_0, \ldots, z_i^*\} \to \R \times W_E$ in $C^\infty_{\rm loc}$ as $n \to \infty$. If $\Gamma_{i+1} \neq \emptyset$, we choose $z^*_{i+1} \in \Gamma_{i+1}$ and denote again by $\tilde v_n$ the subsequence $\tilde v_{n_j}$ satisfying $|\nabla \tilde v_{n_j}(z_j)|_{n_j} \to \infty$ as $j \to \infty$, with $z_j \to z^*_{i+1}$. Each bubbling-off point $z^*_i$ is contained in $(-\infty,0] \times S^1$ and takes away at least $T_{2,E}$ from the $d\lambda_E$-area, i.e., for any fixed $\epsilon>0$ small and all large $n$ we have \begin{equation}\label{energia1}\int_{B_{\epsilon}(z^*_i)} v_n^* d\lambda_n> T_{2,E}  -\epsilon,\end{equation} where $\partial B_{\epsilon}(z^*_i)$ is oriented counter-clockwise. In fact, we can appropriately rescale $\tilde v_n$ arbitrarily close to $z^*_i$ as in \eqref{reescala} in order to find a finite energy plane $\tilde w=(d,w):\C \to \R \times W_E$ satisfying $\int_\C w^*d\lambda_E \geq T_{2,E}$ see Lemma \ref{lemamin}.  Now since $\limsup_n \int_{\R \times S^1}v_n^* d\lambda_n = T_{3,E} - T_{2,E}<\infty$, \eqref{energia1} implies that the procedure above must terminate after a finite number of steps, i.e., there exists $i_0 \in \N^*$ such that $\Gamma_{i_0+1} = \emptyset$ and $\Gamma_{i_0} \neq \emptyset.$ So, up to extraction of a subsequence, we end up finding a $\tilde J_E$-holomorphic map $\tilde v=(b,v):(\R \times S^1) \setminus \tilde \Gamma \to \R \times W_E$ satisfying $\tilde v_n \to \tilde v$ in $C^\infty_{\rm loc}$ as $n\to \infty$, where $\tilde \Gamma:=\{z_0^*,\ldots,z_{i_0}^*\}.$ Moreover, from Fatou's Lemma, we have $E(\tilde v)\leq \limsup_n T_{3,n} = T_{3,E}$.

All punctures in $\tilde \Gamma \cup \{-\infty\}\cup \{+\infty\}$ are non-removable. To see this, observe that \begin{equation}\label{desigt2t30} T_{2,E} \leq \int_{\{s\}\times S^1} v^* \lambda_E \leq T_{3,E},\end{equation} for all values of $s$ where this integral is defined, including $s\in (0,\infty)$ and $s\ll 0$. This implies that $-\infty$ is a negative puncture and $+\infty$ is a positive puncture of $\tilde v$. Moreover, for all $\epsilon>0$, we have $\int_{\partial B_{\epsilon}(z^*_i)}\tilde v^* \lambda_E \geq T_{2,E},$ where $\partial B_{\epsilon}(z^*_i)$ is oriented counter-clockwise. This implies that all punctures $z^*_i$ are negative and thus $+\infty$ is the only positive puncture of $\tilde v$. Since $v((0,\infty) \times S^1)$ lies on $U_E^\delta$ and the loops $t\mapsto v_n(s,t)$ are not linked with $P_{3,E},$ $\forall s,n$ large, it follows from Theorem \ref{teolink} and \eqref{desigt2t30} that if $E>0$ is sufficiently small, then $P_{3,E}$ is the unique asymptotic limit of $\tilde v$ at $+\infty$ and, therefore, $v(s, \cdot) \to x_{3,E}(T_{3,E} \cdot)$ in $C^\infty$ as $s \to \infty$, where $P_{3,E}=(x_{3,E},T_{3,E})$. Moreover, $b_s(s,t) \to T_{3,E}$ as $s \to +\infty$, see \cite{props2}. Since $P_{3,E}$ is simple, we conclude that $\tilde v$ is somewhere injective, i.e., there exists $z_0 \in (\R \times S^1) \setminus \tilde \Gamma$ so that $d\tilde v(z_0)\neq 0$ and $\tilde v^{-1}(\tilde v(z_0))=\{z_0\}$.

We claim that $\tilde v$ is asymptotic to a $p$-cover of $P_{2,E}$ at each of its negative punctures. To see this, observe first that  the image of $v$ does not intersect $\partial S_E = U_{1,E}  \cup P_{2,E}\cup U_{2,E}.$ Otherwise we would have $\int_{(\R \times S^1) \setminus \tilde \Gamma} v^*d\lambda_E >0$ and, from Carleman's similarity principle, these intersections would be isolated. From stability and positivity of intersections of pseudo-holomorphic curves, $v_n(\R \times S^1)$ would intersect $u_{1,n}(\C) \cup P_{2,n}\cup u_{2,n}(\C) $ for all large $n$, a contradiction. Let $Q\in \P(\lambda_E)$ be an asymptotic limit of a negative puncture. It must be contained in $S_E$ and its action satisfies  $T_{2,E} \leq \int_Q \lambda_E \leq T_{3,E}$. Since for $E>0$ sufficiently small, $P_{2,E}$ is the only periodic orbit in $S_E$ satisfying $\int_Q \lambda_E \leq T_{3,E}$ which is not linked to $P_{3,E}$, see Theorem \ref{teolink}, we conclude that  $Q$ coincides either with $P_{3,E}$ or with a $p$-cover of $P_{2,E}$. If it coincides with $P_{3,E}$ then, since we are assuming $\tilde\Gamma \neq \emptyset$, we have $\int_{(\R \times S^1)\setminus \tilde \Gamma} v^*d\lambda_E \leq T_{3,E}- T_{3,E} - T_{2,E}<0,$ a contradiction. Thus $\tilde v$ is asymptotic to a $p-$cover of $P_{2,E}$ at each negative puncture and, therefore, $\tilde v$ is not a cylinder over a periodic orbit. This implies that $\int_{(\R \times S^1) \setminus \tilde \Gamma} v^*d\lambda_E >0$. In particular, by Carleman's similarity principle, the image of $v$ does not intersect $P_{3,E}$ since intersections would be isolated and, by positivity and stability of intersections, this would imply that $v_n(\R\times S^1)$ intersects $P_{3,n}$ for all large $n$, a contradiction. Thus the image of $v$ is contained in $\dot S_E\setminus P_{3,E}$.

Again from the assumption $\tilde \Gamma\neq \emptyset$, we conclude using Proposition \ref{propintersect}-i)-ii) that $v$ admits self-intersections, i.e., there exists a pair of points $z,z'\in (\R \times S^1)\setminus \tilde\Gamma$, with $z\neq z'$, satisfying $v(z)=v(z')$. This implies that $\tilde v$ intersects the $\tilde J_E$-holomorphic map $\tilde v_c:=(b+c,v)$ for a suitable constant $c\in \R$. Since $\tilde v$ is somewhere injective and $\tilde v$ is not a cylinder over a periodic orbit, the intersection pairs $(z,z'),z\neq z',$ with $\tilde v(z)=\tilde v_c(z'),$ are isolated in $((\R \times S^1)\setminus \tilde \Gamma)^2 \setminus \{\rm diagonal\}$. From stability and positivity of intersections of pseudo-holomorphic curves, we conclude that for all large $n$, $\tilde v_n$ also intersects $\tilde v_{n,c}:=(b_n+c,v_n)$, which is a contradiction since $v_n$ is an embedding for all $n$.

We conclude that $\Gamma = \tilde \Gamma=\emptyset$. In this case, from \eqref{separa}, we have $\tilde v_n \to \tilde v=(b,v):\R \times S^1 \to \R \times W_E$ in $C^\infty_{\rm loc}$ as $n \to \infty$, up to extraction of a subsequence. Arguing as above, we know that $\tilde v$ is asymptotic to $P_{3,E}$ at $+\infty$ and that $\tilde v$ is somewhere injective. We claim that $\tilde v$ is asymptotic to $P_{2,E}$ at $-\infty$. To see this observe that \eqref{desigt2t30} holds for all $s\in \R$, which implies that $-\infty$ is a negative puncture of $\tilde v$.  As above, the image of $v$ does not intersect $\partial S_E$ otherwise any intersection would be isolated and from stability and positivity of intersections of pseudo-holomorphic curves, $v_n(\R \times S^1)$ would intersect $u_{2,n}(\C)\cup P_{2,n}\cup u_{2,n}(\C)$ for all large $n$, a contradiction. Let $Q\in \P(\lambda_E)$ be an asymptotic limit of $\tilde v$ at $-\infty$. It must be contained in $S_E$ and its  action satisfies $\int_Q \lambda_E \leq T_{3,E}$. Since for $E>0$ sufficiently small, $P_{2,E}$ is the only periodic orbit in $S_E$ which is not linked to $P_{3,E}$, see Theorem \ref{teolink}, we conclude that  $Q$ coincides either with $P_{3,E}$ or with a $p$-cover of $P_{2,E}$. If it coincides with $P_{3,E}$ then $\int_{\R \times S^1} v^*d\lambda_E =0$ and this implies that $\tilde v$ is a cylinder over $P_{3,E},$ a contradiction with $\tilde v(0,0)\in \R \times N_E^\delta$, see \eqref{separa}. Thus $\tilde v$ is asymptotic to a $p-$cover of $P_{2,E}$. If $p\geq 2$, then $v$ must self-intersect near $P_{2,E}$, see Proposition \ref{propintersect}-i), and, as before, we conclude that for all large $n$, $v_n$ is not an embedding,  a contradiction. We conclude that $p=1$.

Finally we show that there are no bubbling-off points going to $-\infty$. Arguing indirectly, we assume this is the case for a subsequence $\tilde v_{n_j}$, i.e., there exists a sequence $z_j=(s_j,t_j)\in \R \times S^1$, with $s_j \to -\infty$ and $|\nabla \tilde v_{n_j}(z_j)|_{n_j} \to \infty$ as $j \to \infty$. Denote this subsequence again by $\tilde v_n$. We know that, up to extraction of a subsequence, $\tilde v_n \to \tilde v$ in $C^\infty_{\rm loc}(\R \times S^1, \R \times W_E)$ as $n\to\infty$, and that $\tilde v$ is asymptotic to $P_{3,E}$ and $P_{2,E}$ at $s=+\infty$ and $s=-\infty$, respectively. Given $\frac{T_{2,E}}{2}>\epsilon>0$ small, we choose $R>0$ large enough so that $\int_{[-R,R]\times S^1} v^* d\lambda_E > T_{3,E} - T_{2,E} -\epsilon$, thus $\int_{[-R,R]\times S^1} v_n^* d\lambda_n > T_{3,E} - T_{2,E} -\epsilon$ for all large $n$. Now given $\epsilon_0>0$ we have $\int_{B_{\epsilon_0}(z_n)}v_n^*d\lambda_n>T_{2,E} -\epsilon$ for all large $n$. Since $B_{\epsilon_0}(z_n)\cap [-R,R]\times S^1 = \emptyset$ for all large $n$, we conclude that  $$\begin{aligned} T_{3,E}-T_{2,E}= & \lim_n (T_{3,n} - T_{2,n}) \\ = & \lim_n \int_{\R \times S^1} v_n^*d\lambda_n \\ \geq & \limsup_n \int_{([-R,R] \times S^1)\cup B_{\epsilon_0}(z_n) } v_n^*d\lambda_n \\  \geq & T_{3,E}-T_{2,E} +T_{2,E} -2\epsilon\\ > & T_{3,E}-T_{2,E}\end{aligned} $$ for all large $n$, a contradiction.

We have proved that $|\nabla \tilde v_n(z)|_n$ is uniformly bounded in $n\in \N$ and in $z\in \R \times S^1$. As mentioned above,  we have $\tilde v_n \to \tilde v:\R \times S^1 \to \R \times W_E$ in $C^\infty_{\rm loc}$ as $n \to \infty$, up to extraction of a subsequence, where $\tilde v$ is non-constant,  asymptotic to $P_{3,E}$ at its positive puncture $+\infty$ and asymptotic to $P_{2,E}$ at its negative puncture $-\infty$. In particular, we have \begin{equation}\label{energiadl} \int_{\R \times S^1} v^* d\lambda_E = T_{3,E} - T_{2,E}>0.
\end{equation}

Now we prove that $v$ does not intersect $P_{3,E}$, the case $P_{2,E}$ is similar. If $v$ intersects $P_{3,E}$, then $\tilde v$ intersects the $\tilde J_E$-holomorphic cylinder over $P_{3,E}$. By Carleman's similarity principle and \eqref{energiadl} these intersections are isolated. By positivity and stability of intersections, any such intersection implies intersections of $\tilde v_n$ with the $\tilde J_n$-holomorphic cylinder over $P_{3,n}$ for all large $n$, a contradiction. Reasoning in the same way, we get that $v_E(\R \times S^1)$ does not intersect $u_{1,E}(\C) \cup u_{2,E}(\C)$. It follows that $v_E(\R \times S^1) \subset \dot S_E \setminus P_{3,E}$.

In Proposition \ref{asymptoticcylinder} below, we prove that $\tilde v$ converges to $P_{3,E}$ exponentially fast at $+ \infty$. Since $P_{2,E}$ is hyperbolic, $\tilde v$ also converges exponentially fast to $P_{2,E}$ at $-\infty$. From the uniqueness of such $\tilde J_E$-holomorphic cylinders, up to reparametrization and $\R$-translation, see Proposition \ref{propunique}, we conclude that any subsequence of $\tilde v_n$ has the same $C^\infty_{\rm loc}$-limit, denoted by $\tilde v_E=(b_E,v_E):\R \times S^1 \to \R\times W_E$, which, up to a shift in the $S^1$-direction in the domain, is determined by normalization \eqref{separa}.  This finishes the proof of the proposition. \end{proof}

Next we give a better description of the asymptotic behavior of the finite energy cylinder $\tilde v_E=(b_E,v_E)$ obtained in Proposition \ref{propcylinder} at its positive puncture. Our aim is to prove the following more general statement.

\begin{prop}\label{asymptoticcylinder}Let $\tilde v_E=(b_E,v_E):\R \times S^1 \to \R \times W_E$ be the $\tilde J_E$-holomorphic cylinder obtained in Proposition \ref{propcylinder}, which is asymptotic to $P_{3,E}=(x_{3,E},T_{3,E})$ at its positive puncture $+\infty$ and to $P_{2,E}$ at $-\infty$, and is the $C^\infty_{\rm loc}$-limit of the sequence of $\tilde J_n$-holomorphic cylinders $\tilde v_n=(b_n,v_n)$ satisfying \eqref{separa}, where $E>0$ is sufficiently small. Consider Martinet's coordinates $(\vartheta, x,y)\in S^1 \times \R^2$ defined in a small tubular neighborhood $\U\subset W_E$ of $P_{3,E}$, as explained in Appendix \ref{ap_basics}, where the contact form takes the form $\lambda_E\equiv g_E \cdot (d\vartheta +x dy)$ and $P_{3,E} \equiv S^1 \times \{0\}$. Let $A_{P_{3,E}}$ be the asymptotic operator associated to $P_{3,E}$ which assumes the form \eqref{operadorAP} in these local coordinates. Then $v_E(s,t)\in \U$ for all $s$ sufficiently large and, in these local coordinates, the cylinder $\tilde v_E$ is represented by the functions $$(b_E(s,t),\vartheta(s,t),x(s,t),y(s,t)), (s,t)\in \R \times S^1, s\gg0,$$ where
\begin{equation}\label{eqasymptotics} \begin{aligned}|D^\gamma(b_E(s,t) - (T_{3,E}s + a_0))| & \leq A_\gamma e^{-r_0s},\\
|D^\gamma(\vartheta(s,t) - t-\vartheta_0) | & \leq A_\gamma e^{-r_0s},\\
z(s,t):=(x(s,t),y(s,t)) & = e^{\int_{s_0}^s \mu(r) dr}(e(t)+R(s,t)),\\
|D^\gamma R(s,t)|,|D^\gamma (\mu(s)- \delta)| & \leq A_\gamma
e^{-r_0s},
\end{aligned}
\end{equation}for all large $s$ and all $\gamma \in \N \times \N$, where $A_\gamma,r_0>0,\vartheta_0,a_0\in \R$ are suitable constants.  $\vartheta(s,t)$ is seen as
a map on $\R$ and satisfies $\vartheta(s,t+1) = \vartheta(s,t)+1$. Here $\mu(s) \to \delta <0,$  where $\delta$ is an eigenvalue of $A_{P_{3,E}}$ and $e:S^1 \to \R^2$
is an eigensection of $A_{P_{3,E}}$ associated to $\delta$,
represented in these coordinates. \end{prop}

Proposition \ref{asymptoticcylinder} is essentially proved in \cite{convex}. Since our present situation has minor differences, we include the main steps of the proof for completeness. We need a preparation lemma analogous to Lemma 8.1 in \cite{convex}.

\begin{lem}\label{lemvizi}Let $\tilde v_E=(b_E,v_E):\R \times S^1 \to \R \times W_E$ be the finite energy $\tilde J_E$-holomorphic cylinder obtained in Proposition \ref{propcylinder} as the $C^\infty_{\rm loc}$-limit of the sequence of $\tilde J_n$-holomorphic cylinders $\tilde v_n=(b_n,v_n)$ satisfying \eqref{separa}, where $E>0$ is sufficiently small. Given an $S^1$-invariant neighborhood $W$ of $t\mapsto x_{3,E}(T_{3,E}t)$ in the loop space $C^\infty(S^1,W_E)$, there exists $s_0>0$ such that for all $s \geq s_0$ and for every large $n$, the loop $t \mapsto v_n(s,t)$ is contained in $W$. \end{lem}

\begin{proof}First we claim that for every $\epsilon>0$ and every sequence $s_n \to +\infty$, there exists $n_0 \in \N$ such that for all $n \geq n_0$, we have $$\int_{[s_n,\infty) \times S^1} v_n^* d\lambda_n \leq \epsilon.$$ Arguing indirectly we may assume there exists $\epsilon >0$ and a sequence $s_n \to +\infty$ satisfying $$\int_{[s_n,\infty) \times S^1} v_n^* d\lambda_n > \epsilon,\forall n.$$ The $d\lambda_n$-energy of $\tilde v_n$ is equal to $T_{3,n}-T_{2,n}$. We conclude that for all $S_0>0$ and large $n$  we have $$\int_{[-S_0,S_0]\times S^1} v_n^* d\lambda_n \leq T_{3,n} - T_{2,n} - \epsilon.$$ Since $\tilde v_n \to \tilde v_E$ in $C^\infty_{\rm loc}$, we have  $$\int_{[-S_0,S_0]\times S^1} v_E^* d\lambda_E \leq T_{3,E} - T_{2,E} - \epsilon,$$ which implies that $\int_{\R \times S^1} v_E^*d\lambda_E \leq T_{3,E} - T_{2,E} - \epsilon,$ since $S_0$ is arbitrary. This is in contradiction with \eqref{energiadl}.

Now assume that the loops $t\mapsto v_{n_j}(s_{n_j},t)$ are not contained in $W$ for subsequences $\tilde v_{n_j}=(b_{n_j},v_{n_j})$ and $s_{n_j}\to +\infty$, also denoted by $\tilde v_n$ and $s_n$, respectively. Let $$\tilde w_n(s,t) = (d_n(s,t),w_n(s,t))=(b_n(s+s_n,t) - b_n(s_n,0),v_n(s+s_n,t)).$$ Observe that for all $n$, $d_n(0,0)=0$ and $w_n(0,t) \notin W$ for every $t\in S^1.$ From Proposition \ref{propcylinder}, this sequence has $C^\infty_{\rm loc}-$bounds and, therefore, we find a $\tilde J_E$-holomorphic map $\tilde w=(d,w):\R \times S^1 \to \R \times W_E$ such that $\tilde w_n \to \tilde w$ in $C^\infty_{\rm loc}$ as $n\to\infty$. From the previous claim we have, for all fixed $r_0>0$ and all $\epsilon>0$, that $$\begin{aligned} \int_{[-r_0,r_0] \times S^1} w^*d\lambda_E \leq & \limsup_n \int_{[-r_0,r_0] \times S^1} w_n^*d\lambda_n \\ \leq & \limsup_n \int_{[s_n-r_0,s_n+r_0]\times S^1} v_n^* d\lambda_n \\ \leq &  \epsilon,\end{aligned}$$ for all large $n$. This implies that $\int_{\R \times S^1} w^* d\lambda_E = 0.$  Moreover,  $$T_{3,E} \geq \int_{\{0\} \times S^1} w^* \lambda_E = \lim_{n \to \infty} \int_{\{0\} \times S^1} w_n^* \lambda_n =\lim_{n \to \infty} \int_{\{s_n\} \times S^1} v_n^* \lambda_n \geq T_{2,E}.$$ It follows that $\tilde w$ is a cylinder over a periodic orbit $Q=(x,T)\in \P(\lambda_E)$ which, by construction, is geometrically distinct from $P_{3,E}$. Now, since for all large $n$ the loops $t \mapsto w_n(0,t)$ are contained in the component of $W_E \setminus N_E^\delta$ which contains $P_{3,E}$  from \eqref{separa} and such loops converge to $t \mapsto x(tT+c_0)$ in $C^\infty$ as $n \to \infty$, for a suitable constant $c_0$, we conclude that $Q$ is also geometrically distinct from $P_{2,E}$. Moreover, these loops are not linked to $P_{3,E}$, which implies that $Q$ is not linked to $P_{3,E}$ as well. This contradicts Theorem \ref{teolink},  if $E>0$ is taken sufficiently small.  \end{proof}

\begin{proof}[Sketch of the proof of Proposition \ref{asymptoticcylinder}] Lemma \ref{lemvizi} proves the existence of $s_0\gg 0$ such that for all large $n$, $v_n(s,t) \in \U,\forall (s,t)\in \R \times S^1,s\geq s_0$, where $\U\subset W_E$ is a small fixed neighborhood of $P_{3,E}$ where we have Martinet's coordinates $(\vartheta, x,y)$. Thus all analysis can be done in these coordinates. We denote $\lambda_n=g_n\cdot(d\vartheta  + x dy) \to \lambda_E$, where $g_n=f_ng_E \to g_E$ in $C^\infty(\U_0)$ as $n\to \infty$ and $\U_0 \subset S^1 \times \R^2$ is a neighborhood of $S^1 \times \{0\}\equiv P_{3,E}$.

Using a new notation, we write for each $n$ and $s\geq s_0$ $$\begin{aligned} \tilde v_E(s,t)=(a(s,t),\vartheta(s,t), x(s,t),y(s,t)),\\ \tilde v_n(s,t)=(a^n(s,t),\vartheta^n(s,t), x^n(s,t),y^n(s,t)),\end{aligned}$$ where the superscripts are used to simplify notation.  The following equations hold \begin{equation}\label{eq_zs}\begin{aligned}a_s =g_E(v)(\vartheta_t + xy_t),\\ a_t=-g_E(v)(\vartheta_s + xy_s),\\ z_s + M(s,t)z_t + S(s,t)z =0,  \end{aligned} \end{equation}and for all $n$, \begin{equation}\label{eq_zns} \begin{aligned}a^n_s =g_n(v_n)(\vartheta^n_t + x^ny^n_t),\\ a^n_t=-g_n(v_n)(\vartheta^n_s + x^ny^n_s),\\z^n_s +  M^n(s,t)z^n_t + S_n(s,t)z^n  =0, \\ \end{aligned}\end{equation} where $z=(x,y)$, $z^n=(x^n,y^n)$, $v_E=(\vartheta,x,y)$, $v_n=(\vartheta^n,x^n,y^n)$ and $S,S_n,M=M(v_E), M^n=M^n(v_n)$ are $2 \times 2$ matrices, smooth in $(s,t)$. The matrices $M, M^n$ correspond to the complex structures $J_E,J_n:\xi \to \xi$ in the basis $$\{e_1=(0,1,0),e_2=(-x,0,1)\}\subset \xi,$$ and hence satisfy $M^2=-I$ and $(M^n)^2=-I$. Recall that $\tilde v_n$ are $\tilde J_n=(\lambda_n,J_n)$-holomorphic maps and \begin{equation}\label{vntov}\tilde v_n \to \tilde v_E \mbox{ in } C^\infty_{\rm loc}. \end{equation}

In order to apply results from \cite{convex} and obtain the desired asymptotic behavior  for $\tilde v_E$ at $+\infty$, we slightly modify $\tilde v_n$ in the following way: first we define the new variables $\zeta,\zeta^n$ by $$\begin{aligned} z=& \T(m)\zeta \\  z^n= & \T^n(m)\zeta^n,\end{aligned}$$ where \begin{equation}\label{eq_barraT}\begin{aligned} \T(m) = (-J_0 M(m))^{-1/2}, m\in \U_0,\\ \T^n(m)=(-J_0 M^n(m))^{-1/2}, m\in \U_0,\end{aligned} \end{equation} and $J_0$ is the matrix $$J_0 = \left(\begin{array}{cc}0 &  -1 \\ 1 & 0  \end{array}  \right).$$ $\T$ is symplectic, symmetric and satisfies $\T^{-1}M\T = J_0$. The same holds for $\T^n$. Thus, from \eqref{eq_zs} and \eqref{eq_zns}, we obtain \begin{equation}\label{eqzns}\begin{aligned}  \zeta_s + J_0\zeta_t + S'(s,t)\zeta =0, \\ \zeta^n_s +  J_0\zeta^n_t + S_n'(s,t)\zeta^n  =0, \end{aligned}\end{equation} where $S',S_n'$ are modified accordingly and satisfy the same properties of $S,S_n$. Defining $d^n(s,t)=\frac{T_{3,E}}{T_{3,n}}a^n(s,t)$, we have for all $\alpha \in \N \times \N,$ \begin{equation}\begin{aligned} \partial^\alpha[\vartheta^n(s,t)-t-\bar c_n] \to 0,  \\ \partial^\alpha[d^n(s,t) - T_{3,E}s -d_n] \to 0, \end{aligned} \end{equation} as $s\to \infty$, for suitable constants $\bar c_n\in S^1,d_n\in \R$. Let $$ \begin{aligned}  \tilde v'(s,t)=(a(s,t),\vartheta(s,t),\zeta(s,t)),\\ \tilde v_n'(s,t)=(d^n(s,t),\vartheta^n(s,t),\zeta^n(s,t)). \end{aligned} $$ Since  $\T^n \to \T$ in $C^\infty(\U_0)$ and $T_{3,n} \to T_{3,E}$ as $n \to \infty$, it follows from \eqref{vntov} that  $\tilde v_n' \to \tilde v'$ in $C^\infty_{\rm loc}$ as $n \to \infty$.

For each $n$, $S_n'(s,t)$ converges as $s \to \infty$ to a symmetric matrix $S^\infty_n(t)$, which depends on the linearized dynamics of $\lambda_n$ over $P_{3,n}$. In the same way, $S'(s,t)$ converges as $s \to \infty$ to a symmetric matrix $S^\infty(t)$, depending on the linearized dynamics of $\lambda_E$ over $P_{3,E}$. All these computations are shown in \cite{props1} and \cite{convex}.  Again, since $\T^n \to \T$ in $C^\infty(\U_0)$ and $T_{3,n} \to T_{3,E}$ as $n\to \infty$, we can choose $\bar c_n$ appropriately such that, up to extraction of a subsequence, we have \begin{equation}\label{Sinfty} \begin{aligned} \bar c_n & \to \vartheta_0 \mbox{ as } n\to \infty, \\ S^\infty_n(t) & \to S^\infty(t) \mbox{ as } n \to \infty.\end{aligned}\end{equation} Let $A^\infty,A_n^\infty$  be the self-adjoint operators on $W^{1,2}(S^1, \R^2)\subset L^2(S^1,\R^2)$ defined by $$\begin{aligned}A^\infty \eta = -J_0 \dot \eta - S^\infty \eta,\\ A^\infty_n\eta = -J_0 \dot \eta - S^\infty_n \eta.  \end{aligned}$$ From \eqref{Sinfty}, we have $A^\infty_n \to A^\infty$ as $n \to \infty$.

Since $v_n,v_E$ do not intersect $S^1 \times \{0\}$ in local coordinates, we can define the functions $$\xi(s,t) = \frac{\zeta(s,t)}{|\zeta(s)|} \mbox{ and }  \xi^n(s,t) = \frac{\zeta^n(s,t)}{|\zeta^n(s)|},$$ where $|\zeta(s)|,|\zeta^n(s)|$ denote the $L^2$-norm of the loops $t \mapsto \zeta(s,t),\zeta^n(s,t),$ respectively. From \eqref{eqzns}, we see that the following holds $$\begin{aligned}\frac{1}{2}\frac{\frac{d}{ds}|\zeta|^2}{|\zeta|^2} & =\left<-J_0 \xi_t - S \xi ,\xi\right>=:\mu(s),\\ \frac{1}{2}\frac{\frac{d}{ds}|\zeta^n|^2}{|\zeta^n|^2}& =\left<-J_0 \xi^n_t - S_n \xi^n ,\xi^n\right>=:\mu_n(s),\end{aligned}$$ thus $$|\zeta(s)| = e^{\int_{s_0}^s \mu(\tau) d\tau}|\zeta(s_0)| \mbox{ and } |\zeta^n(s)| = e^{\int_{s_0}^s \mu_n(\tau) d\tau}|\zeta^n(s_0)|.$$

Since $P_{3,n}$ is nondegenerate for each $n$, it follows from the results in \cite{props1} that  $\mu_n(s) \to \delta_n$ as $s\to \infty$, where $\delta_n$ is a negative eigenvalue of $A^\infty_n$, i.e., $A^\infty_n e_n = \delta_n e_n$ for an eigensection $e_n$ with $|e_n|=1$. In \cite{convex}, we find the following result.

\begin{prop}[\cite{convex}, Lemmas 8.4 and 8.5]\label{propconvex85} There exists a subsequence of $\delta_n$, still denoted by $\delta_n$, satisfying the following:  $\delta_n \to \delta$, where $\delta$ is a negative eigenvalue of $A^\infty$, and $e_n \to e'$ in $W^{1,2}(S^1,\R^2)$, where $e'$ is an eigensection of $A^\infty$ associated to $\delta$, i.e., $A^\infty e' = \delta e'$ and $|e'|=1$.  Moreover, for each $\epsilon>0$, satisfying $\delta+\epsilon <0$, there exist $s_1 \geq s_0$ and $n_0$ such that $\mu_n(s) \leq \delta +\epsilon, \forall n \geq n_0,s \geq s_1$.\end{prop}

Following \cite{convex}, we find from Proposition \ref{propconvex85} a uniform exponential estimate for $|\zeta^n(s)|$ as follows: if $r:=-(\delta+\epsilon)>0$, then $|\zeta^n(s)| \leq e^{-r(s-s_0)}|\zeta^n(s_0)|,\forall n\geq n_0,s\geq s_1$. Now since $\zeta^n \to \zeta$ in $C^\infty_{\rm loc}$ as $n \to \infty$, we get $|\zeta(s)|\leq e^{-r(s-s_0)}|\zeta(s_0)|,\forall s \geq s_1$. Arguing as in \cite{convex}, this last inequality is enough for proving the asymptotic properties of $\tilde v'$ by using results from \cite{props1}. We conclude that $$\begin{aligned}|\zeta(s)| & = e^{\int_{s_0}^s \mu(\tau) d \tau}|\zeta(s_0)|,\mbox{ where }\mu(s)  \to \delta<0 \mbox{ as } s\to \infty,\\ \xi(s,t) & =\frac{\zeta(s,t)}{|\zeta(s)|}\to e'(t) \mbox{ as } s\to \infty,\mbox{ where }A^\infty e' = \delta e'. \end{aligned}$$ In coordinates $z$, the eigensection $e'$ is represented by $e(t)=\T(t+\vartheta_0,0,0)e'(t).$ Therefore, \eqref{eqasymptotics} holds for suitable constants $a_0,r_0,A_\gamma$.  This concludes the proof of Proposition \ref{asymptoticcylinder}.\end{proof}

\begin{prop}\label{cylinderconclusion}There exists a finite energy $\tilde J_E$-holomorphic rigid cylinder $$\tilde v_E=(b_E,v_E): \R \times S^1 \to \R \times W_E,$$ with a positive puncture at $+\infty$ and a negative puncture at $-\infty$, such that
\begin{itemize}
\item [(i)] $\tilde v_E$ is the $C^\infty_{\rm loc}$-limit of the $\tilde J_n$-holomorphic cylinders $\tilde v_n=(b_n,v_n)$, contained in $\tilde \F_n, \forall n$, with the normalization \eqref{separa} and possibly a shift in the $S^1$-direction in the domain for each $\tilde v_n$.
\item[(ii)] $\tilde v_E$ is asymptotic to $P_{3,E}$ at $+\infty$ and to $P_{2,E}$ at $-\infty$.  In particular, $$\int_{\R \times S^1 } v_E^*d\lambda_E>0.$$
\item[(iii)] The image of $v_E$ is contained in $\dot S_E \setminus P_{3,E}$.
\item[(iv)] The winding numbers of $\tilde v_E$ at $\pm \infty$ are equal to $1$.
\item[(v)] $v_E$ is transverse to the Reeb vector field $X_{\lambda_E}$.
\item[(vi)] $\tilde v_E$ and  $v_E$ are embeddings.
\end{itemize}
An analogous statement holds for the sequence $\tilde v_n'\in \tilde \F_n$.
\end{prop}

\begin{proof}(i), (ii) and (iii) follow from Proposition \ref{propcylinder}. From the asymptotic description given in Proposition \ref{asymptoticcylinder} and since $P_{2,E}$ is hyperbolic,  we have well-defined winding numbers $\wind_\pi(\tilde v_E)$ and $\wind_\infty(\tilde v_E)$. See Appendix \ref{ap_basics} for details. Moreover, since $CZ(P_{3,E})=3$ and $CZ(P_{2,E})=2$, we have  $$\wind_\infty(+\infty) \leq \wind ^{<0}(A_{P_{3,E}})= 1$$ and $$\wind_\infty(-\infty) \geq \wind^{\geq 0}(A_{P_{2,E}})= 1.$$ From \cite{props2}, we know that $$0\leq \wind_\pi(\tilde v_E)=\wind_\infty(\tilde v_E) = \wind_\infty(+\infty)-\wind_\infty(-\infty) \leq 1-1=0.$$ It follows that $\wind_\infty(+\infty)=\wind_\infty(-\infty)=1$ and $\wind_\pi(\tilde v_E)=0$. Since $\wind_\pi(\tilde v_E)$ counts the zeros of $\pi \circ d v_E$, where $\pi:TW_E \to \xi$ is the projection along the Reeb vector field $X_{\lambda_E}$, and all such zeros are positive, we conclude that $v_E$ is an immersion transverse to the Reeb vector field $X_{\lambda_E}$. This proves (iv) and (v).

In order to prove that $\tilde v_E$ is an embedding, we see from its asymptotic behavior near the punctures $\pm \infty$ that $\tilde v_E$ is an embedding near the boundary, meaning that there exists $R>0$ sufficiently large so that $\tilde v_E^{-1}(\tilde v_E(\{|s|>R\}\times S^1)) = \{|s|>R\} \times S^1$ and $\tilde v_E|_{\{|s|>R\} \times S^1}$ is an embedding. Now since $\tilde v_E$ is the $C^\infty_{\rm loc}$-limit of the embeddings $\tilde v_n$, it follows from results of D. McDuff in \cite{dusa} that $\tilde v_E$ is an embedding as well.

Next we show that $v_E$ is an embedding. First we claim  that if $c\neq 0$ then ${\rm image}(\tilde v_E) \neq {\rm image}(\tilde v_{E,c}),$ where $\tilde v_{E,c}(s,t)=(b_E(s,t)+c,v_E(s,t))$. Arguing indirectly, assume there exists $c\neq 0$ such that ${\rm image}(\tilde v_E) ={\rm image}(\tilde v_{E,c}).$ If $(b,z)=\tilde v_E(s_0,t_0)$, we can find $(s_1,t_1)\neq (s_0,t_0)$ such that $\tilde v_{E,c}(s_1,t_1)$ $=(b,z)$ and therefore $\tilde v_E(s_1,t_1)=(b-c,z)$. Repeating this argument we find a sequence $(s_n,t_n)\in \R \times S^1$ such that $\tilde v_E(s_n,t_n)=(b-nc,z)$. If $c>0$ then $s_n \to -\infty$ and from the asymptotic description of $\tilde v_E$ near $-\infty$ this implies that $z=v_E(s_n,t_n)$ approaches $P_{2,E}$ as $n \to \infty$. This is a contradiction since $z\not\in P_{2,E}$. In the same way, if $c<0$ we have $s_n \to +\infty$ and $z=v_E(s_n,t_n)$ approaches  $P_{3,E}$ as $n \to \infty$, also a contradiction. Thus it follows from Carleman's similarity principle that the intersections of $\tilde v_E$ with $\tilde v_{E,c},c\neq 0,$ are isolated. The existence of any such intersection implies, from positivity and stability of intersections of pseudo-holomorphic curves, that $\tilde v_n$ intersects $\tilde v_{n,c}$ for large $n$, where $\tilde v_{n,c}(s,t)=(b_n(s,t)+c,v_n(s,t))$, and leads to a contradiction. Hence $\tilde v_E(\R \times S^1) \cap \tilde v_{E,c} (\R \times S^1) = \emptyset, \forall c \neq 0$. From the asymptotic behavior of $v_E$ near the punctures $\pm \infty$, from the fact that $\tilde v_E$ is an embedding and that $v_E$ is an immersion, we conclude that $v_E$ is also an embedding.  This proves (vi). \end{proof}

\begin{lem}\label{lemvizivE} Let $\tilde v_E=(b_E,v_E):\R \times S^1 \to \R \times W_E$ be the $C^\infty_{\rm loc}$-limit of the $\tilde J_n$-holomorphic rigid cylinders $\tilde v_n=(b_n,v_n)\in \tilde \F_n$ with the normalization \eqref{separa}, as in Proposition \ref{propcylinder}. Given an $S^1$-invariant open neighborhood $\W$ of the loop $S^1 \ni t \mapsto x_{2,E}(tT_{2,E})$ in $C^\infty(S^1,W_E)$,  there exists $s_1 \ll 0$ such that for all $s<s_1$ and all large $n$, the loop $t \mapsto v_n(s,t)$ is contained in $\W$.  \end{lem}

\begin{proof}We first claim that for every $\epsilon>0$ and every sequence $s_n \to -\infty$, there exists $n_1 \in \N$ such that for all $n \geq n_1$, we have $$\int_{(-\infty,s_n]\times S^1} v_n^* d\lambda_n \leq \epsilon.$$ Arguing indirectly we may assume there exist $\epsilon >0$ and a sequence $s_n \to -\infty$ satisfying $$\int_{(-\infty,s_n]\times S^1} v_n^* d\lambda_n > \epsilon,\forall n.$$ The $d\lambda_n$-energy of $\tilde v_n$ is equal to $T_{3,n}-T_{2,n}$. We conclude that for all $S_0> 0$ fixed and large $n$ we have $$\int_{[-S_0,S_0] \times S^1} v_n^* d\lambda_n \leq T_{3,n}-T_{2,n} - \epsilon.$$ Since $\tilde v_n \to \tilde v_E$ in $C^\infty_{\rm loc}$, we have  $$\int_{[-S_0,S_0] \times S^1} v_E^* d\lambda_E \leq T_{3,E}-T_{2,E}  - \epsilon,$$ which implies that $\int_{\R \times S^1} v_E^*d\lambda_E \leq T_{3,E}-T_{2,E}  - \epsilon,$ since $S_0$ is arbitrary. This is in contradiction with $\int_{\R \times S^1} v_E^* d\lambda_E=T_{3,E}-T_{2,E}$ and thus the claim is proved.

Now we argue again indirectly and assume that the loops $t\mapsto v_{n_j}(s_{n_j},t)$ are not contained in $\W$ for subsequences $\tilde v_{n_j}=(b_{n_j},v_{n_j})$ and $s_{n_j} \to -\infty$, also denoted by $\tilde v_n$ and $s_n$, respectively. Let $\tilde w_n=(d_n,w_n):\R \times S^1 \to \R \times W_E$ be defined by $$\tilde w_n(s,t) = (d_n(s,t),w_n(s,t))=(b_n(s+s_n,t) - b_n(s_n,0),v_n(s+s_n,t)).$$ For each $n$, we have $d_n(0,0)=0$ and $w_n(0,t) \notin \W$ for all $t\in S^1$ and, from Proposition \ref{propcylinder},  $\tilde w_n$ has $C^\infty_{\rm loc}-$bounds.  So we find a $\tilde J_E$-holomorphic cylinder $\tilde w=(d,w):\R \times S^1 \to \R \times W_E$ such that $\tilde w_n \to \tilde w$ in $C^\infty_{\rm loc}$ as $n\to\infty$.

For all fixed $r_0>0$ and all $\epsilon>0$, $$\begin{aligned} \int_{[-r_0,r_0] \times S^1} w^*d\lambda_E \leq & \limsup_n \int_{[-r_0,r_0] \times S^1} w_n^*d\lambda_n \\ = & \limsup_n \int_{[-r_0+s_n,r_0+s_n]\times S^1}  v_n^* d\lambda_n \\ \leq &  \epsilon,\end{aligned}$$ for all large $n$. This follows from the first claim and implies that $\int_{\R \times S^1} w^* d\lambda_E = 0.$  Moreover, since $\tilde v_n \to \tilde v_E$ in $C^\infty_{\rm loc}$ and $s_n \to -\infty$ as $n \to \infty$, given any $\epsilon>0$, we have   $$T_{2,n} \leq \int_{\{0\} \times S^1} w_n^* \lambda_n = \int_{\{s_n\} \times S^1} v_n^* \lambda_n<T_{2,E}+\epsilon,$$ for all large $n$. Since $T_{2,n} \to T_{2,E}$ as $n \to \infty$, we get $\int_{\{0\} \times S^1} w^* \lambda_E=T_{2,E}$. It follows from Lemma \ref{lemamin} that $\tilde w$ is a cylinder over $P_{2,E}$. However, this contradicts the fact that the loop $t \mapsto w(0,t)$ is not contained in $\W$. \end{proof}

It follows from Proposition \ref{propcylinder}, Lemmas \ref{lemvizi} and \ref{lemvizivE} that given a neighborhood $\U \subset W_E$ of $v_E(\R \times S^1) \cup P_{2,E} \cup P_{3,E}$, there exists $n_0 \in \N$ so that if $n\geq n_0$, then $v_n(\R \times S^1) \subset \U$. As mentioned before, the case of the sequence $\tilde v_n'$ is totally analogous. This completes the proof of Proposition \ref{prop_step5iib}.

\begin{figure}[h!!]
  \centering
  \includegraphics[width=0.35\textwidth]{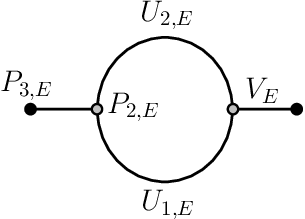}
  \caption{So far we have obtained the binding orbits $P_{2,E}$ and $P_{3,E}$, the pair of rigid planes $U_{1,E}=u_{1,E}(\C)$ and $U_{2,E}=u_{2,E}(\C)$ and the rigid cylinder $V_E=v_E(\R \times S^1)$ connecting $P_{3,E}$ to $P_{2,E}$.}
  \label{fig_cilindro}
\hfill
\end{figure}

\section{Proof of Proposition \ref{prop_step5}-{\rm iii)}}\label{sec_5-iii)}

We start this section recalling what we have obtained so far. Let $\tilde \F_n$ be the stable finite energy foliation on $\R \times W_E$ associated to $\tilde J_n=(\lambda_n,J_n)$ as in Proposition \ref{prop_step4}-i), which projects onto a $3-2-3$ foliation $\F_n$ on $W_E$ with binding orbits $P_{3,n},P_{2,n}$ and $P_{3,n}'$. Let $\tilde J_E=(\lambda_E,J_E)$ be the almost complex structure in $\R\times W_E$ constructed in Propositions \ref{prop_step1} and \ref{prop_step3}, which satisfies $\lambda_n \to \lambda_E$ and $J_n \to J_E$ in $C^\infty$ as $n\to\infty$. Let $P_{3,E},P_{2,E}$ and $P_{3,E}'$ be the periodic orbits of $\lambda_E$, with $P_{3,n} \to P_{3,E}$, $P_{2,n} \to P_{2,E}$ and $P_{3,n}' \to P_{3,E}'$ in $C^\infty$ as $n\to\infty$, for which we can assume $P_{3,n}=P_{3,E}, P_{2,n}=P_{2,E},P_{3,n}'=P_{3,E}'$, $\forall n$, as point sets in $W_E$, according to Proposition \ref{prop_step4}-ii). For each $n$, $\tilde\F_n$ contains a pair of $\tilde J_n$-rigid planes $\tilde u_{1,n}=(a_{1,n},u_{1,n})$ and $\tilde u_{2,n}=(a_{2,n},u_{2,n})$,  both asymptotic to $P_{2,n}$, so that $\tilde u_{1,n} \to \tilde u_{1,E}$ and $\tilde u_{2,n} \to \tilde u_{2,E}$ in $C^\infty_{\rm loc}$ as $n \to \infty$, up to reparametrizations and $\R$-translations. The $\tilde J_E$-rigid planes $\tilde u_{1,E}=(a_{1,E},u_{1,E})$ and $\tilde u_{2,E}=(a_{2,E},u_{2,E})$ are  both asymptotic to $P_{2,E}$ and $u_{1,E}(\C)  \cup P_{2,E}\cup u_{2,E}(\C)$ separates $W_E$ in two components $\dot S_E$ and $\dot S_E'$ containing $P_{3,E}$ and $P_{3,E}'$, respectively. The foliation $\tilde\F_n$ also contains a pair of   $\tilde J_n$-rigid cylinders $\tilde v_{n}=(b_{n},v_{n})$ and $\tilde v_{n}'=(b_{n}',v_{n}')$ such that $\tilde v_n$ and $\tilde v_n'$ are asymptotic to $P_{3,n}$ and $P_{3,n}'$ at their positive punctures, respectively, and both are asymptotic to $P_{2,n}$ at their negative punctures. Up to reparametrizations and $\R$-translations, we find $\tilde J_E$-rigid cylinders $\tilde v_E=(b_E,v_E)$ and $\tilde v_E'=(b_E',v_E')$ asymptotic to $P_{3,E}$ and $P_{3,E}'$ at their positive punctures, respectively, and both asymptotic to $P_{2,E}$ at their negative punctures, such that $\tilde v_n\to \tilde v_E$ and $\tilde v_n'\to \tilde v_E'$ in $C^\infty_{\rm loc}$ as $n \to \infty$. The images $V_E=v_E(\R \times S^1)$ and $V_E'= v_E'(\R \times S^1)$ lie in $\dot S_E$ and $\dot S_E',$ respectively. Finally, $\tilde \F_n$ contains one parameter families of $\tilde J_n$-planes $\tilde w_{\tau,n}=(d_{\tau,n},w_{\tau,n})$ and $\tilde w_{\tau,n}'=(d_{\tau,n}',w_{\tau,n}')$, $\tau \in (0,1)$, with $\tilde w_{\tau,n}$ asymptotic to $P_{3,n}$ and $\tilde w_{\tau,n}'$ asymptotic to $P_{3,n}'$. The compactness properties of $\tilde w_{\tau,n}$ and $\tilde w_{\tau,n}'$ are the matter of this section. We prove the Proposition \ref{prop_step5}-iii), which is restated below.

\begin{prop}\label{prop_step5iiib} If $E>0$ is sufficiently small then, up to suitable reparametrizations and $\R$-translations,  the following holds: there exist smooth one parameter families of $\tilde J_E$-holomorphic embedded planes $\tilde w_{\tau,E}=(d_{\tau,E},w_{\tau,E}),$ $\tilde w_{\tau,E}'=(d_{\tau,E}',w_{\tau,E}'):\C \to \R \times W_E$, $\tau \in (0,1)$, with $w_{\tau,E}$ and  $w_{\tau,E}'$ embeddings,   $\tilde w_{\tau,E}$ asymptotic to $P_{3,E}$, $D_{\tau,E}=w_{\tau,E}(\C),\tau \in (0,1),$ foliating $\dot S_E \setminus (V_E\cup P_{3,E})$,  $\tilde w_{\tau,E}'$ asymptotic to $P_{3,E}'$  and $D_{\tau,E}'=w_{\tau,E}'(\C),\tau \in (0,1),$ foliating $\dot S_E' \setminus (V_E' \cup P_{3,E}')$.  Given small neighborhoods $\U_1,\U_2\subset W_E$ of $P_{3,E} \cup V_{E} \cup P_{2,E} \cup U_{1,E}$ and  $P_{3,E} \cup V_{E} \cup P_{2,E} \cup U_{2,E}$, respectively, we find $0<\tau_1<\tau_2<1$  so that if $\tau\in (0,\tau_1)$ then $D_{\tau,E} \subset \U_1$ and if $\tau\in (\tau_2,1)$ then $D_{\tau,E} \subset \U_2$. An analogous statement holds for the family $D_{\tau,E}',\tau \in (0,1)$. Moreover, given $p_0 \in \dot S_E \setminus (V_E \cup P_{3,E})$, we find a sequence $\tau_n \in (0,1)$ and $\bar \tau \in(0,1)$ so that $p_0 \in w_{\tau_n,n}(\C)$ for each large $n$, $p_0 \in w_{\bar \tau,E}(\C)$ and $\tilde w_{\tau_n,n} \to \tilde w_{\bar \tau,E}$ in $C^\infty_{\rm loc}$ as $n \to \infty$. An analogous statement holds if $p_0 \in \dot S_E' \setminus (V_E' \cup P_{3,E}')$.
\end{prop}

In the following, we only work on $\dot S_E$ since the case $\dot S_E'$ is completely analogous.
Let $\V_E=\dot S_E \setminus (V_E \cup P_{3,E})$.  $\V_E$ is homeomorphic to the open $3$-ball and we shall foliate it with a  one parameter family of planes asymptotic to $P_{3,E}$.

Let us fix an embedded $2$-sphere $N^\delta_E\subset \dot S_E$ so that $N^\delta_E$ separates $P_{3,n}$ and $P_{2,n}\forall n$. In local coordinates $(q_1,q_2,p_1,p_2)$,  it is given by $$N^\delta_E=\{q_1 + p_1 = \delta\}\cap K^{-1}(E),$$ where $\delta>0$ is small. Denote by $U^\delta_E \subset \dot S_E$ the component of $W_E \setminus N^\delta_E$ containing $P_{3,E}$.

The embedded topological $2$-sphere $u_{1,n}(\C) \cup P_{2,n} \cup u_{2,n}(\C)$ separates $W_E$ in two components $\dot S_n, \dot S_n'$ containing $P_{3,n},P_{3,n}'$ respectively. The image $v_n(\R \times S^1)$ is contained in $\dot S_n\setminus P_{3,n}$ and let $\V_n = \dot S_n \setminus (v_n(\R \times S^1) \cup P_{3,n})$. For each $n$, the one parameter family of finite energy $\tilde J_n$-holomorphic planes $\tilde w_{\tau,n}=(d_{\tau,n},w_{\tau,n}):\C \to \R \times W_E$, $\tau \in (0,1)$, is so that $w_{\tau,n}(\C)$ foliates $\V_n$.

By the convergence  of $\tilde u_{1,n},\tilde u_{2,n}$ and $\tilde v_{n}$ to $\tilde u_{1,E},\tilde u_{2,E}$ and $\tilde v_E$, respectively, obtained in Proposition \ref{prop_step5}-i)-ii), we know that given $p_0\in \V_E$, we have $p_0\in \V_n$ for all large $n$. We fix such $p_0\in \V_E$ and consider for each large $n$ the unique, up to $\R$-translation, $\tilde J_n$-holomorphic plane $\tilde w_{\tau_n,n}=(d_{\tau_n,n},w_{\tau_n,n}):\C \to \R \times W_E$, asymptotic to $P_{3,n}$ at $\infty$, satisfying $p_0 \in w_{\tau_n,n}(\C)\subset \V_n$. For simplicity we denote these planes simply by $\tilde w_n=(d_n,w_n)$.

The main purpose of this section is to prove the following proposition.

\begin{prop}\label{propplanes}For $E>0$ sufficiently small, the following holds: let $p_0\in \V_E$ and consider, for each large $n$, a finite energy $\tilde J_n$-holomorphic plane $\tilde w_n=(d_n,w_n)\in \tilde \F_n$ asymptotic to $P_{3,n}$  at $+\infty$ so that $p_0\in w_n(\C)$. Then, up to reparametrization and $\R$-translation, $|\nabla \tilde w_n(z)|_n$ is uniformly bounded in $z\in \C$ and in $n\in \N$ and there exists a finite energy $\tilde J_E$-holomorphic plane $\tilde w=(d,w):\C \to \R \times W_E,$ asymptotic to $P_{3,E}$ at $+\infty$, so that $p_0\in w(\C)\subset \V_E$ and $\tilde w_n \to \tilde w$ in $C^\infty_{\rm loc}$ as $n \to \infty$. Moreover, $\tilde w$ and $w$ are embeddings and $w$ is transverse to the Reeb vector field $X_{\lambda_E}$. \end{prop}

\begin{proof} Let us assume first that there exists $\bar \delta>0$ small such that for a subsequence of $\tilde w_n$, still denoted by $\tilde w_n$, we have \begin{equation}\label{naointersecta} w_n(\C) \cap N^{\bar \delta}_E = \emptyset.\end{equation} In this case we reparametrize $\tilde w_n=(d_n,w_n)$ and slide it in the $\R$-direction so that \begin{equation} \label{reparamwnn1} \int_{\C \setminus \D} w_n^*d\lambda_n = \gamma_0 \mbox{ and } d_n(0) = \min \{d_n(z):z\in \C\}=0, \end{equation} where $0<\gamma_0\ll T_{3,E}$ is fixed and $\D=\{z\in \C:|z|\leq 1\}$ is the closed unit disk.

Arguing indirectly we find a subsequence of $\tilde w_n$,  still denoted by $\tilde w_n$, and a sequence $z_n \in \C$ such that $$|\nabla \tilde w_n(z_n)|_n \to \infty.$$ Take any sequence of positive numbers $\epsilon_n \to 0^+$ satisfying \begin{equation}\label{eq_explode} \epsilon_n |\nabla \tilde w_n(z_n)|_n \to \infty.\end{equation} Using Ekeland-Hofer's Lemma (Lemma \ref{lemtop}) for $X=(\C,|\cdot|_\C)$, $f = |\nabla \tilde w_n(\cdot) |_n$ and $n$ fixed, we assume furthermore, perhaps after slightly changing  $z_n,\epsilon_n$, that \begin{equation}\label{eqalimitn}|\nabla \tilde w_n(z)|_n \leq 2 |\nabla \tilde w_n(z_n)|_n, \mbox { for all } |z-z_n| \leq \epsilon_n.\end{equation} Defining the new sequence of $\tilde J_n$-holomorphic maps $\tilde u_n: \C \to \R \times W_E$  by \begin{equation}\label{defutn} \tilde u_n(z)=(b_n,u_n) = \left(d_n\left(z_n+\frac{z}{R_n}\right)-d_n(z_n),w_n\left(z_n+ \frac{z}{R_n}\right)\right),\end{equation} with $R_n = |\nabla \tilde w_n(z_n)|_n$, we see from \eqref{eqalimitn} and \eqref{defutn} that \begin{equation}\label{utlimita} \begin{aligned} \tilde u_n(0) & \in \{0\} \times W_E, \\ |\nabla \tilde u_n(z)|_n & \leq 2, \forall |z|\leq \epsilon_n R_n.\end{aligned}\end{equation} From an elliptic bootstrapping argument we get $C^\infty_{\text{loc}}$-bounds for the sequence $\tilde u_n$ and, since $\epsilon_nR_n \to \infty$, after taking a subsequence we have $\tilde u_n \to \tilde u$ in $C^\infty_{\text{loc}}$, where $\tilde u=(b,u):\C \to \R \times W_E$ is a $\tilde J_E$-holomorphic plane satisfying $b(0)=0$ and $|\nabla \tilde u(0)|_E=1$, where $|\cdot |_E$ is the norm induced by $\tilde J_E$. Therefore $\tilde u$ is non-constant and from $E(\tilde u_n)\leq E(\tilde w_n)=T_{3,n}$, we have \begin{equation}\label{desigdln} E(\tilde u)\leq T_{3,E} \mbox{ and } 0<\int_{\C} u^*d\lambda_E \leq T_{3,E}.\end{equation} Since there exists no periodic orbit in $U_E^{\bar \delta}\setminus P_{3,E}$ with action $\leq T_{3,E}$ which is not linked to $P_{3,E}$, see Theorem \ref{teolink}, we conclude that $\tilde u$ is asymptotic to $P_{3,E}$ at its positive puncture $+\infty$. In particular, we have \begin{equation} \label{eqTEn} E(\tilde u) = \int_\C u^* d\lambda_E = T_{3,E}.\end{equation}  We must have $\displaystyle \lim_{n \to \infty} |z_n| = 1$, otherwise for any $\epsilon>0$ we take a subsequence with either $|z_n|\geq 1+\epsilon$ or $|z_n|\leq 1-\epsilon,\forall n$, and from \eqref{reparamwnn1} we obtain $$\begin{aligned} \int_{B_R(0)}u^*d\lambda_E = &  \lim_{n \to \infty} \int_{B_R(0)}u_n^*d\lambda_n  \\ \leq & \limsup_{n \to \infty} \int_{B_{\epsilon_nR_n}(0)}u_n^* d\lambda_n \\ = & \limsup_{n \to \infty} \int_{B_{\epsilon_n}(z_n)} w^*_n d\lambda_n  \\ \leq & \limsup_{n \to \infty} \max \{\gamma_0,T_{3,n} - \gamma_0\} =   \max \{\gamma_0,T_{3,E} - \gamma_0\}= T_{3,E}-\gamma_0\end{aligned} $$ for any $R>0$ fixed. This implies $\int_\C u^* d\lambda_E \leq T_{3,E}-\gamma_0 < T_{3,E}$,  contradicting \eqref{eqTEn}. Thus after taking a subsequence we assume $z_n \to z_* \in \partial \D$.  Since $z_*$ takes away $T_{3,E}$ from the $d\lambda_E$-energy, there cannot be any other bubbling-off point for the sequence $\tilde w_n$ other than $z_*$. Hence from \eqref{reparamwnn1}  and elliptic regularity we have $C^\infty_{\text{loc}}$-bounds for $\tilde w_n$ in $\C \setminus \{ z_*\}$. Taking a subsequence we have $\tilde w_n \to \tilde v$ for a $\tilde J_E$-holomorphic map $\tilde v=(a,v):\C \setminus \{z_*\} \to \R \times W_E$ which is non-constant since from \eqref{reparamwnn1} we have \begin{equation}\label{eqS2n}\int_{\{|z|=2\}}v^* \lambda_E = \lim_{n \to \infty} \int_{\{|z|=2\}}w^*_n \lambda_n \geq T_{3,E} - \gamma_0 >0.\end{equation} Moreover, from Fatou's lemma, we get \begin{equation}\label{desigE3n} E(\tilde v) \leq T_{3,E},\end{equation} and from \eqref{eqTEn} we have \begin{equation}\label{eqdl0n}\int_{\C \setminus \{z_*\}} v^* d\lambda_E = 0.\end{equation} From \eqref{eqS2n} and \eqref{eqdl0n}, we see that $z_*$ is non-removable and from the characterization of finite energy cylinders with vanishing $d\lambda_E$-energy, see \cite[Theorem 6.11]{props2}, we find a periodic orbit $P\subset W_E$ so that $\tilde v$ maps $\C \setminus \{z_*\}$ onto $\R \times P$.  However, from \eqref{reparamwnn1}, the $\R$-component of $\tilde v$ is bounded from below by $0$. This contradiction proves that there are no bubbling-off points for $\tilde w_n$.
From \eqref{reparamwnn1}  and usual elliptic regularity we obtain $C^\infty_{\text{loc}}$-bounds for our sequence $\tilde w_n$, thus we find a $\tilde J_E$-holomorphic plane $\tilde w=(d,w):\C \to \R \times W_E$ so that, up to extraction  of a subsequence, $\tilde w_n \to \tilde w$ in $C^\infty_{\rm loc}$ as $n\to\infty$. From  \eqref{reparamwnn1} and since $T_{3,n} \to T_{3,E}$ as $n\to \infty$ we have  $0<E(\tilde w) \leq T_{3,E}$, hence $0<\int_\C w^*d\lambda_E \leq T_{3,E}$. Since there exists no periodic orbit in $U^{\bar \delta}_E \setminus P_{3,E}$ with action $\leq T_{3,E}$ which is not linked to $P_{3,E}$, see Theorem \ref{teolink}, we conclude that $\tilde w$ is asymptotic to $P_{3,E}$ at $+\infty$.

Now assume that for all large $n$, \begin{equation}\label{intersecta} w_n(\C) \cap N^{\delta_0}_E\neq \emptyset \mbox{ for some } \delta_0>0 \mbox{ small.}\end{equation}

Fix $\delta_1>\delta_0> 0$ small and let $g:[0,+\infty) \to [0,1]$ be a smooth non-decreasing function satisfying $g(t)=0$ for $t\in[0,\delta_0]$, $g|_{[\delta_0,\delta_1]}$ is strictly increasing  and $g(t)=1$ if $t\geq \delta_1$. $g$ induces a smooth function  $G:S_E\to [0,1]$ which in local coordinates $x=(q_1,q_2,p_1,p_2)\in V$ is given by $G(x)= g(q_1+p_1)$ and $G\equiv 1$ outside $V$, where $V$ is the neighborhood of the saddle-center  as in Hypothesis 1. The function $G$ extends to a smooth function on $W_E$ also denoted by $G$, by declaring $G|_{\dot S_E'} \equiv 0$.

The function $\bar G_n:\C \to [0,1]$ defined by $\bar G_n(z) = G(w_n(z))$ is smooth for all $n$. From \eqref{intersecta} and since $\tilde w_n$ is asymptotic to $P_{3,n}$ at $+\infty$, the image of $\bar G_n$ coincides with $[0,1]$. Moreover, $\bar G_n(z)=1$ if $|z|$ is large. By Sard's theorem we find $y_n\in (\frac{1}{2},1)$ so that $y_n$ is a regular value for $\bar G_n$. This implies that $K_n:=\bar G_n^{-1}(y_n)\subset \C$ is a non-empty compact set formed by finitely many embedded circles and we let
$$r_n = \inf \{r>0: \exists \mbox { a closed disk } D_r \subset\C \mbox{ with radius } r \mbox{ s.t. } K_n\subset D_r\}.$$
Since $K_n$ contains at least two points, we have $r_n>0$ and there exists a closed disk $D_{r_n}$ with radius $r_n$ containing $K_n$. After a change of coordinates of the type $\C \ni z \mapsto az + b$, we can assume that $D_{r_n}$ coincides with the closed unit disk $\D\subset \C$ and that $-i \in K_n$. From the definition of $r_n$ we can moreover assume that $\partial D_{r_n} \cap K_n$ contains at least two points and there exists $z'_n\in \partial \D \cap K_n$ satisfying $\Im z_n' \geq 0$.

Hence we can  reparametrize $\tilde w_n$ and slide it in the $\R$-direction so that \begin{equation}\label{reparamwnn2}\begin{aligned}& \bar G_n|_{\C \setminus \D} > y_n \Rightarrow w_n(\C \setminus \D) \subset U^{\delta_n}_E,\\ & \bar G_n(-i)=y_n \Rightarrow w_n(-i)\in N^{\delta_n}_E, \\ & \bar G_n(z_n')=y_n \Rightarrow w_n(z_n')  \in N^{\delta_n}_E, \mbox{ with } z_n'\in \partial \D,\Im z_n'  \geq 0,  \\& d_n(2)=0, \end{aligned}\end{equation}where $\delta_n=g^{-1}(y_n)$. With this renormalization we necessarily have from \eqref{intersecta}
\begin{equation}\label{eq_dentro}
\exists z''_n\in \D\setminus \partial \D \mbox{ so that } w_n(z''_n)\in N^{\delta_0}_E.
\end{equation}

Now we show that with this parametrization $\tilde w_n$ does not admit bubbling-off points. Arguing indirectly, assume that up to extraction of a subsequence, we find a sequence $z_n \in \C$ such that $|\nabla \tilde w_n(z_n)|_n \to +\infty$ as $n \to \infty$. First we show that $\limsup_{n \to \infty} |z_n| \leq 1$. Otherwise we find $\epsilon>0$ so that up to extraction of a subsequence we have $|z_n|> 1+\epsilon$. Now after using Ekeland-Hofer's Lemma, we can find a sequence $\epsilon_n \to 0^+$ so that \eqref{eqalimitn} holds and the maps $\tilde u_n=(b_n,u_n):\C \to \R\times W_E$, defined as in \eqref{defutn}, are such that, up to extraction of a subsequence, $\tilde u_n \to \tilde u$ in $C^\infty_{\rm loc}$ as $n \to \infty$, where $\tilde u=(b,u):\C \to \R \times W_E$ is a non-constant finite energy $\tilde J_E$-holomorphic plane. It must satisfy $\int_\C u^*d\lambda_E \leq T_{3,E}$ since $\int_\C u_n^* d\lambda_n = \int_\C w_n^* d\lambda_n = T_{3,n}\to T_{3,E}$ as $n \to \infty$. From \eqref{reparamwnn2} and since there exists no periodic orbit in $U_E^{\delta_0}\setminus P_{3,E}$ with action $\leq T_{3,E}$ which is not linked to $P_{3,E}$, see Theorem \ref{teolink}, $\tilde w$ must be asymptotic to $P_{3,E}$ at $+\infty$ and $E(\tilde u)=T_{3,E}$. This implies $\int_\C u^* d\lambda_E = T_{3,E}$.

Assume that for a subsequence $|z_n|$ is bounded. Then we can also assume $z_n \to z_*\in \C$ with $|z_*|\geq 1+\epsilon.$ Since the bubbling-off point $z_*$ takes away $T_{3,E}$ of the $d\lambda_E$-area, we cannot have any other bubbling-off point for the sequence $\tilde w_n$ in $\C \setminus \{z_*\}$. Thus from \eqref{reparamwnn2} we have $C^\infty_{\rm loc}$-bounds for $\tilde w_n$ in $\C \setminus \{z_*\}$, possibly after translating the sequence $\tilde w_n$ in the $\R$-direction appropriately. Therefore, after taking a subsequence, we find a map $\tilde w=(d,w):\C \setminus \{z_*\} \to \R \times W_E$ so that $\tilde w_n|_{\C \setminus \{z_*\}} \to \tilde w$ in $C^\infty_{\rm loc}$ as $n\to\infty$. Since $z_*$ is a negative puncture of $\tilde w$,  $\tilde w$ is non-constant. Moreover, we must have $\int_{\C \setminus \{z_*\}} w^* d\lambda_E=0$ and $E(\tilde w)\leq T_{3,E}$. So $\tilde w$ is a cylinder over a periodic orbit with action $\leq T_{3,E}$. By \eqref{reparamwnn2} and since there is no periodic orbit in $\dot S_E$ with action $\leq T_{3,E}$ which is not linked to $P_{3,E}$,  see Theorem \ref{teolink}, $\tilde w$ must be a cylinder over $P_{3,E}$.  This contradicts the fact that there are points in $\D$ being mapped by $w$ into $N^{\delta_0}_E$, see \eqref{eq_dentro}. Now if $|z_n|\to \infty$, then in the same way, since $z_n$ takes away $T_{3,E}$ of the $d\lambda_E$-area, there is no bubbling-off point in $\C$ and we have $C^\infty_{\rm loc}$-bounds for the sequence $\tilde w_n$ and after taking a subsequence we have $\tilde w_n \to \tilde w=(d,w):\C \to \R \times W_E$, where $\tilde w$ has finite energy. In this case, we must have $\int_\C w^* d\lambda_E=0$ which implies that $\tilde w$ is constant. This contradicts the fact that there must exist a point $z_1 \in \D$ which is mapped under $w$ into $N^{\delta_0}_E$ and a point $z_2\in \D$ which is mapped under $w$ into $N^{\bar \delta}_E$, with $g(\bar \delta)\geq\frac{1}{2}>0=g(\delta_0) \Rightarrow \bar \delta \neq \delta_0\Rightarrow w(z_1)\neq w(z_2)$.

We have proved that any sequence of bubbling-off points $z_n$ of $\tilde w_n$ satisfies \begin{equation}\label{znlimitado} \limsup_{n \to \infty} |z_n| \leq 1.\end{equation} Let $\emptyset \neq \Gamma \subset \C$ be the set of points $z\in \C$ such that for a subsequence $\tilde w_{n_j}$ of $\tilde w_n$, there exists a sequence $z_{j} \to z$ satisfying $|\nabla \tilde w_{n_j}(z_j)|_{n_j} \to \infty$ as $j \to \infty$.

As explained in the proof of Proposition \ref{propcylinder}, we can use Ekeland-Hofer's Lemma to show that up to extraction of a subsequence we may assume that $\# \Gamma$ is finite. From \eqref{znlimitado}, we know that  $\Gamma \subset \D$.  Thus from the last equation in \eqref{reparamwnn2} and elliptic regularity  we find a non-constant $\tilde J_E$-holomorphic map $\tilde w=(d,w):\C \setminus \Gamma \to \R \times W_E$ satisfying \begin{equation}\label{wnw}\tilde w_n \to \tilde w \mbox{ in }C^\infty_{\rm loc}\mbox{ as }n\to \infty. \end{equation}Moreover, from Fatou's Lemma, we have $E(\tilde w)\leq \limsup_n T_{3,n} = T_{3,E}$. All punctures in $\Gamma$ are negative.  From \eqref{reparamwnn2} and the fact that there is no periodic orbit in $U^{\delta_0}_E \setminus P_{3,E}$ with action $\leq T_{3,E}$ which is not linked to $P_{3,E}$, for $E>0$ sufficiently small, see Theorem \ref{teolink}, $\tilde w$ is asymptotic to $P_{3,E}$ at $+\infty$ and hence $\tilde w$ is somewhere injective.
Note further that $w(\C\setminus \Gamma) \subset \dot S_E$. In fact, if $w(\C \setminus \Gamma) \cap \partial S_E \neq \emptyset,$ then by Carleman's similarity principle, intersections would be isolated and, by stability and positivity of such intersections, they would imply intersections of $w_n(\C)$ with $u_{1,n}(\C) \cup P_{2,n} \cup u_{2,n}(\C)$ for large $n$, a contradiction.

Following our indirect argument that $\Gamma \neq \emptyset$, we first show that  $\#\Gamma =1$, i.e., $\tilde w$ has precisely one puncture $z^*\in \Gamma \subset \D$ at which $\tilde w$ is asymptotic to $P_{2,E}$. From the linking property obtained in Theorem \ref{teolink}, all the asymptotic limits at the negative punctures of $\tilde w$ must coincide either with $P_{3,E}$ or to a cover of $P_{2,E}$. If for a negative puncture $z^*_i$, $\tilde w$ is asymptotic to $P_{3,E}$, then $z^*=z^*_i$ is the only negative puncture, otherwise we would have $\int_{\C \setminus\Gamma} w^*d\lambda_E<T_{3,E}-T_{3,E}=0,$ a contradiction. In this case, we must have $\int_{\C \setminus \{z^*\}}w^*d\lambda_E =0$ and since $\tilde w$ is non-constant, it follows that $\tilde w$ is a cylinder over $P_{3,E}$. Let $\delta_n\in (\frac{1}{2},1)$ be as in \eqref{reparamwnn2}. We may assume that $\delta_n \to \delta'\in[\frac{1}{2},1]$ as $n \to \infty$. If  $z^* \neq -i$, then from \eqref{reparamwnn2} we have $w(-i) \in N^{\delta'}_E$. If $z^* = -i$ then, again from \eqref{reparamwnn2}, we may assume that $z_n' \to z'$, where $\Im z'\geq 0$. In this case we have $w(z') \in N^{\delta'}_E$. In both cases we get a contradiction, since $P_{3,E} \cap N^{\delta'}_E=\emptyset$, which shows that $\tilde w$ is asymptotic to covers of $P_{2,E}$ at all of its negative punctures. In case $\#\Gamma \geq 2$ or in case $\tilde w$ is asymptotic to a $p$-cover of $P_{2,E}$ with $p\geq 2$, $w$ must self-intersect, see Proposition \ref{propintersect}-i)-ii). This follows from the fact that $\tilde w$ is somewhere injective.  However, since $\int_{\C \setminus \Gamma} w^* d\lambda_E>0$, these self-intersections of $w$ correspond to isolated intersections of $\tilde w$ with $\tilde w_c(s,t):=(d(s,t)+c,w(s,t))$ for suitable $c\in \R$. From positivity and stability of intersections of pseudo-holomorphic maps, this implies intersections of $\tilde w_n$ with $\tilde w_{n,c}(s,t):=(d_n(s,t)+c,w_n(s,t))$ for all large $n$, a contradiction.

So far we have proved that  the $\tilde J_E$-holomorphic cylinder $\tilde w=(d,w):\C \setminus \{z^*\} \to \R \times W_E$  is asymptotic to $P_{3,E}$ at $+\infty$ and has precisely one negative puncture $z^*$ at which $\tilde w$ is asymptotic to $P_{2,E}$.

{\it Claim 0. The following holds if $E>0$ is sufficiently small. Given any $S^1$-invariant small neighborhood $\W_1$ of $t\mapsto x_{3,E}(T_{3,E}t)$ in the loop space $C^\infty(S^1, W_E)$, there exists $R_1\gg 0$ such that \begin{equation}\label{vizii2} \mbox{ the loop } S^1 \ni t \mapsto w_n(Re^{2 \pi it}) \mbox{ is contained in } \W_1,\forall R \geq R_1,\forall \mbox{ large } n.\end{equation}}
The proof of Claim 0 follows the same lines of the proofs of Lemmas \ref{lemvizi10} and \ref{lemvizi}. We include it here for completeness. First we claim that for every $\epsilon>0$ and every sequence $R_n \to +\infty$, there exists $n_0 \in \N$ such that for all $n \geq n_0$, we have \begin{equation}\label{eq_limita_energia} \int_{\C \setminus B_{R_n}(0)} w_n^* d\lambda_n \leq \epsilon.\end{equation} Arguing indirectly we may assume there exists $\epsilon >0$ and a sequence $R_n \to +\infty$ satisfying $$\int_{\C \setminus B_{R_n}(0)} w_n^* d\lambda_n > \epsilon,\forall n.$$ The $d\lambda_n$-energy of $\tilde w_n$ is equal to $T_{3,n}$. We conclude that for all $R_0>0$ fixed and large $n$  we have $$\int_{B_{R_0}(0)} w_n^* d\lambda_n \leq T_{3,n} - \epsilon.$$ Since $z^*$ is a negative puncture of $\tilde w$, given any $r_0>0$ fixed sufficiently small we have $\int_{B_{r_0}(z^*)} w_n^* d\lambda_n \geq T_{2,E}-\epsilon/2$ for all large $n$. Since $\tilde w_n \to \tilde w$ in $C^\infty_{\rm loc}$, we have  $$\int_{B_{R_0}(0) \setminus B_{r_0}(z^*)} w^* d\lambda_E \leq T_{3,E}-\epsilon - (T_{2,E} - \epsilon/2)= T_{3,E} - T_{2,E} - \epsilon/2,$$ which implies that $\int_{\C \setminus \{z^*\}} w^*d\lambda_E \leq T_{3,E} - T_{2,E} - \epsilon/2.$ This is a contradiction and proves \eqref{eq_limita_energia}.

In order to prove Claim 0, we argue indirectly assuming that the loops $t\mapsto w_{n_j}(R_{n_j}e^{2 \pi t i})$ are not contained in $\W_1$ for subsequences $\tilde w_{n_j}=(d_{n_j},w_{n_j})$ and $R_{n_j} \to +\infty,$ also denoted by $\tilde w_n$ and $R_n$, respectively. Let $$\tilde v_n(s,t) = (b_n(s,t),v_n(s,t))=\left(d_n\left(R_ne^{2 \pi (s+ i t)}\right) - d_n(R_n),w_n\left(R_ne^{2 \pi (s+ i t)}\right)\right).$$ Notice that for all $n$, $b_n(0,0)=0$ and $v_n(0,t) \notin \W_1$ for every $t\in S^1.$ Moreover, this sequence has $C^\infty_{\rm loc}-$bounds since any bubbling-off point in $\R \times S^1$ would take away at least $T_{2,E}>0$ from the $d\lambda_E$-area, contradicting \eqref{eq_limita_energia}. Thus we find a $\tilde J_E$-holomorphic map $\tilde v=(b,v):\R \times S^1 \to \R \times W_E$ such that $\tilde v_n \to \tilde v$ in $C^\infty_{\rm loc}$ as $n\to\infty$. Also from \eqref{eq_limita_energia}, we have for all fixed $S_0>0$ and all $\epsilon>0$ $$\begin{aligned} \int_{[-S_0,S_0] \times S^1} v^*d\lambda_E \leq & \limsup_n \int_{[-S_0,S_0] \times S^1} v_n^*d\lambda_n \\ \leq & \limsup_n \int_{B_{e^{2 \pi S_0}R_n}(0) \setminus B_{e^{-2 \pi S_0}R_n}(0)} w_n^* d\lambda_n \\ \leq &  \epsilon,\end{aligned}$$ for all large $n$. This implies that $\int_{\R \times S^1} v^* d\lambda_E = 0.$  Moreover,  $$T_{3,E} \geq \int_{\{0\} \times S^1} v^* \lambda_E = \lim_{n \to \infty} \int_{\{0\} \times S^1} v_n^* \lambda_n  = \lim_{n \to \infty} \int_{\partial B_{R_n}(0)} w_n^* \lambda_n \geq T_{2,E}.$$ It follows that $\tilde v$ is a cylinder over a periodic orbit $Q=(x,T)\in \P(\lambda_E)$ which, by construction, is geometrically distinct from $P_{3,E}$ and has action $\leq T_{3,E}$. Now since the loops $t \mapsto v_n(0,t)=w_n(R_n e^{2 \pi i t})$ have image contained in $U^{\delta_0}_E$ from \eqref{reparamwnn2} and converge to $t \mapsto x(tT+c_0)$ in $C^\infty$ as $n \to \infty$, for a suitable constant $c_0$, we conclude that $Q$ is also geometrically distinct from $P_{2,E}$. Moreover, these loops are not linked to $P_{3,E}$, which implies that $Q$ is not linked to $P_{3,E}$ as well. This contradicts Theorem \ref{teolink},  if $E>0$ is taken sufficiently small, proving Claim 0.

We can extract more information about the behavior of $\tilde w$ near $+\infty$ as in Proposition \ref{asymptoticcylinder}. In fact, in Martinet's coordinates near $P_{3,E}$, we express $$\tilde w\left(e^{2 \pi (s + it)}\right)=(d(s,t),\vartheta(s,t),x(s,t),y(s,t)),s\gg 0,$$ and it satisfies the exponential estimates \eqref{eqasymptotics} (with $b_E(s,t)$ replaced with $d(s,t)$), for an eigenvalue $\delta <0$ and a $\delta$-eigensection $e$ of the asymptotic operator $A_{P_{3,E}}$ so that $e$ has winding number equal to $1$, with respect to a global trivialization of $\xi=\ker \lambda_E$. The proof of this fact follows the same lines of the proof of Proposition \ref{asymptoticcylinder}. Moreover, $w$ does not intersect $u_{1,E}(\C)\cup u_{2,E}(\C) \cup P_{2,E} \cup P_{3,E}$ since  it is the $C^\infty_{\rm loc}$-limit of the maps $w_n$, which do not intersect $u_{1,n}(\C) \cup u_{2,n}(\C)\cup P_{2,n} \cup P_{3,n}$. Thus $w(\C \setminus \{z^*\}) \subset \dot S_E \setminus P_{3,E}$ and from the uniqueness of such cylinders, see Proposition \ref{propunique}, $\tilde w$ coincides with the rigid cylinder $\tilde v_E=(b_E,v_E)$ obtained in Section \ref{sec_5-ii)}, up to reparametrization and $\R$-translation.

We conclude from \eqref{vizii2} and from $\tilde w_n \to \tilde w$ in $C^\infty_{\rm loc}$ that given any $R_2>0$ fixed and any open neighborhood $\V_1 \subset W_E$ of $v_E(\R \times S^1)\cup P_{3,E}$, we have \begin{equation}\label{vizii3} w_n(\C \setminus B_{R_2}(z^*))\subset \V_1, \forall \mbox{ large }n.\end{equation}

Next we study the behavior of $\tilde w_n$ near $z^*$ for large $n$ by using the so called soft rescaling procedure near $z^*$  defined in \cite{fols} (see also \cite{HLS} and \cite{HS1}). Since $\tilde w$ is asymptotic to $P_{2,E}$ at $z^*$, we can fix $\epsilon >0$ small so that \begin{equation}\label{soft} 0< \int_{\partial B_\epsilon(z^*)} w^* \lambda_E - T_{2,E} < \frac{T_{2,E}}{4}, \end{equation} where $\partial B_\epsilon(z^*)$ is oriented counter-clockwise. We define a new sequence $z_n \to z^*$ as follows. Let $z_n \in B_\epsilon(z^*)$ and $0<\delta_n<\epsilon$ be such that \begin{equation} \label{soft2} \begin{aligned} d_n(z_n) \leq d_n(\zeta), \forall \zeta \in B_\epsilon(z^*), \\ \int_{B_\epsilon(z^*) \setminus B_{\delta_n}(z_n)} w_n^* d\lambda_n = \frac{T_{2,E}}{2}.\end{aligned}\end{equation} Since $z^*$ is a negative puncture we necessarily have $z_n \to z^*$ and the existence of $\delta_n$ follows from \eqref{soft}. Next we show that $\delta_n \to 0$ as $n \to \infty$. Arguing indirectly assume that up to extraction of a subsequence we have $\delta_n \to \bar \delta>0$ and take $0<\bar \epsilon<\bar \delta$. From \eqref{soft} and \eqref{soft2} we get $$\frac{T_{2,E}}{4} \geq \lim_{n \to \infty} \int_{B_\epsilon(z^*) \setminus B_{\bar \epsilon}(z^*)} w_n^* d\lambda_n \geq \lim_{n \to \infty} \int_{B_\epsilon(z^*) \setminus B_{\delta_n}(z_n)} w_n^* d\lambda_n =\frac{T_{2,E}}{2},$$ which is a contradiction.

Let $R_n$ be a sequence of real numbers satisfying $R_n \to +\infty$ and \begin{equation}\label{Rn} \delta_n R_n <\epsilon/2.\end{equation} Consider the $\tilde J_n$-holomorphic maps $\tilde v_n:B_{R_n}(0) \to \R \times W_E$ defined by $$\tilde v_n(z) = (b_n,v_n)=(d_n(z_n+\delta_n z) - d_n(z_n +2 \delta_n),w_n(z_n + \delta_n z)), \forall z\ \in B_{R_n}(0).$$ Notice that $\tilde v_n(2) \in \{0\} \times W_E$ and that $z=0$ is the minimum point of $b_n$ for all $n$. This last assertion follows from \eqref{soft2} and will be useful in what follows.

{\it Claim I. Up to a subsequence, $\tilde v_n \to \tilde v$ in $C^\infty_{\rm loc}$ as $n \to \infty$, where $\tilde v=(b,v):\C \to \R\ \times W_E$ is a finite energy $\tilde J_E$-holomorphic plane asymptotic to $P_{2,E}$ at its positive puncture $+\infty$.} To prove Claim I, we first show that $|\nabla \tilde v_n(z)|_n$ is uniformly bounded in $z\in B_{R_n}(0)$ and in $n$. Arguing indirectly we may assume that $|\nabla \tilde v_n(\zeta_n)|_n \to \infty$ for a sequence $\zeta_n \in B_{R_n}(0)$ as $n \to \infty$. From \eqref{soft2} and \eqref{Rn} we have $$\int_{B_{R_n}(0) \setminus \D} v_n^* d\lambda_n=\int_{B_{\delta_nR_n}(z_n) \setminus B_{\delta_n}(z_n)} w_n^* d\lambda_n \leq \frac{T_{2,E}}{2},\forall n.$$ Using Lemmas \ref{lemamin} and \ref{lemtop} we see that each bubbling-off point of $\tilde v_n$ takes away at least $T_{2,E}$ of the $d\lambda_E$-area, hence we necessarily have $\limsup_{n \to \infty} |\zeta_n| \leq 1$.

For all fixed $R>0$ large, we have from \eqref{soft} that \begin{equation}\label{soft3} \int_{\partial B_R(0)} v_n^* \lambda_n = \int_{\partial B_{R\delta_n}(z_n)} w_n^* \lambda_n < \int_{\partial B_{\epsilon}(z^*)} w_n^* \lambda_n < \frac{5 T_{2,E}}{4},\end{equation} for all large $n$.
After extracting a subsequence, we have $\tilde v_n \to \tilde v$ in $C^\infty_{\rm loc},$ where $\tilde v=(b,v):\C \setminus \Gamma' \to \R\ \times W_E$ is a finite energy $\tilde J_E$-holomorphic map and  $\Gamma' \subset \D$ is a finite set of bubbling-off points for the sequence $\tilde v_n$. It follows from \eqref{soft3} that $\#\Gamma'\leq 1$. If $\# \Gamma'=1$ then $\tilde v$ is non-constant and $\Gamma'=\{\bar z\}$. In this case $\bar z=0$ since $b_n$ attains its minimum value at $z=0$. From \eqref{soft3} and Lemma \ref{lemamin}, $\tilde v$ is asymptotic to $P_{2,E}$ at its positive puncture $+\infty$ and also at its negative puncture $\bar z$, therefore $\tilde v$ is a cylinder over $P_{2,E}$.

However, we have from \eqref{soft2} that  $$\int_{\partial \D} v_n^*\lambda_n = \int_{\partial B_{\delta_n}(z_n)} w_n^* \lambda_n = \int_{\partial B_\epsilon(z^*)} w_n^* \lambda_n - \frac{T_{2,E}}{2},$$ which implies, from  \eqref{soft}, $$\int_{\partial \D} v^* \lambda_E = \int_{\partial B_\epsilon (z^*) } w^* \lambda_E - \frac{T_{2,E}}{2}\leq \frac{T_{2,E}}{4}+T_{2,E} - \frac{T_{2,E}}{2}=\frac{3T_{2,E}}{4} < T_{2,E},$$ a contradiction with the fact that $\tilde v$ is a cylinder over $P_{2,E}$. Thus $\Gamma'=\emptyset$ and, since $\tilde v_n(2) \in \{0\} \times W_E$ for all $n$, we obtain $C^\infty_{\text{loc}}$-bounds for the sequence $\tilde v_n$ from usual elliptic regularity. Then we find a finite energy $\tilde J_E$-holomorphic plane $\tilde v=(b,v):\C \to \R \times W_E$ so that, up to extraction  of a subsequence, $\tilde v_n \to \tilde v$ in $C^\infty_{\rm loc}$ as $n\to\infty$. From \eqref{soft3} and Lemma \ref{lemamin} we see that $\tilde v$ is asymptotic to $P_{2,E}$ at $+\infty$.
This proves Claim I.

From the uniqueness 
obtained in Proposition \ref{propunique},
we conclude that the $\tilde J_E$-holomorphic plane $\tilde v$ from Claim I must coincide, up to reparametrization and $\R$-translation, with either $\tilde u_{1,E}$ or $\tilde u_{2,E}$,  the rigid planes constructed in Section \ref{sec_step3}. Assuming without loss of generality that $\tilde v$ coincides with $\tilde u_{1,E}$, we can use the fact that $\tilde v_n \to \tilde v$ in $C^\infty_{\rm loc}$ to conclude that given any $R_3>0$ and any neighborhood $\V_2 \subset W_E$ of $u_{1,E}(\C)$, we have $v_n(B_{R_3}(0)) \subset \V_2$ for all large $n$. This implies that \begin{equation}\label{vizii5} w_n(B_{R_3\delta_n}(z_n)) \subset \V_2, \forall \mbox{ large } n.\end{equation}

Let $\W_2$ be an $S^1$-invariant neighborhood of the loop $S^1 \ni t \mapsto x_{2,E}(T_{2,E}t)$ in $C^\infty(S^1,W_E)$. Since $P_{2,E}$ is hyperbolic, we can assume that there is no other periodic orbit in $\W_2$ other than $P_{2,E}$, modulo $S^1$-reparametrizations. Now we use the fact that both $\tilde w$ and $\tilde v$ are asymptotic to $P_{2,E}$ at their punctures $z^*$ and $+\infty$, respectively. Since $\tilde w_n \to \tilde w$ we find $0<\epsilon_0<\epsilon$ small so that if $0<\rho\leq \epsilon_0$ is fixed then \begin{equation}\label{viz1} \mbox{the loop }t \mapsto w_n\left(z_n + \rho e^{2 \pi i t}\right) \mbox{ is contained in } \W_2, \forall \mbox{ large }n.\end{equation}  In the same way, since $\tilde v_n \to \tilde v$, we can find $R_0 \gg 0$ large so that if $R\geq R_0$ is fixed then \begin{equation}\label{viz2} \mbox{the loop }t \mapsto w_n\left(z_n + R \delta_n e^{2 \pi i t}\right) \mbox{ is contained in } \W_2, \forall \mbox{ large }n.\end{equation}  We can also assume that for all $n$ sufficiently large we have \begin{equation} \label{soft4} 0< \frac{T_{2,E}}{2}< \int_{\partial B_{\delta_n R_0}(z_n)} w_n^* \lambda_n. \end{equation}

We prove the following claim which is in essence equivalent to Lemma 4.9 in \cite{fols}, adapted to our situation where we have a sequence of almost complex structures $\tilde J_n$ converging to $\tilde J_E$ in $C^\infty$ as $n \to \infty$ and the only periodic orbit   for $\lambda_E$ with action $<2T_{2,E}$ is the hyperbolic periodic orbit $P_{2,E}$, for $E>0$ sufficiently small.

{\it Claim II. Let $\W_2\subset C^\infty(S^1,W_E)$ be a small $S^1$-invariant neighborhood of the loop $t \mapsto x_{2,E}(T_{2,E}t)$, where $P_{2,E}=(x_{2,E},T_{2,E})$. Then there exist $h>0$ such that if $\tilde u_n=(a_n,u_n):[r_n,R_n] \times S^1 \to \R \times  W_E$ is a sequence of $\tilde J_n$-holomorphic cylinders satisfying \begin{eqnarray} \label{en1} & E(\tilde u_n)\leq \displaystyle\frac{3T_{2,E}}{2}, \\ \label{en2} & \displaystyle\int_{[r_n,R_n] \times S^1} u_n^* d\lambda_n \leq \displaystyle\frac{T_{2,E}}{2},\\ \label{en3} & \displaystyle\int_{\{\rho\} \times S^1} u_n^* \lambda_n \geq \displaystyle\frac{T_{2,E}}{2},\forall \rho \in [r_n,R_n],\end{eqnarray} then each loop $S^1 \ni t\mapsto u_n(s,t)$ is contained in $\W_2$ for all $s\in[r_n+h,R_n-h]$. }

To prove Claim II, we argue indirectly and consider a sequence of $\tilde J_n$-holomorphic cylinders $\tilde u_n=(a_n,u_n):[r_n,R_n] \times S^1 \to \R \times W_E$ satisfying conditions \eqref{en1}-\eqref{en3} and such that \begin{equation}\label{loop1} \mbox{the loop } S^1 \ni t \mapsto u_n(s_n,t) \mbox{ is not contained in } \W_2,\end{equation} for some $s_n \in [r_n+n,R_n-n]$ and $R_n-r_n \geq 2n$. Define the $\tilde J_n$-holomorphic maps $\tilde E_n =(e_n,E_n):[r_n-s_n+\epsilon_1,R_n-s_n-\epsilon_1] \times S^1 \to \R \times W_E$ by $$\tilde E_n(s,t)=(a_n(s+s_n,t)-a_n(s_n,0), u_n(s+s_n,t)),$$ where $\epsilon_1>0$ is any sufficiently small real number. Notice that $e_n(0,0)=0$ and, from \eqref{loop1}, \begin{equation}\label{loop2} \mbox{the loop } S^1 \ni t \mapsto E_n(0,t) \mbox{ is not contained in } \W_2.\end{equation}
$|\nabla \tilde E_n(s,t)|_n$ is uniformly bounded in $(s,t)$ and in $n$, since any bubbling-off sequence of points would take away at least $T_{2,E}$ of the $d\lambda_E$-area, and this is not available from \eqref{en2}. Therefore, the sequence $\tilde E_n$ has $C^\infty_{\rm loc}$-bounds and, up to extraction of a subsequence, we can assume that $\tilde E_n \to \tilde E$ in $C^\infty_{\rm loc}$ as $n\to\infty$, where $\tilde E=(e,E):\R \times S^1 \to \R \times W_E$ is a $\tilde J_E$-holomorphic cylinder.

In view of the properties of $\tilde u_n$ in \eqref{en1}-\eqref{en3}, and consequently similar properties of $\tilde E_n$, we have $$ E(\tilde E) \leq \frac{3T_{2,E}}{2}, \int_{\R \times S^1} E^* d\lambda_E \leq \frac{T_{2,E}}{2} \mbox{ and } \int_{\{\rho\} \times S^1} E^* \lambda_E \geq \frac{T_{2,E}}{2},\forall \rho \in \R.$$ It follows that $\tilde E$ has finite energy, is non-constant, has a negative puncture at $s=-\infty$ and a positive puncture at $s=+\infty$. Moreover, $\tilde E$ must be asymptotic at $+\infty$ to a periodic orbit with action $\leq \frac{3T_{2,E}}{2}$. Thus the same holds for its negative puncture $-\infty$. Since $P_{2,E}$ is the only periodic orbit with such low action, for all $E>0$ sufficiently small, see Lemma \ref{lemamin}, we conclude that $\tilde E$ is a cylinder over $P_{2,E}$.  However, from \eqref{loop2}, the loop $S^1 \ni t \mapsto E(0,t)$ is not contained in $\W_2$, a contradiction. This proves Claim II.

Now consider for each $n$ the $\tilde J_n$-holomorphic cylinder $\tilde C_n:\left[ \frac{\ln R_0 \delta_n}{2 \pi}, \frac{\ln \epsilon_0}{2\pi} \right]\times S^1 \to \R \times W_E$ defined by $$\tilde C_n(s,t)=(c_n(s,t), C_n(s,t))=\tilde w_n\left(z_n+e^{2 \pi(s+it)}\right).$$ From \eqref{soft}, \eqref{soft2} and \eqref{soft4} we have $$\begin{aligned}  & E(\tilde C_n)<\frac{5T_{2,E}}{4}< \frac{3T_{2,E}}{2}, \\  & \int_{\left[ \frac{\ln R_0 \delta_n}{2 \pi}, \frac{\ln \epsilon_0}{2\pi} \right] \times S^1} C_n^* d\lambda_n < \frac{T_{2,E}}{2},\\  & \int_{\{\rho\} \times S^1} C_n^* \lambda_n > \frac{T_{2,E}}{2},\forall \rho \in \left[ \frac{\ln R_0 \delta_n}{2 \pi}, \frac{\ln \epsilon_0}{2\pi} \right],\end{aligned}$$ for all large $n$, where $R_0\gg 0$ is defined above.
From Claim II we find $h>0$ so that the loop $S^1 \ni t \mapsto C_n(s,t)$ is contained in $\W_2$ for $s\in \left[ \frac{\ln R_0 \delta_n}{2 \pi}+h, \frac{\ln \epsilon_0}{2\pi}-h \right]$ and all large $n$. This implies that \begin{equation}\label{vizi3} \mbox{ the loop } S^1 \ni t \mapsto w_n\left(z_n +\rho e^{2 \pi t i}\right) \mbox{ is contained in } \W_2, \end{equation}for all $\rho \in [e^{2 \pi h} R_0 \delta_n, e^{-2 \pi h} \epsilon_0]$ and all large $n.$ We conclude that given any neighborhood $\V_3 \subset W_E$ of $P_{2,E}$, we have
\begin{equation}\label{vizii4}w_n\left(B_{e^{-2 \pi h} \epsilon_0}(z_n) \setminus B_{e^{2 \pi h} R_0 \delta_n}(z_n)\right) \subset \V_3, \forall \mbox{ large }n. \end{equation}

Finally, taking $0<R_2<e^{-2 \pi h} \epsilon_0$, $R_3 > e^{2 \pi h} R_0$, using \eqref{vizii3}, \eqref{vizii5} and \eqref{vizii4}, and that $z_n \to z^*$ as $n \to \infty$, we conclude that given any small neighborhood $\bar \V \subset W_E$ of $u_{1,E}(\C) \cup P_{2,E} \cup v_E(\R \times S^1) \cup P_{3,E}$, with $p_0 \not\in \bar \V$, we have $w_n(\C) \subset \bar \V$ for all large $n$. However, this contradicts the fact that $p_0\in w_n(\C)$ for all large $n$. This contradiction shows that there are no bubbling-off points for $\tilde w_n$.

We have concluded that the sequence $\tilde w_n$ with the reparametrization given either in \eqref{reparamwnn1} or in \eqref{reparamwnn2}, depending on the intersection conditions \eqref{naointersecta} and \eqref{intersecta}, respectively, has $C^\infty_{\rm loc}$-bounds and therefore, from elliptic regularity, we can extract a subsequence also denoted by $\tilde w_n$ so that $\tilde w_n \to \tilde w=(d,w):\C \to \R \times W_E$ in $C^\infty_{\rm loc}$ as $n\to\infty$, where $\tilde w$ is a finite energy $\tilde J_E$-holomorphic plane.  From \eqref{reparamwnn1} and \eqref{reparamwnn2}, we see that $\tilde w$ is non-constant. Moreover, $w(\C)$ is contained in $\V_E=\dot S_E \setminus (V_E \cup P_{3,E})$. In fact, by positivity and stability of intersections, $\tilde w$ cannot intersect $\tilde u_{1,E},\tilde u_{2,E},\tilde v_E,P_{2,E},P_{3,E}$ since this would imply intersections of $\tilde w_n$ with $\tilde u_{1,n},\tilde u_{2,n},\tilde v_n,P_{2,n},P_{3,n}$ for large $n$, a contradiction. Thus, from \eqref{reparamwnn1} and \eqref{reparamwnn2}, we obtain that $\tilde w$ is asymptotic to $P_{3,E}$ at $+\infty$, since there is no periodic orbit in $\dot S_E\setminus P_{3,E}$ with action $\leq T_{3,E}$ which is not linked to $P_{3,E}$, for $E>0$ sufficiently small, according to Theorem \ref{teolink}.

Arguing as in Lemma \ref{lemvizi}, see also \cite[Lemma 8.1]{convex}, we know that given any small open neighborhood $\V_4\subset W_E$ of $P_{3,E}$ we find $R_0>0$ so that $w_n(\C \setminus B_{R_0}(0))\subset \V_4$ for all large $n$. Hence if $\V_4$ is sufficiently small, then $p_0 \not\in \V_4$. It follows that if $\zeta_n\in \C$ is such that $w_n(\zeta_n) = p_0$, then $|\zeta_n|$ is uniformly bounded and we may assume that $\zeta_n \to z_0$ as $n \to \infty$. Since $\tilde w_n \to \tilde w$ in $C^\infty_{\rm loc}$ as $n \to \infty$, we must have $w(z_0)=p_0$.

The finite energy $\tilde J_E$-holomorphic plane $\tilde w=(d,w)$, which, up to extraction of a subsequence, is the $C^{\infty}_{\rm loc}$-limit of $\tilde w_n$ after suitable reparametrizations and $\R$-translations, is asymptotic to $P_{3,E}$ exponentially fast, see \eqref{eqasymptotics2} in Proposition \ref{propplanes2} below. In fact,  there exists a negative eigenvalue $\delta$ and a $\delta$-eigensection $e$ of the asymptotic operator $A_{P_{3,E}}$ so that in suitable Martinet's coordinates near $P_{3,E}$, we can express $\tilde w\left(e^{2 \pi (s+it)}\right)=(d(s,t),\vartheta(s,t),x(s,t),y(s,t)),s \gg 0,$  satisfying the exponential decays given in \eqref{eqasymptotics2}. The winding number of $e$ equals $1$ and, therefore,  $\wind_{\pi}(\tilde w)=\wind_\infty(\tilde w)-1 =0$. It follows that $w$ is an immersion transverse to $X_{\lambda_E}$ and that $\tilde w$ is an immersion. Since $\tilde w$ is an embedding near the boundary and is the limit of the embeddings $\tilde w_n$, it follows that $\tilde w$ is an embedding as well. Also, if $c\neq 0,$ the image of $\tilde w$ does not coincide with the image of $\tilde w_c=(d+c,w)$, see the proof of Proposition \ref{cylinderconclusion}-(iv). This implies, by Carleman's similarity principle, that intersections of $\tilde w$ with $\tilde w_c$ are isolated if $c \neq 0$ and, by positivity an stability of intersections, we must have $\tilde w(\C) \cap \tilde w_c(\C) = \emptyset, \forall c \neq 0$ since $\tilde w_n(\C) \cap \tilde w_{n,c}(\C) = \emptyset, \forall c \neq 0,\forall n,$ with $\tilde w_{n,c} = (d_n + c,w_n)$. We conclude that $w:\C \to W_E$ is also an embedding.

From Theorem 1.4 in \cite{props2} we  get the following uniqueness property: if $\tilde w_1=(d_1,w_1),\tilde w_2=(d_2,w_2)$ are finite energy $\tilde J_E$-holomorphic planes asymptotic to $P_{3,E}$ with the exponential decays \eqref{eqasymptotics2}, then either \begin{equation}\label{uniqplane} w_1(\C)=w_2(\C) \mbox{ or } w_1(\C) \cap w_2(\C) = \emptyset.\end{equation} In fact, this result assumes that the asymptotic limit $P_{3,E}$ is nondegenerate. Since $P_{3,E}$  may be degenerate, we cannot guarantee this assumption. However, its proof only makes use of the asymptotic behavior of the $\tilde J_E$-holomorphic plane $\tilde w$ at $+\infty$ given by the exponential estimates \eqref{eqasymptotics2} in Proposition \ref{propplanes2} below. Thus it also holds in our situation.

It follows that the $C^\infty_{\rm loc}$-limit $\tilde w=(d,w)$ of any subsequence of the sequence $\tilde w_n$, after suitable reparametrization and $\R$-translation, with the property that $p_0 \in w(\C)$, is unique.  The proof of Proposition \ref{propplanes} is now complete. \end{proof}

The following proposition is essentially Theorem 7.2 in \cite{convex}.

\begin{prop}\label{propplanes2}
Let $p_0 \in \V_E$ and $\tilde w=(d,w):\C \to  \R \times W_E$ be a finite energy $\tilde J_E$-holomorphic plane, which is asymptotic to $P_{3,E}=(x_{3,E},T_{3,E})$ at $+\infty$, and is the $C^\infty_{\rm loc}$-limit of the sequence of $\tilde J_n$-holomorphic planes $\tilde w_n=(d_n,w_n)$, satisfying $p_0 \in w_n(\C), \forall n$, as in Proposition \ref{propplanes}. Consider Martinet's coordinates $(\vartheta,x,y) \in S^1 \times \R^2$ on a small tubular neighborhood $\U \subset W_E$ of $P_{3,E}$, as explained in Appendix \ref{ap_basics}, where the contact form takes the form $\lambda_E \equiv g_E(d \vartheta + xdy)$ and $P_{3,E} \equiv S^1 \times \{0\}$. Let $A_{P_{3,E}}$ be the asymptotic operator associated to $P_{3,E}$, which in local coordinates assumes the form \eqref{operadorAP}. Then $w(e^{2\pi (s+it)})=(\vartheta(s,t),x(s,t),y(s,t))\in \U,\forall s$ sufficiently large, $d(s,t)=d(e^{2\pi (s+it)}),$ and  $\tilde w$ is represented by  $$(d(s,t),\vartheta(s,t),x(s,t),y(s,t)), (s,t)\in \R \times S^1, s\gg0,$$ where
\begin{equation}\label{eqasymptotics2} \begin{aligned}|D^\gamma(d(s,t) - (T_{3,E}s + a_0))| & \leq  A_\gamma e^{-r_0s},\\
|D^\gamma(\vartheta(s,t) - t-\vartheta_0) | & \leq  A_\gamma e^{-r_0s},\\
z(s,t):=(x(s,t),y(s,t))  &= e^{\int_{s_0}^s \mu(r) dr}(e(t)+R(s,t)),\\
|D^\gamma R(s,t)|,|D^\gamma (\mu(s)- \delta)| &\leq  A_\gamma
e^{-r_0s},
\end{aligned}
\end{equation}for all large $s$ and all $\gamma \in \N \times \N$, where $A_\gamma,r_0>0,\vartheta_0,a_0\in \R$ are suitable constants.  $\vartheta(s,t)$ is seen as a map on $\R$ and satisfies $\vartheta(s,t+1) = \vartheta(s,t)+1$. Here $\mu(s) \to \delta <0,$  where $\delta$ is an eigenvalue of $A_{P_{3,E}}$ and $e:S^1 \to \R^2$
is an eigensection of $A_{P_{3,E}}$ associated to $\delta$,
represented in these coordinates. Its winding number with respect to a global trivialization equals $1$.
\end{prop}

We denote by $\M_{3,E}$ the space of embedded finite energy $\tilde J_E$-holomorphic planes $\tilde w=(d,w):\C \to \R \times W_E$ asymptotic to $P_{3,E}$ at $+\infty$ and satisfying the exponential decays as in Proposition \ref{propplanes2}. In the terminology of \cite{hryn}, such planes are called fast.

We have seen in Proposition \ref{propplanes} that through any point $p_0\in \V_E$ we can find $\tilde w=(d,w) \in \M_{3,E}$ so that $p_0 \in w(\C)\subset \V_E$. Then it follows from the uniqueness property \eqref{uniqplane} that $w'(\C) \subset \V_E$ for all $\tilde w'=(d',w') \in \M_{3,E}$.

In \cite{props3}, a Fredholm theory with weights is developed in order to give a local description of spaces such as $\M_{3,E}$. Using the exponential asymptotic behavior of the $\tilde J_E$-holomorphic planes $\tilde w=(d,w)\in \M_{3,E}$ given as in Proposition \ref{propplanes2}, one can prove the following statement which is essentially Theorem 2.3 of \cite{hryn}.

\begin{theo}\label{fredholm} Let $\tilde w =(d,w) \in \M_{3,E}$ be a fast finite energy $\tilde J_E$-holomorphic plane. Then there exists  $\epsilon>0$ and an embedding $$\tilde \Phi=(a,\Phi): (-\epsilon,\epsilon) \times \C \to \R \times W_E,$$ so that \begin{itemize}
\item[(i)] $\tilde \Phi(0, \cdot) = \tilde w$.
\item[(ii)] For any $\tau \in (-\epsilon, \epsilon)$, the map $\tilde \Phi_\tau=(a_\tau,\Phi_\tau):= \tilde \Phi( \tau, \cdot ): \C \to \R \times W_E$ is contained $\M_{3,E}$, i.e., $\tilde \Phi_\tau$ is a fast embedded finite energy $\tilde J_E$-holomorphic plane asymptotic to $P_{3,E}$ satisfying $\Phi_\tau(\C)\subset \V_E$, for all $\tau$.
\item[(iii)] The map $$\Phi:(-\epsilon,\epsilon) \times \C \to W_E \setminus P_{3,E}$$ is an embedding.
\item[(iv)] If the sequence $\tilde w_n=(d_n,w_n)\in \M_{3,E}$ satisfies $\tilde w_n \to \tilde w$ in $C^\infty_{\rm loc}$ as $n\to \infty$, then for all large $n$ we find sequences $A_n,B_n \in \C$, with $A_n \to 1\in \C, B_n \to 0\in \C$, $\R \ni c_n \to 0,$ $(-\epsilon,\epsilon) \ni \tau_n \to 0$ so that $$\tilde w_n(z)= \tilde \Phi_{\tau_n,c_n}(A_n z+B_n):= (a_{\tau_n}(A_n z+B_n) + c_n, \Phi_{\tau_n}(A_nz+B_n)).$$ \end{itemize} \end{theo}

Theorem \ref{fredholm} provides a maximal smooth one parameter family of maps $\tilde w_{\tau,E}=(d_{\tau,E},w_{\tau,E})\in \M_{3,E},\tau \in (\tau_-,\tau_+),$ so that  $w_{\tau_1,E}(\C) \cap w_{\tau_2,E}(\C) = \emptyset$ for all $\tau_1 \neq \tau_2$. In fact this family cannot be compactified to an $S^1$-family of such maps since this would provide an open book decomposition with disk-like pages for $W_E$, with binding orbit $P_{3,E}$, which clearly contradicts the fact that there are orbits which are not linked to $P_{3,E}$, such as $P_{2,E}$ and $P_{3,E}'$.

Next we describe how the family $\{\tilde w_{\tau,E}\}$ breaks as $\tau \to \tau_\pm$. This is the content of the following proposition. We assume the  normalization \begin{equation}
\label{eq_normaliza}
\tau_-=0,\\ \tau_+=1,\\
X_{\lambda_E}|_{w_{\tau,E}} \cdot \tau > 0,
\end{equation} i.e., $\tau$ strictly increases in the direction of $X_{\lambda_E}$.

Let $P_{2,E}=(x_{2,E},T_{2,E})$ and $P_{3,E}=(x_{3,E},T_{3,E})$ be the periodic orbits of $\lambda_E$, let $\tilde u_{1,E}$ and $\tilde u_{2,E}$ be the rigid planes and let $\tilde v_E$ and $\tilde v_E'$ be the rigid cylinders in the statement of Proposition \ref{prop_step5}.

\begin{prop}\label{break}For $E>0$ sufficiently small, the following holds. If $\tau_n \in  (0,1)$ is a sequence satisfying $\tau_n \to 0^+$ as $n \to \infty$, then, after suitable reparametrizations and $\R$-translations of $\tilde w_{\tau_n,E}=(d_{\tau_n,E},w_{\tau_n,E}),$ $\tilde v_E$ and $\tilde u_{1,E}$, we have that $\tilde w_{\tau_n,E}$ converges to $\tilde v_E \odot \tilde u_{1,E}$ as $n \to \infty$ in the SFT sense \cite{sft2}, i.e.,
\begin{itemize}
\item [(i)]There exists $z^*\in \D$ so that $\tilde C_n=(c_n,C_n):\R \times S^1 \to \R \times W_E$, defined by $\tilde C_n(s,t)=\tilde w_{\tau_n,E}\left(z^* +e^{2 \pi(s+it)}\right)$, is such that $\tilde C_n \to \tilde v_E$ in $C^\infty_{\rm loc}$ as $n \to \infty$.
\item [(ii)] There exist sequences $z_n \to z^*$, $\delta_n \to 0^+$ and $c_n \in \R$ so that $\tilde P_n=(p_n,P_n): \C \to \R \times W_E$, defined by $\tilde P_n(z)=(d_{\tau_n,E}(z_n +\delta_n z)+c_n,w_{\tau_n,E} (z_n + \delta_n z))$, is such that $\tilde P_n \to \tilde u_{1,E}$ in $C^\infty_{\rm loc}$ as $n \to \infty$.
\item [(iii)] Given an $S^1$-invariant neighborhood $\W_3$ of the loop $S^1 \ni t \mapsto x_{3,E}(T_{3,E}t)$ in $C^\infty(S^1,W_E)$, there exists $R_0\gg 0$  such that the loop $t \mapsto \!w_{\tau_n,E}\!\left(z^* + Re^{2 \pi it}\right)$ is contained in $\W_3$ for all $R\geq R_0$ and all large $n$.
\item [(iv)] Given an $S^1$-invariant neighborhood $\W_2$ of the loop $S^1 \ni t \mapsto x_{2,E}(T_{2,E}t)$ in $C^\infty(S^1,W_E)$, there exist $\epsilon_1>0$ small and $R_1\gg 0$ such that the loop $t \mapsto w_{\tau_n,E}\left(z_n + \rho e^{2 \pi it}\right)$ is contained in $\W_2$ for all $0<R_1 \delta_n \leq \rho \leq \epsilon_1$ and all large $n$.
\end{itemize}
In particular, given any neighborhood $\V_1\subset W_E$ of $u_{1,E}(\C) \cup v_E(\R \times S^1) \cup P_{2,E}\cup P_{3,E}$, we have $w_{\tau_n,E}(\C) \subset \V_1$ for all large $n$. A similar statement holds for any sequence $\tau_n \to 1^-$ with $\tilde u_{1,E}$ replaced with $\tilde u_{2,E}$. The case $S_E'$ is completely analogous.
\end{prop}

\begin{proof}The proof follows the same lines of the proof of Proposition \ref{propplanes}, so we only sketch the main steps. Now the almost complex structure $\tilde J_E=(\lambda_E,J_E)$ is fixed while there we had $\tilde J_n \to \tilde J_E$ in $C^\infty$ as $n\to \infty$. However, this plays no important role in the analysis.

Let $N^\delta_E, E,\delta>0$ small, be the embedded $2$-sphere defined in local coordinates $(q_1,q_2,p_1,p_2)$ by $N^\delta_E= \{q_1+p_1 = \delta\} \cap K^{-1}(E)$.

If there exists $\delta_0>0$ such that $w_{\tau_n,E}(\C) \cap N^{\delta_0}_E=\emptyset$ for all large $n$, then we can suitably reparametrize $\tilde w_{\tau_n,E}$ and translate it in the $\R$-direction as in \eqref{reparamwnn1} so that one has $C^\infty_{\rm loc}$-bounds and $\tilde w_{\tau_n,E} \to \tilde w \in \M_{3,E}$ in $C^\infty_{\rm loc}$ as $n \to \infty$, a contradiction with the maximality of the family $\tilde w_{\tau,E}, \tau \in (0,1)$.

 We have concluded that $w_{\tau_n,E}(\C)$ intersects $N^\delta_E$ for $\delta>0$ arbitrarily small. In this case, we suitably rescale $\tilde w_{\tau_n,E}$ and translate it in the $\R$-direction as in \eqref{reparamwnn2}. Following all the steps from this point in the proof of Proposition \ref{propplanes}, we get all the other desired conclusions.  In the case at hand,  since the family $\tilde w_{\tau,E}$ is not compact, we show that there exists precisely one bubbling-off point $z^* \in \D$ and $\tilde w_{\tau_n,E}$ converges in $C^\infty_{\rm loc}(\C \setminus \{z^*\}, \R \times W_E)$ as $n \to \infty$ to a cylinder asymptotic to $P_{3,E}$ at $+\infty$ and to $P_{2,E}$ at $z^*$. From uniqueness of such cylinders, see Appendix \ref{appunique}, this must coincide with $\tilde v_E$ up to reparametrization and $\R$-translation, proving (i). Performing a soft rescaling of $\tilde w_{\tau_n,E}$ near $z^*$ we find  sequences $z_n \to z^*$, $\delta_n \to 0^+$ and $c_n\in\R$ so that $\tilde P_n$, defined as in (ii), converges to a plane asymptotic to $P_{2,E}$. Again, from the uniqueness of such planes, it must coincide with  $\tilde u_{1,E}$, after  a suitable reparametrization and $\R$-translation. This follows from the normalization \eqref{eq_normaliza}, proving (ii). From a delicate analysis near $P_{3,E}$ and $P_{2,E}$, similar to Claims 0 and II in the proof of Proposition \ref{propplanes}, we prove (iii) and (iv). If $\tau_n \to 1^-$, then $\tilde u_{1,E}$ is necessarily replaced with $\tilde u_{2,E}$.
\end{proof}

This finishes the proof of Proposition \ref{prop_step5iiib}.

\begin{figure}[h!!]
  \centering
  \includegraphics[width=0.35\textwidth]{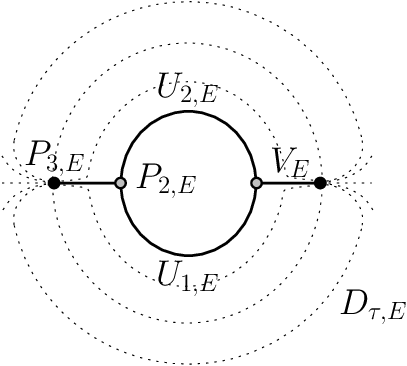}
  \caption{The one parameter family of planes $D_{\tau,E}=w_{\tau,E}(\C)$, $\tau \in (0,1),$ with binding orbit $P_{3,E}$, finishes the construction of the $2-3$ foliation $\F_E$ adapted to $S_E$.}
  \label{fig_completa}
\hfill
\end{figure}

\appendix

\section{Basics on pseudo-holomorphic curves in symplectizations}\label{ap_basics} Let $M$ be a $3$-manifold. A $1$-form $\lambda$ on $M$ is called a
contact form if $\lambda \wedge d\lambda$ never vanishes. The contact structure associated to $\lambda$ is the non-integrable tangent plane distribution $\xi = \ker \lambda\subset TM$. The Reeb vector field $X_\lambda$ associated to $\lambda$ is determined by $i_{X_\lambda}\lambda=1$ and $i_{X_\lambda} d \lambda =0$. The pair $(M,\xi)$ is called a contact manifold and $\lambda$ is a defining contact form for $\xi$. In this case, $\xi$ is called co-oriented by $\lambda$.

Let $\{\varphi_t, t\in \R \}$ be the flow of $\dot w = X_{\lambda}
\circ w$, called the Reeb flow of $\lambda$, assumed to be complete. Since the Lie
derivative $\LL_{X_\lambda} \lambda =i_{X_\lambda} d\lambda + d
i_{X_\lambda} \lambda = 0$, $\varphi_t$ preserves $\lambda$ and, in
particular, preserves the contact structure $\xi,$ i.e.,
$\varphi_{t}^*\xi = \xi, \forall t$. A periodic orbit of the
Reeb flow is a pair $P=(w,T)$, where $w:\R \to M$  satisfies
$w(t)=\varphi_t(w(0))$ and $w(t+T) = w(t),\forall t.$ The period
$T$ of $P$ is assumed to be positive and it coincides with the action of $\lambda$ along $w$, i.e., $$T= \int_{[0,T]} w^*\lambda=: \int_P \lambda.$$ Sometimes we write $w(\R)=P
\subset M$ as an abuse of notation. Given $c\in \R$, we identify
$(w,T)=P\sim P_c=(w_c, T)$, where $w_c(\cdot ) = w(\cdot+c)$, and
denote by $\P(\lambda)$ the set of periodic orbits of
$X_{\lambda}$ modulo this identification. We say that $P=(w,T)$ is
simple if $T$ is the least positive period of $w$. We say that $P$ is unknotted if it is simple and is the boundary of
an embedded disk $D\subset M$. We say that $P$ is nondegenerate
if the linear map $D\varphi_T(w(0))|_{\xi_{w(0)}}: \xi_{w(0)} \to \xi_{w(0)}$
does not have $1$ as an eigenvalue. Otherwise $P$ is said to be
degenerate. The contact form $\lambda$ is called nondegenerate if all its
periodic orbits are nondegenerate.

An overtwisted disk for $\xi$ is an embedded disk $D\subset M$ so that $T_z \partial D \subset \xi_z$ and $T_z D \neq \xi_z, \forall z \in \partial D$. The contact structure $\xi = \ker \lambda$ is called overtwisted if it admits an overtwisted disk. Otherwise, it is called tight.

Given a trivial knot $K\subset M$  transverse to $\xi$, we define its self-linking number $\sl(K)$ in the following way: let $D \subset M$ be an embedded disk satisfying $\partial D = K$ and let $Z$ be a non-vanishing section of $\xi|_D$. Use $Z$ to slightly perturb $K$ to a trivial knot $K'$, disjoint from $K$ and also transverse to $\xi$. We can assume $K'$ is transverse to $D$ as well. We define  $$\sl(K)= \mbox{  algebraic intersection number } K' \cdot D\in \Z.$$ Here $K$ is oriented by $\lambda$, $D$ has the orientation induced by $K$ and $M$ is oriented by $\lambda \wedge d\lambda$. The knot $K'$ inherits the orientation of $K$. This definition does not depend on the choices of $D$ and $Z$ if, for instance, $\xi$ is trivial.

We consider the particular case where
$M=S^3:=\{w=(x_1,x_2,y_1,y_2) \in \R^4: |w|=1\}$, equipped with
a contact form of the type
\begin{equation}\label{contact} \lambda = f \lambda_0|_{S^3},\end{equation} where
$$\lambda_0 =\frac{1}{2} \sum_{i=1}^2 y_i dx_i -x_i dy_i,$$ and
$f: S^3 \to  (0,\infty)$ is a smooth function. Notice that
$$d\lambda_0 = \omega_0:=\sum_{i=1}^2 dy_i \wedge dx_i$$ is the canonical
symplectic form on $\R^4$. By a Theorem of Bennequin, see \cite[Theorem 1]{ben}, the
contact structure $$\xi_0 := \ker \lambda_0|_{S^3} = \ker
\lambda$$ is tight and therefore $\lambda$ is also called tight.
By a classification theorem of Y. Eliashberg, see \cite[Theorem 2.1.1]{eli}, any tight
contact form on a manifold diffeomorphic to $S^3$ is of the form \eqref{contact} up to
diffeomorphism.

Let $S_f = \{ \sqrt{f(w)}w: w\in S^3 \}$. Then $S_f$ is a
star-shaped hypersurface with respect to $0 \in \R^4$ and
$\lambda_f:= \lambda_0|_{S_f}$ is a tight contact form on $S_f$. Its
Reeb vector field $X_{\lambda_f}$ satisfies $\psi_{f*} X_\lambda =
X_{\lambda_f},$ where $\psi_f:S^3 \to S_f$ is the diffeomorphism
given by $\psi_f(w)=\sqrt{f(w)}w, w\in S^3$.

Let $H:\R^4 \to \R$
be a Hamiltonian function such that $S_f = H^{-1}(0),$ where $0$
is a regular value of $H$. Then the Hamiltonian vector field
$X_H$, defined by $i_{X_H} \omega_0 = -dH$, is parallel to
$X_{\lambda_f}$ when restricted to $S_f$, since both vector fields
are sections of the line bundle $\ker \omega_0|_{S_f} \to S_f$.

Now let $(M,\xi=\ker \lambda)$ be a co-oriented closed contact manifold of dimension $3$. The contact structure $\xi$ is a symplectic vector
bundle when equipped with $d\lambda|_{\xi}$. A compatible
complex structure on $\xi$ is a smooth bundle map $J: \xi \to
\xi,J^2 = -\text{Id},$ over the identity, such that  $d\lambda
(\cdot, J \cdot)$ is a positive definite inner product on $\xi$.
The space of such $J's$ is non-empty and contractible in the
$C^\infty$-topology and will be denoted by $\J(\lambda)$.

Let us identify $S^1 = \R / \Z$. Let  $J\in \J(\lambda)$ be a
compatible complex structure on $\xi$.  Let $P=(w,T) \in \P(\lambda)$ be
a periodic orbit and let $w_T:S^1 \to M$ be given by $w_T(t) =
w(Tt)$ $\forall t$. Let $\nabla$ be a symmetric connection on
$M$ and consider the operator $A_P:W^{1,2}(S^1, w_T^* \xi)
\subset L^2(S^1, w_T^* \xi) \to L^2(S^1, w_T^* \xi)$
$$A_P \cdot \eta (t)= J(w_T(t))\cdot (- \nabla_t \eta + T
\nabla_{\eta} X_{\lambda})(t),$$ where $\eta \in W^{1,2}(S^1,
w_T^* \xi)$ is a section of the contact structure along $P$.
Here $\nabla_t$ stands for the covariant derivative in the
direction $\dot w_T=T X_{\lambda}$. One easily checks that $A_P$
does not depend on $\nabla$. Consider the following $L^2$-inner
product \begin{equation} \label{innerp}\left< \eta_1, \eta_2
\right> = \int_{S^1} d\lambda_{w_T(t)} \left(\eta_1(t), J_{w_T(t)}
\cdot \eta_2(t)\right) dt, \,\,\eta_1,\eta_2 \in L^2(S^1, w_T^*
\xi).\end{equation}

We assume that $\xi$ is trivial and fix a global symplectic trivialization of the contact structure
\begin{equation}\label{eq_trivpsi} \Psi: \xi=\ker \lambda \to M \times \R^2.\end{equation}

\begin{theo}[Hofer-Wysocki-Zehnder \cite{props2}]\label{op_ass}
With respect to \eqref{innerp}, $A_P$ is an unbounded self-adjoint closed operator. Its spectrum
 $\sigma(A_P)$ is discrete,
consists of real eigenvalues accumulating only at  $\pm \infty$.
Moreover,
\begin{itemize} \item[(i)] Given $\lambda \in \sigma(A_P)$ and a non-vanishing $v$
satisfying $A_P \cdot v = \lambda v$, there is a well-defined
winding number $\wind(\lambda,v,\Psi)$ with respect to $\Psi$.
This number does not depend on $v$ or on $\Psi$, so we simply
denote it by $\wind(\lambda)$. \item[(ii)] $\lambda_1 \leq
\lambda_2 \Rightarrow \wind(\lambda_1) \leq \wind(\lambda_2).$
\item[(iii)] Given $k \in Z$, there exist precisely two
eigenvalues $\lambda_1,\lambda_2\in \sigma(A_P)$, counting
multiplicities, such that $\wind(\lambda_1) = \wind(\lambda_2) =
k$. \item[(iv)] If $\lambda_1 < \lambda_2 \in \sigma(A_P)$ satisfy
$\wind(\lambda_1) = \wind(\lambda_2)$ and $v_1,v_2$ are
non-vanishing $\lambda_1,\lambda_2$-eigensections of $A_P$,
respectively, then $v_1$ and $v_2$ are pointwise linearly
independent. \item[(iv)] $0 \in \sigma(A_P) \Leftrightarrow P$ is
degenerate.\end{itemize}
\end{theo}

Let $$\begin{aligned} \wind^{< 0}(A_P) := & \max\{\wind(\lambda),
\lambda \in \sigma(A_P) \cap (-\infty,0)\},\\  \wind^{\geq 0}(A_P)
:= & \min\{\wind(\lambda), \lambda \in \sigma(A_P) \cap
[0,\infty)\}, \\ p := & \wind^{\geq 0}(A_P)-\wind^{< 0}(A_P).
\end{aligned}$$ One can check that $p\in \{0,1\}$.

We define the Conley-Zehnder index of $P$ by
\begin{equation}\label{CZ} CZ(P)=2 \wind^{< 0}(A_P) + p= \wind^{< 0}(A_P)+ \wind^{ \geq 0} (A_P).\end{equation} The Conley-Zehnder index $CZ(P)$ can also be defined in terms of the linearized dynamics over $P$. This is discussed in Section \ref{sec_step2}.

Let $P=(w,T)\in \P(\lambda)$ be a simple periodic orbit and let
$w_T:S^1 \to M$ be given by $w_T(t) =
w(Tt)$ $\forall t$. Then there exist $\U\subset M$ and
$\U_0 \subset S^1 \times \R^2$ neighborhoods of $P$ and  $S^1
\times \{0\}$, respectively, and a diffeomorphism $\varphi: \U \to
\U_0$ satisfying:
\begin{itemize} \item $\varphi(w_T(t)) = (t,0,0) \in S^1 \times \{0\}$. \item There exists $g\in C^\infty(\U_0, (0,\infty))$
so that $\lambda = \varphi^* (g\cdot (d\vartheta + x dy)),$ where
$\vartheta$ is the coordinate on $S^1$ and $(x,y)$ are coordinates
on $\R^2$. Moreover $$g(\vartheta,0,0) = T \mbox{ and }
dg(\vartheta,0,0)=0, \forall \vartheta \in S^1.$$ \item In these
coordinates $\xi = \text{span} \{ v_1:=\partial_x ,
v_2:=\partial_y - x\partial_\vartheta\}$. Given $J\in \J(\lambda)$
we can choose $\varphi$ so that $(\varphi_* J)|_{S^1 \times \{0\}}
v_1= v_2$. Using the symplectic trivialization induced by
$\left\{\frac{v_1}{\sqrt{g}},\frac{v_2}{\sqrt{g}}\right\}$, the asymptotic operator $A_P$ gets the local form
\begin{equation}\label{operadorAP} A_P \cdot \eta (t) = -J_0\dot \eta (t)- S(t)\eta(t),t\in S^1,\end{equation} where
$$J_0 = \left( \begin{array}{cc}0 & -1 \\ 1 & 0 \end{array} \right)$$ and
$S(t), t\in S^1,$ is a smooth loop of symmetric matrices which is
related to the linearized dynamics over $P$ restricted to the
contact structure.
\end{itemize}

The coordinates $(\vartheta,x,y)$ near $P$ described above are referred as Martinet's coordinates and the diffeomorphism $\varphi:\U \to \U_0$ is called Martinet's tube.

The symplectization of  $(M, \lambda)$ is the pair
$(\R \times M, d(e^a \lambda)),$ where $a$ is the
$\R$-coordinate. Given $J \in \J(\lambda)$, we have the almost
complex structure $\widetilde J: T(\R \times M) \to T(\R \times
M), \widetilde J^2 = -\text{Id}$, determined by
$$\widetilde J|_{\xi} = J \mbox { and } \widetilde J \partial a=
X_\lambda.$$ Equivalently,  \begin{equation}\label{Jtil}\widetilde J_{(a,w)}(h,k) =
(-\lambda_w(k),J_w \cdot \pi_w (k) + h X_\lambda(w)),\end{equation} where
$(h,k)\in T_{(a,w)}(\R \times M)$ and $\pi:TM \to \xi$ is the
projection along $X_\lambda$, i.e., $\pi_w(k) = k -
\lambda_w(k) X_\lambda(w), \forall k\in T_w M$ and $w\in M$.
Notice that $\widetilde J$ is $\R$-invariant, i.e., $T_c^*
\widetilde J = \widetilde J$ for all $c\in \R,$ where \begin{equation}\label{Tc}T_c(a,w) =(a+c,w), \forall  (a,w)\in \R \times M.\end{equation}

Let $(\Sigma,j)$ be a compact and connected Riemann surface and $\Gamma \subset \Sigma\setminus \partial \Sigma$
be a finite set. Let $\dot \Sigma = \Sigma \setminus \Gamma.$ We
consider finite energy pseudo-holomorphic curves in the symplectization $\R \times M$, i.e., maps
$\tilde u=(a,u):\dot \Sigma  \to \R \times M$ satisfying
$$d \tilde u \circ j = \widetilde J(\tilde u) \circ d \tilde u,$$ and
having finite energy  $0< E(\tilde u) < \infty,$ where
\begin{equation}\label{energiafinita} E(\tilde u) = \sup_{\psi
\in \Lambda} \int_{\dot \Sigma} \tilde u^* d\lambda_\psi,
\end{equation} $\Lambda =  \{ \psi \in C^\infty(\R,[0,1]): \psi' \geq
0\}$ and $\lambda_\psi(a,w) =  \psi(a)\lambda(w), \forall (a,w)\in
\R \times M$.  The curve $\tilde u$ is also called a finite energy $\widetilde J$-holomorphic curve.

Given $z_0\in \dot \Sigma$, there exists a
neighborhood $U_0\subset \dot \Sigma$ of $z_0$ and a
bi-holomorphism $\phi_{z_0}:(\D,i) \to (U_0,j),
\phi_{z_0}(0)=z_0,$ such that $\phi^*j = i$, where $\D \subset \C$
is the closed unit disk centered at $0$ and $i$ is the canonical
complex structure on $\C$. In coordinates $s +it\in\D$, the map
$\tilde u \circ \phi_{z_0}(s,t) = (a(s,t),u(s,t))$ satisfies
$$\begin{aligned} a_s(s,t)
- \lambda (u_t(s,t)) = & 0 \\ a_t(s,t)+ \lambda(u_s(s,t)) = & 0 \\
\pi u_s(s,t) + J(u(s,t))\pi u_t(s,t) = & 0.
\end{aligned}
$$

The points in $\Gamma$ are called punctures. Given $z_0 \in
\Gamma$, let $\phi_{z_0}:\D \to U_0$ be as above. There are
positive exponential coordinates $[0,\infty) \times \R / \Z \ni
(s,t) \mapsto \phi_{z_0}(e^{-(s+it)})$  near $z_0$ and we may
consider the map $(s,t) \mapsto \tilde u \circ
\phi_{z_0}(e^{-(s+it)})$ also denoted by $\tilde
u(s,t)=(a(s,t),u(s,t)).$ Negative exponential coordinates near
$z_0$ are defined similarly with $(s,t) \in (-\infty,0] \times \R
/ \Z$ and we obtain a map $(s,t) \mapsto \tilde u \circ
\phi_{z_0}(e^{s+it})$ also denoted by $\tilde
u(s,t)=(a(s,t),u(s,t))$.

We say that a puncture $z_0 \in \Gamma$ is removable if $\tilde
u=(a,u)$ can be smoothly extended over $z_0$, otherwise we say it is
non-removable. We assume from now on that all punctures are non-removable. It is well known that a finite energy pseudo-holomorphic curve is
asymptotic to periodic orbits at non-removable punctures.

\begin{theo}[Hofer \cite{93}]\label{Ho93} Let $z_0\in \Gamma$ be a non-removable puncture and $(s,t)\in
[0,\infty) \times S^1$ be positive exponential coordinates near
$z_0$. Given a sequence $s_k \to \infty$, there exists a
subsequence still denoted $s_k$ and a periodic orbit $P= (w,T),$
such that $u(s_k, \cdot) \to w(\epsilon T\cdot)$ in $C^\infty(S^1,
M)$ as $s_k \to \infty$. Either $\epsilon=+1$ or $\epsilon=-1$
and this sign  does not depend on the sequence $s_k$.
\end{theo}
We say that the non-removable puncture $z_0 \in \Gamma$ is
positive or negative if $\epsilon=+1$ or $\epsilon = -1$,
respectively, where $\epsilon$ is as in Theorem \ref{Ho93}. This induces the decomposition $\Gamma = \Gamma_+ \cup \Gamma_-$ according to the signs of the punctures.

One has exponential convergence if the asymptotic orbit in Theorem
\ref{Ho93} is nondegenerate, as stated in the following theorem.

\begin{theo}[Hofer, Wysocki, Zehnder \cite{props1}]\label{asymptotic} Let $z_0\in \Gamma$
be a positive puncture of the finite energy pseudo-holomorphic curve $\tilde u=(a,u)$
and $(s,t)\in [0,+\infty) \times S^1$ be positive exponential coordinates  near
$z_0$. Let $P=(w,T)$ be as in Theorem \ref{Ho93} and $k$ be a
positive integer such that $T=kT_0$, where $T_0$ is the least
positive period of $w$. Let $(\vartheta, x,y)\in \U_0\supset S^1
\times \{0\} \equiv P$ be Martinet's coordinates in a neighborhood $\U\subset
M$ of $P_0=(w,T_0)$ given by Martinet's tube $\varphi:\U \to
\U_0$. Assume $P$ is nondegenerate. Then the map
$\varphi \circ u(s,t)=(\vartheta(s,t),x(s,t),y(s,t))\in \U_0$ is
defined for all large $s$, and either $(x(s,t),y(s,t)) \equiv 0$
or there are constants $s_0,A_\gamma,r_0>0, a_0\in \R,$ such that
\begin{equation}\label{uasymptotics} \begin{aligned}|D^\gamma(a(s,t) - (Ts + a_0))| & \leq  A_\gamma e^{-r_0s},\\
|D^\gamma(\vartheta(s,t) - kt) | & \leq  A_\gamma e^{-r_0s},\\
(x(s,t),y(s,t)) & =  e^{\int_{s_0}^s \mu(r) dr}(e(t)+R(s,t)),\\
|D^\gamma R(s,t)|,|D^\gamma (\mu(s)- \mu)| & \leq  A_\gamma
e^{-r_0s},
\end{aligned}
\end{equation}for all large $s$ and $\gamma \in \N \times \N$.  $\vartheta(s,t)$ is seen as
a map on the universal cover $\R$ of $S^1$. Here $\mu(s) \to \mu <0,  \mu \in \sigma(A_P),$ and $e:S^1 \to \R^2$
corresponds to an  eigensection of $A_P$ associated to $\mu$,
represented in coordinates induced by $\varphi$. A similar statement holds if $z_0 \in \Gamma$ is a negative puncture. In this case, we use negative exponential coordinates near $z_0$, $e^{-r_0s}$ is replaced with $e^{r_0s}$ in \eqref{uasymptotics},  $s \to -\infty$ and the eigenvalue $\mu$ of $A_P$ is positive.
\end{theo}

The periodic orbit $P=(w,T)$ given in Theorem \ref{asymptotic} is called the asymptotic limit of $\tilde u$ at the puncture $z_0 \in \Gamma$. In a more general situation as in Theorem \ref{Ho93}, we say that the periodic orbit $P=(w,T)$ is an asymptotic limit of $\tilde u$ at $z_0$.

In some cases, the asymptotic behavior of $\tilde u$ near $z_0$ might be as in Theorem \ref{asymptotic}, even if $P$ is degenerate. In such cases, the eigenvalue $0\in \sigma(A_P)$ plays no role on the asymptotics of $\tilde u$ near $z_0$ and we say that $\tilde u$ has exponential decay near the puncture $z_0$.

Let $\tilde u=(a,u): \dot \Sigma \to \R \times M$ be a finite energy $\widetilde J$-holomorphic curve and let us assume that $\tilde u=(a,u)$ has exponential decay at all of its punctures.  We define its Conley-Zehnder index by $$CZ(\tilde u) = \sum_{z\in \Gamma_+}CZ(P_z)- \sum_{z\in \Gamma_-}CZ(P_z),$$ where  $P_z$ is the asymptotic limit of $\tilde u$ at $z\in \Gamma$.

The Fredholm index of $\tilde u$ is defined by $$\text{Fred}(\tilde u) = CZ(\tilde u)-\chi(\Sigma) + \# \Gamma,$$ where $\chi(\Sigma)$ is the Euler characteristic of $\Sigma$.

The second alternative in Theorem \ref{asymptotic}, where we have the exponential asymptotic behavior of $\tilde u$ as in \eqref{uasymptotics}  for a suitable eigenvalue $\mu$ of $A_P$, implies that $\pi \circ du$ does not vanish near $z_0$. In this case, since $\pi \circ du$ satisfies a Cauchy-Riemann equation, the set where $\pi \circ du$ vanishes must be finite and we can associate to each zero of $\pi \circ du$  a local degree which is always positive. Denote by $\wind_\pi(\tilde u)$ the sum of such local degrees. We also have a well-defined winding number for each puncture as follows: let $z\in \Gamma$, $P$ be the asymptotic limit of $\tilde u$ at $z$ and $e:S^1 \to \R^2$ be the eigensection of $A_P$ describing the behavior of $\tilde u$ near $z$,  in local coordinates as in \eqref{uasymptotics}, with respect to the global trivialization $\Psi$ defined in \eqref{eq_trivpsi}. We define the winding number of $\tilde u$ at $z\in \Gamma$ by $\wind_\infty(z) := \wind\left(t\mapsto e(t)\in \R^2,t\in[0,1]\right)$. The winding number of $\tilde u$ is thus defined by $$\wind_\infty(\tilde u) = \sum_{z\in \Gamma_+} \wind_\infty(z) - \sum_{z\in \Gamma_-} \wind_\infty(z),$$ and it does not depend on $\Psi$. In \cite{props2}, it is proved that \begin{equation}\label{windpi}0\leq \wind_\pi(\tilde u) = \wind_\infty(\tilde u) - \chi(\Sigma) + \# \Gamma.  \end{equation}

If $\Sigma = S^2$, then \begin{equation}\label{casoS2} \begin{aligned}  \text{Fred}(\tilde u)& = CZ(\tilde u) - 2 + \# \Gamma, \\ 0\leq \wind_\pi(\tilde u) & = \wind_\infty(\tilde u) - 2 + \# \Gamma.\end{aligned}\end{equation}

\section{Linking properties}\label{ap_link}

In this appendix we prove the following theorem.

\begin{theo}\label{teolink}For all $E>0$ sufficiently small, the following holds: let $Q_E=(y_E,t_E) \subset \dot S_E$ be an unknotted periodic orbit of $\lambda_E$ with $CZ(Q_E)=3$ and $\sl(Q_E)=-1$.  Let $Q_E'=(y_E',t_E') \subset \dot S_E$  be a periodic orbit of $\lambda_E$ which is geometrically distinct from $Q_E$ and $t_E'\leq t_E$. Then $Q_E'$ is linked to $Q_E$, meaning that $0 \neq [y_E']\in H_1(W_E \setminus Q_E, \Z)$. A similar statement holds for $\dot S_E'$.
\end{theo}

\begin{proof} Arguing indirectly, we assume there exists a sequence $E_n \to 0^+$ as $n \to \infty$ so that $Q_{E_n}'=(y_{E_n}',t_{E_n}') \subset \dot S_{E_n}$ is geometrically distinct from and not linked to $Q_{E_n}=(y_{E_n},t_{E_n})\subset \dot S_{E_n}$, with $t_{E_n}' \leq t_{E_n}$. From Proposition \ref{prop_mu}-i),  since $CZ(Q_{E_n})=3$, $Q_{E_n}$ does not intersect a fixed and small neighborhood $\tilde U_3\subset \U_{E^*}$ of $p_c$, where $E^*>0$ is fixed according to Proposition \ref{prop_step1}. From  Proposition \ref{etabarra}, we find a constant $\bar \eta>0$ so that the argument $\eta(t)$ of any non-vanishing transverse linearized solution of the Reeb flow along a segment of trajectory outside $\tilde U_3$ satisfies $\dot \eta >\bar \eta$, if $E\geq 0$ is sufficiently small. This follows from the convexity assumption on the critical level given by Hypothesis 2. Hence, since $CZ(Q_{E_n})=3$ and $Q_{E_n}'\subset \dot S_{E_n}$,  we find a constant $c>1$ such that \begin{equation}\label{eq_ten}\frac{1}{c}<t_{E_n}'\leq t_{E_n}<c, \forall n.\end{equation}
The first inequality in \eqref{eq_ten} is explained as follows: arguing indirectly we assume that $t_{E_n}' \to 0$ as $n \to \infty$. If there exists $\delta_0>0$ such that in local coordinates $Q_{E_n}'$ does not intersect $B_{\delta_0}(0)$, $\forall n$, then by  Arzel\`a-Ascoli theorem we find a constant Reeb orbit of $\bar \lambda_0$ on $\dot S_0=S_0 \setminus \{p_c\}$, where $\bar\lambda_0$ is as in Proposition \ref{prop_step1}. This is a contradiction and we conclude that $Q_{E_n}'$ must intersect $B_{\delta_0}(0)$. Now since $Q_{E_n}'$ is not a cover of $P_{2,E_n}$, $Q_{E_n}$ must contain a branch $q_n$ inside $B_{2 \delta_0}(0)\setminus B_{\delta_0}(0)$, from $\partial B_{\delta_0}(0)$ to $\partial B_{2\delta_0}(0)$, for all $n$. This follows from the local flow of $K$ near the saddle-center equilibrium point. Therefore we find a uniform constant $\epsilon>0$ so that $$t_{E_n}'=\int_{Q_{E_n}'} \lambda_{E_n}>\int_{q_n} \lambda_{E_n} > \epsilon, \forall n \mbox{ large}, $$ again a contradication.

Thus using Arzel\`a-Ascoli theorem we may extract a subsequence, also denoted by $E_n$,  so that $Q_{E_n} \to Q_0=(y_0,t_0)$, $t_{E_n} \to t_0>0$, where $Q_0$ is a periodic orbit of $\bar \lambda_0$ sitting in $\dot S_0$. Moreover,  $CZ(Q_0)\leq \liminf_{n \to \infty} CZ(Q_{E_n})\leq 3$ and since $\bar \lambda_0$ is dynamically convex, see Proposition \ref{prop_ghomi}, we get $CZ(Q_0)=3$. $Q_0$ must be simple, otherwise we would have $CZ(Q_0) \geq 5$. This implies that $Q_0$ is unknotted and $\sl(Q_0)=-1$, since $Q_{E_n} \to Q_0$ in $C^\infty$ as $n\to\infty$ and the same properties hold for each $Q_{E_n}$.

Now we study the compactness of the sequence $Q_{E_n}'$. Assume that $Q_{E_n}'$ stays away from $p_c$. From \eqref{eq_ten}, we can extract a subsequence, still denoted by $E_n$, so that $Q_{E_n}' \to Q_0'$ in $C^\infty$ as $n\to\infty$, where $Q_0'\subset \dot S_0$ is a periodic orbit of $\bar \lambda_0$. If $Q_0'$ is a $p$-cover of $Q_0$, then $p=1$ from \eqref{eq_ten}. In this case, since $CZ(Q_{E_n})=CZ(Q_0)=3$, $Q_{E_n}'$ must be linked to $Q_{E_n}$ for large $n$, a contradiction. Hence $Q_0'$ is geometrically distinct from $Q_0$. Since $Q_{E_n}'$ is not linked to $Q_{E_n}$, necessarily $Q_0'$ is not linked to $Q_0$ as well. Now take a small neighborhood $U_0\subset \U_{E^*}$ of $p_c$ not intersecting $Q_0$ and $Q_0'$. The set $S_0 \setminus U_0$ can be smoothly extended to a strictly convex sphere-like hypersurface $\tilde S_0\subset \R^4$ by a theorem of M. Ghomi, see \cite{ghomi}. We end up with a strictly convex hypersurface admitting geometrically distinct closed characteristics $Q_0$ and $Q_0'$ which are not linked. However, this contradicts Corollary 6.7 from \cite{convex}.

Now assume that $Q_{E_n}'$ gets closer to $p_c$ as $n \to \infty$. Then we can find $\delta>0$ small such that $Q_{E_n}' \cap A_\delta \neq \emptyset$ for all $n$ large, where $A_\delta\subset \U_{E^*}$ is an annulus centered at $p_c$ determined in local coordinates $(q_1,q_2,p_1,p_2)$ by $A_\delta = B_\delta(0) \setminus B_{\frac{\delta}{2}}(0)$. Let $\gamma_\delta^s$ be the intersection of the local stable manifold of $p_c$ with $A_\delta$  given  by $\left\{p_1=q_2=p_2=0,  \frac{\delta}{2} \leq q_1<  \delta\right\}$ and let $\gamma_\delta^u$ be the intersection of the local unstable manifold of $p_c$ with $A_\delta$  given  by $\left\{q_1=q_2=p_2=0,  \frac{\delta}{2}\leq p_1<  \delta\right\}$. Each segment of $Q_{E_n}'$ intersecting $A_\delta$ gets either closer to $\gamma_\delta^s$ or to $\gamma_\delta^u$ as $n \to \infty$. Since $\int_{\gamma_\delta^s} \bar \lambda_0 , \int_{\gamma_\delta^u} \bar \lambda_0 > c_1>0$ for some constant $c_1>0$, and $\lambda_{E_n} \to \bar \lambda_0$ in $A_\delta$, see Proposition \ref{prop_step1}, we conclude that each component $\gamma_n$ of $Q_{E_n}' \cap A_\delta$ satisfies $\int_{\gamma_n} \lambda_{E_n} >c_1$ for all large $n$. Since the action of $Q_{E_n}'$ is uniformly bounded, see \eqref{eq_ten}, the number of components of $Q_{E_n}' \cap A_\delta$ is uniformly bounded in $n$. We may assume that the number of such components is constant equal to $2p_0\geq 2$, $p_0$ of them close to $\gamma_\delta^s$ and $p_0$ close to $\gamma_\delta^u$. This implies that $Q_{E_n}'$ has exactly $p_0$ components outside $B_\delta(0)$. Again from \eqref{eq_ten} and from the fact that the Liouville vector fields $Y_{E_n}$ and $\bar X_0$ coincide outside $B_\delta(0)$ and near $\dot S_0$ for all large $n$, see Proposition \ref{prop_step1},  we can use Arzel\`a-Ascoli theorem in order to find a subsequence of $E_n$ and a homoclinic orbit $\gamma_0\subset \dot S_0$ to $p_c$ so that, after a suitable reparametrization, $Q_{E_n}'$ converges in $C^0$ to a $p_0$-cover of the simple closed curve $\bar \gamma_0:=\gamma_0\cup\{p_c\}\subset S_0$ as $n \to \infty$. From the assumptions on $Q_{E_n}$ and $Q_{E_n}'$, $\bar \gamma_0$ is not linked to $Q_0$ in $S_0$. This contradicts Proposition \ref{prophomoc} proved below and finishes the proof of Theorem \ref{teolink}.
\end{proof}

\begin{prop}\label{prophomoc}Assume $p_c$ admits a homoclinic orbit $\gamma_0 \subset \dot S_0$. Let $Q_0\subset \dot S_0$ be an unknotted periodic orbit with $CZ(Q_0)=3$ and $\sl(Q_0)=-1.$  Then the closed curve $\bar \gamma_0 = \gamma_0 \cup \{p_c\}\subset S_0$ is linked to $Q_0$ in $S_0$. \end{prop}

In order to prove Proposition \ref{prophomoc} we start by regularizing $S_0$ near $p_c$ to obtain a dynamically convex tight Reeb flow on a smooth sphere-like hypersurface so that $Q_0$ and $\bar \gamma_0$ correspond to periodic orbits. Then we apply Corollary 6.7 from \cite{convex} to get the desired linking property.

Fix $\delta_1>0$ small and choose a cut-off symmetric function $f:\R \to [0,1]$ so that $f'(x)\leq0 \, \forall x\geq 0,$ $f(x)=1$ if $0\leq x \leq \frac{\delta_1}{2}$, and $f(x)=0$ if $x \geq \delta_1.$ We consider $S_0$ as a subset of $\R^4$. Fix a small neighborhood $U_0\subset \R^4$ of $p_c$ not intersecting $Q_0$ and let $S_\epsilon\subset \R^4$ be the hypersurface which coincides with $S_0$ outside $U_0$ and inside $U_0$ corresponds in the local coordinates $z=(q_1,q_2,p_1,p_2)$ to the set $K_\epsilon^{-1}(0)$ where $$K_\epsilon(z) = \bar K(I_1,I_2)+\epsilon f(q_1+p_1),$$
$\epsilon>0$ small. Then $S_\epsilon$ is diffeomorphic to the $3$-sphere and in local coordinates it projects on $\{q_1,p_1 \geq 0\}$, not containing the origin. Moreover $S_\epsilon \to S_0$ in $C^0$ as $\epsilon \to 0$ and $S_\epsilon \to S_0$ in $C^\infty$ as $\epsilon \to 0$ outside any fixed neighborhood of $p_c$. Note that $S_\epsilon$  coincides with $S_0$ in the set $\{q_1+p_1 \geq \delta_1\}$.

\begin{lem}\label{lemSeps}For all $\epsilon>0$ small, the Hamiltonian flow on $S_\epsilon$ is equivalent to the Reeb flow of a dynamically convex tight contact form on $S^3$.
\end{lem}

\begin{proof}Let $\tilde S_\epsilon= \varphi^{-1}(S_\epsilon \cap U)\subset \{q_1\geq 0,p_1 \geq 0\}, 0\not \in \tilde S_\epsilon,$ where $\varphi:V \to U$ is the symplectic diffeomorphism established in Hypothesis 1. Let $\delta_0\gg \bar \delta_0 >0$ be as in Proposition \ref{propYtransversal}-ii). We may assume that $\delta_1> 4\delta_0$ and, if $z\in B_{\delta_0}(0) \subset V,$ then $q_1+p_1<\delta_1/2$, where $\delta_1>0$ is given above.

Let $Y$ be the Liouville vector field on $V$ constructed in Section \ref{sec_prop_step1}, see equation \eqref{defi_Y}. As mentioned in the proof of Proposition \ref{prop_step1}, $Y$ induces a Liouville vector field, denoted here by $\bar X_0$ and defined near $S_0 \subset \R^4$, by patching $\varphi_* Y$ together with $X_0=\frac{w-w_0}{2},w_0=\varphi(z_0),z_0=\frac{1}{\sqrt{2}}(\delta_0,0,\delta_0,0).$ Proposition \ref{propYtransversal}-i) and Remark \ref{remconvexity} show that $\bar X_0$ is transverse to $\dot S_0$.

From Proposition \ref{propYtransversal}-i), $Y$ is positively transverse to $\tilde S_\epsilon$ outside $B_{\bar \delta_0}(0)$ if $\epsilon>0$ is sufficiently small. This follows from the fact that $\tilde S_\epsilon \to \varphi^{-1}(S_0\cap U)$ in $C^\infty$ outside $B_{\bar \delta_0}(0)$ as $\epsilon \to 0^+$. Hence $\bar X_0$ is transverse to $S_\epsilon$  outside $\varphi(B_{\bar \delta_0}(0))$ if $\epsilon>0$ is sufficiently small. Now observe that, from the choices of $\bar \delta_0, \delta_0$ and $\delta_1$, we have $z\in  B_{\bar \delta_0}(0)\cap  \tilde S_\epsilon \Rightarrow z\in K^{-1}(-\epsilon)$. From Proposition \ref{propYtransversal}-ii), this implies that $Y$ is positively transverse to $\tilde S_\epsilon$ if $z \in B_{\bar \delta_0}(0) \cap \tilde S_\epsilon$. Therefore $\bar X_0$ is transverse to $S_\epsilon$ for all $\epsilon>0$ small and then $\lambda_\epsilon:= i_{\bar X_0} \omega_0|_{S_\epsilon}$ is a contact form on $S_\epsilon$. Since $S_\epsilon$ bounds a symplectic manifold, $\lambda_\epsilon$ is tight, see \cite{eli}.

To see that $\lambda_\epsilon$ is dynamically convex we use the frame $\{X_1,X_2\}$, defined in \eqref{mfX}, to analyze the linearized flow as in the proof of Proposition \ref{prop_step2} in Section \ref{sec_step2}.  Using Lemma \ref{lem2} with $\W_0:=\varphi(B_{\delta_0}(0))$, we find small neighborhoods $U_{2\pi} \subset U_* \subset \W_0$ of $p_c$ such that the argument variation $\Delta \eta$ of a non-vanishing transverse linearized solution along a segment of trajectory in $U_*\cap S_\epsilon\subset H^{-1}(-\epsilon)$ which intersects $U_{2\pi}$ satisfies \begin{equation}\label{eeqdelta}\Delta \eta > 2 \pi.\end{equation} Since $S_\epsilon \to S_0$ in $C^\infty$ outside $U_{2\pi}$ as $\epsilon \to 0^+$, we can assume that \begin{equation}\label{eeqeta} \dot \eta >\bar \eta >0\end{equation} outside $U_{2\pi}$ for some $\bar \eta>0$ and all $\epsilon>0$ small. Here we use the parametrization induced by $K_\epsilon$ and  by $H$ outside $U$ and also when $q_1+p_1\geq \delta_1$. So, arguing indirectly, assume there exists a sequence $\epsilon_n \to 0^+$ and periodic orbits $P_n \subset S_{\epsilon_n}$ with \begin{equation} \label{eeqcz} CZ(P_n) \leq 2\end{equation} and periods $T_n>0,\forall n$. From \eqref{eeqdelta} and \eqref{eeqeta} we have $P_n \cap U_{2\pi} = \emptyset$ for all large $n$. From \eqref{eeqeta} and \eqref{eeqcz}, $T_n$ is uniformly bounded so, by Arzel\`a-Ascoli theorem, we can extract a subsequence still denoted by $P_n$ so that $P_n \to \bar P_0$ in $C^\infty$ as $n\to \infty$, where $\bar P_0$ is a periodic orbit contained in $S_0\setminus U_{2 \pi}.$ From \eqref{eeqcz},  $CZ(\bar P_0) \leq 2$. This contradicts Proposition \ref{prop_ghomi} and finishes the proof of the lemma.
\end{proof}

\begin{lem}\label{lemhomoc}In the hypotheses of Proposition \ref{prophomoc}, we have: for each $\epsilon>0$ small, there exists a homeomorphism $h: S_0 \to S_\epsilon$ satisfying the following properties:
 \begin{itemize}
 \item[(i)] $h$ is the identity map outside a small neighborhood of $p_c$ in $S_0$;
 \item[(ii)] $h|_{\dot S_0}$ is a smooth diffeomorphism;
 \item[(iii)] $h(\bar \gamma_0)= P_0,$ where $P_0\subset S_\epsilon \setminus Q_0$ is a closed characteristic of $S_\epsilon$, which is linked to $Q_0$.
 \end{itemize}
\end{lem}

\begin{proof}To define $h:S_0 \to S_\epsilon$, we first require that $h$ coincides with the identity map outside $U_0$ and also at the points $w=\varphi(q_1,q_2,p_1,p_2) \in S_0 \cap U_0$ satisfying $\{q_1+p_1\geq \delta_1\}$ in local coordinates. At the remaining points $w\in S_0 \cap U_0$, which  satisfy $q_1,p_1\geq 0, 0\leq q_1+p_1 \leq \delta_1$, we define $h$ in local coordinates as follows. Let $\bar q_1 \geq 0$ be such that $q_1p_1 = \bar q_1 2 \delta_1$.  The negative Hamiltonian flow of $K_\epsilon$ through $(\bar q_1, q_2,2\delta_1, p_2)$ projected into the $(q_1,p_1)$-plane hits the line $\{(q_1,p_1) + \lambda(1,1), \lambda \in \R\}$ at a single point $(\tilde q_1,\tilde p_1)$. This follows from the fact that $X_{K_\epsilon}$ is always linearly independent to the vector $(1,0,1,0)$ for any $\epsilon>0$ small and $q_1+p_1 \geq 0, (q_1,p_1)\neq (0,0)$. We let $h(w):= \varphi(\tilde q_1,q_2,\tilde p_1,p_2)$. Clearly $h(w)\in S_\epsilon$. Notice that $h$ fixes the coordinates $(q_2,p_2)$.  The maps defined above are smooth in coordinates $(q_1,p_1)$ and this implies that $h$ is continuous in $S_0$ and smooth in $\dot S_0=S_0\setminus \{p_c\}$.

Its inverse $h^{-1}:S_\epsilon \to S_0$ is the map defined in the following way: if $w\in S_\epsilon$ and $w\not\in U_0$ or if in local coordinates $w$ is such that $q_1 + p_1 \geq \delta_1$, then $h^{-1}(w)=w$. If $w\in S_\epsilon \cap U_0$ and in local coordinates is such that $0<q_1+p_1 \leq \delta_1$, then the positive flow of $K_\epsilon$ through $z=\varphi^{-1}(w)=(q_1,q_2,p_1,p_2)$ hits the hyperplane $\{p_1 = 2 \delta_1\}$ at a point $(\bar q_1,\bar q_2,2\delta_1, \bar p_2)$ satisfying $\bar q_1\geq 0$. Let $\lambda\in\R$ be the root of the equation $\bar q_1 2 \delta_1=(q_1-\lambda)(p_1-\lambda)$ given by $$\lambda = \frac{q_1+p_1 - \sqrt{(q_1-p_1)^2 + 8\delta_1 \bar q_1}}{2}\geq 0.$$ We define $\tilde z=z - \lambda(1,0,1,0)$ and let $h^{-1}(w)= \varphi(\tilde z).$ The map $(q_1,p_1) \mapsto (\bar q_1,2\delta_1)$ is smooth since it is the restriction of a smooth flow map. The map $(\bar q_1,2\delta_1) \mapsto \tilde z$ is smooth except if both equalities $q_1=p_1$ and $\bar q_1=0$ hold. However, this point is mapped to $0$ and  $(\bar q_1,2\delta_1) \mapsto \tilde z$ is only continuous at this point. This implies that $h^{-1}$ is also continuous and therefore $h$ is a homeomorphism and $h$ restricts to a smooth diffeomorphism at $\dot S_0$.

Observe that $h(Q_0)=Q_0$, since $U_0$ was chosen so that $Q_0 \subset S_0 \setminus U_0$. Now for each $ w\in \bar \gamma_0$ outside $U_0$  or  which satisfies $q_1+p_1 \geq \delta_1$  in local coordinates, we have $h(w)=w$. In local coordinates,  $\bar \gamma_0$  is given by $\bar \gamma_{0,\text{loc}}:=\{q_1p_1=0, q_2=p_2=0,0\leq q_1+p_1\leq \delta_1\}$. We can assume that $\bar \gamma_0\cap \{0 \leq q_1+p_1 \leq \delta_1\} = \bar \gamma_{0,\text{loc}}$. By construction of $h$ and by the symmetry of $X_{K_\epsilon}$ with respect to $\{q_1=p_1\}$, we see that $\bar \gamma_{0,\text{loc}}$ is mapped into a trajectory of $X_{K_\epsilon}$ which, in local coordinates, is contained in $\{q_2=p_2=0\} \cap \tilde S_\epsilon$, starts at $\{q_1=\delta_1,q_2=p_1=p_2=0\}$ and ends at $\{p_1=\delta_1,q_1=q_2=p_2=0\}$. This segment of trajectory is symmetric with respect to $\{q_1=p_1\}$ and its internal points are contained in $\{q_1>0,p_1>0\}$.  Thus $P_0:=h(\bar \gamma_0)$ is a closed characteristic of $S_\epsilon$. Now $Q_0$ is also a closed characteristic of $S_\epsilon$ which is geometrically distinct from $P_0$ by construction. From Lemma \ref{lemSeps}, if $\epsilon>0$ is small enough then the Hamiltonian flow on $S_\epsilon$ is equivalent to the Reeb flow of a dynamically convex tight contact form on $S^3$. Since $Q_0$ is unknotted, $\sl(Q_0)=-1$ and $CZ(Q_0)=3$, we can use Corollary 6.7 from \cite{convex} to conclude that $P_0$ is linked to $Q_0$.
\end{proof}

Proposition \ref{prophomoc} is a direct consequence of Lemma \ref{lemhomoc}.

\section{Uniqueness and intersections of pseudo-holomorphic curves}\label{appunique}

In this appendix we use Siefring's intersection theory \cite{Si2} to prove uniqueness of rigid planes and rigid cylinders in $W_E$ and we obtain some intersection properties of pseudo-holomorphic semi-cylinders asymptotic to $P_{2,E}$. We only treat the case  $S_E$ since the case  $S_E'$ is completely analogous. We start with uniqueness.

In Proposition \ref{prop_step3}, we prove that, for all $E>0$ sufficiently small, there exists a pair of finite energy $\tilde J_E$-holomorphic planes $\tilde u_{1,E}=(a_{1,E},u_{1,E}),\tilde u_{2,E}=(a_{2,E},u_{2,E}):\C \to \R \times W_E$, both asymptotic to $P_{2,E}$ at $+\infty$, so that $u_{1,E}(\C)=U_{1,E},$ $u_{2,E}(\C)=U_{2,E}$ are the hemispheres of $\partial S_E$. Moreover, we prove in Proposition \ref{prop_step5}-ii), the existence of a finite energy $\tilde J_E$-holomorphic cylinder $\tilde v_E=(b_E,v_E):\R \times S^1 \to \R \times W_E$, asymptotic to $P_{3,E}$ and $P_{2,E}$ at $s=+\infty$ and $s=-\infty$, respectively, and so that $v_E(\R \times S^1)\subset \dot S_E\setminus P_{3,E}$. The asymptotic behavior of $\tilde v_E$ at $s=+\infty$ is given in Proposition \ref{asymptoticcylinder}.

\begin{prop}\label{propunique}Up to reparametrization and $\R$-translation, $\tilde u_{1,E}$ and $\tilde u_{2,E}$ are the unique finite energy $\tilde J_E$-holomorphic planes asymptotic to $P_{2,E}$ at $+\infty$. In the same way, $\tilde v_E$ is the unique finite energy $\tilde J_E$-holomorphic cylinder with projected image in $\dot S_E\setminus P_{3,E}$, which is asymptotic to $P_{3,E}$ and to $P_{2,E}$ at $s=+\infty$ and $s=-\infty$, respectively, and so that its asymptotic behavior at $s=+\infty$ is described by a negative eigenvalue of $A_{P_{3,E}}$ with winding number $1$ as in Proposition \ref{asymptoticcylinder}.\end{prop}

\begin{proof}  Let us start with the case of the rigid planes. Arguing indirectly, we assume that $\tilde u=(a,u):\C \to \R \times W_E$ is a finite energy $\tilde J_E$-holomorphic plane asymptotic to $P_{2,E}$ at $+\infty$, which does not coincide with $\tilde u_{1,E}$ and $\tilde u_{2,E}$ up to reparametrization and $\R$-translation. This is equivalent to saying that \begin{equation}\label{nigual} u(\C) \neq u_{1,E}(\C) \mbox{ and } u(\C) \neq u_{2,E}(\C).\end{equation} Since the hyperbolic orbit $P_{2,E}$ is unknotted, $CZ(P_{2,E})=2$ and $\pi_2(S^3)$ vanishes, we have from Theorem 1.3 in \cite{props2} that \begin{equation}\label{uint} u(\C)\cap P_{2,E}=\emptyset\end{equation} and $u$ is an embedding. Equations \eqref{nigual}, \eqref{uint} and  Theorem 1.4 from \cite{props2} imply that \begin{equation}\label{disjunto} u(\C) \cap u_{1,E}(\C) = \emptyset \mbox{ and } u(\C) \cap u_{2,E}(\C) = \emptyset.\end{equation}

Denote by $\wind(e_+)$ the winding number of the asymptotic eigensection $e_+$ of $A_{P_{2,E}}$, which describes $\tilde u$ near $+\infty$, computed with respect to a global trivialization of $\xi$. See Appendix \ref{ap_basics} for definitions. Since $e_+$ is associated to a negative eigenvalue of $A_{P_{2,E}}$ and $CZ(P_{2,E}) =2$, we have that $\wind(e_+)\leq \wind^{<0}(A_{P_{2,E}})=1$. Then we obtain $$0\leq \wind_{\pi}(\tilde u) = \wind_\infty(\tilde u)-1 = \wind(e_+) -1 \leq 0,$$ which implies that \begin{equation}\label{gaps} \begin{aligned} & \wind_\infty(\tilde u) =  \wind(e_+) =1, \\ & \wind_\pi(\tilde u) = 0,\\ & d_0(\tilde u):=  \wind^{<0}(A_{P_{2,E}}) - \wind(e_+)=0.\end{aligned}\end{equation}

In \cite{Si2}, R. Siefring introduces a generalized intersection number between pseudo-holomorphic curves $\tilde u$ and $\tilde v$, denoted by $[\tilde u] *[\tilde v]$, which counts the actual algebraic intersection number  plus an asymptotic intersection number at their punctures.

Recall that $u_{1,E}$ and $u_{2,E}$ don't intersect $P_{2,E}$. Moreover, the eigensections $e_1$ and $e_2$ of $A_{P_{2,E}}$ describing $\tilde u_{1,E}$ and $\tilde u_{2,E}$ at $+\infty$, respectively, also have winding number $1$, and therefore, $d_0(\tilde u_{1,E})=d_0(\tilde u_{2,E})=0$.

Using equations \eqref{uint}, \eqref{gaps} and Corollary 5.9 from \cite{Si2} (see conditions (1) and (3)), we obtain \begin{equation} \label{genint0} [\tilde u]*[\tilde u_{1,E}]=0 \mbox{ and } [\tilde u]*[\tilde u_{2,E}]=0.\end{equation}

\begin{definition}[Siefring~\cite{Si2}]
Let $\tilde u,\tilde v$ be finite energy $\tilde J_E$-holomorphic curves asymptotic to the same nondegenerate periodic orbit $P\in \P(\lambda_E)$ at certain  punctures of $\tilde u$ and $\tilde v$, respectively. We say that $\tilde u$ and  $\tilde v$ approach $P$ in the same direction at these punctures if $\eta_+=c \eta_-$ for a positive constant $c$, where $\eta_+$ and $\eta_-$ are the eigensections of the asymptotic operator $A_{P}$ which describe $\tilde u$ and $\tilde v$, respectively, near the respective punctures, as in Theorem \ref{asymptotic}. In case $\eta_+=c \eta_-$ with $c<0,$ we say that $\tilde u$ and $\tilde v$ approach $P$ through opposite directions.
\end{definition}

Since $\wind(e_+)=\wind(e_1)=\wind(e_2)=1$ and there exists only one negative eigenvalue $\delta$ of $A_{P_{2,E}}$  with winding number equal to $1$, we conclude that $e_+,e_1$ and $e_2$ are all $\delta$-eigensections. Since the eigenspace of $\delta$-eigensections  is one-dimensional, we find a positive constant $c>0$ such that either $e_+ = c e_1$ or $e_+=c e_2$. Suppose without loss of generality that $e_+ = c e_1, c>0$. We conclude that \begin{equation}\label{samedirection} \tilde u \mbox{ and } \tilde u_{1,E} \mbox{ approach } P_{2,E} \mbox{ in the same direction.}\end{equation} Using \eqref{disjunto}, \eqref{samedirection} and Theorem 2.5 from \cite{Si2}, we conclude that $[\tilde u]*[\tilde u_{1,E}]>0,$ contradicting \eqref{genint0}. We have proved that $\tilde u_{1,E}$ and $\tilde u_{2,E}$ are the unique finite energy $\tilde J_E$-holomorphic planes asymptotic to $P_{2,E}$, up to reparametrization and $\R$-translation, and they approach $P_{2,E}$ through opposite directions.

Now we deal with the cylinder case. Again we argue indirectly and assume the existence of a finite energy $\tilde J_E$-holomorphic cylinder $\tilde v=(b,v):\R \times S^1 \to \R \times W_E$ which is asymptotic to $P_{3,E}$ at $s=+\infty$, asymptotic to $P_{2,E}$ at $s=-\infty$, $v(\R \times S^1) \subset \dot S_E\setminus P_{3,E}$, and that \begin{equation}\label{diferente} v(\R \times S^1) \neq v_E(\R \times S^1). \end{equation} This implies, by Carleman's similarity principle, that the images of $v_E$ and $v$ do not coincide near any neighborhood of $s=-\infty$.

We also assume that there exists an eigensection $e_+$ of $A_{P_{3,E}}$ with winding number equal to $1$ associated to a negative eigenvalue, which describes $\tilde v$ near $s=+\infty$ as in Theorem \ref{asymptotic}.

Denote by $e_-$ the eigensection of $A_{P_{2,E}}$ associated to a positive eigenvalue which describes $\tilde v$ near $s=-\infty$ as in Theorem \ref{asymptotic}. Since $CZ(P_{2,E})=2$, we have $\wind(e_-)\geq \wind^{ \geq 0}(P_{2,E}) =1$. Moreover, we have $$0\leq \wind_\pi(\tilde v) = \wind_\infty(\tilde v)= \wind(e_+) - \wind(e_-) = 1 - \wind(e_-)\leq 0.$$ We conclude that $\wind_\pi(\tilde v)=0$ and $\wind(e_+)=\wind(e_-) =1$. Denote by $e_{1+}$ the eigensection of $A_{P_{3,E}}$ which describes $\tilde v_E$ near $+\infty$ and by $e_{1-}$ the eigensection of $A_{P_{2,E}}$ which describes $\tilde v_E$ near $-\infty$. We also have $\wind(e_{1+})=\wind(e_{1-})=1$, see Proposition \ref{cylinderconclusion}.  Since there exists only one positive eigenvalue of $A_{P_{2,E}}$ with winding number $1$ and its eigenspace is one-dimensional, we conclude that there exists a constant $c$ so that $e_-=ce_{1-}$. Since $v(\R \times S^1),v_E(\R \times S^1) \subset \dot S_E$ we conclude that $c>0$, i.e., $\tilde v$ and $\tilde v_E$ approach $P_{2,E}$ in the same direction. \eqref{diferente} and Theorem 2.5 from \cite{Si2} imply that \begin{equation}\label{genint1} [\tilde v]*[\tilde v_E]>0.\end{equation}

Now since $v(\R\times S^1) \cap (P_{2,E} \cup P_{3,E})=\emptyset$ we can apply Corollary 5.9 from \cite{Si2} to conclude that $[\tilde v]*[\tilde v_E]=0$, contradicting \eqref{genint1}. This proves that $\tilde v_E$ is the unique finite energy $\tilde J_E$-holomorphic cylinder with the properties given in the statement of this proposition. \end{proof}

\begin{remark}This type of uniqueness property of pseudo-holomorphic curves, as proved in Proposition \ref{propunique}, can be found in \cite{HMS}.\end{remark}

Now we prove the following intersection properties of semi-cylinders asymptotic to covers of $P_{2,E}.$

\begin{prop}\label{propintersect}The following assertions hold: \begin{itemize}
\item[i)] Let $\tilde w=(d,w):(-\infty,0] \times S^1 \to \R \times W_E$ be a somewhere injective finite energy $\tilde J_E$-holomorphic semi-cylinder, which is asymptotic to a $p_0$-cover of $P_{2,E}$ at its negative puncture $-\infty,$ with $p_0 \geq 2$. Assume also that $w((-\infty,0] \times S^1) \subset \dot S_E$. Then $w$ admits self-intersections, i.e., there exist $z_1\neq z_2 \in (-\infty,0] \times S^1$ so that $w(z_1)=w(z_2)$. In particular, there exists $c\in \R$  so that $\tilde w(z_1)=\tilde w_c(z_2)$, where $\tilde w_c:=(d+c,w)$.
\item[ii)] Let $\tilde w_1=(d_1,w_1),\tilde w_2=(d_2,w_2):(-\infty,0] \times S^1 \to \R \times W_E$ be  finite energy $\tilde J_E$-holomorphic semi-cylinders, which are both asymptotic to $P_{2,E}$ at their negative punctures $-\infty.$  Assume  that $w_1((-\infty,0]\times S^1),$ $w_2((-\infty,0] \times S^1) \subset \dot S_E$. Then $w_1$ intersects $w_2$, i.e., there exist $z_1,z_2\in (-\infty,0] \times S^1$ so that $w_1(z_1)=w_2(z_2)$. In particular, there exists $c\in \R$ so that $\tilde w_1(z_1)=\tilde w_{2,c}(z_2)$, where $\tilde w_{2,c}:=(d_2+c,w_2)$.

\end{itemize}
\end{prop}

\begin{proof}
The asymptotic operator $A_{P_{2,E}}$ admits two eigenvalues $\delta^-<0$ and $\delta^+>0$  with winding number $1$, whose eigenspaces are both one-dimensional and contain a pair of $\delta^-$-eigensections $e_1,e_2$ and a pair of $\delta^+$-eigensections $e_-,e_-'$, respectively, which describe, in this order, the asymptotic behavior of the rigid planes $\tilde u_{1,E}, \tilde u_{2,E}$ near the puncture $+\infty$ and  of the rigid cylinders $\tilde v_{E}, \tilde v_{E}'$ near the negative puncture $-\infty$, as in Theorem \ref{asymptotic}. From Theorem \ref{op_ass}-(iv),  the  $\delta^-$-eigensections $e_-$ and $e_-'$ are transverse to $\partial S_E=u_{1,E}(\C) \cup P_{2,E} \cup u_{2,E}(\C)$ at $P_{2,E}$ and, moreover, $e_-$ points inside $\dot S_E$ and $e_-'$ points inside $\dot S_E'$.


The periodic orbit $P_{2,E}^{p_0}=(x_{2,E},p_0 T_{2,E})$ is also hyperbolic and has Conley-Zehnder index  $CZ(P_{2,E}^{p_0})=2p_0$. Its asymptotic operator $A_{P_{2,E}^{p_0}}$ admits an eigensection $e_-^{p_0}$ associated to the least positive eigenvalue $\delta^{p_0} := p_0 \delta^+$ with winding number $p_0$. In fact, one can check that $e_-^{p_0}$ can be chosen to be equal to $e_-$ covered $p_0$ times.  All the  eigensections of $A_{P_{2,E}^{p_0}}$ associated to positive eigenvalues $\mu > \delta^{p_0}$ have winding number $>p_0$. Then, since we are assuming that $w((-\infty,0] \times S^1) \subset \dot S_E$, it follows that $\tilde w$ is described by $e_-^{p_0}$ near $s=-\infty$.

To see that $w$ self-intersects, we argue indirectly assuming it doesn't. Consider Martinet's coordinates $(\vartheta,x,y)\in S^1 \times \R^2$ near $P_{2,E}\equiv S^1 \times \{0\},$ so that in these coordinates we have $\partial_x \sim e_-$, where $\sim$ means a positive multiple of each other. Therefore, with respect to the frame $\{\partial_x,\partial_y\}$, the winding number of $e_-$ is equal to zero. From the asymptotic description of $\tilde w$ near $-\infty$, we have that the image $F_0:=\{w(s,t),(s,t)\in (-\infty,s_0] \times S^1\}, s_0 \ll 0,$ is a connected strip and the intersection of $F_0$ with the plane $\{\vartheta =t\}$ contains $p_0$ disjoint branches of curves converging to $(t,0,0)$ as $s \to -\infty$, and all of them tangent to $\partial_x$ at $(t,0,0)$, for each $t\in S^1$. Hence, in these coordinates, $F_0 \subset \{x > 0\}$. These $p_0$ branches vary smoothly in $t\in S^1$ and, since we are assuming they are disjoint, it follows that  they admit a natural order in $\R^2 \equiv \{\vartheta =t\}$ which is independent of $t$. This is not possible since $F_0$ is connected and $p_0 \geq 2$, so the branches should non-trivially permute when going through each period of $P_{2,E}$, implying intersections of the branches. This contradiction shows that $w$ must self-intersect. This proves i).

To prove ii), we use Siefring's description of the difference of two distinct pseudo-holomorphic semi-cylinders \cite{Si1}. After a change of coordinates in the domain, which converges to the identity as $s \to -\infty$, we can write $$w_i(s,t) = \exp_{x_{2,E}(T_{2,E}t)}e^{\delta_i s}(\eta_i(t)+r_i(s,t)), s \ll 0,$$ $i=1,2$, where $\delta_i>0$ is an eigenvalue of the asymptotic operator $A_{P_{2,E}}$, $\eta_i$ is a $\delta_i$-eigensection and $r_i(s,t)$ decays exponentially fast with all of its derivatives as $s\to -\infty$, $i=1,2$. Here $\exp$ denotes the exponential map associated to the usual metric on $W_E$ induced by $\tilde J_E=(\lambda_E, J_E)$. Since $w_i((-\infty,0] \times S^1) \subset \dot S_E$, we have that $\delta_i=\delta^+,i=1,2,$ where $\delta^+$ is the least positive eigenvalue of $A_{P_{2,E}}$. Its winding number $\wind(\delta^+)$ is  equal to $1$, when computed with respect to a global trivialization of $\xi$. Since $\delta^+$ is the only positive eigenvalue with winding number $1$, since its eigenspace is one-dimensional and since $w_i((-\infty,0] \times S^1) \subset \dot S_E$, $i=1,2,$ we conclude that $\tilde w_1$ and $\tilde w_2$ approach $P_{2,E}$ through the same direction, i.e., $\eta_1(t)=c\eta_2(t),\forall t,$ for a positive constant $c>0$. After a translation in the domain, we can assume that $c=1 \Rightarrow \eta_1(t)=\eta_2(t), \forall t.$

Let $W_i(s,t):=e^{\delta_i s}(\eta_i(t)+r_i(s,t))$, $i=1,2$. From Theorem 2.2 in \cite{Si1}, adapted to negative punctures, we obtain that either $W_1(s,t)\equiv W_2(s,t)$ or there exists a positive eigenvalue $\bar\delta$ of $A_{P_{2,E}}$ and a $\bar\delta$-eigensection $\bar e$ such that $W_1(s,t)-W_2(s,t)=e^{\bar\delta s}(\bar e(t)+ r(s,t))$, where $r(s,t)$ decays exponentially fast with all of its derivatives as $s\to -\infty$. In the second case, $$e^{\delta^+ s}(r_1(s,t)-r_2(s,t))=W_1(s,t)-W_2(s,t)=e^{\bar\delta s}(\bar e(t)+ r(s,t)), s \ll 0,$$ and, from the exponential decay of $r_i,$ $i=1,2$, we conclude that $\bar\delta>\delta^+$, which implies $\wind(\bar\delta) > \wind(\delta^+)=1.$ Therefore, in both cases, we find infinitely many intersections of $w_1$ with $w_2$, for $|s|$ arbitrarily large.  \end{proof}

\noindent {\it Acknowledgements.} This project started many years ago while PS was visiting Courant Institute as a post-doc under the supervision of Prof. H. Hofer. PS would like to thank the institute for hospitality and H. Hofer for fruitful conversations about this question. NP was supported by FAPESP 2009/18586-6 and 2014/08113-1. PS was partially supported by CNPq 301715/2013-0 and by FAPESP 2011/16265-8 and 2013/20065-0.


\begin{thebibliography}{cc}

\bibitem{AR} S. Addas-Zanata and C. Grotta-Ragazzo. {\it Conservative dynamics: unstable sets for saddle-center loops.} J. Diff. Eqs, {\bf 197}, (2004), 118--146.

\bibitem{AFKP} P. Albers, U. Frauenfelder, O. van Koert and G. P. Paternain.  {\it Contact geometry of the restricted three-body problem.} Communications on Pure and Applied Mathematics, 65(2), (2012), 229--263.


\bibitem{BMO} E. Barrab\'es, J. M. Mondelo and M. Oll\'e. {\it Dynamical aspects of multi-round horseshoe-shaped homoclinic orbits in the RTBP.} Celestial Mechanics and Dynamical Astronomy 105.1-3, (2009), 197--210.

\bibitem{ben} D. Bennequin, {\it Entrelacements et \'equation de Pfaff,} Ast\'erisque, {\bf 107-108} (1983), 83--161.

\bibitem{ber1} P. Bernard, {\it Homoclinic orbits in families of hypersurfaces with hyperbolic periodic orbits.} Journal of Differential Equations {\bf 180,} 2, (2002) 427--452.

\bibitem{ber2} P. Bernard, {\it Homoclinic orbit to a center manifold.} Calculus of Variations and Partial Differential Equations {\bf 17,} 2 (2003) 121--157.

\bibitem{BGS} P. Bernard,  C. Grotta-Ragazzo and Pedro A. Santoro Salom\~ao. {\it Homoclinic orbits near saddle-center fixed points of Hamiltonian systems with two degrees of freedom,} Ast\'erisque, {\bf 286} (2003), 151--165.

\bibitem{BLJ} P. Bernard, E. Lombardi and T. J\'ez\'equel. {\it Homoclinic connections with many loops near a $0^2 iw$ resonant fixed point for Hamiltonian systems.} preprint arXiv:1401.1509 (2014).



\bibitem{sft2} F. Bourgeois, Y. Eliashberg, H. Hofer, K. Wysocki, E.  Zehnder. {\it Compactness results in symplectic field theory.} Geometry and Topology, {\bf 7}, 2, (2003) 799--888.

\bibitem{con1} C. Conley. {\it Twist mappings, linking, analyticity, and periodic solutions which pass close to an unstable periodic solution.} Topological dynamics (1968) 129--153.

\bibitem{con2} C. Conley. {\it Low energy transit orbits in the restricted three-body problems.} SIAM Journal on Applied Mathematics {\bf 16,} 4 (1968) 732--746.


\bibitem{con3} C. Conley. {\it On the ultimate behavior of orbits with respect to an unstable critical point I. Oscillating, asymptotic, and capture orbits.} Journal of Differential Equations, {\bf 5,} 1 (1969) 136--158.

\bibitem{CM} V. Coti Zelati and M. Macri. {\it Multibump solutions homoclinic to periodic orbits of large energy in a centre manifold.} Nonlinearity {\bf 18}, 6 (2005), 2409.

\bibitem{CM2} V. Coti Zelati and M. Macri. {\it Existence of homoclinic solutions to periodic orbits in a center manifold.} Journal of Differential Equations {\bf 202}, 1 (2004), 158--182.

\bibitem{del} D. Delatte. {\it On normal forms in Hamiltonian dynamics, a new approach to some convergence questions}. Erg. Th. Dyn. Sys., {\bf 15}, (1995), 49--66.

\bibitem{eli}Y. Eliashberg. {\it Contact 3-manifolds twenty years since J. Martinet's work.} Ann. Inst. Fourier (Grenoble) {\bf 42} (1992), 165--192.


\bibitem{FS}J. W. Fish, R. Siefring. {\it Connected Sums and Finite Energy Foliations I: Contact connected sums}, preprint arXiv:1311.4221 (2013).




\bibitem{ghomi} M. Ghomi. {\it Strictly convex submanifolds and hypersurfaces of positive curvature.} Journal of Differential Geometry, {\bf 57}, 2 (2001), 239--271.




\bibitem{rag} C. Grotta-Ragazzo. {\it Irregular dynamics and homoclinic orbits to Hamiltonian saddle centers.} Communications on pure and applied mathematics {\bf 50,} 2 (1997) 105--147.

\bibitem{GS1} C. Grotta-Ragazzo  and Pedro A. S. Salom\~ao. {\it The Conley-Zehnder index and the saddle-center equilibrium.} Journal of Differential Equations {\bf 220}, 1 (2006), 259--278.

\bibitem{GS2} C. Grotta-Ragazzo and Pedro A. S. Salom\~ao. {\it Global surfaces of section in non-regular convex energy levels of Hamiltonian systems.} Mathematische Zeitschrift {\bf 255}, 2 (2007), 323--334.

\bibitem{93}H. Hofer. \textit{Pseudoholomorphic curves in symplectisations with applications to the Weinstein conjecture in dimension three.} Invent. Math. {\bf 114} (1993), 515--563.

\bibitem{char1}H. Hofer, K. Wysocki and E. Zehnder. \textit{A characterization of the tight three sphere.} Duke Math. J. {\bf 81} (1995), no. 1, 159--226.

\bibitem{char2}H. Hofer, K. Wysocki and E. Zehnder. \textit{A characterization of the tight three sphere II.} Commun. Pure Appl. Math. {\bf 55} (1999), no. 9, 1139--1177.

\bibitem{props1}H. Hofer, K. Wysocki and E. Zehnder. \textit{Properties of pseudoholomorphic curves in symplectisations I: Asymptotics.} Ann. Inst. H. Poincar\'e Anal. Non Lin\'eaire {\bf 13} (1996), 337-379.

\bibitem{props2}H. Hofer, K. Wysocki and E. Zehnder. \textit{Properties of pseudoholomorphic curves in symplectisations II: Embedding controls and algebraic invariants.} Geom. Funct. Anal. {\bf 5} (1995), no. 2, 270--328.

\bibitem{props3}H. Hofer, K. Wysocki and E. Zehnder. \textit{Properties of pseudoholomorphic curves in symplectizations III: Fredholm theory.} Topics in nonlinear analysis, Birkh\"auser, Basel, (1999), 381--475.

\bibitem{convex}H. Hofer, K. Wysocki and E. Zehnder. \textit{The dynamics of strictly convex energy surfaces in $\R^4$.} Ann. of Math. {\bf 148} (1998), 197-289.

\bibitem{fols}H. Hofer, K. Wysocki and E. Zehnder. \textit{Finite energy foliations of tight three-spheres and Hamiltonian dynamics.} Ann. Math {\bf 157} (2003), 125--255.


\bibitem{HMO} P. Holmes, A. Mielke  and O. O'Reilly. {\it Cascades of homoclinic orbits to, and chaos near, a Hamiltonian saddle-center.} Journal of Dynamics and Differential Equations {\bf 4,} 1 (1992) 95--126.


\bibitem{hryn}U. Hryniewicz. \textit{Fast finite-energy planes in symplectizations and applications.} Trans. Amer. Math. Soc. {\bf 364}, (2012), 1859--1931.


\bibitem{hryn2}U. Hryniewicz. {\it Systems of global surfaces of section for dynamically convex Reeb flows on the $3$-sphere,}  Journal of Symplectic Geometry {\bf 12}, 4, (2014).


\bibitem{HLS} U. Hryniewicz, J. Licata and Pedro A. S. Salom\~ao. {\it A dynamical characterization of universally tight lens spaces}, Proc. London Math. Soc.  {\bf 110}, 1, (2015), 213--269.


\bibitem{HMS} U. Hryniewicz, A. Momin and Pedro A. S. Salom\~ao. {\it A Poincar\'e-Birkhoff theorem for tight Reeb flows on $S^3$,} Inventiones Mathematicae {\bf 199,} 2 (2015), 333--422.


\bibitem{HS1} U. Hryniewicz and Pedro A. S. Salom\~ao. {\it On the existence of disk-like global sections for Reeb flows on the tight $3$-sphere.} Duke Mathematical Journal {\bf 160}, 3 (2011), 415--465.


\bibitem{HS2} U. Hryniewicz and Pedro A. S. Salom\~ao. {\it Elliptic bindings for dynamically convex Reeb flows on the real projective three-space.} preprint arXiv:1505.02713 (2015).

\bibitem{HS3} U. Hryniewicz and Pedro A. S. Salom\~ao. {\it Global properties of tight Reeb flows with applications to Finsler geodesic flows on $S^2$,} Mathematical Proceedings of the Cambridge Philosophical Society, Cambridge University Press, {\bf 154}, (2013) 1--27.

\bibitem{LK} L. M. Lerman and O. Yu. Koltsova. {\it Periodic and homoclinic orbits in two-parameter unfolding of a Hamiltonian system with a homoclinic orbit to a saddle-center.} Intern. Jour. Biff. Chaos {\bf 5}, (1995), 397--408.

\bibitem{dusa} D. McDuff. \textit{The local behavior of $J$-holomorphic curves in almost complex $4$-manifolds}, J. Diff. Geom. {\bf 34}, (1991), 143--164.

\bibitem{McG} R. McGehee. {\it Some homoclinic orbits for the restricted three-body problem}, University of Wisconsin--Madison, 1969.

\bibitem{MHO} A. Mielke, P. Holmes, O. O'Reilly. {\it Cascades of homoclinic orbits to, and chaos near, a Hamiltonian saddle center.} J. Dyn. Diff. Eqns. {\bf 4}, (1992), 95--126.

\bibitem{moser} J. Moser. {\it On the generalization of a theorem of A. Liapounoff.} Communications on Pure and Applied Mathematics, {\bf 11}, 2 (1958), 257--271.

\bibitem{Russ}  H. R\"{u}ssmann,  \textit{\"{U}ber das
Verhalten analytischer Hamiltonscher differentialgleichungen in der N\"{a}he
einer Gleichgewichtsl\"{o}sung}, Math. Ann., Princeton
University Press, {\bf 154}, (1964), 285--300.

\bibitem{Sa1} Pedro A. S. Salom\~ao. {\it Convex energy levels of Hamiltonian systems,} Qualitative Theory of Dynamical Systems {\bf 4}, 2, (2004), 439--454.

\bibitem{Si1}R. Siefring. \textit{Relative asymptotic behavior of pseudoholomorphic half-cylinders.} Comm. Pure Appl. Math., {\bf 61}(12):1631--1684, 2008.

\bibitem{Si2} R. Siefring. \textit{Intersection theory of punctured pseudoholomorphic curves .}  Geometry \& Topology, {\bf 15} (2011),  2351--2457.


\end{thebibliography}
\end{document}